\documentclass[a4paper,reqno]{book}

 \usepackage[english]{babel}
 \usepackage{amsmath,amsfonts,amssymb,amsthm, mathtools}
 \usepackage{graphicx} 
 \usepackage{subcaption,paralist}
 \usepackage[numbers]{natbib}
 
    	\usepackage{hyperref}
	\usepackage{xcolor}
	\hypersetup{
    	colorlinks,
    	linkcolor={red!80!black},
    	citecolor={blue!70!black},
    	urlcolor={blue!90!black}
		}
    \usepackage{authblk}

\newtheorem{theorem}{Theorem}[chapter]
\newtheorem{corollary}[theorem]{Corollary}
\newtheorem{lemma}[theorem]{Lemma}
\newtheorem{proposition}[theorem]{Proposition}

\theoremstyle{definition}
\newtheorem{definition}[theorem]{Definition}
\newtheorem{remark}[theorem]{Remark}

\numberwithin{equation}{chapter}

\renewcommand{\leq}{\leqslant}
\renewcommand{\le}{\leqslant}
\renewcommand{\geq}{\geqslant}
\renewcommand{\ge}{\geqslant}

\def\N{\mathbb{N}}

\def\R{\mathbb{R}}

\usepackage{natbib}
\usepackage{chapterbib}

\title{Civil Wars: \\ A New Lotka-Volterra \\ Competitive System\\
and Analysis of Winning Strategies}

\author{Elisa Affili, Serena Dipierro, \\
Luca Rossi, and Enrico Valdinoci}

\begin{document}

\maketitle

\chapter*{Preface: this monograph at a glance}
In this monograph, we introduce a new model in population dynamics that describes
two species, or communities, sharing the same environmental resources in a situation of open hostility. 

Though the main methodology fits into the broad realm of mathematical biology, relying on methods from dynamical systems, ordinary differential equations, optimization and optimal control,
the model proposed and the results presented here are completely new.

\medskip

The model is deduced from basic principles, accounting for
competition and hostility between two species sharing the same environment. The fact of sharing this environment constitutes
one of the salient features of the model in the description of a situation that, for human populations, is typical of civil wars. The model is also adapted to describe economic situations where 
two companies compete for the same market.

Our key assumption is that one of the two populations deliberately seeks for hostility through ``targeted attacks''. 
Hence the interaction is described not in terms
of random encounters but rather via the
strategic decisions of one population
that can attack the other according to different levels of aggressiveness.

This leads to a non-variational model for the two populations in conflict, taking into
account structural parameters such as the relative fit of the two populations with
respect to the available resources and the effectiveness of the attack strikes of the
aggressive population.
One of the features that distinguishes this model from usual competitive systems is that 
it allows one of the population to go extinct {\em in finite time}.

The analysis that we perform focuses
on the dynamical properties of the system, by detecting
and describing all  possible equilibria and their basins of attraction.
Moreover, we analyze the strategies that may lead to the victory of the aggressive
population, i.e.~the choice of the aggressiveness parameter,
in dependence of the structural constants of the system and possibly varying in time
in order to make the attacks effective, which take to the extinction in finite time
of the defensive population. 

{F}rom the technical point of view, analyzing the case of two populations allows us to exploit techniques typical of two-dimensional dynamical systems, in which trajectories tend to separate different regions and in which asymptotic behaviors are regulated by the Poincar\'e-Bendixson Theorem (in particular,
the system does not exhibit strange attractors).
The analysis of the linear stability of the system, combined with
a careful detection of stable and unstable orbits, also allows us to fully characterize the behavior of the equilibria, in dependence of
the parameters involved in the description of the model (such as the fitness of the populations to the environment, the level of aggressiveness of the hostile population, and the effectiveness and cost of the attacks).

A number of bespoke analytical arguments are utilized to detect
and identify the set of points for which finite time extinction is possible. 
In particular, different constraints on the aggressiveness parameter lead to different sets of initial points for which the victory is possible, highlighting that bang-bang strategies are necessary and sufficient to perform the task. 
We also analyze time minimizing strategies, for which variational tools such as the Pontryagin's Maximum Principle
play a role in the optimality conditions.

\medskip
Besides its mathematical interest, we think that the subject of this monograph is also topical and of great impact. 
Indeed, during the last century, the average duration of civil wars has significantly grown: since the end of World War II, this average duration has risen
from about one-and-a-half-year to over four years (see~\cite{NEV}).
This increased average length of civil wars has also resulted in an increased number of wars ongoing at any one time, thus contributing to a rise of tensions and potentially conflict situations, which may also contribute to the surge of new civil wars
in a spiral of concurrent effects (for instance,
about twenty contemporary civil wars took place
close to the end of the Cold War).

Given the localized structure of these conflicts,
civil wars often entail a  large numbers of casualties, also among civilians (it is estimated that civil wars have caused the deaths of over 25 million people since 1945). 

The intensity of the conflict and the high number of collateral damages typically cause a severe consumption of significant resources, a rise of speculative financial operations, and a radical pauperization of the territory.

Causes for civil wars are varifold, including
ethnic and religious fractionalizations, poverty and social inequalities,
governance and political issues. Classical analysis of civil wars
focused on ``greed versus grievance'' as the two baseline arguments as causes of civil wars, where
``greed'' collects economic motivations making the
best interests for individuals to join a rebellion 
and ``grievance'' is shorthand for all issues of ethnic,
religious, and social tensions that contribute to the
development of a conflict (see~\cite{GRGG}).

Various causes to start and prolong civil wars have also been detected (see~\cite{KEEN}), such as the possibility for elite groups
to control economic resources and power positions, or
to obtain private profit by mobilizing violent riots.

Among the myriad of factors that may stimulate or consolidate civil wars, we also mention the lack of accountability by political leaders
(see~\cite{fe0pkjfeSmb}) and
the size of a country's population (see~\cite{47852tjg-156yhng}).

A strategical configuration of the territory can also
be a reinforcing factor for a civil war: for instance,
high levels of population dispersion 
and the presence of mountainous terrain can favor riots, making the population harder to control (see~\cite{47852tjg-156yhng}).
\medskip

Humans are not the only population that exhibit organized aggressive behavior.  
This tendency is common to several primate groups but also to evolutionarily very different species. 
In his work devoted to the study of ants, Wilson states:
``The foreign policy aim of ants can be summed up as follows: restless aggression, territorial conquest, and genocidal annihilation of neighboring colonies whenever possible'' (see \cite{holldobler1990ants}).
Our model shows that this effect is not negligible in population dynamics models and opens the way for new studies in mathematics for ecology.

Besides the specific application in the study of aggressiveness phenomena in population dynamics, we stress that the model that we present here is also well-suited
to describe confrontation between rival players or agents in various different contexts,
such as strategic games or marketing models.
For instance, starting from the Bass diffusion model \cite{bass1969new}, which describes the number of buyers or adopters 
of a product, and considering two products competing for the same market,
leads to a Lotka–Volterra competitive system, see \cite{KrishnanBass,Bonaldo}. Our model then describes the situation where 
the two products are fabricated by distinct companies, and one of the two companies
resorts to aggressive policies (such as misleading
advertising, or releasing computer viruses) to set the rival product out of the
market.

\newpage$\,$\vfill

{\bf Keywords:} Dynamics of populations, biological mathematics, models for competing species, conflicts, controllability, reachable sets.\medskip\medskip

{\bf 2010 Mathematics Subject Classification:} 92D25, 37N25, 92B05, 34A26, 93B03.

\chapter*{Acknowledgments}

It is a pleasure
to thank Emmanuel Tr\'elat
for very interesting discussions.

Elisa Affili and Luca Rossi have been supported by the Italian INDAM-GNAMPA.

Serena Dipierro has been supported by
the Australian Research Council DECRA DE180100957
{\em PDEs, free boundaries and applications}. 

Luca Rossi has been supported by
the French project 
ANR-23-CE40-0023-01
{\em ``{\sc ReaCh}'' Réaction-diffusion: nouveaux défis}.

Enrico Valdinoci has been supported by
the Australian Laureate Fellowship
FL190100081
{\em Minimal surfaces, free boundaries and partial differential equations}.

\tableofcontents

\chapter{Introduction}\label{ONE:C}

\begin{center}
\begin{minipage}{25em}
\noindent{\bf Abstract of Chapter~\ref{ONE:C}.}
{\sl In this chapter, we describe the content of this book, in which we will introduce a new mathematical model to analyze the situation of two biological populations competing for the same resources, in a mutual conflict caused by
an aggressive population which, depending on the parameters of the system, may attack the other.
We employ ``civil war'' as a stylised expression that we employ to resume the two main features of our model: 
aggressive behaviour from one hand, and competition for the resources from the other.

The main questions that
we deal with in this monograph are
the characterization of the initial conditions for which there exists a winning strategy,
the success of the constant strategies, compared to all  possible strategies,
the construction of a winning strategy for a given initial datum, and the
existence of a single winning strategy independently of the initial datum.

Our analysis will characterize the equilibria of the system and their features 
in terms of the different parameters
of the model (such as relative fitness to the environment, aggressiveness
and effectiveness of strikes). Moreover, we will study the initial configurations which
may lead to the victory of the aggressive population, discussing the possible strategies to achieve the victory.

The analysis of the optimal strategies is complex, also because
there exists initial configurations
for which the aggressive population cannot obtain the victory of the war, no matter what strategy is adopted. 

Different scenarios will arise in dependence of the different parameters of the system.
For example, for populations with the same fit to the environment,
the constant strategies suffice for the aggressive population to possibly
achieve the victory, but
for populations with different fit to the environment
the constant strategies do not exhaust all the possible winning strategies,
but it is still enough to consider strategies with at most one jump
discontinuity.

The utility of the strategy is also subject to different possible parameters, such as the duration of the war. For example, we show that among all the winning strategies,
jump discontinuous strategies may not be optimal.}\end{minipage}\end{center}
\bigskip\bigskip\bigskip\bigskip\bigskip\bigskip

Mathematical biology is a traditional field of investigation that bridges together different branches of pure and applied mathematics in strong connection with several disciplines in biology, such as ethology, behavioral ecology, cytology, evolutionary biology, cancer modeling, neuroscience, etc. (see~\cite{murray1, murray2}).

The birth of mathematical biology dates back to the 13th century, when the famous Fibonacci sequence was introduced to calculate the growth of rabbit populations. The tradition of mathematical biology consolidated in the 18th and 19th centuries, also due to the works of Thomas Malthus, who described the population growth in terms of exponential functions, and Pierre Fran\c{c}ois Verhulst, who introduced the mathematical notion of competition for resources and formulated the logistic equation (see e.g.~\cite{CRAMER2004613} for an overview of the history of logistic models).

Since then, one of the traditional domains of mathematical biology has focused on population dynamics. This field of research typically leverages methods from differential equations and dynamical systems to understand the size and features of biological populations.

The effectiveness of population dynamics has very often reached further out from its original targets and has provided extremely solid links with other branches of science and mathematics. For instance, the methods developed for problems in population dynamics frequently find applications in epidemiology (see e.g.~\cite{Saito2020.06.25.20139865} for recent applications
of logistic models to concrete descriptions of the COVID-19 pandemic). The nowadays very popular family of ``SIR'' models introduced by Kermack and McKendrick in epidemiology is a particular instance of population dynamics systems in which the population is divided into compartments: Susceptibles, Infected, Removed (see~\cite{KMKSIR}).

Moreover, the questions posed in the setting of population dynamics often require the involvement of different mathematical specializations, such as game theory, control theory, optimization, etc.
(see e.g.~\cite{MR2270822, MR1747250, MR2605611, MR3408563}).

The development of population dynamics has also deeply impacted social sciences and anthropology, providing quantitative settings to describe complicated social interactions, such as crime occurrences, migrations, foundations of political parties, voting processes and circulation of ideas
(see~\cite{MR3163243, MR3488848, MR4231145}
and the references therein).

Finally, ideas from population dynamics have also been applied in the context of economics, predicting the phenomena of adoption of a new product or technology (the Bass model, see e.g.~\cite{bass1969new, RePEc:inm:ormnsc:v:31:y:1985:i:12:p:1569-1585}), as well as knowledge diffusion in macroeconomics, and competition between old technologies ad newer ones (Lotka-Volterra, see e.g.~\cite{WATANABE2003437, MR3412958}).

\section{Themes and aims of this work}

Among the several models dealing with the
dynamics of biological systems, the case of populations engaging into a mutual conflict
seems to be unexplored.
This work aims at laying the foundations of a new model describing
two populations competing for the same resources with one aggressive population
which may attack the other:
concretely, one may think of
a situation in which
two populations live together in the same territory and share the same
environmental resources,
till one population wants to prevail and try to overwhelm the other one.

We consider this situation as a ``civil war'', since the two populations share
land and resources; the two populations may be equally fit to the environment
(and, in this sense, they are ``indistinguishable'', up to the aggressive attitude of
one of the populations), or they can have a different compatibility to the resources
(in which case one may think that the conflict could be motivated by the different
accessibility to environmental resources).

Given the lack of reliable data related to civil wars, a foundation of
a solid
mathematical theory for this type of conflicts may only leverage the deduction
of the model from first principles: we follow this approach to obtain
the description of the problem in terms of a system of two
ordinary differential equations, each describing the evolution in time
of the density
of one of the two populations.

The method of analysis that we adopt is a combination
of techniques from different fields, including ordinary differential equations,
dynamical systems and optimal control.

This viewpoint allows us to rigorously investigate the model,
with a special focus on a number of mathematical features of
concrete interest, such as the possible extinction of one of the two populations
and the analysis of the strategies that lead to the victory of the aggressive population.

In particular, we analyze the {\em dynamics of the system},
characterizing the equilibria and their features (including possible basins of attraction)
in terms of the different parameters
of the model (such as relative fitness to the environment, aggressiveness
and effectiveness of strikes). Moreover, we study the initial configurations which
may lead to the victory of the aggressive population, also taking into account
different possible {\em strategies} to achieve the victory: roughly speaking,
we suppose that the aggressive population may adjust the parameter
describing the aggressiveness in order to either dim or 
exacerbate the conflict with the aim of destroying the second population
(of course, the war has a cost in terms of life
for both the populations,
hence the aggressive population must select the appropriate strategy in terms
of the structural parameters of the system). We show that the initial data
allowing the victory of the aggressive population
does not exhaust the all space, namely {\em there exist initial configurations
for which the aggressive population cannot make the other extinct}, regardless of the strategy adopted during the conflict. 

Furthermore, {\em for identical populations with the same fit to the environment,
the constant strategies suffice} for the aggressive population to possibly
achieve the victory: namely, if an initial configuration admits a piecewise continuous in time
strategy that leads to the victory of the aggressive population,
then it also admits a constant in time strategy that reaches the same objective
(and of course, for the aggressive population, the possibility of focusing
only on constant strategies would entail concrete practical advantages).

Conversely, {\em for populations with different fit to the environment
the constant strategies do not exhaust all the winning strategies}:
that is, in this case, there are initial conditions which allow
the victory of the aggressive population only under the exploitation
of a strategy that is not constant in time.

In any case, we also prove that {\em strategies with at most one jump
discontinuity are sufficient} for the aggressive population:
namely, independently from the relative fit to the environment,
if an initial condition allows the aggressive population to reach the victory
through a piecewise continuous in time
strategy, then the same goal can be reached using a ``bang-bang''
strategy with at most one jump.

We also discuss the {\em winning strategies that minimize
the duration of the war}: in this case, we will show that
jump discontinuous strategies may be not sufficient and interpolating
arcs have to be taken into account.

\section{Disclaimer}

In no way do the authors of this book suggest that the model has implications of military or sociological type. This model is designed to expand the family of Lotka-Volterra systems by introducing a novel element, justifiable as an aggression term within the diverse interpretations of Lotka-Volterra systems -- whether viewed through the lens of population dynamics or as models of business competition. Conversely, the incorporation of this new term significantly alters the system's behavior, thereby raising questions distinct from those prevalent in existing literature, in particular from a control theory point of view. We suggest that our approach could be useful in other contexts. Through this book, our intention is to inspire and furnish an accessible resource for researchers possessing a robust mathematical background, particularly those intrigued by the development of competitive systems and models necessitating investigations into controllability.

\section{Organization of this monograph}

In Chapter~\ref{chapt:model} we describe in detail the model starting from prime principles and in
Chapter~\ref{chapt:result} we present the main results of this monograph.

After having clarified the main notation used throughout this monograph in Chapter~\ref{TOOLB},
in Chapter~\ref{IKJM:plrg777}
we will exploit methods from ordinary differential equations and dynamical systems
to describe the equilibria of the system and their possible basins of attraction.
The dependence of the dynamics on the structural parameters,
such as fit to the environment, aggressiveness, and efficacy of attacks,
is discussed in detail in Chapter~\ref{ss:dependence}.

Chapter~\ref{STRATE} is devoted to the analysis of the strategies
that allow the first population to eradicate the second one (this part needs
an original combination of methods from dynamical systems and optimal control theory).

We conclude our work with a brief chapter offering guidance on classical bibliography. This section aims to assist readers seeking background on the theories we employ.

\chapter{Description of the model}\label{chapt:model}

\begin{center}
\begin{minipage}{25em}
\noindent{\bf Abstract of Chapter~\ref{chapt:model}.} {\sl
Here we will introduce our new model of civil war starting from first principles.
The classical logistic description of two competing populations will be complemented by terms modeling a conflict originating from the aggressiveness of a population
and accounting for the death rate caused by intentional strikes.}\end{minipage}\end{center}\bigskip\bigskip\bigskip\bigskip\bigskip\bigskip

We now describe in detail our model of conflict between the
two populations and the attack strategies pursued by the aggressive population.
We assume that two populations 
compete for the same resources; this leads to the  
standard competitive Lotka-Volterra system for their densities~$u$ and~$v$, as introduced\footnote{This model
was originally designed to describe
a predator-prey system. Given its broad flexibility,
it is often regarded as a paradigmatic model for competition.} in~\cite{LOTK, zbMATH02586915}, see
also~\cite{ lotka,volterra}.

We then incorporate the fact that one population --the one with density $u$--
deliberately attacks the other. As a result, both populations 
suffer some losses. 

The key point in our analysis
is that the clashes do not depend on the chance of meeting between
 the two populations, given by the quantity~$uv$, as it happens in many other works in the literature  (starting from the publications of Lotka and Volterra,~\cite{lotka,volterra}), but they are sought by the first population and
only depend  on the size~$u$ of the first population and on its level of aggressiveness~$a$ (or the portion of the population devoted to the attacks).

The resulting model is
\begin{equation}\label{model}
	\left\{
	\begin{array}{llr}
		\dot{u}&= u(1-u-v) - acu, & {\mbox{ for }}t>0,\\
		\dot{v}&= \rho v(1-u-v) -au, & {\mbox{ for }}t>0,
	\end{array}
	\right.
\end{equation}
where~$a$,~$c$ and~$\rho$ are positive real numbers. Here, the coefficient~$\rho$ models the second population's fitness with respect to the first one when resources are abundant for both; it is linked with the exponential growth rate of the two species. The parameter~$c$ stands for the quotient of endured per inflicted  damages for the first population. 
Deeper justifications to the model~\eqref{model} will be given in
Section~\ref{ss:derivation}.
The complete description of the trajectories of the dynamical system~\eqref{model} is presented in Section~\ref{ss:notation}.

Notice that the size of the second population~$v$ may become negative in finite time while the first population is still alive. The situation where~$v=0$ and~$u>0$ represents the extinction of the second population and the victory of the first one. 

To describe our results, for communication convenience (and in spite of our
personal fully pacifist beliefs)
we take the perspective of the first population,
that is, the aggressive one; the objective
of this population is to overwhelm the other one, and, to achieve that, it can influence the system by tuning the parameter~$a$. 

{F}rom now on, we may refer to the parameter~$a$ as the \textit{strategy}, which may also depend on time, and we will say that it is \textit{winning} if it leads to the victory of the first population. 

The main questions that
we deal with in this monograph are:
\begin{enumerate}
	\item The characterization of the {\em initial conditions for which there exists a winning strategy}.
	\item The {\em success of the constant strategies}, compared to all  possible strategies.
	\item The {\em construction of a winning strategy} for a given initial datum.
	\item The {\em existence of a single winning strategy independently of the initial datum}.
\end{enumerate}

We discuss all these topics in Section~\ref{ss:strategy}, presenting
concrete answers to each of these problems.
 
Also, since to our knowledge, this is the first time that system~\eqref{model} is considered,
in Sections~\ref{ss:notation} and~\ref{ss:dynamics} we discuss the dynamics and some exciting results about the dependence of the basins of attraction on the other parameters. 

It would also be extremely interesting to add the space component to our model, by considering a system of reaction-diffusion equations. This will be the subject of further work.

\section{Motivations} \label{ss:derivation}

The classic Lotka-Volterra equations for modeling predator-prey systems were first introduced independently in \cite{1eddd3e8-a442-3aa5-91bf-ea3ab4ee18ac} and \cite{volterra1926variatzioni}. 
Later, the models were extended to other types of interaction between two populations, including competition (see~\cite{volterra}), and to model other phenomena involving competition, for example in technology substitution~\cite{substitution}.  
The competitive Lotka-Volterra system concerns the sizes~$u_1(t)$ and~$u_2(t)$ of two species competing for the same resources. The system that
the couple~$(u_1(t), u_2(t))$ solves is
 \begin{equation}\label{lv}
 \begin{cases}
 \dot{u}_1=r_1  u_1\left(\sigma-\displaystyle
  \frac{u_1+\alpha_{12} u_2}{k_1}  \right), & t>0,\\
 \dot{u}_2
=r_2 u_2\left(\sigma- \displaystyle\frac{u_2+\alpha_{21} u_1}{k_2}  \right), & t>0,
 \end{cases}
 \end{equation}
where~$r_1$,~$r_2$,~$\sigma$,~$\alpha_{12}$,~$\alpha_{21}$,~$k_1$ and~$k_2$ are nonnegative real numbers.

Here, the coefficients~$\alpha_{12}$ and~$\alpha_{21}$ represent the competition between individuals of different species, and indeed they
appear multiplied by the term~$u_1 u_2$, which represents a probability of meeting. 

The coefficient~$r_i$ is the exponential growth rate of the~$i-$th population, that is, the reproduction rate that is observed when the resources are abundant. 
The parameters~$k_i$ are called carrying capacity and represent the number of individuals of the~$i-$th population that can be fed with the resources of the territory, that are quantified by~$\sigma$. 
It is however usual to renormalize the system in order to reduce the number of parameters. 
In general,~$u_1$ and~$u_2$ are normalized so that they vary in the interval~$[0,1]$,
thus describing densities of populations.
 
The behavior of the system depends substantially on the values of~$\alpha_{12}\frac{k_2}{k_1}$ and~$\alpha_{21}\frac{k_1}{k_2}$ with respect to the threshold value~$1$, or simply on $\alpha_{12}$ and $\alpha_{21}$ if $k_1=k_2$  (see e.g.~\cite{nonlinear}). In this latter case, if~$\alpha_{12}<1<\alpha_{21}$, then the first species~$u_1$
has an advantage over the second one~$u_2$
and will eventually prevail;
if~$\alpha_{12}$ and~$\alpha_{21}$ are both strictly above~$1$, 
	then the first population that penetrates the environment (that is, the one that has a greater size at the initial time) will persist while the other will extinguish;
	 if instead~$\alpha_{12}$ and~$\alpha_{21}$ are both equal or below~$1$, then an attractive coexistence equilibrium appears.

Some modification of the Lotka-Volterra model were made in stochastic analysis by adding a noise term of the form~$-f(t)u_i$ in the~$i-$th equation, finding some emerging phenomena of phase transition,
see e.g.~\cite{noise}. 

The ODE system~\eqref{lv} has been extended to study the case of two competitive populations that diffuse in space. Many different types of diffusion have been compared and one can find a huge literature on the topic, see ~\cite{mimura, crooks, massaccesi} for some examples and~\cite{murray2} for a more general overview. 
We point out that other dynamical systems exhibiting finite time extinction of one or more species living in some heterogeneous environments
have been considered in the literature, see for example  the model in~\cite{gaucel} for the predator-prey behavior of cats and birds, that has been thereafter widely studied. 
\medskip

In this monograph,
we focus not only on {\em
basic competition for resources}, but also on {\em
situations of open hostility}.
In social sciences, war models are in general little studied; indeed, the collection of data up to modern times is hard for the lack of reliable sources. 
Also, there is still much discussion about what factors are involved and how to quantify them: in general, the outcome of a war does not only depend on the availability of resources, but also on more subtle factors as the commitment of the population and the knowledge of the battlefield,
see e.g.~\cite{toft2005state}. 
Instead, the causes of war were investigated by the statistician L.F. Richardson, who proposed some models for predicting the beginning of a conflict,
see~\cite{richardson1960arms}. 

In addition to the human populations, behavior of hostility between groups of the same species has been observed in chimpanzee. Other species with complex social behaviors are able to coordinate attacks against groups of different species: ants versus termites, agouti versus snakes, small birds versus hawk and owls, see e.g.~\cite{animalwar}.  

The model that
we present here
is clearly a simplification of reality. Nevertheless,
we try to capture some important features of conflicts
between rational and strategic populations, 
introducing in the
mathematical modeling the idea that a conflict may be sought
and the parameters that influence its development may be
conveniently adjusted. 

Specifically, in our model, the
interactions between populations are not merely driven
by chance but rather the strategic decisions of the population
play a crucial role in the final outcome of the conflict, and
we consider this perspective as an interesting novelty in the mathematical
description of competitive environments.

At a technical level, our aim is to introduce a model for conflict between
two populations~$u$ and~$v$,
starting from the model when the two populations compete for food and modifying it to add the information about the clashes. 
We imagine that
each individual of the first population~$u$ decides to attack an individual of the second population with some probability~$a$ in a given period of time. 
As an outcome, the  individual of the first population has a probability~$\zeta_u$ of being killed and a probability~$\zeta_v$ of killing one opponent.
One may think that hostilities take the form of ``duels'', that is, one-to-one fights, whose single outcome does not 
depend on the total number of individuals of the populations. 
Notice that in some duel the fighters might  both be killed. 
Thus, after one time-period, the casualties for the first and second populations
are~$a\zeta_u u$ and~$a\zeta_v u$
respectively.
The same conclusions are found if we imagine that the first population forms an army to attack the second, which tries to resist by recruting an army of proportional size. At the end of each battle, a ratio of the total soldiers is dead, and this is again of the form~$a\zeta_u u$ for the first population and~$a\zeta_v u$ for the second one.

Another effect that
we want to take into account is the drop in  the fertility of the population  during wars. 
This seems due to the fact that families suffer some income loss during war time, because of a lowering of the average productivity and lacking salaries only partially compensated by the state; another reason possibly discouraging couples to have children is the increased chance of death of the parents during war. 
As pointed out in~\cite{fertility}, in some cases the number of lost births during wars are comparable to the number of casualties. 
However, it is not reasonable to think that this information should be included in the exponential growth rates~$r_u$  and~$r_v$, because the fertility drop really depends on the intensity of the war. For this reason, 
we introduce a population loss rate for $u$ and $v$ given by $c_u a u$ and $c_v a u$ respectively,
where $c_u\geq 0$ and~$c_v\geq 0$ ar given parameters.

Finally, for simplicity,
we also suppose that the clashes take place apart from inhabited zone, without having influence on the harvesting of resources. 
\medskip

\section{The notion of aggressiveness}

Concering the notion of ``aggressiveness’’, we remark that obviously in our simplified model this term has merely a mathematical meaning, and not a social, psychological, or legal connotation.

{F}rom the historical point of view, the notion of ``aggression'' in relation to military actions was probably formalized for the first time in 1919, on the occasion of the Treaty of Versailles (Article 231, often referred to as the ``War Guilt Clause'', stated that ``The Allied and Associated Governments affirm and Germany accepts the responsibility of Germany and her allies for causing all the loss and damage to which the Allied and Associated Governments and their nationals have been subjected as a consequence of the war imposed upon them by the {\em aggression} of Germany and her allies'').

Similar clauses were also used in the Treaty of Saint-Germain-en-Laye (1919), in the Treaty of Neuilly (1919), in the Treaty of Trianon (1920) and in the Treaty of S\'evres (1920). Articles of this type have been used as legal bases to extract money reparations for the war's devastations and costs. Notwithstanding the existence of an International Criminal Court, the definition for a war of aggression is not univocal and it is often controversial.

The notion of ``aggressiveness’’ has also been commonly employed in connection with wars in several historical contexts. For instance, in relation to ancient empires, the word ``aggressive'' has also been very often adopted by scholars (e.g. ``Roman aggression'' \cite[page~229]{RAAF},
as well as ``...the most aggressive ancient and modern civilized states.'', see~\cite[page~33]{LAWRENCE}; ``the populous and aggressive Parthian (Persian) Empire'', see~\cite[page~76]{LAWRENCE}; also about ancient Egypt ``In the Old Kingdom, warfare was supposed to be aggressive'' \cite[page~77]{RAAF}; in relation to the Hellenistic World ``aggressive kings such as Philip and Alexander'' \cite[page~171]{RAAF};
and, with regard to Eastern imperial dynasties, ``the Japanese court faced the aggressive Tang Empire'' \cite[page~52]{RAAF}; as regards the prehistoric and pre-Columbian Mesoamerican societies,  ``the major aggressor was the Culhua-Mexica (or Aztec) empire, which rather easily overran the area in 1486 and in 1506’’, see~\cite[page~340]{PREME}; ``Aztec aggression’’, see~\cite[page~357]{PREME}; ``It is difficult to say exactly What the statuittle doubt that it was a politically weak, militarily aggressive, and probably tributary group’’, see~\cite[page~398]{PREME}; ``The Warrau appear to have been pushed into very marginal swamp areas by their notoriously expansive and aggressive neighbors, the Caribs and Arawaks'' \cite[page~203]{CONST}).

In this context, there are also classical examples of ``peaceful’’ populations (such as ``the certain peaceful Inuit groups, the Semai, and the La Paz Zapotec, as well as on several other peaceful cultures'', see~\cite[page~721]{LAPA}).

However, aggressiveness and peacefulness may vary within different communities of the same population (e.g., ``it would be incorrect to generalize that Zapotec culture overall has a low level of aggression based solely on data from the peaceful La Paz community, or conversely, to generalize on the basis of fieldwork in a different Zapotec community than La Paz that all Zapotec communities are violent. Jean Briggs voices a similar caution that not all Inuit bands are as peaceful as the groups she describes'', see~\cite[page~727]{LAPA}).

Also in cases of Indian massacres in North America, the role of the ``aggressive'' or ``colonialistic'' population has also often emerged quite clearly (in fact, scholars speak about ``colonial aggression'', see e.g.~\cite[pages~133 and~247]{ALFRED}), though we have also occurrences of aggressive behaviors of indigenous populations (see~\cite[pages~128--129]{LAWRENCE} in the context of ``Apache-Navajo aggressiveness'' and ``the aggressive Mohave'', see also~\cite[page~58]{CONST} for ``The Comanche of the Southern Plains and the Yanomama have been described as particularly aggressive'').

The distinction between aggressive and peaceful populations is also a notion adopted by scholars (e.g., ``the aggressive groups acquired territory at the expense of more passive ones'', see~\cite[page~129]{LAWRENCE}; ``regions and periods of frequent bitter warfare are often centered on especially aggressive societies that spoil their neighborhood'', see~\cite[page~177]{LAWRENCE}; ``the Semai [...] tradition of flight from violence is a consequence of countless defeats and slave raiding at the hands of the more numerous and aggressive Malays. In other words, the Sexnai can be characterized as defeated refugees'', see~\cite[page~206]{LAWRENCE}).

Obviosuly, aggressive treats may also change in time and according to circumstances (e.g. ``The hyperaggressive Norsemen have become the pacific Scandinavians'', see~\cite[page~130]{LAWRENCE}).

The strict link between the notions of ``war'' and ``aggression'' are at the basis of~\cite{ARMS}.

Of course, the ``aggressiveness’’ of a population is sometimes highly influenced by its political leader (see e.g.~\cite[page~175]{LAWRENCE} for ``Napoleon's aggressive use of'' the French Revolution).  Moreover, in cases of national, ethnical, racial or religious genocides, the role of an ``aggressive'' population is usually very apparent.

Aggressiveness is not only found in humans but is a well-studied behavior in the animal realm. 
Konrad Lorenz, one of the founders of ethology, also wrote an influential book fully dedicated on the topic of ``aggression’’, see~\cite{KONR}. In his opinion, 
``the aggression of so many animals towards members of their
own species is in no way detrimental to the species but, on the
contrary, is essential for its preservation'' \cite{KONR}. The same concepts were also confirmed and presented by Dawkins in his influential book ``The selfish gene'' \cite{dawkins2016selfish}. Aggression is a complex behavior that can manifest in a variety of ways, including physical attacks, threats, and dominance displays.

Aggression is in fact observed in a wide range of animal species, from insects to fish, birds, and mammals \cite{davies2012introduction}.
Interspecific aggressive behavior seems to be crucial in order to defend territory, protect the offspring, and establish dominance hierarchies,  that ensure breeding rights to the triumphing male.
In ecosystems around the world, top predators kill, harass, and steal food from smaller predators. These direct, aggressive interactions, generally referred to as interference competition, are widespread and substantial, and can have profound consequences for the distributions and population dynamics of smaller predators. These patterns of suppression and coexistence vary across systems and species, see \cite{swanson2016absence}.

Aggressive business policies, including predatory pricing, excessive discounting, and exclusive dealing, are tactics employed by companies to gain or maintain a competitive advantage in the market. These practices can harm consumers and rival businesses, and they are often subject to legal scrutiny \cite{bain1956barriers}.
Many retailers use excessive discounting as a promotional tactic to attract customers during periods of low demand. However, this practice can harm smaller competitors and lead to higher prices in the long run. Some authors refer to this strategy as ``price wars'', see \cite{HEIL200183}.

Moreover, with a slight abuse of notation, it is fascinating to include in the study of ``civil wars'' possibly the most ancient human conflict on a large scale, namely the long war of attrition between the individuals of {\em Homo sapiens} and those of {\em Homo neanderthalensis}, caused by the expansion of the sapiens out of Africa about 60 or 70 thousand years ago, which led to the extinction of the Neanderthals around 40 thousand years ago, with a concrete overlap of the two species for between about 2 and 5 thousand years, see~\cite{EXMNEA}.

On the one hand, this conflict cannot be classified as a civil war in the modern sense of the term, also because neanderthalensis and sapiens are recognized as two separate species. On the other hand, the two species present a strikingly similar anatomy and share 99.7\% of DNA, and there is even
strong indication of interbreeding, see~\cite{20448178}.

Both species were certainly acquainted with war actions: for instance, signs of warfare are typically considered skull traumas and parry fractures, which seem to be especially common in young males, see~\cite{NAKAHASHI201783}.

The reasons for the supremacy of the sapiens species are still under intense debate, they may include a refined symbolic intelligence, a more articulated language, the adoption of superior ranged weapons, more advanced social systems, a more specialized division of labor, as well as possibly a more aggressive and better organized expansion of the sapiens which broke the preexisting demographic balance, see e.g.~\cite{507197, ojqld904uythgiwhgvnbvnbJKXSN}.
Let us also mention the mathematical modeling of Sapiens-Neanderthal interaction proposed 
in~\cite{flores} where the author adopt the point of view of considering pure quadratic competition between the two species and 
explicitly neglect the war perspective.
\medskip

\section{Derivation of the model}

Now we derive the system of equations from an heuristic analysis.
As in the Lotka-Volterra model, it is assumed that the change of the size of the population in an interval of time~$\Delta t$ is proportional to the size of the population~$u(t)$, that~is
\begin{equation*}
	u(t+\Delta t)-u(t) \approx u(t) f(u,v)
\end{equation*}
for some appropriate function~$f(u,v)$. In particular,~$f(u,v)$ should depend on resources that are available and reachable for the population. 
The maximum number of individuals that can be fed with all the resources of the environment is~$k$; taking into account all the individuals of the two populations, the available resources are 
\begin{equation*}	
	k-u-v.	
\end{equation*}
Notice that we suppose here that each individual consumes the same amount of resources, independently of its belonging. In our
model, this assumption
is reasonable since
all the individuals belong to the same species. 
Also, the competition for the resources only depends  on the number of individuals, independently on their identity.

Furthermore, our model is sufficiently
general to take into account the fact that the growth rate of the populations can be possibly different. In practice,
this possible difference
could be the outcome of
a cultural distinction, or it may be also due to some slight genetic differentiation, as it happened 
in the case of Homo Sapiens and Neanderthal mentioned in the previous section.

Let us call~$r_u$ and~$r_v$ the fertility of the first and second populations respectively. The contribution
to the population growth rate  is  given by
$$ f(u,v) := r_u \left(1-\frac{u+v}{k}   \right),~$$ and these effects
can be comprised in a typical Lotka-Volterra system.

Instead, in our model, we also take into
account the possible death rate due to casualties.
In this way,
we obtain a term such as~$-a\zeta_u$ to be added to~$f(u,v)$. 
The fertility losses give another term~$-ac_u$ for the first population.
We also perform the same analysis for the second population, with the appropriate coefficients.

With these considerations, the system of
the equations that we obtain is
\begin{equation}\label{model1}
\left\{
\begin{array}{llr}
\dot{u}&= r_u u\left(1- \dfrac{u+v}{k} \right) - a(c_u + \zeta_u)u, & t>0,\\
\dot{v}&=r_v v\left(1- \dfrac{v+u}{k} \right) - a(c_v + \zeta_v)u, & t>0.
\end{array}
\right.
\end{equation}
As usual in these kinds of models, we can rescale the variables and the coefficients  in order to
find an equivalent model with fewer parameters.

Hence, we perform the changes of variables
\begin{equation}\label{changeofvar}\begin{split}
&\tilde{u}(\tilde t)= \dfrac{u(t)}{k}, \quad \tilde{v}(\tilde t)=\dfrac{v(t)}{k}, 
\quad {\mbox{ where }}\quad\tilde{t}= r_u t,\\
&
\tilde{a}= \dfrac{a(c_v+\zeta_v)}{r_u}, \quad \tilde{c}= \dfrac{c_u+\zeta_u}{c_v+\zeta_v} \quad {\mbox{ and }}\quad \rho= \frac{r_v }{r_u },
\end{split}
\end{equation}
and, dropping the tildas for the sake of readability, we finally get the system
in~\eqref{model}.
We will also refer to it as the civil war model.

{F}rom the change of variables in~\eqref{changeofvar}, we notice in particular that~$a$ may now take values in~$[0,+\infty)$.

\medskip

\section{Interpretation of the model in an economic key}
The competitive Lotka-Volterra system is already used to study some market phenomena as technology substitution, see e.g.~\cite{substitution, bhargava1989generalized, watanabe2004substitution}, and our model aims at adding new features to such models.

Concretely, in the technological competition model, one can think that~$u$ and~$v$
represent the capitals of two companies, producing for instance computers, or cell phones, etc. 
In this setting, to start with, one can suppose that the first company produces a very successful product, 
say computers with a certain operating system, in an infinite market, reinvesting a proportion~$r_u$ of the profits into the production of additional items, which are purchased by the market, and so on: in this way, one obtains a linear equation of the type~$\dot u=r_u u$, with exponentially growing solutions. 
The case in which the market is not infinite, but admits a saturation income threshold~$k$, would correspond to the equation
$$\dot u=r_u u\left(1-\frac{u}{k}\right).$$ 
Then, when a second computer company comes into the business, selling computers with a different operating system to the same market, one obtains the competitive system of equations
\begin{equation}\label{Bass}
 \begin{cases}
\dot u=r_u u\displaystyle\left(1-\frac{u+v}{k}\right),\\
\dot v=r_v v\displaystyle\left(1-\frac{v+u}{k}\right).
\end{cases}
\end{equation}
At this stage, the first company may decide to use an ``aggressive'' strategy in order to harm the rival company
and set it out of the market
(for instance through the  
spreading of a virus attacking the other company's operating system, or by some marketing campaigns).
Once the competition of the second company is removed, the first company can then exploit the market in a monopolistic regime.
To model this strategy, one can suppose that the first company 
invests a proportion of its capital in the project and
diffusion of the virus, according to a quantifying parameter~$a_u\ge0$, thus producing the equation
\begin{equation}\label{MAR1} \dot u=r_u u\left(1-\frac{u+v}{k}\right)-a_u u.\end{equation}
This directly impacts the capital
of the second company proportionally to the virus
spread,
since the second company has to spend money to project and release antiviruses,
as well as to repay unsatisfied customers,
hence resulting in a second equation of the form
\begin{equation}\label{MAR2} \dot v=r_v v\left(1-\frac{v+u}{k}\right)-a_v u.\end{equation}
The case~$a_u=a_v$ would correspond to an ``even'' effect in which the costs of producing the virus is in balance with
the damages that it causes.
It is also realistic to take into account the case~$a_u<a_v$
(e.g., the first company manages to
produce and diffuse the virus at low cost,
with high impact on the functionality of the operating
system of the second company) as well as the case~$a_u>a_v$ (e.g., the cost of producing and diffusing
the virus is high with respect to the damages caused).

We remark that equations~\eqref{MAR1} and~\eqref{MAR2} can be set into the form~\eqref{model1}, thus showing the interesting versatility of our model also
in financial mathematics.

Even the original Bass model \cite{bass1969new}, introduced to describe the number of adopters 
of a durable good, can be extended to the case of two products competing for the same market,
leading to a competitive system of the Lotka–Volterra type, see e.g.~\cite{KrishnanBass,Bonaldo}.
In such a framework, if we neglect ``innovators'', new people adopt a product by imitation, 
namely, calling $u$ and $v$ the portions of adopters of the two goods,
their rate of change is proportional to the portion of people that
do not adopt neither product, i.e.~$(1-u-v)$, times the portion that already adopted the product,
i.e.~$u$ or $v$. The system then takes the form~\eqref{Bass} with $k=1$.
One can then envision the fact that one of the two companies producing the goods resorts to
some type of aggressive marketing policy in order to harm the rival.
Without describing the specific mechanisms of such a policy, we just make the general assumption that, on one hand,  
it requires some consumption of resources by the company, and, on the other hand, it
produces some damages to the rival one; the outcome is the reduction of the rates of change of adopters
of the two products, and we assume that these are proportional. Under these assumptions, we end up with 
the system~\eqref{MAR1}-\eqref{MAR2}.

Finally, it is natural to envision other contexts 
where our model  could be pertinent, such as dynamic games. 
Let us mention for instance real-time strategy computer games
that combine
resource management and war confrontation, see~\cite{starcraft}.

\bigskip

As a final comment, let us stress that certainly our model does not aim to capture all the complexity of the phenomena intertwined with civil wars, and other mathematical approaches to the problem can certainly be very beneficial for a deeper understanding of the problem.
Other possible tools of investigations naturally include (but are not limited to) kinetic models and
Boltzmann-type equations, mean-field games,
agent-based models, and active particles methods, see e.g.~\cite{MR2295621,
MR2346927, MR3204371, MR3268061,
MR3847177, MR3969953, 
AUFISH}
and the references therein.

Thus, our objective here is just to
propose a simple, stylized model to describe such a complex scenario as the interaction between rival species and communities. The outcomes of our approach are that, on one hand, it allows us for a rigorous analytic investigation of the mathematical model proposed. On the other hand, we believe that our results may capture some qualitative features of the real phenomenon under study, as we are now going to showcase in some detail.

\chapter{Statement of the main results}\label{chapt:result}

\begin{center}
\begin{minipage}{25em}
\noindent{\bf Abstract of Chapter~\ref{chapt:result}.} {\sl 
Here we describe our main results in mathematical details.
First of all, we analyze the equilibria of the system, their linear stability,
and the presence of a stable or center manifold, in dependence of the parameters. We also characterize the strategies that lead to the victory of the aggressive population.

In this analysis, the case of populations exactly with the same fitness to the environment plays a special role, since in this situation the final outcome of the war is determined solely by the initial conditions and the specific strategy followed by the aggressive population does not play a major role.

Instead, when the fitness levels of the two populations are different, complex scenarios arise and constant strategies are not sufficient to ensure victory starting from favorable conditions. The set of winning strategies can however be greatly simplified, by reducing it to the case of piecewise constant functions with at most one discontinuity.

Finally, among all the possible winning strategies, we aim at detecting the one that minimizes the length of the war. For this, constant strategies are not enough, nor piecewise constant
strategies with a jump discontinuity, and the quickest victory could be achieved through a strategy assuming some values along a singular arc.}\end{minipage}\end{center}
\bigskip\bigskip\bigskip\bigskip\bigskip\bigskip

We now describe in some technical detail the results that we obtain on the civil war model.

\section{Basic results on the dynamics}\label{ss:notation}

We denote by~$(u(t), v(t))$ a solution of~\eqref{model} starting from a point $$(u(0),v(0))\in [0,1] \times [0,1].$$

We will also refer to the \textit{orbit} of~$(u(0), v(0))$  as the following subset of $\R^2$:
$$\{(u(t), v(t))\ :\ t\in \R\},$$
thus both positive and negative times, while the \textit{trajectory} 
(referred in some texts as simply \emph{solution} or \emph{phase curve through the point} $(u(0),v(0))$ 
at time $t=0$, see \cite{smale1974differential,MR1056699}) is the set
$$\{(u(t), v(t))\ :\ t\geq0\}.$$
\medskip

As already mentioned in the discussion below formula~\eqref{model},
$v$ can reach the value~$0$ and even negative values in finite time. However, we suppose that the dynamics stops when the value~$v=0$ is reached for the first time.
At this point, the conflict ends with the victory of the first population~$u$, that can continue its evolution with a classical Lotka-Volterra equation of the form
\begin{equation*}
\dot{u}= u (1- u)
\end{equation*}
and that would certainly fall into the attractive equilibrium~$u=1$.
The only other possibility is that the solution remains in the set~$[0,1]\times(0,1]$ for all times.
Indeed, on the rest of the boundary of this square, there holds
\begin{eqnarray*} &&u(t)=0\implies \dot{u}(t)=0,\\&& u(t)=1\implies \dot{u}(t)\leq0 \\{\mbox{and }} &&
v(t)=1\implies \dot{v}(t)\leq0.\end{eqnarray*}

\begin{remark}\label{rmk:dichotomy}
	In a nutshell, for any solution with initial datum~$(u(0),v(0))\in [0,1] \times [0,1]$, one of the following 
	situations occurs:
	\begin{enumerate}
		\item[(1)] $(u(t), v(t))\in [0,1]\times [0,1]$ for all~$t\geq 0$.
		\item[(2)] There exists a unique~$T\geq0$ such that $v(T)=0$, $u(T)>0$ and~$(u(t), v(t))\in [0,1]\times (0,1]$ for all~$t\in[0,T)$.
	\end{enumerate} 
\end{remark}

Owing to this dichotomy, we define the \textit{stopping time} of the solution~$(u(t), v(t))$~as
\begin{equation}\label{def:T_s}
	T_s (u(0), v(0)) := 
	\left\{ 
	\begin{array}{ll}
		+\infty & \text{if situation (1) occurs}, \\
		T & \text{if situation (2) occurs}.
	\end{array}
	\right.
\end{equation} 
{F}rom now on, we will implicitly  consider solutions~$(u(t),v(t))$ only for~$t\in [0, T_s(u(0), v(0)) )$.

\medskip

We now specify the class of admissible strategies~$a(t)$. In view of the applications,
one is led to allow $a(t)$ to be non-constant and discontinuous, so we consider the following class of admissible
strategies:
\begin{equation}\begin{split}\label{DEFA}
 \mathcal{A}&\; :=
\big\{a: [0, +\infty) \to [0, +\infty) {\mbox{ s.t.~$a$ is continuous}}\\
&\qquad \qquad {\mbox{except at most at a finite number of points}}\big\}.\end{split}\end{equation}
A \emph{solution related to a strategy~$a(t)\in \mathcal{A}$} is a pair
$$(u(t), v(t)) \in C^0( (0,+\infty))\times 
C^0( (0,+\infty)),$$ which is~$C^1$ outside the  points of discontinuity of~$a(t)$,
and solves the system~\eqref{model} outside these points.
Moreover, once the initial datum is imposed, the solution is assumed to be
continuous up to~$t=0$. 
Existence and uniqueness of solutions in this setting is standard:
one just considers a juxtaposition of classical Cauchy problems starting at the times 
coinciding with the discontinuity points of $a$. 

\medskip

We then analyze the dynamics of~\eqref{model} with a particular emphasis on possible strategies. To do this, we consider the \textit{basin of attraction} of the equilibrium~$(0,1)$,~i.e.
\begin{equation}\label{DEFB}\begin{split}
	\mathcal{B}&\;:= \Big\{ (u(0),v(0))\in [0,1]\times[0,1] \;{\mbox{ s.t. }}\;\\
	&\qquad\qquad T_s (u(0), v(0)) = +\infty, \ (u(t),v(t)) \overset{t\to\infty}{\longrightarrow} (0,1)  \Big\},
\end{split}\end{equation}
which corresponds to the set of the initial points for which the first population gets extinct (in infinite time) and the second one survives.

Furthermore, we set
\begin{equation}\label{DEFE}\begin{split}
\mathcal{E}:=\;& \big\{ (u(0),v(0))\in [0,1]\times[0,1] \\&\qquad {\mbox{ s.t. }}\;
 T_s(u(0),v(0))< + \infty \big\},\end{split}
\end{equation}
namely
the set of initial points for which we eventually have the victory of the first population and the extinction of the second one. 

Of course, the sets~$\mathcal{B}$ and~$\mathcal{E}$ depend on the parameters~$a$,~$c$, and~$\rho$; we will express this dependence by writing~$\mathcal{B}(a,c,\rho)$
and~$\mathcal{E}(a,c,\rho)$ when it is needed, and omit it otherwise for the sake of readability. The dependence on parameters will be carefully studied
in Chapter~\ref{ss:dependence}.

\section{Constant strategies}\label{ss:dynamics}

The first step towards the understanding of the dynamics of~\eqref{model} 
consists in the analysis
of the behavior of the system for constant coefficients. 

To this end, we introduce some notation.
Following the terminology of~\cite[Section~1.1.]{MR1056699},
we say that an equilibrium point (or fixed point) of the dynamics
is a (hyperbolic) {\em sink}
if all the eigenvalues of the linearized map have strictly
negative real parts, a (hyperbolic) {\em source}
if all the eigenvalues of the linearized map have strictly
positive real parts, and a (hyperbolic) {\em saddle}
if some of the eigenvalues of the linearized map have strictly
positive real parts
and some have strictly negative real parts
(since in this monograph we work in dimension~$2$,
this means that one
eigenvalue has positive real part
and the other has negative real part).

We also recall that
sinks are asymptotically stable (and sources are
asymptotically stable for the reversed-time dynamics), see e.g.~\cite[Theorem~1.1.1]{MR1056699}.

With this terminology, we state the following theorem:

\begin{theorem}[Dynamics of system~\eqref{model}] \label{thm:dyn}
For given positive constants $a$, ${\rho}$, and $c$, the system~\eqref{model} has the following features:
\begin{itemize}
	\item[(i)] When~$0<ac<1$, there are $3$ equilibria:~$(0,0)$ is a source,~$(0,1)$ is 
	a sink, and 
	\begin{equation}\label{usvs}
		(u_s, v_s):= \left( \frac{1-ac}{1+{\rho}c} {\rho}c, \frac{1-ac}{1+{\rho}c} \right) \in (0,1)\times (0,1)
	\end{equation}
	is a saddle.

	 \item[(ii)] When~$ac>1$, there are 2 equilibria:~$(0,1)$ is a sink and~$(0,0)$ is a saddle.  
		\item[(iii)] When~$ac=1$, there are 2 equilibria:~$(0,1)$ is a sink and~$(0,0)$
		corresponds to a positive eigenvalue and a null one.
	\item[(iv)] We have 
	\begin{equation} \label{fml:division}
	[0,1]\times [0,1] = \mathcal{B} \cup \mathcal{E} \cup \mathcal{M}
	\end{equation}
	where~$\mathcal{B}~$ and~$\mathcal{E}$ are defined in~\eqref{DEFB}
and~\eqref{DEFE}, respectively, and~$\mathcal{M}$ is a smooth curve.

	\item[(v)] There holds that
	
	\vspace{-8pt}
	
	\begin{itemize}[\quad$\bullet$]
		\item  $\mathcal{M}$ is the {\em stable manifold} of $(u_s,v_s)$ if\,~$0<ac<1$;	
		\item  $\mathcal{M}$ is the {\em center manifold} of $(0,0)$ if\,~$ac=1$;
		\item  $\mathcal{M}$ is the {\em stable manifold} of $(0,0)$ if\,~$ac>1$.
	\end{itemize}

\vspace{-8pt}
		
	Moreover, trajectories starting in~$\mathcal{M}$ remain in~${\mathcal{M}}$ and
	converge to~$(u_s,v_s)$ if~$0<ac<1$
	and to~$(0,0)$ if~$ac\ge1$ as~$t$ goes to~$+\infty$.
\end{itemize}
\end{theorem}

Figure~\ref{fig:dyn} below depicts the different dynamics in the cases $ac<1$ and $ac>1$.

In the case $ac\neq1$, i.e.~when $\mathcal{M}$ is the stable manifold of a saddle point,
the properties of $\mathcal{M}$ stated in Theorem~\ref{thm:dyn} follow from the general theory of 
dynamical systems (see e.g. \cite{dynsyst}).

Instead, the case~$ac=1$ needs a special treatment,
due to the degeneracy of one eigenvalue, and an ad-hoc argument will
be exploited
to show that also in this degenerate case orbits starting in~$\mathcal{M}$
are asymptotic to~$(0,0)$ in the~future.

As a matter of fact,~$\mathcal{M}$
acts as a dividing wall between the two basins of attraction, as described in~(iv)
of Theorem~\ref{thm:dyn} and in the forthcoming Proposition~\ref{prop:char}.

Moreover, in the forthcoming
Propositions~\ref{lemma:M}
and~\ref{M:p045}
we will show that~$\mathcal{M}$ can be written as the graph of a function. This is particularly useful because, by studying the properties of this function, we gain relevant pieces of information on the sets~$\mathcal{B}$ and~$\mathcal{E}$
in~\eqref{DEFB} and~\eqref{DEFE}.

We point out that in Theorem~\ref{thm:dyn}
we find that the set of initial data~$[0,1]\times[0,1]$ splits into three parts:
the set~$\mathcal{E}$, given in~\eqref{DEFE}, made
of points going to the extinction of the second population in finite time; the set~$\mathcal{B}$, given in~\eqref{DEFB}, which is the
basin of attraction of the equilibrium~$(0,1)$;
the set~$\mathcal{M}$, which is
a manifold of dimension~$1$ that separates~$\mathcal{B}$ from~$\mathcal{E}$. 

In particular,
Theorem~\ref{thm:dyn} shows that, also for our model, the Gause principle of exclusion is respected; that is, in general, two competing populations cannot (stably) coexist in the same territory, see e.g.~\cite{fath2018encyclopedia}. 

One peculiar feature of our system is that, if the aggressiveness is too strong, the equilibrium~$(0,0)$ changes its ``stability'' properties, passing from a source (as in (i) of
Theorem~\ref{thm:dyn})
to a saddle point (as in (ii) of
Theorem~\ref{thm:dyn}). This shows that the war may have self-destructive outcomes, therefore it is important for the first population to analyze the situation in order to choose a proper level of aggressiveness. 

\begin{figure}
	\begin{subfigure}{.5\textwidth}
		\centering
		\includegraphics[width=1.86\linewidth]{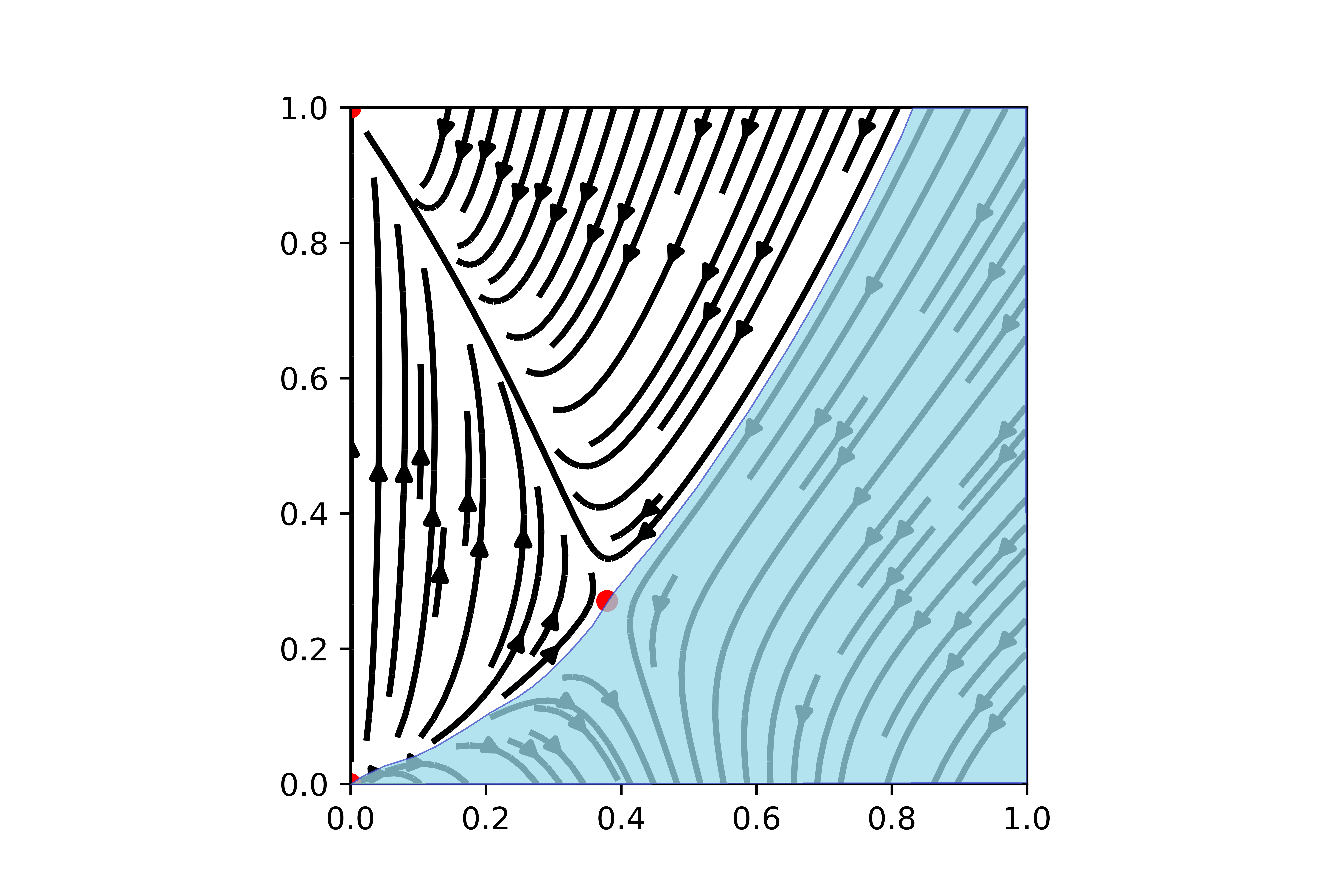}
		\caption{$a=0.5$,~$c=0.7$,~$\rho=2
		$ \\(note that~$ac<1$)}
	\end{subfigure}\\
	\begin{subfigure}{.5\textwidth}
		\centering
		\includegraphics[width=1.86\linewidth]{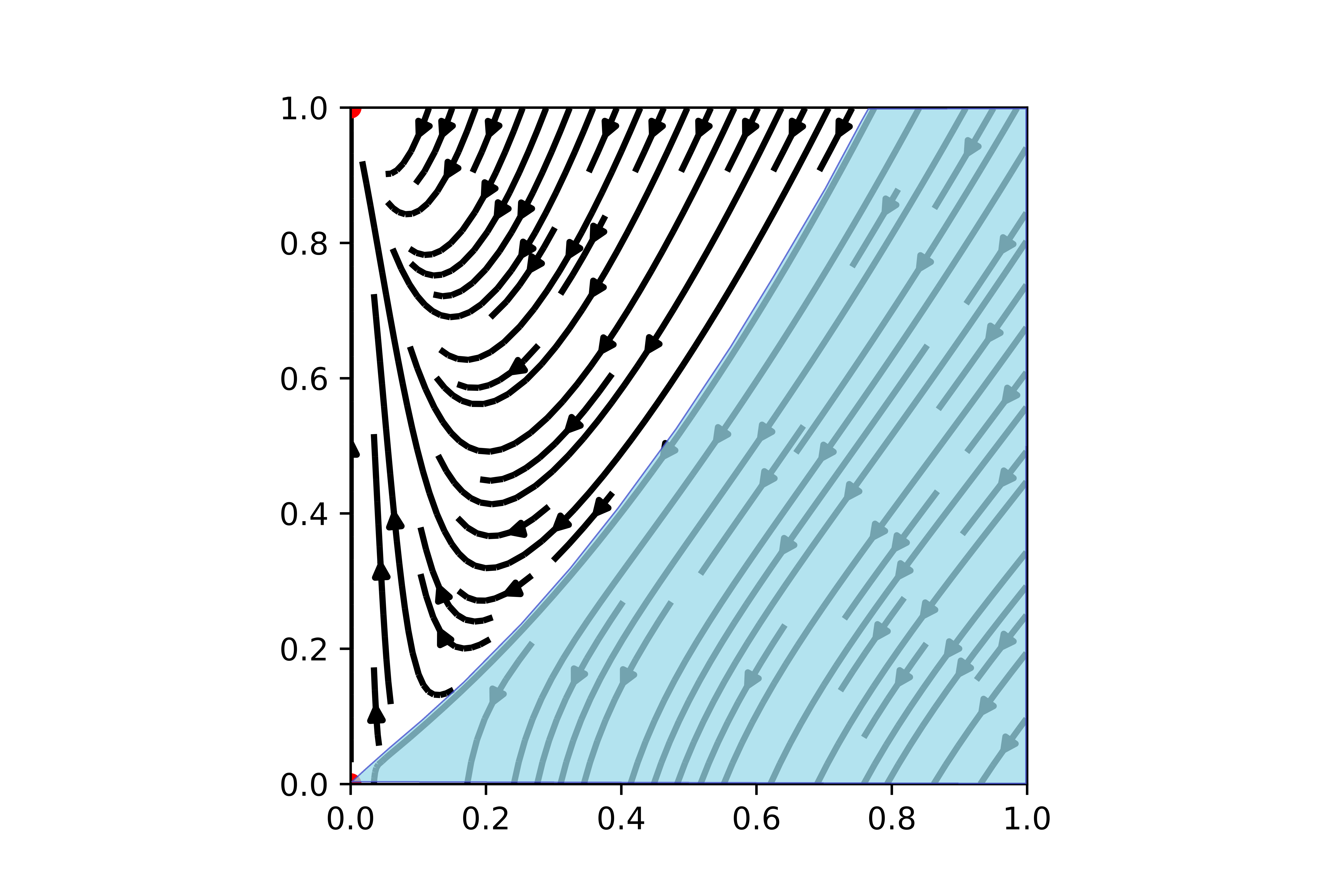}
		\caption{$a=1.5$,~$c=0.7$,~$\rho=2$ \\(note that~$ac>1$)}
	\end{subfigure}
	\caption{\em The figures show a phase portrait for the indicated values of the coefficients. In black, the orbits of the points. The red dots represent the equilibria. The light blue region correspond to $\mathcal{E}$,
	while the white region to~${\mathcal{B}}$.}
	\label{fig:dyn}
\end{figure}

\section{Winning strategies}\label{ss:strategy}

We now deal with the problem of choosing the strategy~$a$, 
in the set of all admissible strategies $\mathcal{A}$ (defined by~\eqref{DEFA}),
such that the first population wins, that is a problem of \emph{target reachability} for a control-affine system (see also \cite{boscain2003optimal}).
 As we will see, the problem is not \emph{controllable}, meaning that, starting from a given initial point, it is not always possible to reach a given target.
\medskip

Let us introduce some terminology, that will be employed throughout this monograph.
For a given strategy $a(\cdot)\in\mathcal{A}$, 
we let $\mathcal{E}(a(\cdot))$ denote the set defined by~\eqref{DEFE} corresponding to 
the system~\eqref{model} with $a\equiv a(t)$; namely, this is the set 
 of initial data~$(u_0,v_0)$ 
such that~$T_s(u_0,v_0)< +\infty$ for the strategy $a(\cdot)$.

Then, for a given set of strategies $\mathcal{T}\subset\mathcal{A}$,
we set
\begin{equation}\label{DEFNU}
\mathcal{V}_{\mathcal{T}}:= \underset{a(\cdot)\in \mathcal{T}}{\bigcup} \mathcal{E}(a(\cdot)),
\end{equation}
which represents the set of initial conditions for which~$u$ is able to win by choosing a suitable strategy in~$\mathcal{T}$; we call~$\mathcal{V}_{\mathcal{T}}$ the \emph{victory set} with strategies in~$\mathcal{T}$.
We also say that~$a(\cdot)$ is a \emph{winning strategy} for the point~$(u_0,v_0)$
if~$(u_0,v_0)\in \mathcal{E}(a(\cdot) )$.

Moreover, we set
\begin{equation}\label{u0v0}
(u_s^0, v_s^0):= \left(\frac{\rho c}{1+\rho c}, \frac{1}{1+\rho c}\right).
\end{equation}
Notice that~$(u_s^0, v_s^0)$ is the limit point as~$a\to0$ of the sequence of saddle points~$\{(u_s^a, v_s^a)\}_{a>0}$
defined in~\eqref{usvs}.
\medskip

With this notation,
the first question that we address is for which initial configurations it is possible for the population~$u$
to have a winning strategy, that is, to characterize the victory set. For this, we allow the strategy to take all the values in~$[0, +\infty)$ (the details on the behavior of the system for $a=0$ are contained in Proposition~\ref{prop:bhvaPRE}).
In this setting, we have the following result:

\begin{theorem}\label{thm:Vbound} We have:
	\begin{itemize}
		\item[(i)] For~$\rho=1$, we have that
		\begin{equation}\label{Vbound1}
			\mathcal{V}_{\mathcal{A}} = \Big\{ (u,v)\in[0,1] \times [0,1] \;
			{\mbox{ s.t. }}\;  v-\frac{u}{c}<0\Big\}.
		\end{equation}
		\item[(ii)] 
		For~$\rho<1$, we have that
		\begin{equation}\label{bound:rho<1}
		\begin{split}
		\mathcal{V}_{\mathcal{A}} &\;= \Bigg\{ (u,v)\in[0,1] \times [0,1] \;{\mbox{ s.t. }}\\
		&\qquad\qquad v< \gamma_0(u) \ \text{if} \ u\in [0, u_s^0], \\ 
		&\qquad\qquad
		 v< \frac{u}{c} + \frac{1-\rho}{1+\rho c} \ \text{if} \ u\in \left(u_s^0, 
		 \frac{\rho c(c+1)}{1+\rho c}\right]
		 \Bigg\},
		\end{split}
		\end{equation}
		where 
		\begin{equation}\label{def:gamma0}
			\gamma_0(u):= \frac{v_s^0}{{(u_s^0)}^{\rho}} u^{\rho}.
		\end{equation}
		\item[(iii)] For~$\rho>1$, we have that
		\begin{equation}\label{bound:rho>1}
		\begin{split}
		\mathcal{V}_{\mathcal{A}} &\;= \Bigg\{ (u,v)\in[0,1] \times [0,1]\; {\mbox{ s.t. }}\\ &\qquad\qquad
		 v< \frac{u}{c} \ \text{if} \ u\in [0, u_{\infty}],\\&\qquad
\qquad
v< \zeta(u)  \ \text{if} \ u\in\left(u_{\infty}, \frac{c}{(c+1)^{\frac{\rho-1}\rho}}\right] 
 \Bigg\},
		\end{split}
		\end{equation}
		where
		\begin{equation}\label{ZETADEF}
		u_{\infty}:= \frac{c}{c+1}
		\quad {\mbox{ and }}\quad \zeta (u):= \frac{u^{\rho}}{c \, u_{\infty}^{\rho-1}}\, .
		\end{equation}   
	\end{itemize}
\end{theorem}

With this result in hand, it is natural to wonder whether the scenario changes if one restricts to
constant strategies.
Indeed, in practice, 
these are certainly easier to implement.
The next result addresses this problem by showing that when~$\rho=1$
constant strategies are as good as all strategies,
but instead when $\rho\ne 1$ victory cannot be achieved by only
exploiting constant strategies:

\begin{theorem}\label{thm:W}
	Let $\mathcal{K}\subset \mathcal{A}$ be the set of constant functions. Then the following holds:
\begin{itemize}
\item[(i)] For~$\rho= 1$, we have that~$ \mathcal{V}_{\mathcal{A}}=\mathcal{V}_{\mathcal{K}}=\mathcal{E}(a)$ for 
any $a>0$.
\item[(ii)] For~$\rho\neq 1$, we have that~$\mathcal{V}_{\mathcal{K}} \subsetneq \mathcal{V}_{\mathcal{A}}$.
\end{itemize}	
\end{theorem} 

The result of Theorem~\ref{thm:W}, part~(i),
reveals a special rigidity of the case~$\rho=1$
in which, no matter which strategy~$u$ chooses,  the victory depends only on the initial conditions, but it is independent of the strategy~$a(t)$.

Instead, as stated in
Theorem~\ref{thm:W}, part~(ii),
for~$\rho \neq 1$ the choice of~$a(t)$ plays a crucial role in determining which population is going to win and constant strategies do not exhaust all the
possible winning strategies.

Roughly speaking, when~$\rho=1$ constant strategies suffice
to detect all possible winning configurations, while when~$\rho\ne1$ non-constant strategies are necessary to detect
all winning configurations.

We stress that~$\rho=1$ plays also
a special role in the biological interpretation of the model, since in this case the two
populations have the same fit to the environmental resource, and hence, in a sense,
they are indistinguishable, up to the possible aggressive behavior of the first population.

Next, we show that the set~$\mathcal{V}_{\mathcal{A}}$ can be recovered if we use piecewise constant functions with at most one discontinuity, that we call Heaviside functions. 

\begin{theorem}\label{thm:H}
	There holds that~$\mathcal{V}_{\mathcal{A}} = \mathcal{V}_{\mathcal{H}}$, where~$\mathcal{H}$ is the set of Heaviside functions.
\end{theorem}

In proving Theorem~\ref{thm:H} we will actually answer to the third question mentioned in the Introduction:
for each point in $\mathcal{V}_{\mathcal{A}}$ we either have a constant winning strategy or
a winning strategy of the type 
\begin{equation*}
a(t) = \left\{
\begin{array}{lr}
a_1  &{\mbox{ if }} t<T ,\\
a_2  &{\mbox{ if }} t\geq T,
\end{array}
\right.
\end{equation*}
for a suitable~$T\in(0,T_s)$ and
for~$a_i$ very small and~$a_j$ very large, the values of $i$ and $j$ depending on~$\rho$. 
Our construction also enlightens the fact that the choice of the strategy depends on the initial datum, answering 
to the fourth question as well. 

It is interesting to observe that the winning strategy that switches abruptly from a small to a large value
could be considered, in the optimal control terminology, as a ``bang-bang'' strategy.
Even in a target reachability problem, the structure predicted by Pontryagin's Maximum Principle is brought in light: the bounds of the set~$\mathcal{V}_{\mathcal{A}}$, as
given in Theorem~\ref{thm:Vbound}, depend on the bounds that
we impose on the strategy, that are,~$a \in[0,+\infty)$.

It is natural to consider also the case
in which the level of aggressiveness 
is constrained between a minimal and maximal threshold,
which corresponds to imposing~$a\in[m,M)$ for given~$0\leq m\leq M\leq +\infty$, with $M>0$.
In this setting, we denote by~$\mathcal{A}_{m,M}$ the class of piecewise continuous strategies~$a(\cdot)$
in~${\mathcal{A}}$ such that~$
m\leq a(t)\leq M$ for all~$t>0$ and we call
\begin{equation}\label{SPE}
\mathcal{V}_{m,M}:=\mathcal{V}_{\mathcal{A}_{m,M}}=\underset{{a(\cdot)\in \mathcal{A}}\atop{m\leq a(t)\leq M}
}{\bigcup} \mathcal{E}(a(\cdot)).\end{equation}
Observe that in the case $M=+\infty$, the strategy actually satisfies $m\leq a(t) < +\infty$ since $a(\cdot)\in\mathcal{A}$.
Then we have the following: 

\begin{theorem}\label{thm:limit}
	Let $M$ and $m$ be two real numbers such that $0\leq m\leq M\leq +\infty$ with $M>0$ and either $m\neq 0$ or $M\neq +\infty$. Then, for $\rho\neq 1$ we have the strict inclusion
$$\mathcal{V}_{{m,M}}\subsetneq \mathcal{V}_{\mathcal{A}}.$$
\end{theorem}

Notice that for $\rho=1$, Theorem \ref{thm:W} gives instead that $\mathcal{V}_{{m,M}}= \mathcal{V}_{\mathcal{A}}$.

\section{Time minimizing strategy}

Once established that it is possible to win starting at a certain initial condition, we are interested in knowing which of the possible strategies is best to choose. One condition that may be taken into account is the duration of the war. Now, this question can be written as a minimization problem with a proper functional to minimize and therefore the classical Pontryagin theory applies. 

To state our next result, we consider a given $(u_0, v_0)\in \mathcal{V}_{m,M}$ and,
recalling the setting in~\eqref{SPE}, we define
\begin{equation*}
\mathcal{S}(u_0, v_0) := \Big\{ a(\cdot)\in \mathcal{A}_{m,M}
\;\mbox{ s.t. }\; (u_0, v_0) \in \mathcal{E}(a(\cdot))  \Big\}.
\end{equation*}
This is the set of all bounded strategies for which the trajectory starting at~$(u_0, v_0)$ leads to the victory of the first population.

To each~$a(\cdot)\in\mathcal{S}(u_0, v_0)$ we associate the stopping time defined in~\eqref{def:T_s}, and we express its dependence on~$a(\cdot)$ by writing~$T_s(a(\cdot))$.

In this setting, we provide the following statement concerning the strategy leading
to the quickest possible victory for the first population:

\begin{theorem}\label{thm:min}
	Given a point~$(u_0, v_0)\in \mathcal{V}_{m,M}$, there exists a winning strategy~$\tilde{a}(t)\in
	\mathcal{S}(u_0, v_0)$ for which
	\begin{equation*}
		T_s(\tilde{a}(\cdot)) = \underset{a(\cdot)\in\mathcal{S}}{\min} T_s(a(\cdot)).
	\end{equation*}
	
	Moreover, the optimal strategy satisfies
	\begin{equation*}
	\tilde{a}(t)\in \left\{m, \ M, \    a_s(t) \right\},
	\end{equation*}	 
	where 
\begin{equation}\label{KSM94rt3rjjjdfe}
	{a}_s(t) := \dfrac{(1-\tilde{u}(t)-\tilde{v}(t))[\tilde{u}(t) \, (2c+1-\rho c)+\rho c]}{\tilde{u}(t) \, 2c(c+1)},
\end{equation}	
and $(\tilde{u}(t), \tilde{v}(t) )$ is the trajectory emerging from~$(u_0,v_0)$
associated with $\tilde{a}(t)$.
\end{theorem}

The surprising fact given by Theorem~\ref{thm:min}
is that the
minimizing strategy is not only of bang-bang type, but it may assume some values along a \emph{singular arc}, given by~$a_s(t)$.
This possibility is realized in some concrete cases, as we verified by running some numerical simulations, whose results can be visualized in Figure~\ref{fig:min}. 

\begin{figure} 
	\begin{subfigure}{.5\textwidth}
		\centering
		\includegraphics[width=1.9\linewidth]{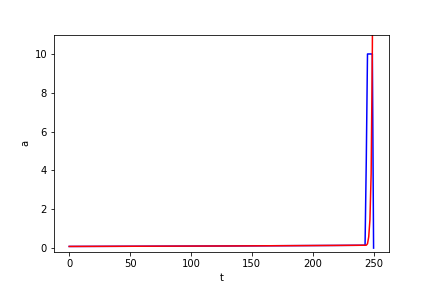}
	\end{subfigure}\\
	\begin{subfigure}{.5\textwidth}
		\centering
		\includegraphics[width=1.9\linewidth]{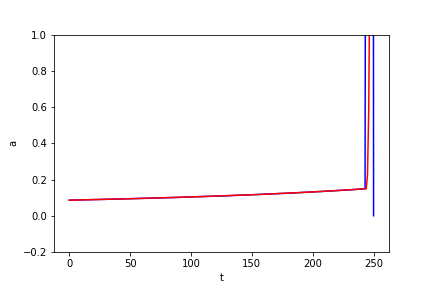}
	\end{subfigure}
	\caption{\em The figure shows the result of a numerical simulation searching a minimizing time strategy~$\tilde{a}(t)$ for the problem starting in~$(0.5, 0.1875)$ for the
	parameters~$\rho=0.5$,~$c=4.0$,~$m=0$ and~$M=10$. In blue, the value
	found for~$\tilde{a}(t)$; in red, the value of~$a_s(t)$ for the corresponding trajectory~$(u(t), v(t))$. As one can observe,~$\tilde{a}(t)\equiv a_s(t)$ in a long trait.
	The simulation was done using AMPL-Ipopt on the server NEOS and pictures have been made with Python. 
	}
	\label{fig:min}
\end{figure}

\section{Discussion of the results}

We now give the interpretation of the analyses of the aggressive competition model that we propose. We discuss the results in the key of the economic model because it seems to us the interpretation most applicable to reality, while we again dissociate ourselves from the use of violence between human beings. 

We emphasize that, unlike the Lotka-Volterra model, our system allows for the case in which one of the two firms (the ``attacked'' one) fails in finite time 
(cf.~Theorem \ref{thm:dyn}). Once the competition is eliminated, the first firm obtains a monopoly in the market, with all the benefits of the case and the consequent problems for consumers.

In the absence of aggression ($a=0$), the two companies would eventually 
end up on a coexisting equilibrium in the market (see Proposition~\ref{prop:bhvaPRE}). 
When bringing in an aggressive strategy, we can generally see three aspects. 
When the aggressive company ($u$) is already clearly preponderant in the market compared to the competitor ($v$), regardless of the strategy adopted, it will succeed in wiping out the competitor from the market. In contrast, if the competitor is very established in the market, it will not always be possible to supplant it, even in the case where the aggressive firm is more efficient than the opponent ($\rho <1$), and in this case the aggressive strategy harms the attacking firm much more. 

Interestingly, in cases close to the limit, the aggressive firm could only completely supplant the second with aggressive strategies that would bring itself close to failure.

As for the cases of intermediate initial situations, we observe that indeed the choice of strategy influences the final outcome. We can draw the following conclusions, depending on whether the market is initially 
saturated or not, and whether the aggressive firm is more or less efficient than its competitor in reusing its capital to generate more.

We emphasize that the case of an oversaturated market may occur, for example, when there is a shrinking pool of buyers due to an economic or demographic crisis, or when some consumers have the products of both companies for a period out of curiosity.

A more specific description of the different scenarios can be summarized as follows:

\begin{itemize}
	\item \textit{Scenario 1: the two firms have the same efficiency ($\rho=1$)}. \\
	In this case, the choice of the strategy does not affect the outcome. Namely, 
	one of the two firms will eventually prevail according only to the initial conditions, independently of
	the aggressive strategy. The strategy just modulates the speed of the dynamics.
	This is shown in 
	Theorem~\ref{thm:Vbound} part (i) and Lemma \ref{lemma:rho=1}. 
    \item \textit{Scenario 2: the first firm is more efficient than the second ($\rho <1$)}.\\ 
    If the market is not oversaturated ($u+v\leq1$), then it is more convenient for the first firm to ``let the market flow'' and use light aggression ($a$ very small).  If the market, on the other hand, is oversaturated ($u+v>1$), it is convenient for the first firm to adopt a very aggressive strategy ($a$ very large), to bring the market to an undersaturated condition which is, however, as convenient as possible for it, and once this intermediate goal is achieved, to continue with less intense aggression. In particular, in this latter case, nonconstant strategies are better than constant strategies.
These results are contained in Theorem \ref{thm:Vbound}, \ref{thm:W}, and \ref{thm:H}, as well as Proposition \ref{prop:construction} part 1.
 \item \textit{Scenario 3: the first firm is less efficient than the second ($\rho >1$)}.\\ 
 If the market is not over-saturated ($u+v\leq1$), a very aggressive strategy allows the first firm to obtain monopoly even in cases where light aggressiveness would 
 not allow it. 
 
 If, on the contrary, the market is oversaturated ($u+v<1$), it is convenient for the first firm not to use an aggressive strategy ($a=0$)
 until the market reaches an unsaturated state, then, adopting a strongly aggressive strategy, the firm will be able to 
 eventually wipe the rival out
(as long as the initial data belong to the set presented in the formula~\eqref{DEFQ}).
 
These results are presented in Theorems  \ref{thm:Vbound}, \ref{thm:W}, and \ref{thm:H}, and Proposition~\ref{prop:construction} part 2.
\end{itemize}

We also analyzed, in the case of initial situations that allow the first firm for 
``winning'' strategies, which strategy eliminates competitors from the market in the fastest way possible. What is highlighted is that this strategy can be very sophisticated, in particular it can alternate between very high and very low values of aggressiveness, or follow a certain function (the singular arc function $a_s$ defined in Theorem \ref{thm:min}). Although it is difficult to give a general expression for the fastest strategy, it is possible to calculate or simulate it numerically from the initial data using well-known optimal control tools (see Figure \ref{fig:min} and the proof of the Theorem \ref{thm:min} in Section \ref{s:dimthmmin}).

\chapter{Toolbox}\label{TOOLB}

\begin{center}
\begin{minipage}{25em}
\noindent{\bf Abstract of Chapter~\ref{TOOLB}.}
{\sl In this chapter we collect some auxiliary results about dynamical systems which will come in handy during the proofs of the main results.}\end{minipage}\end{center}
\bigskip\bigskip\bigskip\bigskip\bigskip\bigskip

We start with some technical notation that we will often use in this work,
in addition to the basic ones presented in Section~\ref{ss:notation}.
%
We will sometimes
indicate the solutions~$(u(t),v(t))$  of this system by~$\phi_p(t)$, where~$p=(u(0),v(0))$,
in order to stress out the dependence on the initial position.

	Throughout this monograph, 
	solutions, trajectories and orbits are always associated with system \eqref{model}. 

Given $(u_0,v_0)\in [0,1]\times[0,1]$ such that $T_s(u_0,v_0)=+\infty$, we define the $\omega-$\emph{limit set} of $(u_0,v_0)$ as 
\begin{equation*}\label{def:omega}\begin{split}
	\omega(u_0,v_0):=\;&\big\{ (x,y)\in\R^2 \;{\mbox{ s.t. }}\\&\qquad \phi_{(u_0,v_0)}(t)\in [0,1]\times[0,1] \; \text{ for all } \; t\geq 0, \\&\qquad {\mbox{and there exists }} \;\{t_i\}_{i\in\N}\; {\mbox{ s.t. }}\;\ t_i\to+\infty\\&\qquad {\mbox{and }}\;\underset{i\to+\infty}{\lim} \phi_{(u_0,v_0)}{(t_i)}=(x,y) \big\}.\end{split}
\end{equation*}
We also define the limit in the past as the $\alpha-$\emph{limit set} of~$(u_0,v_0)$ if~$\phi_{(u_0,v_0)}(t)\in [0,1]\times[0,1]$ for all~$t\leq 0$, that is
\begin{equation*}\label{def:alpha}
\begin{split}
\alpha(u_0,v_0):=\;&\big\{ (x,y)\in\R^2 \;{\mbox{ s.t. }}\\&\qquad \phi_{(u_0,v_0)}(t)\in [0,1]\times[0,1] \; \text{ for all } \; t\leq 0, \\&\qquad {\mbox{and there exists }} \;\{t_i\}_{i\in\N}\; {\mbox{ s.t. }}\;\ t_i\to-\infty\\&\qquad {\mbox{and }}\;\underset{i\to+\infty}{\lim}\phi_{(u_0,v_0)}{(t_i)}=(x,y) \big\}.
\end{split}
\end{equation*}
We will refer to a periodic trajectory as a \emph{closed orbit} (see for example~\cite{dynsyst}).

Also, when we talk about \emph{open} or \emph{closed} sets contained in $[0,1]\times [0,1]$, 
it is always understood {with respect to the relative topology of $[0,1]\times [0,1]$}.

\begin{figure}
	\centering
\includegraphics[scale=0.4]{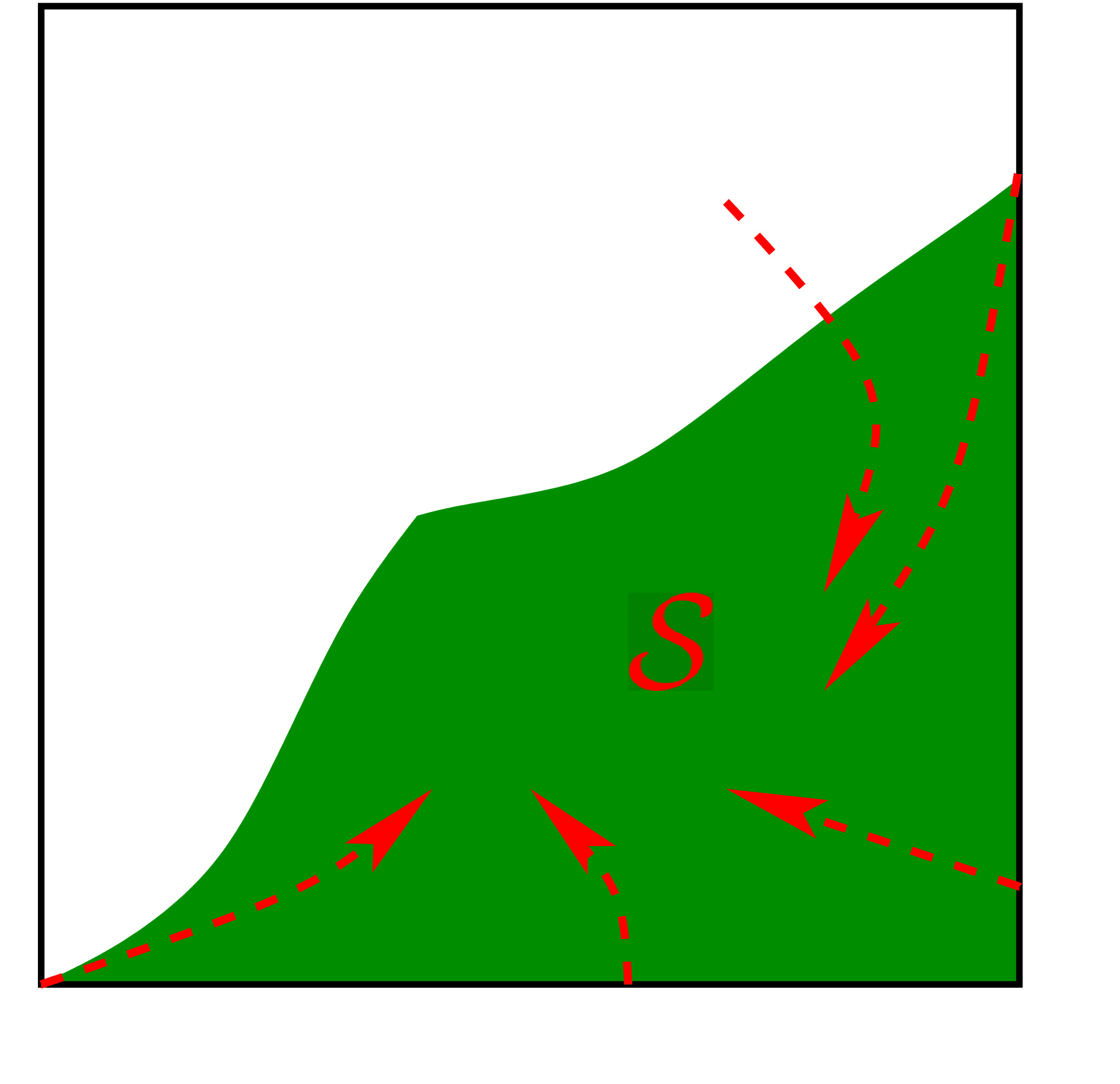}
\caption{Trajectories entering~${\mathcal{S}}$.}
\label{fig:intervals:KMD9uoijhr2lkngr}
\end{figure}

\begin{definition}
Given an open set $\mathcal{S}$ of $[0,1]\times [0,1]$, we say that a trajectory \emph{enters}  $\mathcal{S}$ through
a point~$p\in\partial S$ if, letting $(u(t),v(t))$ be a solution generating such a trajectory, 
there exist a time $T\geq0$ and a strictly  
decreasing sequence $(t_n)_{n\in\N}$ converging to $T$ such that
\begin{itemize}
	\item $(u(T),v(T))=p$,
	\item~$(u(t_n),v(t_n))\in \mathcal{S}$ for all~$n\in\N$.
\end{itemize}\end{definition}
See Figure~\ref{fig:intervals:KMD9uoijhr2lkngr} for a sketch of this notion of entering.

\begin{remark}\label{rmk:enter}
	Notice that the side $\{0\}\times [0,1]$ coincide with the orbit starting in the equilibrium $(0,0)$ and arriving to the equilibrium $(0,1)$.
	We mostly consider sets which are subgraphs of a continuous function (see Lemma \ref{lemma:entrance}). Thus, by the uniqueness of the Cauchy problem, no trajectory can enter these sets  through the left side $\{0\}\times [0,1]$.
\end{remark}

\begin{definition}\label{def:exiting}
If $\mathcal{S}$ is a closed set in the topology of $[0,1]\times[0,1]$,
we also say that a trajectory \emph{exits} the set $\mathcal{S}$ through
a point~$p\in\partial S$ if it enters $[0,1]\times[0,1]\setminus \mathcal{S}$ through $p$ (see Figure~\ref{fig:intervals:KMD9uoijhr2lkngr2}).
\end{definition}

\begin{figure}
	\centering
\includegraphics[scale=0.4]{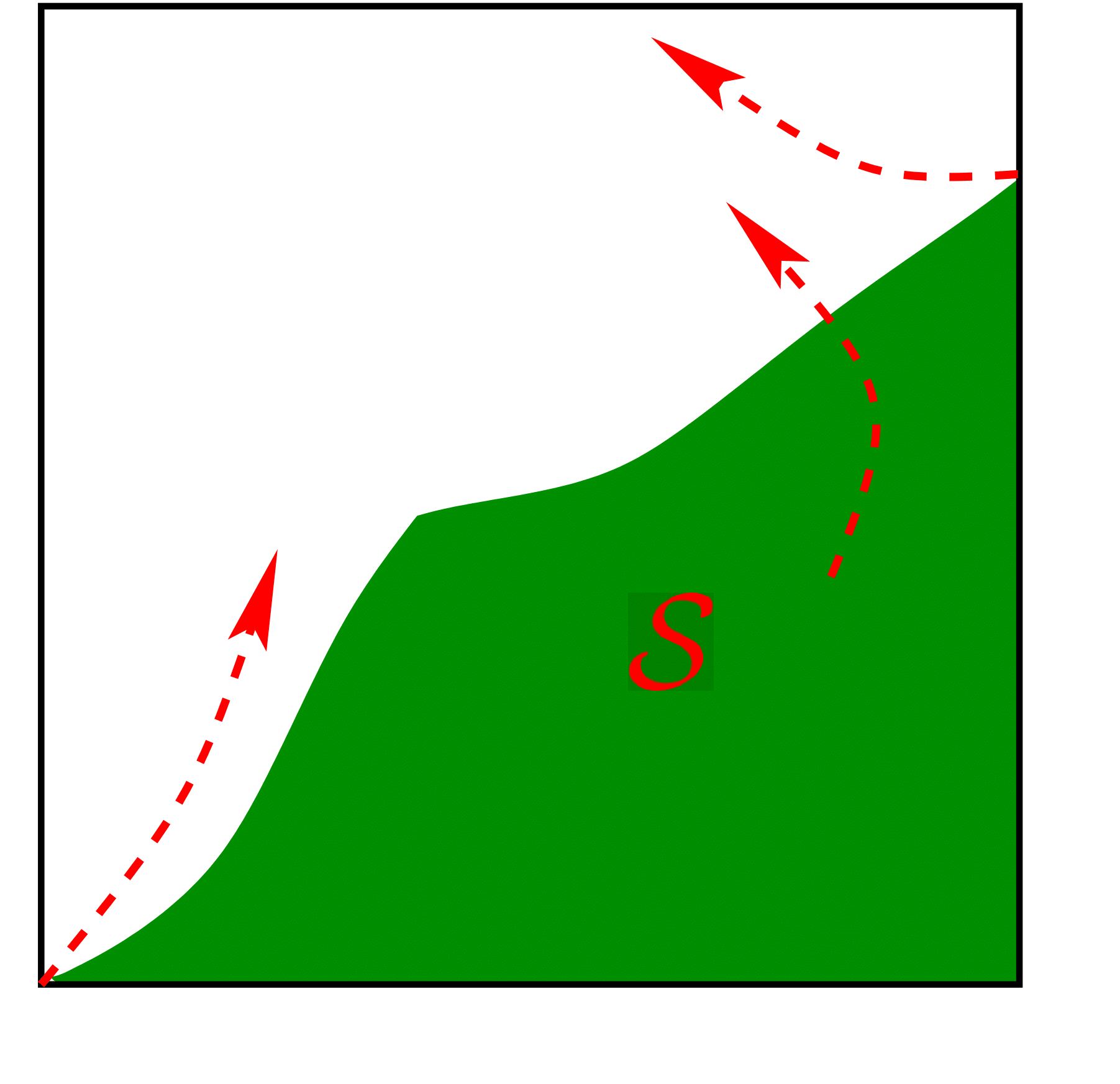}
\caption{Trajectories exiting~${\mathcal{S}}$.}
\label{fig:intervals:KMD9uoijhr2lkngr2}
\end{figure}

\begin{remark}
	As we already discussed in Remark \ref{rmk:dichotomy}, a trajectory either stays in $[0,1]\times[0,1]$ for all positive times, or
has a finite stopping time. This is why we do not take into account the situations in which a trajectory could exit ``through the sides of $[0,1]\times[0,1]$'' in all the rest of the monograph without further mention. 
\end{remark}

Now, we enunciate a useful lemma about entrance and exit of the trajectories in a set. 
In the following statement, we have a set~$\mathcal{S}$ that is the set of points given by the subgraph of a function~$g(u)$.

We define the outward unit normal vector to the surface $\partial \mathcal{S}$ at the point~$(\check u, g(\check u))$ by
\begin{equation}\label{01u2ohf90437ut943EDFRHJKD-0ihHYGF}
\nu = \left(- \frac{g'(\check u)}{\sqrt{1+(g'(\check u))^2}}, \frac{1}{\sqrt{1+(g'(\check u))^2}}  \right)
\end{equation}
whenever $g'(\check u)$ exists. We observe that ``outward'' here is intended with respect to the subgraph of~$g$.

We also extend this notion of outward unit normal vector
at points where  $\partial \mathcal{S}$ has a corner or a vertical tangency\footnote{Actually, later on, we will focus our attention
on monotone increasing functions~$g$, thus the case in which~$\underset{u\to \check u}{\lim} g'(u)=-\infty$
will not be used (we mentioned it at this level just for completeness).}
(hence~$g'(\check u)$ does not exist or is infinite). In this situation,
\begin{itemize}
	\item if $\underset{u\to \check u}{\lim} g'(u)=+\infty$, then we take $\nu=(-1,0)$
	(i.e., the one obtained from~\eqref{01u2ohf90437ut943EDFRHJKD-0ihHYGF} by formally replacing~$g'(\check u)$ by~$+\infty$);
		\item if $\underset{u\to \check u}{\lim} g'(u)=-\infty$, then we take $\nu=(1,0)$
			(i.e., the one obtained from~\eqref{01u2ohf90437ut943EDFRHJKD-0ihHYGF} by formally replacing~$g'(\check u)$ by~$-\infty$);
	\item if $\ell_+=\underset{u\to \check u^+}{\lim} g'(u)$ and $\ell_-=\underset{u\to \check u^-}{\lim} g'(u)$ exist (possibly infinite) but are different, we admit that $\partial \mathcal{S}$ has two outward unit normal vectors at $\check u$ (i.e., the ones obtained from~\eqref{01u2ohf90437ut943EDFRHJKD-0ihHYGF} by replacing~$g'(\check u)$ by~$\ell_-$ and~$\ell_+$, possibly using the conventions in the first two points on this lists).
\end{itemize} 

Furthermore, for the sake of clarity, for all $(\check u, \check v)\in\partial \mathcal{S}$ we denote by~$N(\check u, \check v)$  \emph{the set of outward unit normal vectors 
	to $\mathcal{S}$ at~$(\check u, \check v)$}.
Notice that for the points $(\check u, g(\check u))$ this set has one element when~$g$ is differentiable at~$\check u$ or one of the first two cases in the previous list occurs, while when the last case of the previous list occurs, it has two elements.  

In this chapter we call
\begin{equation*}
F(u,v)= u(1-u-v-ac), \qquad G(u,v)=\rho v (1-u-v)-ac,
\end{equation*}
so that system \eqref{model} becomes
\begin{equation*}
\left\{
\begin{array}{llr}
\dot{u}&= F(u,v), & {\mbox{ for }}t>0,\\
\dot{v}&= G(u,v), & {\mbox{ for }}t>0,
\end{array}
\right.
\end{equation*}
where $F$ and $G$ are locally Lipschitz-continuous functions.

In the forthcoming Lemma \ref{lemma:entrance}, we will require that for all~$(\check u, g(\check u))\in\partial \mathcal{S}$, the trajectory starting at $(\check u, g(\check u))$ satisfies 
\begin{equation*}
	(F(\check u, g(\check u)), G(\check u, g(\check u))) \cdot \nu \geq 0 \qquad \text{for all } \ \nu \in N(\check u, g(\check u)). 
\end{equation*}

\begin{remark}\label{rmk:trajectory}
	Notice that, if $(\bar{u}, \bar{v})$ is not an equilibrium for the system, 
	the trajectory starting at some $(\bar{u}, \bar{v})$ has the vector  $(F(\check u, g(\check u)), G(\check u, g(\check u)))$ as tangent vector at $(\bar{u}, \bar{v})$.
	Hence, the scalar product $(F(\check u, g(\check u)), G(\check u, g(\check u))) \cdot \nu $ gives us information on the relative position of the trajectory and the vector~$\nu$.
	In particular, if the scalar product has a sign, this tell us in which direction the trajectory crosses the graph of $g$.
\end{remark}

However, we want to specify the structure of the set where the scalar product is equal to zero for at least one normal vector. 
In fact, we only treat the cases where the product is zero in a finite number of closed intervals and isolated singletons. 
The reason for this is to avoid pathological cases, i.e. when there is a dense sequence of singletons where the scalar product is zero.
 
Let $C$ be a finite subset of~$\N$ and let $\{K_j\}_{j\in C}$ be a collection of closed intervals, possibly coinciding with singletons.
Then, notice that in the topology of $[0,1]$, the set $$[0,1]\setminus \underset{j\in C}{\bigcup} K_j$$ is an open set and can be written as a collection of open intervals~$\{I_k \}_{k\in A}$ for some finite subset~$A$ of~$ \N$.
This setting can be visualized in Figure~\ref{fig:intervals}.

\begin{figure}
	\centering
\includegraphics[scale=0.2]{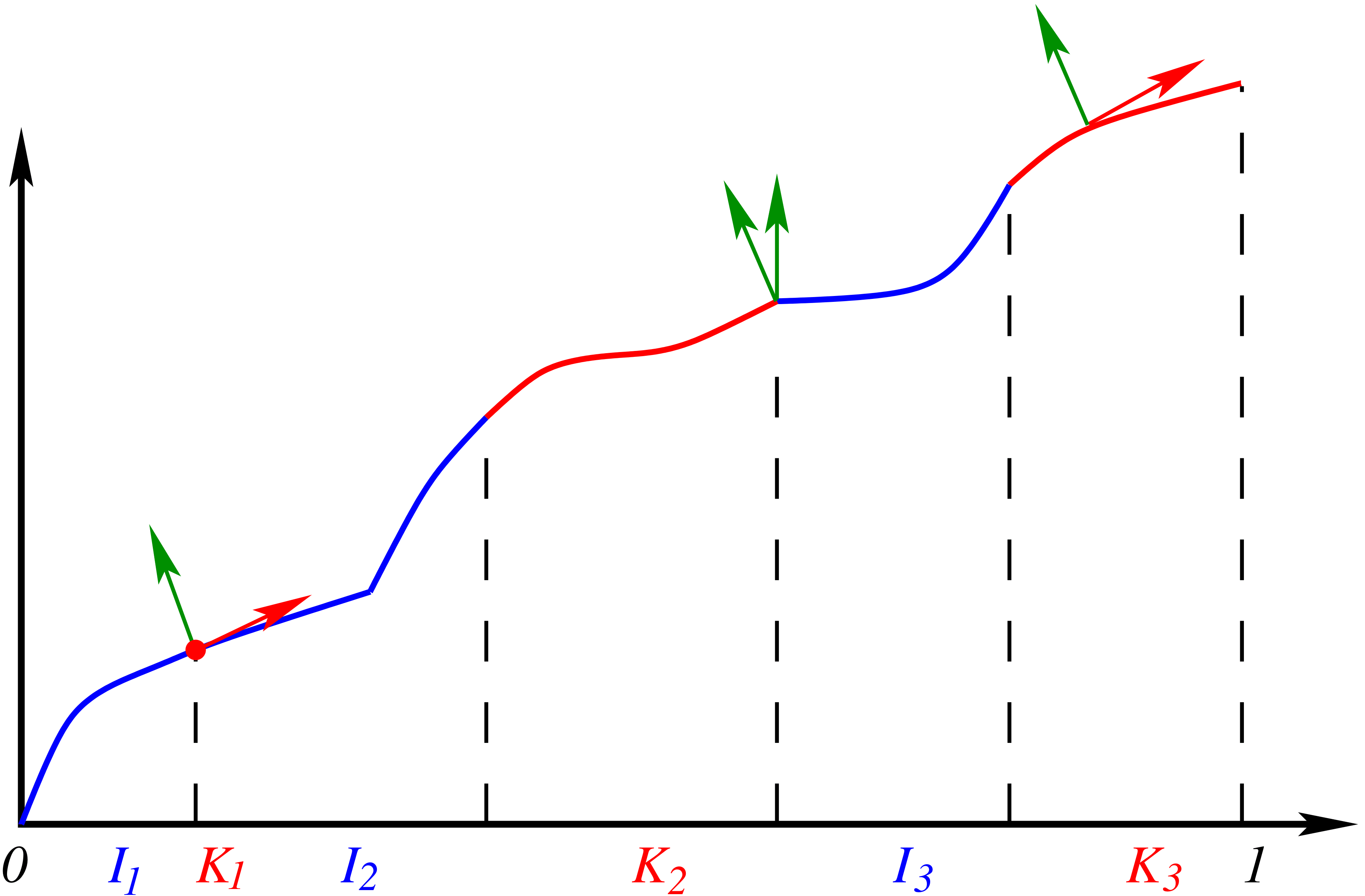}	
\caption{A possible choice for the function $g(u)$ and intervals $\{I_k\}_{k\in A}$, $\{ K_j \}_{j\in C}$ satisfying the hypothesis of Lemma \ref{lemma:entrance}. In blue, the traits for $u\in I_k$ with $k=1,2,3$. In red, the points and traits for $u\in K_j$ with $j=1,2,3$.	
	In green, the outward unit normal vectors; in red, the tangent vectors to the graph of $g$.}
\label{fig:intervals}
\end{figure}

Now we are ready to give the following:

\begin{lemma}\label{lemma:entrance}
	Let
	\begin{equation*}
		\mathcal{S}:= \{ (u,v)\in [0,1]\times[0,1] \;{\mbox{ s.t. }}\;  v < g(u) \},
	\end{equation*}
	where  $g:[0,1]\to [0,+\infty)$
	is a monotone increasing, continuous function such that there exists a finite (possibly empty) set~$Z\subset[0,1]$
		for which~$g\in C^1([0,1]\setminus Z)$, and in addition
		the limits
		$$\lim_{u\to z^\pm}g'(u)$$
		exist for all $z\in Z$ (possibly distinct and possibly equal to $+\infty$).
		Let $N(\check u,g(\check u))$ be {the set of outward unit normal vectors 
			to $\mathcal{S}$ at~$(\check u,g(\check u))$, as defined above}.
	Let $\{I_k\}_{k\in A}$ and $\{K_j\}_{j\in C}$ be two finite collections of disjoint intervals in $[0,1]$,
	the $I_k$ being open and the $K_j$ being closed (possibly coinciding with singletons)
	in the topology induced from $[0,1]$, 
	such that $$\underset{k\in A}{\bigcup} I_k  \cup \underset{j\in C}{\bigcup} K_j=[0,1].$$
	
	Suppose that for all $\check u\in \underset{k\in A}{\bigcup} I_k$
	it holds that
		\begin{equation}\label{property}
		\underset{\nu \in N(\check u, g(\check u))}{\min} (F(\check u, g(\check u)),G(\check u, g(\check u))) \cdot \nu > 0,
		\end{equation}
and that for all $\check u\in \underset{j\in C}{\bigcup} K_j$
	it holds that
		\begin{equation}\label{property2}
		\underset{\nu \in N(\check u, g(\check u))}{\min} (F(\check u, g(\check u)),G(\check u, g(\check u))) \cdot \nu = 0.
		\end{equation}

	Then, no trajectory enters $\mathcal{S}$ through
	a point of the set
$$\partial \mathcal{S}\cap \{ (u,v) \;{\mbox{ s.t. }}\; u\in[0,1],\ v=g(u)\}.$$
\end{lemma}

\begin{proof}
	Let us call 
	$$\mathcal{G}:=\{ (u,v) \;{\mbox{ s.t. }}\; u\in[0,1],\ v=g(u)\}.$$
	We suppose that
there exists $u_M\in(0,1]$ such that~$g(u)<1$ for all~$u\in[0,u_M)$ and~$(u_M, g(u_M))\in \partial ([0,1]\times[0,1])$; the case~$g(u)>1$ for all~$u$ is trivial because in this case~$\mathcal{S}$ coincides with~$[0,1]\times[0,1]$ and~$\partial \mathcal{S}\cap  \mathcal{G}= \varnothing$.

	We argue by contradiction and suppose that there exists a trajectory entering $\mathcal{S}$ through
		a point $(\check u,g(\check u))$, with $\check u\in[0, u_M]$. Let $(\widehat u,\widehat v)$ be the solution generating 
		such a trajectory, with (up to translation in time) $(\widehat u(0),\widehat v(0))=(\check u,g(\check u))$.
	
	Notice 
	that either $\check u\in I_k$ for a unique $k\in A$, or $\check u\in K_j$ for 
	a unique $j\in C$.

	The assumptions on the regularity of $g(u)$ imply that $\partial \mathcal{S}\cap \mathcal{G}$ has a unique normal vector,
	except at the corner points, where it has 2.
	
We now distinguish three cases. In the first two cases, we distinguish whether $g$ is differentiable  at $\check u$, or~$g$ is not differentiable at~$\check u$.
In the last case we analyze the role of the extrema~$\check u=0$ and~$\check u=u_M$.
	
	\smallskip
		{\em Case 1: $\check u\notin Z \cup \{0, u_M\}$.
			}
   
		In this case, $g$ is differentiable in $\check u$.
		Notice that, since $g$ is piecewise $C^1$ with a finite number of non differentiability points, then there exists a neighborhood of $\check u$ (in $[0,1]$) where $g$ is $C^1$.

	If $\check u\in I_k$ for some $k\in A$, then by \eqref{property} and Remark \ref{rmk:trajectory} one has a 
	contradiction with the fact that $(\widehat u,\widehat v)$
	 enters~$\mathcal{S}$ through the point~$(\check u, g(\check u))$.
	
	If $\check u\in \text{int} K_j$ for some $j\in C$, then by~\eqref{01u2ohf90437ut943EDFRHJKD-0ihHYGF} and~\eqref{property2}
	there exists a neighborhood $U$ of $\check u$ such that for each $u\in U$ it holds that
	\begin{equation}\label{FG}
	-g'(u)F(u, g(u))+G(u, g(u)) = 0.
	\end{equation} 
	Let $\varphi$ be the solution of the Cauchy problem
	\begin{equation}\label{1813}
	\begin{cases}
	\varphi'(t)=F(\varphi(t),g(\varphi(t))),\\
	\varphi(0)=\check u,
	\end{cases}
	\end{equation}
	which exists for $t\geq0$ sufficiently small. It follows from~\eqref{FG} that~$(\varphi(t),g(\varphi(t)))$ solves~\eqref{model} for~$t\geq0$ sufficiently small,
	hence for these~$t$'s, we have that $$(\widehat u(t),\widehat v(t))=(\varphi(t),g(\varphi(t))).$$
	This contradicts the fact that the associated trajectory enters~$\mathcal{S}$ through~$(\check u,g(\check u)))$.
	
	If $\check u\in\partial K_j$ for some fixed $j\in C$, we distinguish two cases depending on whether $K_j$ is a singleton or not. 

	If $K_j=\{\check u\}$, recalling that the trajectory generated by 
	$(\hat u(t),\hat v(t))$ enters~$\mathcal{S}$ through~$(\check u, g(\check u))$ (at time $0$),
	there exists an arbitrarily small time $\tau>0$ such that
	$(\hat u(\tau),\hat v(\tau))\in \mathcal{S}$.
	Then, by the continuous dependence with respect to initial data, and because $\mathcal{S}$
	is an open set,
	there exists a ball  $B$ centered at~$(\check u,g(\check u))$ and of radius $\varepsilon$
	sufficiently small such that~$\phi_q(\tau)\in\mathcal{S}$  for every~$q\in B$. 
	In addition, for given $\delta>0$, up to reducing the time $\tau>0$ and the radius $\varepsilon$ of the ball $B$ if need be, we have that 
	\begin{equation*}\label{Kj}
		\forall t\in[0,\tau],\ \forall q\in B,\quad
		\phi_q(t)\in (\check u-\delta,\check u+\delta)\times[0,1].
	\end{equation*}	
As a consequence, for $q\in B\setminus \mathcal{S}$, there must exist $\overline{t}\in (0, \tau)$ 
and $u\in(\check u-\delta,\check u+\delta)$ such that 
$$\phi_{q}(\overline{t})=(\underline{u}, g(\underline{u}))\in\partial \mathcal{S}.$$
Then, taking $\delta$ smaller than the distance between the point $\check u$ and the compact set $\bigcup_{j'\neq j}K_{j'}$,
we infer
$$\underline u\in K_j \cup \underset{k\in  A}{\bigcup} I_k.$$
In addition, since $g$ is differentiable at $\check u$, it must be differentiable in a neighborhood of $\check u$, by hypothesis,
hence, for even smaller $\delta$ we have that $g$ is differentiable at
 $\underline{u}$. 
	But we have shown before that trajectories cannot enter~${\mathcal{S}}$ through a point~$(u, g(u))$ with $u\in I_k$ and $g$ differentiable at $u$, 
	therefore it has to be $\underline{u}=\check u$. Summing up, we have shown that	
	\begin{equation}\label{2002}
		\forall q\in B\setminus \mathcal{S},\ \exists \bar t\in(0, \tau),\quad
	\phi_{q}(\bar t)=(\check u, g(\check u)).
	\end{equation}
	This means that $B\setminus \mathcal{S}$ is a subset of the trajectory
	$$\{\phi_{(\check u, g(\check u))}(t)\ :\ t\in[-\tau,0]\},$$
	which is impossible because the trajectory has zero measure whereas
	$B\setminus \mathcal{S}$ has positive measure, being~$g$ a continuous function.

Thus, we are left with the case where $\check u\in\partial K_j$ and $K_j$ is an interval. Then, it must be that $\check u\in \partial I_k$ for some $j\in A$. 
	Without loss of generality we can suppose that~$I_k=(u_1,\check u)$ and $K_j=[\check u, u_2]$ for some $0\leq u_1< \check u < u_2\leq u_M$.
	
Let us denote by~$\mu$ the vector that is tangent to the graph of $g(u)$ at $(\check u, g(\check u))$ and such that $\mu$ has positive components.
	
	We now distinguish three cases.
	
	If $(F(\check u, g(\check u)), G(\check u, g(\check u))) \cdot \mu =0$, then this and \eqref{property2} give that $(\check u,g(\check u))$ is an equilibrium, hence no trajectory can enter through~$(\check u,g(\check u))$.
	
	If $(F(\check u, g(\check u)), G(\check u, g(\check u))) \cdot \mu >0$,  then
	at least one between $F(\check u, g(\check u))$ and $G(\check u, g(\check u))$ must be positive. 
	Moreover, by \eqref{property2}, we have
	\begin{equation}\label{1837}
	-g'(\check u)F(\check u, g(\check u))+G(\check u, g(\check u)) = 0.
	\end{equation} 
	Notice that, if $F(\check u, g(\check u))<0$, by the fact that $g$ is increasing and~\eqref{1837} we  get that~$G(\check u, g(\check u))<0$, which contradicts our assumption.

Consequently,
	\begin{equation}\label{1838}
		F(\check u, g(\check u))>0.
	\end{equation}
Also, there is a right neighborhood $U$ of $\check u$ such that for all $u\in U$ \eqref{FG} holds true. Let $\varphi(t)$ be a solution of \eqref{1813}. Then, by \eqref{1838}, we get that $\varphi(t)$ is increasing, thus $\varphi(t)\in U$ for small $t$. 
	Hence, \eqref{FG} holds true and $(\varphi(t), g(\varphi(t)))$ is a solution of the system in~\eqref{model} starting at $(\check u, g(\check u))$ and laying on the graph of $g$ for small $t$. 
	This contradicts the fact that a trajectory enters~$\mathcal{S}$ through~$(\check u, g(\check u))$. 
	
	We are left with the case $(F(\check u, g(\check u)), G(\check u, g(\check u))) \cdot \mu <0$. 
	Then, arguing as in the previous case, we can prove that $F(\check u, g(\check u))<0$. By continuity, we have that 
	\begin{equation}\label{2032}
		F(u, g(u))<0
	\end{equation}
	for $u\in U$, being~$U$ a right neighborhood of $\check u$. 

	Let~$\bar u \in U$. 
	Then, by arguing as in the previous case, we have that the trajectory of $(\bar u, g(\bar{u}))$ can be written as $(\psi(t), g(\psi(t) ))$ where $\psi(t)$ is a solution of 
	\begin{equation*}
	\begin{cases}
	\psi'(t)=F(\psi(t),g(\psi(t))),\\
	\psi(0)=\bar u.
	\end{cases}
	\end{equation*}
	By \eqref{2032}, we have that $\psi(t)$ is decreasing. Moreover, since \eqref{2032} holds true for all $u\in U$, it follows that $\psi (\bar{t})=\check u$ for some $\bar t >0$.  Hence, since for all $u\in U\setminus \{\check u\}$ the point $(u, g(u))$ belongs to the trajectory of $(\bar{u}, g(\bar{u}))$, no trajectory can enter $\mathcal{S}$ through $(u, g(u))$ for $u\in U\setminus \{\check u\}$. 
	
	Now take $ V \subset I_k$ where $V$ is a left neighborhood of $\check u$ where $g$ is differentiable. Then for all $u\in V\setminus \{\check u\}$, as seen in the first case, no trajectory can enter $\mathcal{S}$ through $(u, g(u))$. 
	
	Now, take a ball $B$ of radius $\varepsilon$ centered at $(\check u, g(\check u))$. By arguing as in the case when $K_j$ is a singleton, we can prove that for all $q\in B\setminus \mathcal{S}$ (which has positive measure) it holds that $\phi_q(\bar t) =(\check u, g(\check u))$ for some $\bar t \in (0, \delta)$ because no trajectory can enter through other points in a neighborhood of $\check u$. So  
	\begin{equation*}
	\phi_{B\setminus \mathcal{S}}(t) \subset \Upsilon  \quad \text{for} \ t\in(\delta, 2\delta)
	\end{equation*}
	where
	\begin{eqnarray*}\Upsilon&: = &\big\{ (u,v)\in[0,1]\times[0,1] \ ;{\mbox{ s.t. }}\\&&\qquad\quad (u,v)= \phi_{(\check u, g(\check u))}(t) \ \text{for} \ t\in[0,2 \delta]  \big\}. \end{eqnarray*}
	But $\Upsilon$ has measure 0 and $\phi_{B\setminus \mathcal{S}}(t)$ has positive measure, thus giving a contradiction.
	
	\medskip
	\emph{Case 2: $\check u\in Z \setminus \{0, u_M\}$.}

    Since $\check u\in Z$, then $g$ is not differentiable at $\check u$.
	By hypothesis, there are a finite number of non differentiability points, hence they are isolated points. 
	So, given a neighborhood $U$ of $\check u$, it holds that $g$ is differentiable in $U\setminus \{\check u\}$. Hence, no trajectory can enter~${\mathcal{S}}$ through a point of the form $(u, g(u))$ with~$u\in U\setminus \{\check u\}$,
in light of Case~1.
	
	Now, take a ball $B$ of radius $\varepsilon$ centered in $(\check u, g(\check u))$. By the continuity of $g$, we see that $B\setminus \mathcal{S}$ has positive measure. By arguing as in the case when $K_j$ is a singleton, for all $q\in B\setminus \mathcal{S}$  it holds that $\phi_q(\bar t) =(\check u, g(\check u))$ for some $\bar t \in (0, \delta)$ because no trajectory can enter through other points in a neighborhood of $\check u$. So  
	\begin{equation*}
	\phi_{B\setminus \mathcal{S}}(t) \subset \Upsilon  \quad \text{for} \ t\in(\delta, 2\delta)
	\end{equation*}
	where
	\begin{eqnarray*}\Upsilon&: =& \big\{ (u,v)\in[0,1]\times[0,1] \;{\mbox{ s.t. }}\\&&\qquad\quad (u,v)= \phi_{(\check u, g(\check u))}(t) \ \text{for} \ t\in[0,2 \delta]  \big\}. \end{eqnarray*}
	But $\Upsilon$ has measure 0 and $\phi_{B\setminus \mathcal{S}}(t)$ has positive measure, thus providing the desired contradiction.

	\medskip
	\emph{Case 3: $\check u=0$ or $\check u=u_M$.}
	
	For the sake of concreteness, let us consider the case $\check u=0$, the other one being analogous.
	Let us consider a right neighborhood $U$ of $\check u=0$ (in the topology of $[0,1]$). By Cases~1 and~2, 
\begin{equation}\label{previousaffirmations}\begin{split}
&{\mbox{no trajectory can enter~${\mathcal{S}}$ through a point}}\\&{\mbox{of the form $(u, g(u))$ with~$u\in U\setminus \{0\}$.}}\end{split}\end{equation}
	
	Suppose that a trajectory enters $\mathcal{S}$ through $(0, g(0))$. 
	Also, consider a ball $B$ centered at $(0, g(0))$ of radius $\varepsilon$ sufficiently small and
	the set 
	$$D=B\cap ([0,1]\times [0,1]) \setminus \mathcal{S}, $$
	which has positive measure. 

By continuity with respect to initial data, 
	there exists $\delta>0$
	such that~$\phi_q(\delta)\in\mathcal{S}$    for every~$q\in D$.

As a consequence, there exists a point $(\bar{u}, \bar{v})\in \partial \mathcal{S}$ such that $\phi_q(\bar{t})=(\bar{u}, \bar{v})$ for some $\bar{t}\in (0, \delta)$. 
	Notice that $(\bar{u}, \bar{v})$ cannot be of the form $(u, g(u))$ with~$u\in U\setminus \{\check u\}$, in light of~\eqref{previousaffirmations}. 
	
Also, the trajectory of $q$ is contained in $[0,1]\times [0,1]$ for all $t<T_s(q)$ (where $T_s$ is the stopping time defined in \eqref{def:T_s}). Thus, it cannot be that $\phi_q(\bar{t})=(0, \underline{v})$ with $\underline{v}\leq g(0)$, unless   $\phi_q(\bar{\tau})=(u_0, g(\check u))$ for some $\bar{\tau}\in (0,\delta)$.
	
	Therefore, for $t\in (\delta, 2\delta)$ we have that $\phi_D(t) \subset \Upsilon$
		where
	\begin{eqnarray*}\Upsilon&: =& \big\{ (u,v)\in[0,1]\times[0,1] \;{\mbox{ s.t. }}\\&&\qquad\quad (u,v)= \phi_{(\check u, g(\check u))}(t) \ \text{for} \ t\in[0,2 \delta]  \big\}. \end{eqnarray*}
	But $\Upsilon$ has measure 0 and $\phi_{B\setminus \mathcal{S}}(t)$ has positive measure, thus providing the desired contradiction
and completing the proof of Lemma~\ref{lemma:entrance}.
\end{proof}

We also provide a stronger statement for exiting trajectories. Notice that here we take a closed set $\mathcal{S}$ to use the definition \ref{def:exiting} of exiting trajectories.

\begin{lemma}\label{lemma:exit}
    Let
	\begin{equation*}
	\mathcal{S}:= \{ (u,v)\in [0,1]\times[0,1] \;{\mbox{ s.t. }}\; v \leq g(u) \},
	\end{equation*}
	where $g$ satisfies the same hypotheses as in Lemma \ref{lemma:entrance}. Suppose also that $I_k$ and $K_j$ are as in Lemma~\ref{lemma:entrance}.
	
	Suppose that for all $\check u\in \underset{k\in A}{\bigcup} I_k$
	it holds that
	\begin{equation}\label{property21}
	\underset{\nu \in N(\check u, g(\check u))}{\max} (F(\check u, g(\check u)),G(\check u, g(\check u))) \cdot \nu < 0,
	\end{equation}
	and that for all $\check u\in \underset{j\in C}{\bigcup} K_j$
	it holds that
	\begin{equation}\label{property22}
	\underset{\nu \in N(\check u, g(\check u))}{\max} (F(\check u, g(\check u)),G(\check u, g(\check u))) \cdot \nu = 0.
	\end{equation}

	Then, no trajectory exits $\mathcal{S}$. 
\end{lemma}

\begin{proof}
	Let us suppose by contradiction that there exists a trajectory exiting $\mathcal{S}$ through a point $(\check u, \check v)$. 
	
	If
$$(\check u, \check v)\in \partial \mathcal{S}\cap\{ (u,v) \:{\mbox{ s.t. }}\; u\in[0,1],\ v=g(u)\},$$ then repeating the arguments of Lemma~\ref{lemma:entrance} we get a contradiction.
	
	We also observe that no trajectory can exit $\mathcal{S}$ by leaving $[0,1]\times[0,1]$. This rules out all the possible cases.
\end{proof}

\medskip

Finally, to lighten the text, all along this monograph, we will call \emph{outward normal derivative} at some point  $(\check u,\check u)\in\partial \mathcal{S}$ the scalar product
\begin{equation*}
	(F(\check u, \check v), G(\check u, \check v)) \cdot \nu 
\end{equation*}
with $\nu \in N(\check u, \check v)$.

Also, we call the \emph{inward normal derivative} at some point  $(\check u,\check v)\in\partial \mathcal{S}$ the scalar product
\begin{equation*}
-(F(\check u, \check v), G(\check u, \check v)) \cdot \nu 
\end{equation*}
with $\nu \in N(\check u, \check v)$.


\chapter{Basins of attractions}\label{IKJM:plrg777}

\begin{center}
\begin{minipage}{25em}
\noindent{\bf Abstract of Chapter~\ref{IKJM:plrg777}.} {\sl
In this chapter we provide some useful results on the behavior of the solutions
of the system in~\eqref{model} and on the basins of attraction in the case of constant strategies, also characterizing a separating invariant manifold in dependence of the structural parameters.}\end{minipage}\end{center}
\bigskip\bigskip\bigskip\bigskip\bigskip\bigskip

In this chapter we provide some useful results on the behavior of the solutions
of the system in~\eqref{model} and on the basins of attraction in the case of constant strategies $a$. 
In particular,
we provide the proof of Theorem~\ref{thm:dyn}
and
we state a characterization of  the sets~$\mathcal{B}$
and~$\mathcal{E}$ given in~\eqref{DEFB} and~\eqref{DEFE}, respectively, see Propositions~\ref{prop:char}.

This material will be extremely useful for the analysis of the strategy that we operate later. 

\smallskip

We are now in a position to derive the first three statements of Theorem~\ref{thm:dyn}.

\begin{proof}[Proof of (i), (ii) and~(iii) of Theorem~\ref{thm:dyn}]
We first consider equilibria with first coordinate~$u=0$. In this case, 
from the second equation in~\eqref{model}, we have that
the equilibria must satisfy~$\rho v(1-v)=0$, thus~$v=0$ or~$v=1$.
As a consequence,~$(0,0)$ and~$(0,1)$ are two equilibria of the system.

Next, we consider equilibria with first coordinate~$u>0$.
From the first equation in~\eqref{model} we get
	\begin{equation}\label{curve:u'}
	1-u-v-ac=0,
	\end{equation}
while, from the second one,
	\begin{equation} \label{curve:v'}
	\rho v(1-u-v)-au=0.
	\end{equation}
Putting together~\eqref{curve:u'} and~\eqref{curve:v'} we get~\eqref{usvs}.

{F}rom now on, we distinguish the three situations in~(i), (ii), (iii) of 
Theorem~\ref{thm:dyn}.

\medskip
{\emph{(i)}} If~$0<ac<1$, we have that the point~$(u_s,v_s)$ given in~\eqref{usvs}
lies in~$(0,1)\times(0,1)$. As a result, in this case the system has~$3$ equilibria,
given by~$(0,0)$,~$(0,1)$ and~$(u_s,v_s)$.

The Jacobian of the system~\eqref{model} is
	\begin{equation} \label{Jmatrix}
J(u,v)=	\begin{pmatrix}
	1-2u-v -ac & -u \\ 
	-\rho v-a & \rho (1-u-2v)
	\end{pmatrix}.
	\end{equation}
	At the point~$(0,0)$, the matrix has eigenvalues~$\rho >0$ and~$1-ac >0$, thus~$(0,0)$ is a source. 
	
	At the point~$(0,1)$, the Jacobian~\eqref{Jmatrix} has eigenvalues~$-ac <0$ and~$-\rho <0$, thus~$(0,1)$ is a sink.
	
	At the point~$(u_s,v_s)$, by exploiting the relations~\eqref{curve:u'} and~\eqref{curve:v'} we have that
	\begin{equation*}
	J(u_s,v_s)=\begin{pmatrix}
	-u_s  & -u_s \\ 
	-\rho v_s-a & \rho (ac-v_s)
\end{pmatrix},
	\end{equation*}
	which, by the change of basis given by the matrix
	$$
\begin{pmatrix}
	-\frac1{u_s}  & 0 \\ 
	-\frac1{u_s}\left[\left(\frac{u_s}{c}+a\right)\left(\frac{\rho c-c}{1+\rho c}\right)+ac	\right] & \frac{\rho c-c}{1+\rho c}
\end{pmatrix},
	$$
	becomes	
	\begin{equation*}
\begin{pmatrix}
	1  & 1 \\ 
	ac & \rho ac
\end{pmatrix}.	\end{equation*}
	The characteristic polynomial of this matrix 
	is $$\lambda^2-\lambda(1+\rho a c)+\rho a c-ac,$$ that has two real roots, as one can see by inspection. Hence, the matrix~$J(u_s, v_s)$ has two real eigenvalues.
	
	Moreover,
	the determinant of~$J(u_s, v_s)$ is $$-\rho ac u_s-au_s <0,$$ which implies
	that~$J(u_s, v_s)$ has one positive and one negative eigenvalues. These considerations
	give that~$(u_s, v_s)$ is a saddle point. This completes the proof
	of~(i) in Theorem~\ref{thm:dyn}.
\medskip
	
{\emph{(ii) and (iii)}} We assume that~$ac\ge1$.
We observe that the equilibrium described by the
coordinates~$(u_s,v_s)$ in~\eqref{usvs}
coincides with~$(0,0)$ for~$ac=1$,
and lies outside~$[0,1]\times[0,1]$ for~$ac>1$. 
As a result, when~$ac\ge1$ the system has~$2$ equilibria, given by~$(0,0)$ and~$(0,1)$.

Looking at the Jacobian in~\eqref{Jmatrix},
one sees that
at the point~$(0,1)$, it has eigenvalues~$-ac <0$ and~$-\rho <0$,
and therefore~$(0,1)$ is a sink when~$ac\ge1$.

Furthermore, from~\eqref{Jmatrix}
one finds that
if~$ac>1$ then~$J(0,0)$ has the positive eigenvalue~$\rho$ and the negative
eigenvalue~$1-ac$, thus~$(0,0)$ is a saddle point.

If instead~$ac=1$, then~$J(0,0)$
has one positive eigenvalue and one null eigenvalue, as desired.
\end{proof}

To complete the proof of Theorem~\ref{thm:dyn},
we will deal with the cases~$ac\neq\!1$ and~$ac=\!1$ separately. This analysis will be performed
in the forthcoming Sections~\ref{sec:deg0} and~\ref{sec:deg} respectively.
The completion of the proof of
Theorem~\ref{thm:dyn} will then be
given in Section~\ref{sec:deg2}.

\section{Characterization of~$\mathcal{M}$ when~$ac\ne1$}
\label{sec:deg0}

We point out that
in the proof of~(i) and~(ii) in Theorem~\ref{thm:dyn} we found a saddle point 
in both cases.
By the Stable Manifold Theorem (see for example~\cite{dynsyst}), the point~$(u_s, v_s)$ in~\eqref{usvs} in the case~$0<ac<1$ and the
point~$(0,0)$ in the case~$ac> 1$ have a stable manifold and an unstable manifold.
These manifolds are unique,
they have dimension~$1$, and they are tangent to the eigenvectors of the linearized system.

We will denote by~$\mathcal{M}$ the stable manifold associated
with these saddle points.
Since we are interested in the dynamics in the square~$[0,1]\times[0,1]$, with a slight abuse of notation we will only consider the restriction of ~$\mathcal{M}$ in~$[0,1]\times[0,1]$.

We now analyze
some properties of~$\mathcal{M}$:

\begin{proposition} \label{lemma:M}
	For $ac\ne1$ the set~$\mathcal{M}$ can be written as the graph of a unique increasing~${C}^2$ function $$\gamma:[0,u_{\mathcal{M}}] \to [0, v_{\mathcal{M}}]$$ for some
	$$(u_{\mathcal{M}}, v_{\mathcal{M}}) \in
	\big(\{1\}\times[0,1]\big)\cup
\big((0,1]\times\{1\}\big),$$ such that~$\gamma(0)=0$,~$\gamma(u_{\mathcal{M}})=v_{\mathcal{M}}$ and
	\begin{itemize}
		\item if~$0<ac<1$,~$\gamma(u_s)=v_s$ and in $u=u_s$ the function $\gamma(u)$ is tangent to the line $(v-v_s)c-(u-u_s)=0$;
		\item if~$ac> 1$, in~$u=0$ the function~$\gamma$ is tangent to the
		line~$(\rho-1+ac)v-au=0$.  
	\end{itemize}
\end{proposition}

As a byproduct of the proof of Proposition~\ref{lemma:M}, we also
obtain some useful information on the structure of the stable manifold and
the basins of attraction, that we summarize here below:

\begin{corollary}\label{lemma:M1}
Suppose that~$0<ac<1$. Then,
the curves~\eqref{curve:u'} and~\eqref{curve:v'}, loci of the points such that~$\dot{u}=0$ and~$\dot{v}=0$ respectively, divide the square~$[0,1]\times[0,1]$ into four regions:
	\begin{equation}\begin{split}\label{DEFA1234}
	&\mathcal{A}_1 := \big\{ (u, v) \in [0,1]\times[0,1] \;{\mbox{ s.t }}\; \dot{u}\leq 0,\; \dot{v}\geq 0 \big\}, \\&
	\mathcal{A}_2 := \big\{ (u, v) \in [0,1]\times[0,1] \;{\mbox{ s.t }}\; \dot{u}\leq 0,\; \dot{v}\leq 0 \big\}, \\&
	\mathcal{A}_3 := \big\{ (u, v) \in [0,1]\times[0,1] \;{\mbox{ s.t }}\;\dot{u}\geq 0,\; \dot{v}\leq 0 \big\}\\{\mbox{and }}\quad&
	\mathcal{A}_4 := \big\{ (u, v) \in [0,1]\times[0,1] \;{\mbox{ s.t }}\;\dot{u}\geq 0, \;\dot{v}\geq 0 \big\}.
	\end{split}\end{equation}
		
Furthermore, the sets~$\mathcal{A}_1\cup \mathcal{A}_4$
and~$\mathcal{A}_2\cup\mathcal{A}_3$ are separated by the curve~$\dot{v}=0$, given by the graph of the continuous function
	\begin{equation}\label{f:sigma}
	\sigma(v):= 1- \frac{\rho v^2+a}{\rho v+a},
	\end{equation}
that satisfies~$\sigma(0)=0$,~$\sigma(1)=0$
and~$0<\sigma(v)<1$ for all~$v\in (0,1)$.

In addition,
\begin{equation}\label{aggiunto}
{\mbox{$\mathcal{M}\setminus \{(u_s,v_s) \}$
is contained in~$\mathcal{A}_2\cup\mathcal{A}_4$,}}
\end{equation}
\begin{equation}\label{primaBIS}
(\mathcal{A}_3 \setminus \{ (0,0), (u_s, v_s) \} ) \subseteq \mathcal{E}
\end{equation}
and
\begin{equation}\label{prima2BIS}
\mathcal{A}_1\setminus \{(u_s,v_s) \} \subset \mathcal{B},\end{equation}
where the notation in~\eqref{DEFB} and~\eqref{DEFE} has been utilized.
\end{corollary}

To visualize the statements in Corollary~\ref{lemma:M1}, one can see Figure~\ref{fig:zone}.

	\begin{figure}
		\centering
		\includegraphics[width=1.1\linewidth]{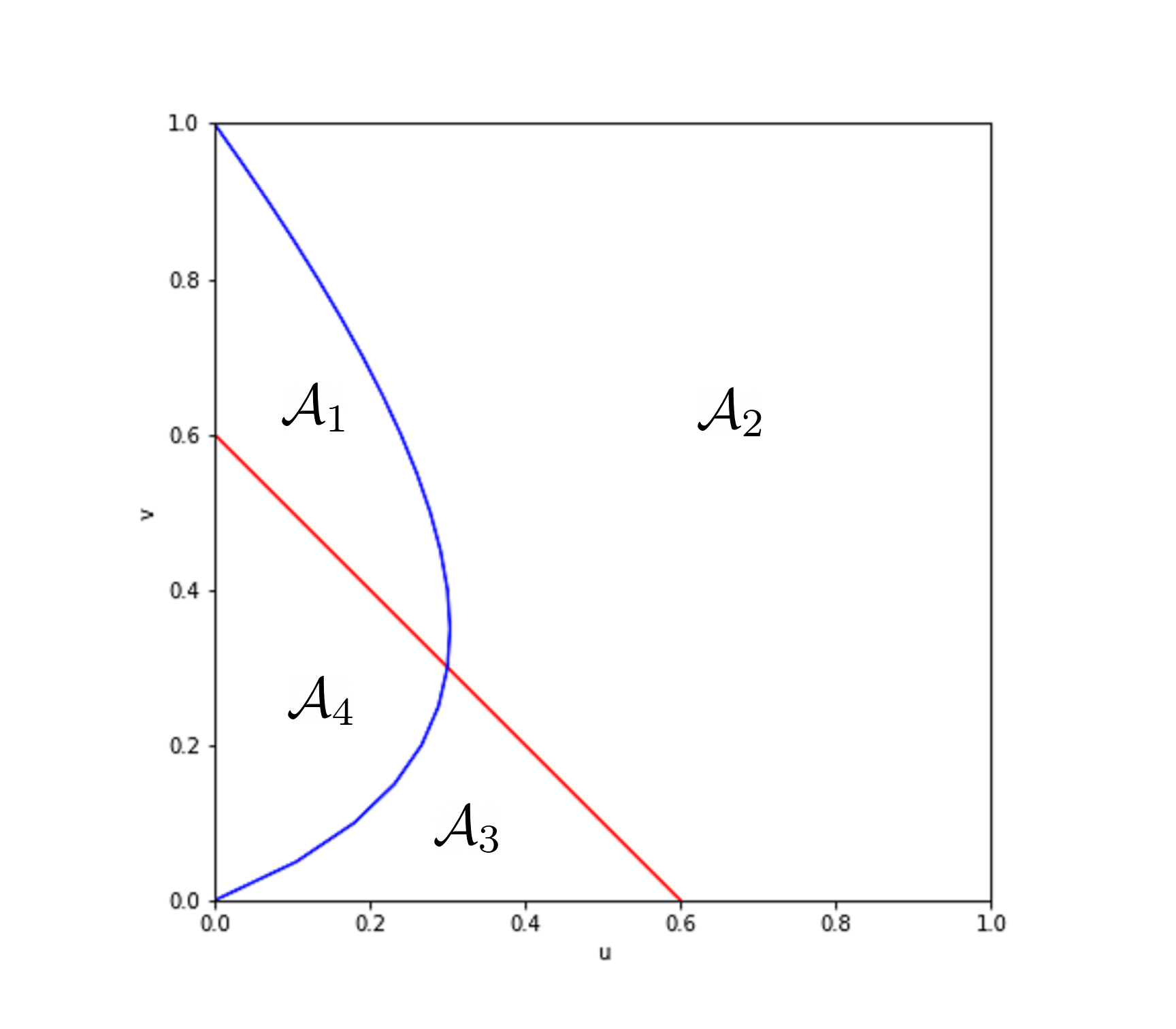}
		\caption{\em Partition of~$[0,1]\times[0,1]$ in the case~$a=0.8$,~$c=0.5$,~$\rho=2$,
		as given by~\eqref{DEFA1234}.
		In red, the curve~$\dot{u}=0$. In blue, the curve~$\dot{v}=0$, parametrized
	by the function~$\sigma$ in~\eqref{f:sigma}.}
		\label{fig:zone}
	\end{figure}

\begin{corollary}\label{lemma:M2}
Suppose that~$ac>1$.
Then , we have that~$\dot{u}\leq 0$ in~$[0,1]\times [0,1]$,
and the curve~\eqref{curve:v'}
divides the square~$[0,1]\times[0,1]$ into the regions
	\begin{equation}\begin{split}\label{DEFA12}
	&\mathcal{A}_1 := \big\{ (u, v) \in [0,1]\times[0,1]\;{\mbox{ s.t. }}\; \dot{u}\leq 0,\; \dot{v}\geq 0 \big\} \\
	{\mbox{and }}\quad&
	\mathcal{A}_2 := \big\{ (u, v) \in [0,1]\times[0,1]\;{\mbox{ s.t. }}\; \dot{u}\leq 0, \;
	\dot{v}\leq 0 \big\}.
	\end{split}\end{equation}

Furthermore, the sets~$\mathcal{A}_1$
and~$\mathcal{A}_2$ are separated by the curve~$\dot{v}=0$, given by the graph of the continuous function~$\sigma$ given in~\eqref{f:sigma}.

In addition,
\begin{equation}\label{aggiun2}
\mathcal{M}\subset \mathcal{A}_2.\end{equation} 
\end{corollary}

Proposition~\ref{lemma:M} and Corollaries~\ref{lemma:M1}
and~\ref{lemma:M2}
are a bit technical, but provide fundamental information
to obtain a characterization of the sets~$\mathcal{E}$ and
$\mathcal{B}$, given in the forthcoming Proposition~\ref{prop:char}.

We now provide the proof of Proposition~\ref{lemma:M}
(and, as a byproduct, of Corollaries~\ref{lemma:M1} and~\ref{lemma:M2}).

\begin{proof}[Proof of Proposition~\ref{lemma:M} and
Corollaries~\ref{lemma:M1} and~\ref{lemma:M2}] In the forthcoming argument
we treat separately the cases $0<ac<1$ and~$ac> 1$.
We start with the case~$0<ac<1$, and divide the proof into three further parts.
	\medskip
	
\emph{Step 1: localizing~$\mathcal{M}$.}
Recall that a orbit is closed if the associate trajectory is periodic (see for example \cite{dynsyst}).
We first observe that
\begin{equation}\label{pouyi86}
\begin{split}
&{\mbox{there are no closed orbits other than fixed points}}\\
&{\mbox{entirely contained
		in~$\mathcal{A}_i$, $i=1, \dots, 4,$}}
	\end{split}\end{equation}
because~$\dot{u}$ and~$\dot{v}$ have a sign.

With the notation introduced in~\eqref{DEFA1234},
we prove that
	\begin{equation}\label{dkoegjerig94768}\begin{split}&
{\mbox{all points $(u_0,v_0)\in\mathcal{A}_3\setminus \{(0,0), (u_s,v_s) \}$}}\\&{\mbox{have~$T_s(u_0,v_0)<+\infty$.}}\end{split}\end{equation} 
To this aim, we show that 
\begin{equation}\label{pouyi861}{\mbox{no trajectory exits~$\mathcal{A}_3$.}}\end{equation}
First of all, we notice that no trajectory can exit $\mathcal{A}_3$ through $(0,0)$ or $(u_s,v_s)$ since they are equilibria. 

We remark that the side connecting~$(0,0)$ and~$(u_s, v_s)$ can be written as the set
of points belonging to
$$\big\{ (u,v)\in [0,1]\times (0,v_s) \;{\mbox{ s.t. }}\;
u=\sigma(v) \big\},$$
where the function~$\sigma$ is defined in~\eqref{f:sigma}.
In this set, it holds that~$\dot{v}=0$ and~$\dot{u}>0$, thus the normal derivative pointing outward~$\mathcal{A}_3$ is negative, so the trajectories cannot exit~$\mathcal{A}_3$ passing through this side. 

Furthermore, on the side connecting~$(u_s, v_s)$ with~$(1-ac, 0)$, that lies on the straight line~$v=1-ac-u$, we have that~$\dot{u}= 0$  and~$\dot{v}<0$ for ~$(u,v)\neq (u_s,v_s)$, so also here the outer normal derivative is negative. Therefore, the trajectories cannot exit~$\mathcal{A}_3$ passing through this side either. This completes the proof of \eqref{pouyi861}.

Now we show that
\begin{equation}\begin{split}\label{pouyi862}
&{\mbox{there are no equilibria where a trajectory}}\\&{\mbox{lying
in the interior
of~$\mathcal{A}_3$ can converge.}}\end{split}\end{equation}
Indeed, if~$p=(u_0,v_0)\in\mathcal{A}_3\setminus \{(0,0)\}$, then $u_0>0$. 
Now, $(0,0)\not\in\omega(p)$, because~$\dot{u}\geq 0$ in~$\mathcal{A}_3$. 
Also, for all~$(u_0, v_0)\in \mathcal{A}_3 \setminus(u_s, v_s)$, we have that~$v_0<v_s$.

On the other hand,~$\dot{v}\leq 0$ in~$\mathcal{A}_3$, so no trajectory that is entirely
contained in~$\mathcal{A}_3$ can converge to~$(u_s, v_s)$. These observations prove~\eqref{pouyi862}.

As a consequence of~\eqref{pouyi86},~\eqref{pouyi861},~\eqref{pouyi862} and the
Poincar\'e-Bendixson Theorem (see e.g.~\cite{MR1056699}, see also Lemma 9.0.4), we have that
all the points in the interior of~$\mathcal{A}_3$ must have $T_s(u_0,v_0)<+\infty$.

These considerations complete the proof of~\eqref{dkoegjerig94768}.
Accordingly, recalling the definition of~$ \mathcal{E}$ in~\eqref{DEFE}, we see
that
\begin{equation}\label{prima}
(\mathcal{A}_3 \setminus \{ (0,0), (u_s, v_s) \} ) \subseteq \mathcal{E}.
\end{equation}
In a similar way one can prove that all trajectories starting in~$\mathcal{A}_1\setminus \{(u_s,v_s) \}$ must converge to~$(0,1)$, which, recalling the definition of~$\mathcal{B}$
in~\eqref{DEFB}, implies that
\begin{equation}\label{prima2}
\left(\mathcal{A}_1\setminus \{(u_s,v_s) \} \right) \subset \mathcal{B}.\end{equation}
Thanks to~\eqref{prima} and~\eqref{prima2},
we have that the stable manifold~$\mathcal{M}$ has no intersection with~$\mathcal{A}_1\setminus \{(u_s,v_s) \}~$ and~$\mathcal{A}_3\setminus \{(0,0),(u_s,v_s) \}~$,
and therefore~$\mathcal{M}$ must lie
in~$\mathcal{A}_2\cup \mathcal{A}_4$.

Also, we know that~$\mathcal{M}$ is tangent to an eigenvector in~$(u_s, v_s)$,
and we observe that
\begin{equation}\label{poui985004}
{\mbox{$(1, -1)$ is not an eigenvector of the linearized system.}}\end{equation}
Indeed, if~$(1, -1)$ were an eigenvector, then
$$
\begin{pmatrix}
	1-ac-2u_s-v_s & -u_s \\ 
	-\rho v_s-a & \rho-\rho u_s-2\rho v_s
\end{pmatrix}\cdot
	\begin{pmatrix}
	1\\-1\end{pmatrix}=\lambda\begin{pmatrix}
	1\\-1\end{pmatrix},
	$$
so from the first component we would get~$\lambda=0$, which is not an eigenvalue of the saddle point $(u_s, v_s)$. This
establishes~\eqref{poui985004}.

In light of~\eqref{poui985004}, we conclude that~$\mathcal{M}\setminus \{(u_s,v_s) \}$
must have intersection with both~$\mathcal{A}_2$ and~$\mathcal{A}_4$.
	
	\medskip

	\emph{Step 2: defining~$\gamma(u)$ and tangential property.}
	Since~$\dot{u}> 0$ and~$\dot{v}>0$ in the interior of~$\mathcal{A}_4$,
	the portion of~$\mathcal{M}$
	in~$\mathcal{A}_4$ can be described globally as the graph of a
	monotone increasing smooth function~$\gamma_1:U\to[0,v_s]$,
	for a suitable interval~$U\subseteq[0,u_s]$ with~$u_s\in U$, and such that~$\gamma_1(u_s)=v_s$. 
	
	We stress that, for~$u>u_s$, the points~$(u,v)\in \mathcal{M}$ belong
	to~$\mathcal{A}_2$.
	
	Similarly, in the interior of~$\mathcal{A}_2$ we have that~$\dot{u}< 0$ and~$\dot{v}<0$.
	Therefore, we find that~$\mathcal{M}$ can be represented in~$\mathcal{A}_2$
	as the graph of a monotone increasing smooth function~$\gamma_2: V\to [v_s, 1]$, for a suitable interval~$V\subseteq[u_s,1]$ with~$u_s\in V$, and such that~$\gamma_2(u_s)=v_s$.  Notice that in the second case the trajectories and the parametrization run in opposite directions.
	
	Now, we define
	\begin{equation*}
		\gamma(u) := \begin{cases}
		\gamma_1(u)  & {\mbox{ if }}u\in U, \\
		\gamma_2(u)  &{\mbox{ if }} u\in V,
		\end{cases}
	\end{equation*}
	and we observe that it is an increasing smooth function locally
	parametrizing~$\mathcal{M}$ around~$(u_s,v_s)$ (thanks to the
	Stable Manifold Theorem).
	
	We point out that, in light of the
	Stable Manifold Theorem, the stable manifold~$\mathcal{M}$ is globally
	parametrized by
	an increasing smooth function on a set~$W\subset[0,1]$.
	
	We now study the tangent to $\gamma(u)$ at $u=u_s$.
	First, let us compute the derivative of $\gamma(u)$, which is given by
	\begin{equation*}
	\frac{d \gamma (u)}{du} = \frac{d v(t)}{dt} \cdot \frac{dt}{du(t)} = \frac{\dot v}{\dot u}.
	\end{equation*}
	Notice that, as $u\to u_s^{-}$ and $v\to v_s^{-}$, owing the expression of~$u_s$ and~$v_s$ given in formula~\eqref{usvs}, we get
	\begin{equation*}
		\underset{\substack{u\to u_s \\ v \to v_s}}{\lim} \, \frac{\dot v}{\dot u} = \underset{\substack{u\to u_s \\ v \to v_s}}{\lim} \, \frac{\rho v(1-u-v)-au}{u(1-u-v-ac)} = \frac{1}{c}.
	\end{equation*} 
	From this we can say that $$\gamma'(u)=\frac{1}{c}.$$
	This gives us that $\gamma(u)$ is tangent to the line
	$(v-v_s)c-(u-u_s)=0$, as desired.

	\medskip

	\emph{Step 3: showing that~$\gamma(0)=0$ and~$\gamma(u_{\mathcal{M}})=v_{\mathcal{M}}$
	for some~$(u_{\mathcal{M}},v_{\mathcal{M}})\in\partial\big([0,1]\times[0,1]\big)$.}
	We first prove that
	\begin{equation}\label{r4gyghj}
	\gamma(0)=0.\end{equation}
	For this, we claim that
	\begin{equation}\label{098ouitdbnb}
	{\mbox{no trajectory enters ~$\mathcal{A}_4\setminus \{(u_s,v_s)\}$.}}
	\end{equation}
		
	Indeed, it is easy to see that points in the form $(0,v_0)$ converge to $(0,1)$. Hence, by the uniqueness of the trajectory passing through a nonfixed point,  we get that no trajectory can enter through the side of $\mathcal{A}_4$ laying on $\{u=0\}$.
		
	No trajectory can enter through $(u_s,v_s)$ or $(0,0)$ since they are equilibria.	
		
	As for the side connecting~$(0,0)$ to~$(u_s, v_s)$, excluding the extrema, one has that~$\dot{u}>0$ and~$\dot{v}=0$, and so the inward pointing normal derivative is negative. Therefore,
	no trajectory can enter~$\mathcal{A}_4$ on this side, see Remark \ref{rmk:trajectory}.
	
Moreover, on the side connecting~$(u_s, v_s)$ to~$(0, 1-ac)$ the inward pointing normal derivative is negative, because~$\dot{u}=0$ and~$\dot{v}>0$, thus we have that no
trajectory can enter~$\mathcal{A}_4$ on this side either.
These considerations prove~\eqref{098ouitdbnb}.

Furthermore, by \ref{pouyi86}, we have that no closed orbits are allowed in~$\mathcal{A}_4$.

{F}rom~\eqref{098ouitdbnb},~\eqref{pouyi86} and the
Poincar\'e-Bendixson Theorem (see e.g.~\cite{TESCHL}), we conclude that,
given a point~$(\tilde u,\tilde v)\in\mathcal{M}$ in the
interior of~$\mathcal{A}_4$, the~$\alpha$-limit set of~$(\tilde u,\tilde v)$,
that we denote by~$\alpha{(\tilde u,\tilde v)}$, exists and  
\begin{equation}\label{09765gjkd}\begin{split}&
{\mbox{$\alpha{(\tilde u,\tilde v)}$ can be either an equilibrium}}\\&{\mbox{or a union of (finitely many)
equilibria}}\\&{\mbox{and non-closed orbits connecting these equilibria.}}\end{split}
\end{equation}

We stress that, being~$(\tilde u,\tilde v)$ in the interior of~$\mathcal{A}_4 $, we have that
\begin{equation}\label{8yfe993vcem}
\tilde u<u_s  \  \vee  \ \tilde v<v_s .
\end{equation}
Now, we observe that
\begin{equation}\label{degfiewgh}
{\mbox{$\alpha{(\tilde u,\tilde v)}$ cannot contain the saddle point~$(u_s,v_s)$.}}\end{equation}
Indeed, suppose by contradiction that~$\alpha{(\tilde u,\tilde v)}$ does
contain~$(u_s,v_s)$. Then,
we denote by $$\phi_{(\tilde u,\tilde v)}(t)=\big(u_{(\tilde u,\tilde v)}(t),v_{(\tilde u,\tilde v)}(t)\big)$$ the solution of~\eqref{model}
with~$\phi_{(\tilde u,\tilde v)}(0)=(\tilde u,\tilde v)$, and we have that
there exists a sequence~$t_j\to-\infty$ such
that~$\phi_{(\tilde u,\tilde v)}(t_j)$ converges to~$(u_s,v_s)$ as~$j\to+\infty$.
In particular, in light of~\eqref{8yfe993vcem},
there exists~$j_0$ sufficiently large
such that
$$ u_{(\tilde u,\tilde v)}(0)=\tilde u<u_{(\tilde u,\tilde v)}(t_{j_0}) \quad{\mbox{ or }}\quad  \ v_{(\tilde u,\tilde v)}(0)=\tilde v<v_{(\tilde u,\tilde v)}(t_{j_0}).$$
Consequently, there exists~$t_\star\in(t_{j_0},0)$ such that
$$\dot
u_{(\tilde u,\tilde v)}(t_\star)<0 \quad{\mbox{ or }}\quad\dot
v_{(\tilde u,\tilde v)}(t_\star)<0.$$

As a result, it follows that~$\phi_{(\tilde u,\tilde v)}(t_\star)\not\in\mathcal{A}_4$.
This, together with the fact that~$\phi_{(\tilde u,\tilde v)}(0)\in\mathcal{A}_4$,
is in contradiction with~\eqref{098ouitdbnb}, and the proof of~\eqref{degfiewgh}
is thereby complete.

Thus, from~\eqref{09765gjkd} and~\eqref{degfiewgh},
we deduce that~$\alpha_{(\tilde u,\tilde v)}=\{(0,0)\}$.
This gives that~$(0,0)$ lies on the stable manifold~$\mathcal{M}$,
and therefore the proof of~\eqref{r4gyghj} is complete.

Now, we show that
\begin{equation}\label{ifregkjh0000}\begin{split}&
{\mbox{there exists~$(u_{\mathcal{M}},v_{\mathcal{M}})\in\partial\big([0,1]\times[0,1]\big)$}}\\&{\mbox{such that~$\gamma(u_{\mathcal{M}})=v_{\mathcal{M}}$.}}\end{split}
\end{equation}
To prove it, we first observe that
\begin{equation}\label{ifregkjh0000pre}\begin{split}&
{\mbox{non-fixed point orbits converging to~$(u_s,v_s)$}}\\ &{\mbox{are not contained in~$ \mathcal{A}_2$}}.\end{split}
\end{equation}
Indeed, we suppose by contradiction that
{there exists a nontrivial orbit contained {in}~$\mathcal{A}_2$
converging to~$(u_s,v_s)$.}
We remark that, in this case, 
{the orbit contianed in~$\mathcal{A}_2$ cannot be a close orbit,}
because~$\dot{u}$ and~$\dot{v}$ have a sign in~$\mathcal{A}_2$.

Then, by the Poincar\'e-Bendixson Theorem (see e.g.~\cite{TESCHL}), we conclude that, given a point~$(\tilde u,\tilde v)\in\mathcal{M}$ in the interior of~$\mathcal{A}_2$, the~$\alpha$-limit set of~$(\tilde u,\tilde v)$,
that we denote by~$\alpha{(\tilde u,\tilde v)}$, exists and it is an equilibrium or a union of (finitely many) equilibria and non-closed orbits connecting these equilibria.

We notice that the set~$\alpha{(\tilde u,\tilde v)}$ cannot contain~$(0,1)$ or $(0,0)$,
since they lay outside~$\mathcal{A}_2$.

So, since $(u_s,v_s)$ is the only equilibria, the $\alpha-$ limit set must coincide with it and the orbit must be a omocline. This is in contradiction with \eqref{pouyi86}, proving \eqref{ifregkjh0000pre}.

Now, we observe that
the inward pointing normal derivative
at every point in~$\mathcal{A}_2 \cap \mathcal{A}_3 \setminus\{(u_s, v_s)\}$
is negative, since~$\dot{u}=0$ and~$\dot{v}<0$. Hence, no trajectory can enter
from this side (see Remark \ref{rmk:trajectory}).

Also, 
the inward pointing normal derivative
at every point in~$\mathcal{A}_1 \cap \mathcal{A}_2 \setminus\{(u_s, v_s)\}$
is negative, since~$\dot{u}>0$ and~$\dot{v}=0$. Hence, no trajectory can enter
from this side either.

These observations and~\eqref{ifregkjh0000pre} give the desired result
in~\eqref{ifregkjh0000}, and thus Proposition~\ref{lemma:M}
is established in the case~$ac<1$.
	
	\medskip
	
	Now we treat the case~$ac>1$, using the same ideas. In this setting,
	$\mathcal{M}$ is the stable manifold associated with the saddle point~$(0,0)$.
	We point out that, in this case, for all points in~$[0,1]\times [0,1]$ we have
	that~$\dot{u}\leq 0$.
	
	Hence, the curve of points satisfying~$\dot{v}=0$, that was also given in~\eqref{curve:v'}, divides the square~$[0,1]\times[0,1]$ into two regions~$
	\mathcal{A}_1$ and~$\mathcal{A}_2$, defined in~\eqref{DEFA12}.	
	
Now, one can repeat verbatim the arguments in {\emph{Step 1}} with obvious modifications,
to find that~$
\mathcal{M}\subset \mathcal{A}_2$.

Since the derivatives of~$u$ and~$v$ have a sign in~$\mathcal{A}_2$, and the
set~$\mathcal{M}$ in this case is the trajectory of a point converging to~$(0,0)$, the set~$\mathcal{M}$ can be represented globally as the graph of a smooth increasing function~$\gamma: U\to [0,1]$ for a suitable interval~$U\subseteq[0,1]$ containing the origin.

As a consequence, the condition~$\gamma(0)=0$ is trivially satisfied in this setting.
The existence of a suitable~$(u_{\mathcal{M}},v_{\mathcal{M}})$ can be derived reasoning as in {\emph{Step 3}}
with obvious modifications.

Now, we prove that
\begin{equation}\label{ofriyty98579}\begin{split}&
{\mbox{at~$u=0$ the function~$\gamma$ is tangent}}\\&
{\mbox{to the line~$(\rho-1+ac)v-au=0$.}}\end{split}\end{equation}
For this, we recall~\eqref{Jmatrix} and we see, by inspection, that	
the Jacobian matrix~$J(0,0)$ has two eigenvectors, namely~$(0,1)$ and~~$(\rho-1+ac, a)$. The first one is tangent to the line~$u=0$, that is the unstable manifold of~$(0,0)$,
as one can easily verify.

Thus, the second eigenvector is the one tangent to~$\mathcal{M}$, as prescribed by the Stable Manifold Theorem (see e.g.~\cite{dynsyst}).

Hence, in~$(0,0)$ the manifold~$\mathcal{M}$ is tangent to the line $$(\rho-1+ac)v-au=0$$ and so is the function~$\gamma$ in~$u=0$. This proves~\eqref{ofriyty98579}, and thus
Proposition~\ref{lemma:M}
is established in the case~$ac>1$ as well.
\end{proof}	

\section{Characterization of~$\mathcal{M}$ when~$ac=1$}\label{sec:deg}

Here we will prove the counterpart of Proposition~\ref{lemma:M}
in the degenerate case~$ac=1$.

To this end, looking at the velocity fields,
we first observe that
\begin{equation}\label{NOEX}
\begin{split}&
{\mbox{trajectories starting in~$(0,1)\times(-\infty,1)$ at time~$t=0$}}\\&{\mbox{remain in~$(0,1)\times(-\infty,1)$ for all time~$t>0$.}}\end{split}
\end{equation}
We also point out that
\begin{equation}\label{CALR}
\begin{split}&
{\mbox{trajectories entering the region}} \\&{\mathcal{R}}:=
\{u\in(0,1),\,u+v<0\} \quad {\mbox{at some time~$t_0\in\R$,}}\\&{\mbox{remain in that region for all time~$t>t_0$,}}\end{split}
\end{equation}
since $$\dot v=\rho v(1-u-v)-au=-\rho u-au<0$$
along~$\{u\in(0,1),\,u+v=0\}$ (see Remark \ref{rmk:trajectory}).

Also, by the
Center Manifold Theorem
(see e.g. Theorem~1 on page~16
of~\cite{MR635782} or pages 89-90 in~\cite{MR1031257}),
there exists a collection~$\mathcal{M}_0$
of invariant curves, which are all
tangent at the origin to the eigenvector corresponding to the null eigenvalue,
that is the straight line~$\rho v-au=0$.

Then, we define $$\mathcal{M}:=
\mathcal{M}_0\cap ([0,1]\times[0,1])$$ and we observe that this
intersection is nonvoid, given the tangency property of~$\mathcal{M}_0$
at the origin.

In what follows, for every~$t\in\R$, we denote by $$(u(t),v(t))=\phi_p(t)$$ the orbit of~$p\in\mathcal{M}\setminus\{(0,0)\}$. We start by providing an observation
related to negative times:

\begin{lemma} \label{NOCPA}
Suppose that~$ac=1$.
If~$p\in\mathcal{M}\setminus\{(0,0)\}$
then~$\phi_p(t)$ cannot approach the origin for negative values of~$t$.
\end{lemma}

\begin{proof}
We argue by contradiction
and denote by~$t_1,\dots,t_n,\dots$ a sequence of such negative values of~$t$, for which~$t_n\to-\infty$ and
$$ \lim_{n\to+\infty}\phi_p(t_n)=(0,0).$$
Up to a subsequence, we can also suppose that
\begin{equation}\label{bejv0565etP}
u(t_{n+1})<u(t_n).
\end{equation}
In light of~\eqref{CALR}, we have that, for all~$T\le0$,
\begin{equation}\label{0okf3233}
\phi_p(T)\not\in{\mathcal{R}}.
\end{equation}
Indeed, if~$\phi_p(T)\in{\mathcal{R}}$, we deduce from~\eqref{CALR}
that~$\phi_p(t)\in{\mathcal{R}}$ for all~$t\ge T$. In particular,
we can take~$t=0\ge T$ and conclude that~$p=\phi_p(0)\in{\mathcal{R}}$,
and this is in contradiction with the assumption that~$p\in{\mathcal{M}}\setminus\{(0,0)\}$.

As a byproduct of~\eqref{0okf3233}, we obtain that, for all~$T\le0$,
$$\phi_p(T)\in
\{u\in(0,1),\,u+v\ge0\}\subseteq\{\dot u=-u(u+v)\le0\}.$$
In particular
$$ u(t_n)-u(t_{n+1})=\int_{t_{n+1}}^{t_n}\dot u(\tau)\,d\tau\le0,$$
which contradicts~\eqref{bejv0565etP},
and consequently we have established the desired result.
\end{proof}

Now we show that the~$\omega$-limit set of any point lying
on the global center manifold coincides with the origin, according to the next result:

\begin{lemma}\label{lerptj:lemma}
Suppose that~$ac=1$.
If~$p\in\mathcal{M}$, then its~$\omega$-limit is~$\{(0,0)\}$.
\end{lemma}

\begin{proof}
We observe that, for every~$t>0$,
\begin{equation}\label{T0z}
\phi_p(t)\in[0,1]\times[0,1].
\end{equation}
By Remark \ref{rmk:dichotomy}, the other possibility would be $T_s(p)<+\infty$.
Therefore, to prove~\eqref{T0z}, we suppose, by contradiction,
that there exists~$t_0\ge0$ such that~$\phi_p({t_0})\in[0,1]\times\{0\}$,
that is~$v(t_0)=0$.

Since~$(0,0)$ is an equilibrium, it follows that~$u(t_0)\ne0$.
In particular,~$u(t_0)>0$ and accordingly $$\dot v(t_0)=-au(t_0)<0.$$

This means that~$v(t_0+\varepsilon)<0$ for all~$\varepsilon\in
(0,\varepsilon_0)$ for a suitable~$\varepsilon_0>0$.
Looking again at the velocity fields, this entails that~$\phi_p(t)\in(0,1)\times(-\infty,0)$
for all~$t>\varepsilon_0$.

Consequently,
$\phi_p(t)$ cannot approach the
straight line~$\rho v-au=0$ for~$t>\varepsilon_0$.

This, combined with Lemma~\ref{NOCPA}, says that the trajectory
emanating from~$p$ can never
approach the
straight line~$\rho v-au=0$ at the origin, in contradiction with the definition
of~${\mathcal{M}}$,
and thus the proof of~\eqref{T0z}
is complete.

{F}rom~\eqref{T0z} 
and the Poincar\'e-Bendixson Theorem (see e.g.~\cite{TESCHL}),
we deduce that the~$\omega$-limit of~$p$
can be either a cycle, or an equilibrium,
or a union of (finitely many)
equilibria and non-closed orbits connecting these equilibria.
We observe that the~$\omega$-limit of~$p$ cannot be a cycle, since~$\dot u$
has a sign in~$[0,1]\times[0,1]$.

Moreover, it cannot contain the sink~$(0,1)$, due to
Lemma~\ref{NOCPA}. Hence, the only possibility is that
the~$\omega$-limit of~$p$ coincides with~$(0,0)$,
which is the desired result.

We also remark that the $\alpha-$limit of $p$ cannot be a cycle in $[0,1]\times[0,1]$ since $\dot{u}$ has a sign. Moreover, it cannot contain $(0,1)$, which is a sink. If the orbit of $p$ is all contained in $[0,1]\times[0,1]$, then the $\alpha-$limit set cannot contain $(0,0)$, since this would generate a close trajectory. Therefore, the orbit of $\phi_p(t)$ must intersect the complementary set of $[0,1]\times[0,1]$.
\end{proof}

As a consequence of Lemma~\ref{lerptj:lemma} and the fact
that~$\dot u<0$ in~$(0,1]\times[0,1]$, we obtain the following statement:

\begin{corollary}\label{N7u43566r}
For~$ac=1$, every trajectory in~$\mathcal{M}$ has
the form~$\{\phi_p(t),\,t\in\R\}$, with
$$\lim_{t\to+\infty}\phi_p(t)=(0,0)$$
and there exists~$t_p\in\R$ such that $$\phi_p(t_p)\in\big(\{1\}\times[0,1]\big)
\cup\big([0,1]\times\{1\}\big).$$
\end{corollary}

The result in Corollary~\ref{N7u43566r} can be sharpened
in view of the following statement (which can be seen as the counterpart
of Proposition~\ref{lemma:M}
in the degenerate case~$ac=1$): namely, since the center manifold can in principle contain
many different trajectories
(see e.g. Figure~5.3 in~\cite{MR635782}), we provide a
tailor-made argument that excludes this possibility in the specific case
that we deal with.

\begin{proposition}\label{M:p045}
For $ac=1$, the set	
$\mathcal{M}$ contains one, and only one, orbit,
which is asymptotic to the origin as~$t\to+\infty$, and that can be written
as a graph $$\gamma:[0,u_{\mathcal{M}}]\to[0,v_{\mathcal{M}}],$$ for some $$(u_{\mathcal{M}},v_{\mathcal{M}})\in\big(\{1\}\times[0,1]\big)\cup
\big((0,1]\times\{1\}\big),$$ where~$\gamma$ is an increasing~$C^2$ function such that~$\gamma(0)=0$,
$\gamma(u_{\mathcal{M}})=v_{\mathcal{M}}$ and the graph of~$\gamma$ at the origin is tangent to the line~$\rho v-au=0$.
\end{proposition}

\begin{proof} First of all, we show that
\begin{equation}\label{LIMSDD-0}
{\mbox{$\mathcal{M}$ contains one, and only one, orbit.}}\end{equation}
Suppose, by contradiction, that~${\mathcal{M}}$
contains two different orbits, that we denote by~${\mathcal{M}}_-$
and~${\mathcal{M}}_+$.
Using Corollary~\ref{N7u43566r},
we can suppose that~${\mathcal{M}}_+$
lies above~${\mathcal{M}}_-$
and
\begin{equation}\label{CQPSKD}
\begin{split}&
{\mbox{the region~${\mathcal{P}}\subset[0,1]\times[0,1]$ contained}}\\&{\mbox{between~${\mathcal{M}}_+$
and~${\mathcal{M}}_-$ lies in~$\{\dot u<0\}$.}}\end{split}
\end{equation}
Consequently, for every~$p\in{\mathcal{P}}$,
it follows that
\begin{equation}\label{9ikfjty} \lim_{t\to+\infty}\phi_p(t)=(0,0).\end{equation}
In particular, we can take an open ball~$B\subset
{\mathcal{P}}$ in the vicinity of the origin, denote by~$\mu(t)$ the Lebesgue measure of $${\mathcal{S}}(t):=\{\phi_p(t),\;
p\in B\},$$ and write that~$\mu(0)>0$
and
\begin{equation}\label{LIMSDD} \lim_{t\to+\infty}\mu(t)=0.\end{equation}
We point out that~${\mathcal{S}}(t)$
lies in the vicinity of the origin for all~$t\ge0$, thanks to~\eqref{CQPSKD}.
As a consequence, for all~$t$,~$\tau>0$, changing variable
$$ y:=\phi_{x}(\tau)=x+\int_0^\tau \frac{d\phi_x(\theta)}{d\theta}\,d\theta=
x+\tau \frac{d\phi_x(0)}{dt}+O(\tau^2),
$$
we find that
\begin{eqnarray*}
\mu(t+\tau)&=&\int_{{\mathcal{S}}(t+\tau)}dy\\&=&
\int_{{\mathcal{S}}(t)}\big|\det \big(D_x \phi_x(\tau)\big)\big|\,dx\\&=&
\int_{{{\mathcal{S}}(t)}}\left|\det D_x \left(x+\tau \frac{d\phi_x(0)}{dt}
+O(\tau^2)
\right)\right|\,dx\\
&=&
\int_{{{\mathcal{S}}(t)}}\left( 1+\tau\,{\rm Tr}\left(D_x \frac{d\phi_x(0)}{dt}\right)
+O(\tau^2)
\right)\,dx
\\&=&\mu(t)+\tau
\int_{{{\mathcal{S}}(t)}} {\rm Tr}\left(D_x \frac{d\phi_x(0)}{dt}\right)\,dx
+O(\tau^2)
,
\end{eqnarray*}
where~${\rm Tr}$ denotes the trace of a~$(2\times2)$-matrix.

As a consequence,
\begin{equation}\label{090987t4oyyorfg4}
\frac{d\mu}{dt}(t)=
\int_{{{\mathcal{S}}(t)}} {\rm Tr}\left(D_x \frac{d\phi_x(0)}{dt}\right)\,dx.\end{equation}
Also, using the notation~$x=(u,v)$, we can write~\eqref{model} when~$ac=1$ in the form
$$ \frac{d\phi_x}{dt}(t)=
\dot x(t)=\left(
\begin{matrix}
\dot u(t)\\
\dot v(t)
\end{matrix}\right)=
\left(
\begin{matrix}
-u(t)(u(t)+v(t))\\
\rho v(t)(1-u(t)-v(t))-au(t)
\end{matrix}\right).
$$
Accordingly,
\begin{eqnarray*} &&D_x \frac{d\phi_x(0)}{dt}\\&&\qquad=
\left(
\begin{matrix}
-\partial_u\big(u(u+v)\big)&-\partial_v\big(u(u+v)\big)
\\
\partial_u\big(\rho v (1-u-v)-au\big)&\partial_v\big(\rho v (1-u-v)-au\big)
\end{matrix}\right),
\end{eqnarray*}
whence
\begin{equation}\label{0oj476ytgf9846}
\begin{split}&
{\rm Tr}\left(D_x \frac{d\phi_x(0)}{dt}\right)\\=\;&-\partial_u\big(u(u+v)\big)
+\partial_v\big(\rho v (1-u-v)-au\big)\\=\;&-2u-v+\rho(1-u-v)-\rho v\\=\;&\rho+O(|x|)
\end{split}
\end{equation}
for~$x$ near the origin.

As a result, recalling~\eqref{9ikfjty}, we can take~$t$ sufficiently large,
such that~${{{\mathcal{S}}(t)}}$ lies in a neighborhood of the origin, exploit~\eqref{0oj476ytgf9846}
to write that $${\rm Tr}\left(D_x \frac{d\phi_x(0)}{dt}\right)\ge\frac\rho2$$
and then~\eqref{090987t4oyyorfg4} to conclude that
$$ \frac{d\mu}{dt}(t)\ge \frac{\rho}2
\int_{{{\mathcal{S}}(t)}} dx=\frac{\rho}{2}\,\mu(t).$$
This implies that~$\mu(t)$ diverges (exponentially fast)
as~$t\to+\infty$, which is in contradiction with~\eqref{LIMSDD}.
The proof of~\eqref{LIMSDD-0}
is thereby complete.

Now, we check the other claims in the statement of Proposition~\ref{M:p045}.
The asymptotic property as~$t\to+\infty$ is a consequence of Corollary~\ref{N7u43566r}.
Also, the graphical property as well as the monotonicity
property of the graph follow from the fact that~${\mathcal{M}}\subset\{\dot u<0\}$.
The smoothness of the graph follows from the smoothness of the center manifold.
The fact that~$\gamma(0)=0$ and
$\gamma(u_{\mathcal{M}})=v_{\mathcal{M}}$ follow also from
Corollary~\ref{N7u43566r}. The tangency property at the origin is a consequence of
the tangency property of the center manifold to the center eigenspace.
\end{proof}

As a byproduct of the proof of Proposition~\ref{M:p045}
we also obtain the following information:

\begin{corollary}\label{lemma:Mdeg}
Suppose that~$ac=1$.
Then , we have that~$\dot{u}\leq 0$ in~$[0,1]\times [0,1]$,
and the curve~\eqref{curve:v'}
divides the square~$[0,1]\times[0,1]$ into two regions~$\mathcal{A}_1$
and~$\mathcal{A}_2$,
defined in~\eqref{DEFA12}.

Furthermore, the sets~$\mathcal{A}_1$
and~$\mathcal{A}_2$ are separated by the curve~$\dot{v}=0$, given by the graph of the continuous function~$\sigma$
given in~\eqref{f:sigma}.

Finally, there holds that
\begin{equation}\label{aggiun2BIS}
\mathcal{M}\subset \mathcal{A}_2.\end{equation} 
\end{corollary}

\section{Study of the dynamics}
\label{sec:deg2}

We observe that, by the Stable Manifold Theorem
and the Center Manifold Theorem, the statement in~(v) of Theorem~\ref{thm:dyn}
is obviously fulfilled. 

Hence, to complete the proof of Theorem~\ref{thm:dyn},
it remains to show that the statement in~(iv) holds true.
To this aim, exploiting
the useful pieces of
information in Propositions~\ref{lemma:M} and~\ref{M:p045}, we first give a characterization of the sets~$\mathcal{E}$ and~$\mathcal{B}$:

\begin{proposition} \label{prop:char}
	The sets in~\eqref{DEFB} and~\eqref{DEFE}
	are characterized by
	\begin{equation}\label{char:E}
	\begin{split}
	\mathcal{E}=&\big\{ (u,v)\in [0,u_{\mathcal{M}}]\times [0,1]\;{\mbox{ s.t. }}\; v < \gamma(u) \big\}\\&\qquad
\cup\,\big((u_{\mathcal{M}},1]\times[0,1]\big)
	\end{split}
	\end{equation}
	and
	\begin{equation}\label{char:B}
	\mathcal{B}=\Big\{ (u,v)\in [0,u_{\mathcal{M}}]\times [0,1]\;{\mbox{ s.t. }}\;  v>\gamma(u)   \Big\},
	\end{equation}
	where $\gamma$ is the parametrization of~$\mathcal{M}$,
	as given by Propositions~\ref{lemma:M} (when~$ac\neq1$)
	and~\ref{M:p045} (when~$ac=1$).
\end{proposition}

In~\eqref{char:E} we use the convention that $(u_{\mathcal{M}},1]\times[0,1]=\varnothing$ 
in the case~$u_{\mathcal{M}}=1$.
The sets~$\mathcal{E}$,~$\mathcal{B}$ can be visualized in two particular cases
 in Figure~\ref{fig:char}
 .

\begin{figure} 
	\begin{subfigure}{.5\textwidth}
		\centering
		\includegraphics[width=1.8\linewidth]{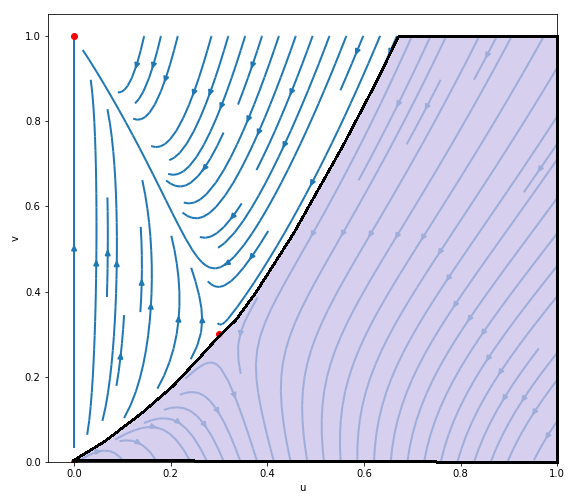}
		\caption{$a=0.8$,~$c=0.5$,~$\rho=2$}
	\end{subfigure}\\
	\begin{subfigure}{.5\textwidth}
		\centering
		\includegraphics[width=1.8\linewidth]{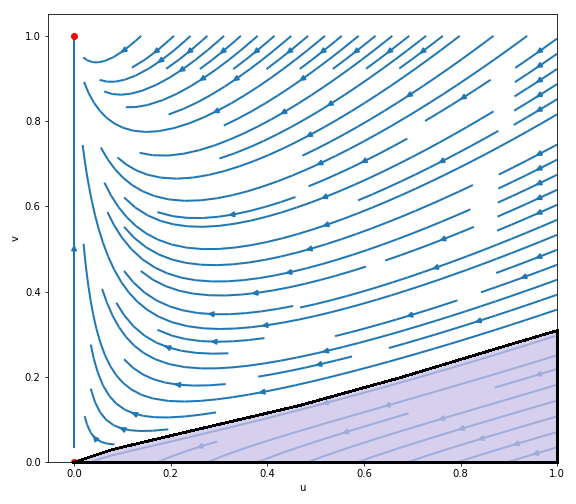}
		\caption{$a=0.8$,~$c=3$,~$\rho=2$}
	\end{subfigure}
	\caption{\em The figures show the phase portrait for the indicated values of the coefficients. In blue, the orbits of the points. The red dots show the equilibria. In violet, the set~$\mathcal{E}$.}
	\label{fig:char}
\end{figure}

\begin{proof}[Proof of Proposition~\ref{prop:char}]
	We let~$\gamma$ be the parametrization of~$\mathcal{M}$,
	as given by Propositions~\ref{lemma:M} (when~$ac\neq1$)
	and~\ref{M:p045} (when~$ac=1$).
	Let us call $\mathcal{X}$ and~$\mathcal{Y}$ the sets in~\eqref{char:E},~\eqref{char:B}, that is,
	\begin{eqnarray*}
		&&\mathcal{X}:= \big\{ (u,v)\in [0,u_{\mathcal{M}}]\times [0,1]\;{\mbox{ s.t. }}\; v < \gamma(u) \big\}\\&&\qquad\qquad\cup\,\big((u_{\mathcal{M}},1]\times[0,1]\big),\\
		{\mbox{ and }}&&
		\mathcal{Y}:= \big\{ (u,v)\in [0,u_{\mathcal{M}}]\times [0,1] \;{\mbox{ s.t. }}\;  v > \gamma(u) \big\}
	\end{eqnarray*}
	(the second set in the definition of $\mathcal{X}$ is understood to be $\varnothing$ 
	if~$u_{\mathcal{M}}=1$).  
	The goal is to prove that~$\mathcal{X}\equiv\mathcal{E}$ and~$\mathcal{Y}\equiv\mathcal{B}$. 
	We recall from Propositions~\ref{lemma:M} and~\ref{M:p045} that
	$(\mathcal{X},\mathcal{Y},\mathcal{M})$ is a partition of $[0,1]\times[0,1]$.
	Hence, since the sets $\mathcal{E}$, $\mathcal{B}$, $\mathcal{M}$ are disjoint,
	if we show that $\mathcal{X}\subseteq\mathcal{E}$ and $\mathcal{Y}\subseteq\mathcal{B}$
	we are done. 
	
	We first deal with the inclusion~$\mathcal{X}\subseteq\mathcal{E}$. 
	Namely, recalling~\eqref{DEFE},
	we will show~that
	\begin{equation}\label{alltra}
	{\mbox{all~$(u_0,v_0)\in\mathcal{X}$ have~$T_s(u_0,v_0)<+\infty$.}}\end{equation} 
	For this, we first notice that, gathering together~\eqref{aggiunto}, \eqref{primaBIS},~\eqref{prima2BIS}, and~\eqref{pouyi86}, we find that
	\begin{equation}\label{07960789djiewf2}
	{\mbox{no closed orbit exists in~$[0,1]\times[0,1]$}}\end{equation}
	(in the case~$0<ac<1$, and this holds true in the case~$ac\ge1$ where~$\dot u$ has a sign). 
	
	In addition, 
	\begin{equation}\label{07960789djiewf}\begin{split}&
	{\mbox{the~$\omega$-limit of any point in~$\mathcal{X}$, if it exists,}}\\ &{\mbox{does not contain equilibria.}}\end{split}\end{equation}
	Indeed, by Propositions~\ref{lemma:M} (when~$(ac\neq1$)
	and~\ref{M:p045} (when~$ac=1$),
	we have that~$\gamma(0)=0<1$, and therefore~$(0,1)\notin \overline{\mathcal{X}}$.
	Moreover, if~$ac<1$, 
	a trajectory in~$\mathcal{X}$ cannot converge to~$(u_s, v_s)$, since~$\mathcal{X}$ does not
	contain points of the corresponding stable manifold~$\mathcal{M}$, nor to~$(0,0)$,
	since this is a repulsive equilibrium, cf.~Theorem~\ref{thm:dyn}.
	If instead~$ac\geq 1$, then trajectories cannot converge to~$(0,0)$, since~$\mathcal{X}$ does not
	contain points of~$\mathcal{M}$, which in this case coincides with the stable or center manifold of $(0,0)$.
	These observations complete the proof of~\eqref{07960789djiewf}.
	
	Moreover, by Lemma \ref{lemma:exit}, no trajectory can exit $\mathcal{X}$, since $\partial \mathcal{X} \setminus \partial ([0,1]\times [0,1])$ coincide with a trajectory (which contains also the point at the boundary $(u_{\mathcal{M}},\gamma (\mathcal{M}) )$).
	
	{F}rom the latter observation,~\eqref{07960789djiewf2},~\eqref{07960789djiewf} and the 
	Poincar\'e-Bendixson Theorem (see e.g.~\cite{TESCHL}),
	we have that
	every trajectory with initial point$$(u_0,v_0)\in\mathcal{X}$$ cannot remain in $[0,1]\times[0,1]$
	for all $t>0$,
	that is, $T_s(u_0,v_0)<+\infty$.
	This shows that $\mathcal{X}\subseteq\mathcal{E}$.

	We now claim that
	\begin{equation}\label{po0584yiugherghegkbis}
	\big((u_{\mathcal{M}},1]\times[0,1]\big)\subseteq\mathcal{E}.
	\end{equation}
	To this end, we observe that there are neither cycles nor equilibria in~$(u_{\mathcal{M}},1]\times[0,1]$, and therefore
	we can use the Poincar\'e-Bendixson Theorem (see e.g.~\cite{TESCHL}) to conclude that for any point~$(u_0,v_0)\in(u_{\mathcal{M}},1]\times[0,1]$, its $\omega-$limit set, if exists, must satisfy 
	\begin{equation}\label{1602}
		\omega(u_0,v_0)\cap ((u_{\mathcal{M}},1]\times[0,1]) = \varnothing.
	\end{equation}
	Now, a trajectory can exit~$(u_{\mathcal{M}},1]\times[0,1]$ only
	from the side~$\{u_{\mathcal{M}}\}\times(0,1)$, and entering the set~$\mathcal{X}$, and therefore~\eqref{po0584yiugherghegkbis}
	is a consequence of~\eqref{alltra} in this case.
	
	Thanks to \eqref{1602}, the only other possibility is that $T_s(u_0,v_0)<+\infty$,
	giving directly~\eqref{po0584yiugherghegkbis}.
	
{F}rom~\eqref{alltra} and~\eqref{po0584yiugherghegkbis} we obtain~\eqref{char:E},
	as desired.
	
	We now prove~\eqref{char:B}, namely we show that
	\begin{equation}\label{089ghvbdflpoiuytr}\begin{split}
&	{\mbox{for all~$(u_0,v_0)\in \mathcal{Y}$ we have that}}\\&{\mbox{$(u(t), v(t))\to (0,1)$
			as~$t\to +\infty$.}}\end{split}\end{equation}
	Hence, again by Lemma \ref{lemma:exit}, we have that no trajectory exits $\mathcal{Y}$, since no trajectory can cross~$\mathcal{M}$. 
	
	Also, $$\mathcal{Y }\cap ((0,1)\times \{0\})=\varnothing.$$	
	Thus, for $(u_0,v_0)\in\mathcal{Y}$ it must be $T_s(u_0,v_0)=+\infty$ and
	\begin{equation}\label{1628}
		\omega(u_0,v_0)\subset \mathcal{Y}.
	\end{equation}	
	Let us now investigate $\omega(u_0,v_0)$.		
	To this end, we observe that~$(u_s, v_s)$ (if~$0<ac<1$) and~$(0,0)$ are not in~$\mathcal{Y}$.
	Moreover,
	no trajectory starting in~$\mathcal{Y}$ converges to~$(u_s, v_s)$ (if~$0<ac<1$), nor to~$(0,0)$,
	since~$\mathcal{Y}$ does not contain points on~$\mathcal{M}$.
	
	In addition, recalling~\eqref{07960789djiewf2}, we have that
	there are no limit cycles in~$\mathcal{Y}$. 
	As a consequence, by the
	Poincar\'e-Bendixson Theorem (see e.g.~\cite{TESCHL}),
	we have that
	every trajectory starting in~$\mathcal{Y}$ converges to~$(0,1)$, proving \eqref{089ghvbdflpoiuytr}.
	Hence, the proof of~\eqref{char:B} is complete as well.
\end{proof}

With this, we are now able to complete the
proof of Theorem~\ref{thm:dyn}:

\begin{proof}[Proof of (iv) of Theorem~\ref{thm:dyn}]
	The statement in~(iv) of Theorem~\ref{thm:dyn}
	is a direct consequence of the parametrization
	of the manifold~$\mathcal{M}$,
	as given by Proposition~\ref{lemma:M} for~$ac\neq 1$ and by Proposition~\ref{M:p045} for~$ac=1$,
	and the characterization of the sets~$\mathcal{B}$
	and~$\mathcal{E}$, as given by Proposition~\ref{prop:char}.
\end{proof}

\chapter{Parameters dependence}\label{ss:dependence}

\begin{center}
\begin{minipage}{25em}
\noindent{\bf Abstract of Chapter~\ref{ss:dependence}.}
{\sl Here we analyze the bifurcation patterns of the model under consideration in dependence of the structural parameters. Since the system does not possess a variational structure, a bespoke analysis is needed for this.}\end{minipage}\end{center}
\bigskip\bigskip\bigskip\bigskip\bigskip\bigskip

In this chapter we discuss the dependence on the
parameters involved in the system~\eqref{model}.

The dynamics of the system in~\eqref{model} depends qualitatively
only on~$ac$, but of course the position of the saddle equilibrium
and the size and shape of the basins of attraction depend quantitatively
upon all the parameters. Here
we perform a deep analysis on each parameter separately.

We notice that the system in~\eqref{model} does
not present a variational structure, due to the presence
of the
terms~$-acu$ in the first equation and~$-au$
in the second one, that are of first order in~$u$.
Thus, the classical methods of the calculus of variations cannot
be used and we have to make use of ad-hoc arguments,
of geometrical flavour.

\section{Dependence on the parameter~$c$}

We start by studying the dependence on~$c$, that represents the losses
(soldier death and missing births) caused
by the war for the first population with respect to the second one.
In the following proposition, we will express the dependence on~$c$
of the basin of attraction~$\mathcal{E}$
in~\eqref{DEFE} by writing explicitly~$\mathcal{E}(c)$.

\begin{proposition}[Dependence of the dynamics on~$c$] \label{prop:bhvc}
With the notation in~\eqref{DEFE}, we have that
	\begin{itemize}
		\item[(i)] If~$0< c_1 < c_2$, then 
		$\mathcal{E}(c_2) \subset \mathcal{E}(c_1)~$. 
		
		\item[(ii)] 	It holds that
		\begin{equation}\label{242}
		\underset{c>0}{\bigcap} \, \mathcal{E}(c)= (0,1]\times \{0\}.
		\end{equation}
		
	\end{itemize}
\end{proposition}

We remark that the behavior for~$c$
small is included by~(i) of Theorem~\ref{thm:dyn}: in this case,
there is a saddle point $(u_s,v_s\in(0,1)\times(0,1)$ and for all $c>0$ we get~$\mathcal{E}(c)\neq (0,1]\times [0,1]$. We do not investigate the case $c=0$, which results in a special case where the saddle $(u_s,v_s)$ and the sink $(0,1)$ collapse in a point with one negative and one zero eigenvalue.

On the other hand, as~$c\to +\infty$, the set~$\mathcal{E}(c)$ gets
smaller and smaller until the first population has no
chances of victory if the second population has a positive size. 

As one would expect, Proposition~\ref{prop:bhvc}
tells us that
the greater the cost of the war for the first population,
the fewer possibilities of victory there are for it.

\begin{proof}[Proof of Proposition~\ref{prop:bhvc}]
{\emph{(i)}} 
We take~$c_2 > c_1 > 0$. 
Now, in the notation
of Propositions~\ref{lemma:M} (if~$ac\neq1$)
and~\ref{M:p045} (if~$ac=1$), thanks to the characterization
in~\eqref{char:E}, the inclusion in~(i) reduces to 
\begin{equation}\label{19071}
\gamma_{c_1}(u)>\gamma_{c_2}(u)
\quad \text{for any} \ u\in (0, u_{\mathcal{M}}^{c_1}].
\end{equation}
Notice that this gives us automatically $ u_{\mathcal{M}}^{c_1}< u_{\mathcal{M}}^{c_2}$, and the inclusion $$(\mathcal{E}(c_2) \cap (u_{\mathcal{M}}^{c_1}, 1]\times [0,1] )\subset (\mathcal{E}(c_1) \cap (u_{\mathcal{M}}^{c_1}, 1]\times [0,1])  $$ is trivial.

Suppose by the absurd that \eqref{19071} is not true, hence there is some value $\bar{u}\in(0, u_{\mathcal{M}}^{c_1}]$ such that $\gamma_{c_1}(\bar u)<\gamma_{c_2}(\bar u)$.  
Let us now consider a value $\bar{v} \in (\gamma_{c_1}(\bar u),\gamma_{c_2}(\bar u))$. 
Then,  $(\bar u,\bar v)\in\mathcal{E}(c_2) \setminus \mathcal{E}(c_1)$. 
Observe that $(0,1]\times \{0\} \subset \mathcal{E}(c_1)$ and $T_s(\bar u, \bar v)<+\infty$.
In particular, 
\begin{equation}\label{20151}\begin{split}
&	{\mbox{there must be a time $T\in (0, T_s(\bar{u}, \bar{v}))$}}\\ &{\mbox{such that the trajectory $\phi_{(\bar u, \bar v)}^{c_2}(t)$ enters $\mathcal{E}(c_1)$.}}\end{split}
\end{equation}

Moreover, by Remark \ref{rmk:dichotomy} (see also Remark \ref{rmk:enter}), the trajectory~$\phi_{(\bar u, \bar v)}^{c_2}(t)$ does not leave $[0,1]\times [0,1]$ for $t\in(0, T_s(\bar{u}, \bar{v}))$, therefore it enters~$\mathcal{E}(c_1)$ from the side given by $\mathcal{M}(c_1)$.

Let us now compute the normal derivative to $\mathcal{M}(c_1)$ at the point $(\tilde{u}, \tilde v) = \phi_{(\bar u, \bar v)}^{c_2}(T)$. 
We use the notation
\begin{align*}
	&\dot{u}_1 := \tilde{u}(1-ac_1-\tilde{u}-\tilde{v}),  \\
	&\dot{u}_2 := \tilde{u}(1-ac_2-\tilde{u}-\tilde{v})  \\
	{\mbox{and }}\qquad&\dot{v} := \rho \tilde{v}(1-\tilde{u}-\tilde{v})-a\tilde{u}.
\end{align*}
Since $\mathcal{M}(c_1)$ is the graph of $\gamma_{c_1}$ and by the properties of $\mathcal{M}$ given in Corollary \ref{lemma:M1} for $0<ac<1$, Corollary \ref{lemma:M2} for $ac>1$, and Proposition \ref{M:p045} for $ac=1$, the normal vector $\nu$ to $\mathcal{M}(c_1)$ at $(\tilde{u}, \tilde v)$ is given by:
\begin{enumerate}
	\item if $\dot{v}>0$: \begin{equation*}
	\nu :=\left( \frac{-\dot{v}}{||(\dot{u}_1, \dot{v})||}, \frac{\dot{u}_1}{||(\dot{u}_1, \dot{v})||}  \right).
	\end{equation*}
	\item if $\dot{v}< 0$: \begin{equation*}
	\nu:=\left( \frac{\dot{v}}{||(\dot{u}_1, \dot{v})||}, \frac{-\dot{u}_1}{||(\dot{u}_1, \dot{v})||}  \right).
	\end{equation*}
	\item if $\dot{v}=0$ (that is, when we are at $(u_s^{c_1}, v_s^{c_1})$ for $0<ac<1$), by Proposition \ref{lemma:M}: \begin{equation*}
	\nu := \frac{1}{\sqrt{1+c_1^2}}\left(-1,c_1   \right).
	\end{equation*}
\end{enumerate}

Now, for each case, we compute the product between the normal vector and the direction $(\dot{u}_2, \dot{v})$ of the trajectory $\phi_{(\bar u, \bar v)}^{c_2}(t)$ at $(\tilde{u}, \tilde{v})$. 
We get
\begin{enumerate}
	\item if $\dot{v}>0$, then 
	\begin{equation*}
		\frac{1}{||(\dot{u}_1, \dot{v})||} \dot{v} \left(\dot u_1-\dot u_2 \right)= \frac{a \tilde{u}\dot v(c_2-c_1) }{||(\dot{u}_1, \dot{v})||} >0;
	\end{equation*}
	\item if $\dot{v}<0$, then
	\begin{equation*}
	\frac{1}{||(\dot{u}_1, \dot{v})||} \dot{v} \left(\dot u_2-\dot u_1 \right)= \frac{a \tilde{u}\dot v(c_1-c_2) }{||(\dot{u}_1, \dot{v})||} >0;
	\end{equation*}
	\item if $\dot{v}=0$, owning the formula \eqref{usvs} for $(u_s^{c_1}, v_s^{c_2})$, then
	\begin{equation*}
		\frac{1}{\sqrt{1+c_1^2}}(-\dot{u}_2+c \dot{v})=\frac{1}{\sqrt{1+c_1^2}}\frac{1-ac_1}{1+\rho c_1}\rho c_1 (ac_2-a c_1) >0.
	\end{equation*}
\end{enumerate}
Since the scalar product of the normal to $\mathcal{M}(c_1)$ and the trajectory is always positive, the trajectory cannot enter in $\mathcal{E}(c_1)$ (see Remark \ref{rmk:trajectory}), contradicting \eqref{20151}. Hence, \eqref{19071} holds true.

	\medskip
	
	{\emph{(ii)}} We first show that for all~$\varepsilon>0$ there exists~$c_{\varepsilon}>0$ such that for all~$c\ge c_{\varepsilon}$ it holds that
	\begin{equation} \label{859}
		\mathcal{E}(c) \subset \big\{ 
		(u,v)\in [0,1]\times [0,1]\;{\mbox{ s.t. }}\; v < \varepsilon u    \big\}.
	\end{equation}
	The inclusion in~\eqref{859} is also equivalent to 
	\begin{equation} \label{255}
	 \big\{ (u,v)\in [0,1]\times [0,1]\;{\mbox{ s.t. }}\; v > 
	  \varepsilon u   \big\} \subset \mathcal{B}(c),
	\end{equation}
	and the strict inequality is justified by the fact that~$\mathcal{E}(c)$
	and~$\mathcal{B}(c)$ are separated by~$\mathcal{M}$, according
	to Proposition~\ref{prop:char}.
	We now establish the inclusion in~\eqref{255}.
	For this, let
\begin{equation}\label{255BIS}
\mathcal{T}_{\varepsilon}:= \big\{ 
	(u,v)\in [0,1]\times [0,1] \;{\mbox{ s.t. }}\; v \geq  \varepsilon u    \big\}. \end{equation}
Now, we can choose~$c$ large enough such that the condition~$ac\geq 1$
is fulfilled. In this way, thanks to~(ii) and~(iii)
of Theorem~\ref{thm:dyn},
the only equilibria are the points~$(0,0)$ and~$(0,1)$.

Now, the component of the
velocity in the inward normal direction to~$\mathcal{T}_{\varepsilon}$
on the side~$\{v=\varepsilon u\}$ is given by
	\begin{eqnarray*}
	&&(\dot u,\dot v)\cdot \frac{(-\varepsilon,1)}{\sqrt{1+\varepsilon^2}}=
	\frac{\dot{v}-\varepsilon \dot{u}}{\sqrt{1+\varepsilon^2}}
	\\&&\qquad =\frac1{{\sqrt{1+\varepsilon^2}}}\big(
	  \rho v(1-u-v) -au -\varepsilon u(1-u-v) + \varepsilon acu \big)\\
		&&\qquad=\frac1{{\sqrt{1+\varepsilon^2}}}\big[
		 (\rho v-\varepsilon u)(1-u-v)  +  (\varepsilon c -1)au\big]\\
	&&\qquad=\frac1{{\sqrt{1+\varepsilon^2}}}\big[
		 (\rho \varepsilon u-\varepsilon u)(1-u-\varepsilon u)  +  (\varepsilon c -1)au\big]	 ,
	\end{eqnarray*}
	that is positive for
	\begin{equation}\label{possibly}
		c > c_{\varepsilon} := \frac{2\varepsilon(1+\rho) +a}{\varepsilon a}.
	\end{equation}
	This says that no trajectory in~$\mathcal{T}_{\varepsilon}$ can exit~$\mathcal{T}_{\varepsilon}$ from the side~$\{v=\varepsilon u\}$ (see Remark \ref{rmk:enter}).
	
	The other parts of~$\partial \mathcal{T}_{\varepsilon}$ belong to~$\partial(
	(0,1)\times(0,1))$
	but not to~$[0,1]\times \{0 \}$.
	As a consequence, no trajectory can exit $\mathcal{T}_{\varepsilon}$, so
\begin{equation}\label{po123097}
{\mbox{every trajectory in~$\mathcal{T}_{\varepsilon}$ is well defined for all~$t\ge0$ and belongs to~$\mathcal{T}_{\varepsilon}$.}}\end{equation}	
{F}rom this,~\eqref{07960789djiewf2} and the 
Poincar\'e-Bendixson Theorem (see~\cite{TESCHL}), we conclude that
the~$\omega$-limit of any trajectory starting in~$\mathcal{T}_{\varepsilon}$
can be either an equilibrium or a union of (finitely many)
equilibria and non-closed orbits connecting these equilibria.

Now, we claim that, possibly taking~$c$ larger in~\eqref{possibly},
\begin{equation}\label{po1230972}
\mathcal{M}\subset \big([0,1]\times[0,1]\big)\setminus\mathcal{T}_{\varepsilon}.
\end{equation}
Indeed, suppose by contradiction that there exists~$(\tilde u,\tilde v)\in\mathcal{M}
\cap\mathcal{T}_{\varepsilon}$. Then, in light of~\eqref{po123097}, a trajectory passing
through~$(\tilde u, \tilde v)$ and converging to~$(0,0)$ has to be entirely contained
in~$\mathcal{T}_{\varepsilon}$.

On the other hand, by Propositions~\ref{lemma:M} and~\ref{M:p045},
we know that at~$u=0$ the manifold~$\mathcal{M}$ is tangent to the
line~$(\rho-1+ac)v-au=0$.  
Hence, if we choose~$c$ large enough such that
$$ \frac{a}{\rho-1+ac}<\varepsilon,$$
we obtain that this line is below the line~$v=\varepsilon u$, thus reaching
a contradiction. This establishes~\eqref{po1230972}.

{F}rom~\eqref{po1230972}, we deduce that,
given~$(\tilde u,\tilde v)\in\mathcal{T}_{\varepsilon}$,
and denoting~$\omega_{(\tilde u,\tilde v)}$ the~$\omega$-limit of~$(\tilde u,\tilde v)$,
\begin{equation}\label{podnjewbf215}
\omega_{(\tilde u,\tilde v)}\neq \{(0,0)\},
\end{equation}
provided that~$c$ is taken large enough.

Furthermore,~$\omega_{(\tilde u,\tilde v)}$ cannot consist of the two equilibria~$(0,0)$
and~$(0,1)$ and non-closed orbits connecting these equilibria, due to the fact that~$(0,1)$
is a sink. 

As a consequence of this and~\eqref{podnjewbf215},
we obtain that~$\omega_{(\tilde u,\tilde v)}=\{(0,1)\}$
for any~$(\tilde u,\tilde v)\in\mathcal{T}_{\varepsilon}$, provided that~$c$
is large enough.

Thus, recalling~\eqref{DEFB} and~\eqref{255BIS}, this proves~\eqref{255},
and therefore~\eqref{859}.
	
Now, using~\eqref{859}, we see that for every~$\varepsilon>0$,
\begin{eqnarray*}\underset{c>0}{\bigcap}  \mathcal{E}(c) &\subseteq& \mathcal{E}(c_{\varepsilon})
\\&\subseteq& \big\{
	 (u,v)\in [0,1]\times [0,1]\; {\mbox{ s.t. }}\; v < \varepsilon u   \big \}.
\end{eqnarray*}
Accordingly,
	\begin{eqnarray*}
	\underset{c>0}{\bigcap}  \mathcal{E}(c)& \subseteq  &\underset{\varepsilon>0}{\bigcap}  \big\{
	 (u,v)\in [0,1]\times [0,1]\; {\mbox{ s.t. }}\; v < \varepsilon u   \big \}\\& =& (0,1] \times \{0\},
	\end{eqnarray*}
	which gives~\eqref{242}, as desired.
\end{proof}

\section{Dependence on the parameter~$\rho$}

Now we analyze the dependence of the dynamics on the parameter~$\rho$, that is the fitness of the second population~$v$ with respect to the fitness of the first one~$u$.

In the following proposition, we will make it explicit the dependence on~$\rho$ by
writing~$\mathcal{E}(\rho)$ and~$\mathcal{B}(\rho)$.

\begin{proposition}[Dependence of the dynamics on~$\rho$] \label{prop:bhvA}
With the notation in~\eqref{DEFB} and~\eqref{DEFE}, we have that
	\begin{itemize}
		\item[(i)]
		When~$\rho=0$, for any~$v \in [0,1]$ the point~$(0,v)$  is an equilibrium. If~$v\in(1-ac,1]$, then it corresponds to a strictly
negative eigenvalue and a null one.
If instead~$v\in[0,1-ac)$, then it corresponds to a strictly
positive eigenvalue and a null one.
		
		Moreover, 
		\begin{equation}\label{first}
		\mathcal{B}(0)=  \varnothing,\end{equation} and for any~$\varepsilon<  ac/2~$ and
		any~$\delta< \varepsilon c/2$  we have that
\begin{equation}\label{first2}
[0,1]\times [0,1-ac) \subseteq \mathcal{E}(0)  \subseteq \mathcal{T}_{\varepsilon, \delta}  ,
\end{equation} where
		\begin{equation}\label{TEPS}
		\mathcal{T}_{\varepsilon, \delta}:=\big\{ (u,v)\in[0,1] \times [0,1]\;
		{\mbox{ s.t. }}\; \delta v-\varepsilon u \leq \delta(1-\varepsilon) \big\}.
		\end{equation}
		\item[(ii)]
		For any~$\varepsilon<  ac/3~$ and any~$\delta< \varepsilon c/2$ it holds that 
		\begin{equation*}
		\underset{{0<\rho<a/3}}{\bigcup} \mathcal{E}(\rho) \subseteq  \mathcal{T}_{\varepsilon, \delta} ,
		\end{equation*}
		where~$ \mathcal{T}_{\varepsilon, \delta}$ is defined in~\eqref{TEPS}.
		\item[(iii)]
		It holds that 
		\begin{equation}\label{ir4t4y4y}
		\underset{\omega>0}{\bigcap} \, \underset{{\rho>\omega}}{\bigcup} \mathcal{E}(\rho) =  (0,1] \times \{0\}.
		\end{equation}	
	\end{itemize}
\end{proposition}

We point out that
the case~$\rho =0$ is not comprehended in Theorem~\ref{thm:dyn}.
As a matter of fact, the dynamics of this case is qualitatively very different from all the other cases. Indeed, for~$\rho =0$ the domain~$[0,1] \times [0,1]$ is not divided into~$\mathcal{E}$ and~$\mathcal{B}$, since more attractive equilibria appear on the line~$\{0\}\times(0,1)$. Thus, even if the second population cannot grow, it still has some chance of victory. 

As soon as~$\rho~$ is positive, on the line~$u=0$ only the equilibrium~$(0,1)$ survives, and it attracts all the points that were going to the line~$\{0\}\times(0,1)$ for~$\rho =0$.

When~$\rho \to +\infty$, the basin of attraction of~$(0,1)$ tends to invade the domain, thus the first population tends to have almost no chance of victory and the second population tends to win.
However, the dependence on the parameter~$\rho~$ is not monotone as one could think, at least not in~$[0,+\infty)\times[0,+\infty)$.

Indeed, by performing some simulation, one could find some values~${\rho}_1$ and~${\rho}_2$, with~$0<{\rho}_1 < {\rho}_2$, and a point~$(u^*, v^*)\in [0,+\infty)\times[0,+\infty)$ such that~$(u^*, v^*) \notin \mathcal{E}({\rho}_1)$ and~$(u^*, v^*) \in \mathcal{E}({\rho}_2)$, see~Figure~\ref{fig:trajrho}. 

\begin{figure} 
	
	\begin{subfigure}{0.8\textwidth} \label{A3to01}
		\includegraphics[width=16cm, height=10cm]{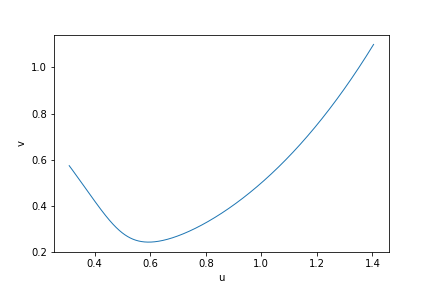} 
		\caption{$a=0.2$,~$c=0.1$, and~$\rho=3$}
		\label{fig:subim1}
	\end{subfigure}\\
	\begin{subfigure}{0.8\textwidth} \label{A7notto01}
		\includegraphics[width=16cm, height=10cm]{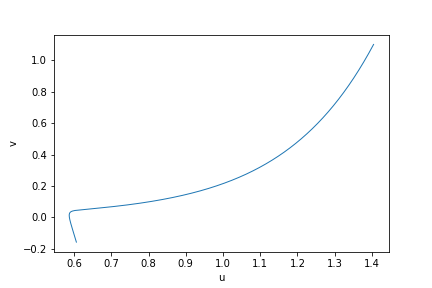}
		\caption{$a=0.2$,~$c=0.1$, and~$\rho=7$}
		\label{fig:subim2}
	\end{subfigure}
	
	\caption{\em Figure (a) and Figure (b) show the trajectory starting from the point~$(u_0,v_0)=(1.4045, 1.1)$ for~$\rho=3$ and~$\rho=7$ respectively. For~$\rho=3$ the trajectory leads to the equilibrium~$(0,1)$, so~$(u_0,v_0)\notin \mathcal{E}(\rho=3)$, while for~$\rho=7$ the second population goes to extinction in finite time, so~$(u_0,v_0)\in \mathcal{E}(\rho=7)$. 
	}
	\label{fig:trajrho}
\end{figure}

This means that, sometimes, a big value of fitness for the second population may lead to extinction while a small value brings to victory. This is counterintuitive, but can be easily explained: the parameter~$\rho~$ is multiplied by the term~$1-u-v$, that is negative past the counterdiagonal of the square~$[0,1]\times[0,1]$. So in the model~\eqref{model}, as well as in any model of Lotka-Volterra type, the population that grows faster is also the one that suffers more the consequences of overpopulation. Moreover, the usual dynamics of Lotka-Volterra models is altered by the presence of the term~$-au$, and this leads to the lack of monotonicity that we observe.
\medskip

We now give the proof of Proposition~\ref{prop:bhvA}:

\begin{proof}[Proof of Proposition~\ref{prop:bhvA}]
	{\emph{(i)}}
	For~$\rho=0$, the equation~$\dot{v}=0$ collapses to~$u=0$. Since for~$u=0$ also the equation~$\dot{u}=0$ is satisfied,
	each point on the line~$u=0$ is an equilibrium. 
	
	Calculating the eigenvalues for the points~$(0, \tilde{v})$, with~$\tilde v\in[0,1]$,
	using the Jacobian matrix in~\eqref{Jmatrix},
	one gets the values~$0$ and~$1-ac-\tilde{v}$.
	Accordingly, this entail that, if~$\tilde{v} < 1-ac$, the point~$(0, \tilde{v})$ corresponds to a strictly negative eigenvalue
	and a null one, while if~$\tilde{v}>1-ac$
then~$(0, \tilde{v})$ corresponds to a strictly negative
eigenvalue and a null one.
	These considerations proves the first statement in~(i).

We now study the behavior of the points in $[0,1]\times[0, 1-ac)$. Notice that the only part of its boundary that is inside $(0,1)\times(0,1)$ is the side~$(0,1]\times\{1-ac\}$.
We notice also that in the whole square~$(0,1]\times[0,1]$ we have
$$\dot{v}= -au < 0,$$ so trajectories cannot exit $[0,1]\times[0, 1-ac)$ (see Remark \ref{rmk:enter}).

 This also gives that there is no trajectory that can go to~$(0,1)$, and there is no cycle.
In particular this implies~\eqref{first}.

Thus a trajectory starting in $[0,1]\times[0, 1-ac)$ either converges to one of the equilibria on the
	side~$\{0\}\times[0,1]$, or has a finite stopping time.
	
In particular, since~$\{0\}\times [0,1-ac)$ consists of repulsive equilibria, we have that
$$ [0,1]\times[0, 1-ac)\subseteq \mathcal{E}(0),~$$
that is, trajectories starting in~$[0,1]\times[0, 1-ac)$ go to the extinction of~$v$.
This proves the first inclusion in~\eqref{first2}.

To prove the second inclusion in~\eqref{first2},
we first show that
\begin{equation}\label{pot43 yb9 49y}\begin{split}&
{\mbox{points in~$\big([0,1]\times[0,1]\big)\setminus\mathcal{T}_{\varepsilon, \delta}$
	are mapped}}\\&{\mbox{into~$\big([0,1]\times[0,1]\big)\setminus\mathcal{T}_{\varepsilon, \delta}$ itself.}}\end{split}\end{equation}
Indeed, 
on the line~$\{\delta v-\varepsilon u = \delta(1-\varepsilon)\}$ we have that the inward-pointing normal derivative is given by
	\begin{equation}\begin{split}\label{fagiano}&
	(\dot{u},\dot{v})\cdot\frac{(-\varepsilon,\delta)}{\sqrt{\varepsilon^2+\delta^2}}\\
=\;&
\frac1{\sqrt{\varepsilon^2+\delta^2}}	\big(\delta \dot{v}- \varepsilon \dot{u}\big)\\
 =\;&\frac1{\sqrt{\varepsilon^2+\delta^2}}\big( -\delta a u - \varepsilon u(1-u-v) +\varepsilon ac u\big)\\ 
	=\;&
	\frac{ u}{\sqrt{\varepsilon^2+\delta^2}}\left[\varepsilon\left(-1+ac+u+\frac{\varepsilon}{\delta}u +1-\varepsilon\right)-\delta a\right] \\
	=\;&\frac{ 1}{\sqrt{\varepsilon^2+\delta^2}}\left[u^2\left( 1+\frac{\varepsilon}{\delta}\right) + u(\varepsilon ac -\delta a -\varepsilon^2)\right].
	\end{split}\end{equation}
	The first term is always positive; the second one is positive for the choice \begin{equation*}
	\delta < \frac{\varepsilon c}{2} \quad {\mbox{ and }}\quad\varepsilon< \frac{ac}{2}.
	\end{equation*} 
	Hence, under the assumption in~(i), on the line~$\{\delta v-\varepsilon u = \delta(1-\varepsilon)\}$ the inward-pointing normal derivative is positive, which implies
	that no trajectories in~$\big([0,1]\times[0,1]\big)\setminus\mathcal{T}_{\varepsilon, \delta}$ can exit from~$\big([0,1]\times[0,1]\big)\setminus\mathcal{T}_{\varepsilon, \delta}$ (see Remark \ref{rmk:enter}). This establishes~\eqref{pot43 yb9 49y}.
	
As a consequence of~\eqref{pot43 yb9 49y}, we obtain also the second
inclusion~\eqref{first2}, as desired.
	\medskip
	
{\emph{(ii)}} 
	We claim that
	\begin{equation}\label{poi8562eq2dvfkgjlkuykyuasdawre}
	\big([0,1]\times[0,1]\big)\setminus\mathcal{T}_{\varepsilon, \delta}
 \subseteq \mathcal{B}(\rho),\end{equation}
	for all~$0<\rho< a/3$.
	To this end, we observe that, in order to determine
	the sign of the inward pointing normal derivative on the side~$\{\delta v -\varepsilon u = \delta(1-\varepsilon)\}$, by~\eqref{fagiano} we have to 
check that~$\delta \dot{v}- \varepsilon\dot{u}\ge0$. In order to simplify the calculation, we use the change of coordinates~$x:=u$ and~$y:=1-v$.
In this way, one needs to verify that
\begin{equation}\label{oewut45b7v85476}
\delta \dot{y}+\varepsilon \dot{x} < 0 \qquad{\mbox{ 
on the line}}\quad\{\delta y + \varepsilon x = \delta \varepsilon\}.\end{equation} For this, we compute
	\begin{equation}
	\label{cneq0}
	\begin{split}&
	\delta \dot{y}+\varepsilon \dot{x} \\=\;& \delta \rho (y-1)(y-x)+\delta a x + \varepsilon x (y-x) - \varepsilon acx \\
	 =\;& -\delta \rho (1-y) y+x \big(
 \delta \rho (1-y) +\delta a + \varepsilon (y-x) -\varepsilon ac   \big) \\
	 =\;& -\delta \rho (1-y) y +  x \big(  \delta \rho  -\delta \rho y 
+\delta a + \varepsilon y-\varepsilon x -\varepsilon a c   \big)\\
\le\;&  x \big(  \delta \rho  -\delta \rho y 
+\delta a + \varepsilon y-\varepsilon x -\varepsilon a c   \big)  .
	\end{split}
	\end{equation}
Now we choose~$\delta<\varepsilon c / 2$ and we
recall that~$\rho < a/3$. Moreover, we notice
that
$$y= \varepsilon-\frac{\varepsilon}{\delta}x\le\varepsilon,
$$ 
and therefore~$\varepsilon y \leq \varepsilon^2$. Thus, we have that
	\begin{eqnarray*}
-\delta \rho y + \delta \rho  +\delta a + \varepsilon y-\varepsilon x 
-\varepsilon a c  & \le& \frac{\varepsilon ac }{6} + \frac{\varepsilon ac }{2} +
 \varepsilon^2
-\varepsilon a c\\&=& \varepsilon\left( \frac{2}{3} ac  + \varepsilon
- a c\right)
	\end{eqnarray*}
	that is negative for~$\varepsilon < ac/3$. Plugging this
information into~\eqref{cneq0}, we obtain~\eqref{oewut45b7v85476}, as desired.

This proves that
trajectories in~$
\big([0,1]\times[0,1]\big)\setminus\mathcal{T}_{\varepsilon, \delta}$
cannot exit~$
\big([0,1]\times[0,1]\big)\setminus\mathcal{T}_{\varepsilon, \delta}$ (see Remark \ref{rmk:enter}).
This, the fact that there are no cycles in~$[0,1]\times[0,1]$
and the Poincar\'e-Bendixson Theorem (see e.g.~\cite{TESCHL})
give that
trajectories in~$\big([0,1]\times[0,1]\big)\setminus\mathcal{T}_{
\varepsilon, \delta}$ converge to~$(0,1)$,
that is the only equilibrium in~$\big([0,1]\times[0,1]\big)\setminus
\mathcal{T}_{\varepsilon, \delta}$. Hence, 
\eqref{poi8562eq2dvfkgjlkuykyuasdawre}
is established.

{F}rom~\eqref{poi8562eq2dvfkgjlkuykyuasdawre}
we deduce that
$$ \mathcal{E}(\rho)\subseteq\mathcal{T}_{\varepsilon, \delta}$$
for all~$0<\rho<a/3$, which implies the desired result
in~(ii).
	\medskip
		
{\emph{(iii)}} 
We consider~$\varepsilon_1>\varepsilon_2 >0$ to be
taken sufficiently small in what follows,
and we show that
there exists~$R>0$, depending on~$\varepsilon_1$
and~$\varepsilon_2$, such that for
all~$\rho\geq R$ it holds that
\begin{equation}\label{qeruyjy8790}
\mathcal{R}_{\varepsilon_1, \varepsilon_2}:=
[0, 1-\varepsilon_1]\times [\varepsilon_2,1] \subseteq \mathcal{B}(\rho).\end{equation}
For this, we first observe that 
\begin{equation}\label{po089egdgdkjfkghjighywrv58465v8}
{\mbox{no trajectory starting
in~$\mathcal{R}_{\varepsilon_1, \varepsilon_2}$
can exit the set.}}\end{equation}
Indeed, looking at the velocity fields on the side~$\{1-\varepsilon_1\} \times [\varepsilon_2, 1]$, the normal inward derivative is
	\begin{equation*}
	-\dot{u}=-[u(1-u-v)-acu] = -(1-\varepsilon_1)(\varepsilon_1-v-ac),
	\end{equation*}
	and this is positive for~$\varepsilon_1\leq ac$ (which is fixed
	from now on).
	In addition, on the side~$[0,1-\varepsilon_1]\times\{ \varepsilon_2 \}$, the inward normal derivative is
	\begin{eqnarray*}
	\dot{v}&=& [\rho v(1-u-v)-au] \\&=& 
	 \rho \varepsilon_2(1-u-\varepsilon_2) - au\\&
	 \ge&	\rho \varepsilon_2(\varepsilon_1-\varepsilon_2) - a(1-
	 \varepsilon_1),
	\end{eqnarray*}
	and this is positive
	for 
\begin{equation}\label{rhodef}
 \rho > \frac{a(1-\varepsilon_1)}{\varepsilon_2(\varepsilon_1
	-\varepsilon_2)}=:R.\end{equation}
These observations complete the proof
of~\eqref{po089egdgdkjfkghjighywrv58465v8} (see Remark \ref{rmk:enter}).

{F}rom~\eqref{07960789djiewf2},
\eqref{po089egdgdkjfkghjighywrv58465v8}
and the Poincar\'e-Bendixson Theorem (see e.g.~\cite{TESCHL}), we have that
all the trajectories in the interior 
of~$\mathcal{R}_{\varepsilon_1, \varepsilon_2}$
must converge to
either an equilibrium or a union of (finitely many)
equilibria and non-closed orbits connecting these equilibria.

In addition, we claim that, if~$0<ac<1$, recalling~\eqref{usvs}
and possibly enlarging~$\rho$ 
in~\eqref{rhodef},
\begin{equation}\label{possibly2}
(u_s,v_s)\notin \mathcal{R}_{\varepsilon_1, \varepsilon_2}.
\end{equation}
Indeed, we have that~$u_s \to 1-ac$ and~$v_s \to 0$,
as~$\rho \to +\infty$. Hence, we can choose~$\rho$ large enough
such that the statement in~\eqref{possibly2} is satisfied.

As a consequence of~\eqref{possibly2},
we get that all the trajectories in the interior 
of~$\mathcal{R}_{\varepsilon_1, \varepsilon_2}$
must converge to
the equilibrium~$(0,1)$,
and this establishes~\eqref{qeruyjy8790}.

Accordingly,~\eqref{qeruyjy8790} entails that, for~$\varepsilon_1>
\varepsilon_2>0$ sufficiently small, there exists~$R>0$, depending on~$\varepsilon_1$
and~$\varepsilon_2$, such that for
all~$\rho\geq R$
\begin{equation*}
\mathcal{E}(\rho)
\subset\big((0,1]\times[0,\varepsilon_2)\big)\cup
\big( (1-\varepsilon_1,1]\times(\varepsilon_2,1]\big)
.\end{equation*}
This implies~\eqref{ir4t4y4y}, as desired.
\end{proof}

\section{Dependence on the parameter~$a$}

The consequences of the lack of variational structure
become even more extreme when we observe the dependence
of the dynamics on
the parameter~$a$, that is the aggressiveness of the first population
towards the other. 
Throughout this chapter, we take~$\rho>0$ and~$c>0$,
and we perform our analysis
taking into account the limit cases~$a\to0$ and~$a\to+\infty$.
We start analyzing the dynamics of~\eqref{model}
in the case~$a=0$.

\begin{proposition}[Dynamics of~\eqref{model} when~$a=0$] \label{prop:bhvaPRE}
For~$a=0$
the system~\eqref{model} has the following
features:
\begin{itemize}
\item[i)] The system has the equilibrium~$(0,0)$, which is a source,
and a straight line of equilibria~$(u,1-u)$, for all~$u\in[0,1]$,
which correspond to a strictly negative eigenvalue and a null one.
\item[ii)] Given any~$(u(0), v(0))\in (0,1)\times(0,1)$ we have that
		\begin{equation}\label{form}
		(u(t), v(t)) \to(\bar{u}, 1-\bar{u})\quad{\mbox{ as }}t\to+\infty, 
		\end{equation}
		where~$\bar{u}\in(0,1)$ satisfies
		\begin{equation}\label{1650}
	\frac{v(0) }{u^{\rho}(0)}\bar{u}^{\rho} + \bar{u} -1=0.
		\end{equation}
\item[iii)] The equilibrium~$(u_s^0, v_s^0)$ given in~\eqref{u0v0}
has a stable manifold, which can be written as the graph of an
increasing smooth function $$\gamma_0:[0,u_{\mathcal{M}}^0]\to[0,v_{\mathcal{M}}^0],$$
as given in~\eqref{def:gamma0},
for some $$(u_{\mathcal{M}}^0,v_{\mathcal{M}}^0)\in\big(\{1\}\times[0,1]\big)\cup
\big((0,1]\times\{1\}\big),$$ such that~$\gamma_0(0)=0$ and
$\gamma_0(u_{\mathcal{M}}^0)=v_{\mathcal{M}}^0$.

More precisely, 
		\begin{equation}\label{def:um0}
		 u_{\mathcal{M}}^0:=\min \left\{1, \frac{u_s^0}{(v_s^0)^{\frac{1}{{\rho}}}}\right\},	\end{equation} 
		 being~$(u_s^0,v_s^0)$ defined in~\eqref{u0v0}.
\end{itemize}
\end{proposition}

We point out that formula~\eqref{form}
says that for~$a=0$ every point in the interior
of~$[0,1]\times[0,1]$ tends to a coexistence equilibrium.
The shape of the trajectories depends on~$\rho$, being
convex in the case~$\rho>1$,
a straight line in the case~$\rho=1$, and concave in the case~$\rho< 1$. This means that if the second population~$v$ is alive at the initial time, then it does not get extinct in finite time.

\begin{proof}[Proof of Proposition~\ref{prop:bhvaPRE}]
{\emph{(i)}}	
	For~$a=0$, we look for the equilibria of the system~\eqref{model}
 by studying when~$\dot{u}=0$ and~$\dot{v}=0$. It is easy to see that
	the point~$(0,0)$ and all the points on the line~$u+v=1$ are the only equilibria.
	
The Jacobian of the system (see~\eqref{Jmatrix}, with~$a=0$) at the point~$(0,0)$
has two positive eigenvalues,~$1$ and~$\rho~$, and therefore~$(0,0)$ is a
source.
 
Furthermore,
the characteristic polynomial at a point~$(\tilde{u}, \tilde{v})$ on the line~$u+v=1$
is given by
	$$(\lambda+\tilde{u})(\lambda+\rho \tilde{v})-\rho \tilde{u}\tilde{v}
=\lambda(\lambda+\tilde{u} +\rho \tilde{v}),$$
and therefore, the eigenvalues are~$0$ and~$-\tilde{u} -\rho \tilde{v}<0$.
\medskip

{\emph{(ii)}}  We point out that
when~$a=0$ 
\begin{equation}\label{intprim678}
{\mbox{$\mu(t):=v(t)/u^{\rho} (t)$ is a prime integral for the system.}}
\end{equation}
Indeed,
	\begin{equation*}\begin{split}
	\dot{\mu}=\;& \frac{\dot{v}u^\rho - {\rho} u^{{\rho}-1}  \dot{u} v }{u^{2{\rho}}}\\=\;& u^{{\rho}-1} \frac{{\rho}uv(1-u-v)- {\rho} uv(1-u-v) }{u^{2{\rho}}}\\=\;&0.\end{split}
	\end{equation*}  
As a result, the trajectory starting at a point~$(u(0),v(0) )\in(0,1)\times(0,1)$ lies on the curve
\begin{equation}\label{intprim}
v(t)=\frac{v(0)}{ u^{\rho}(0)}\,  u^{\rho} (t).\end{equation}
Notice that these curves do not intersect the line $\{v=0\}$ and that they are not periodic orbits, since $\dot{u}$ and $\dot{v}$ have the same sign.
Hence, the  $\alpha-$limit point of~$(u(0),v(0) )$ is  an equilibrium on this curve. Since~$(0,0)$ is a source, the only possibility
is that the trajectory starting at~$(u(0),v(0) )$ converges to an
equilibrium~$(\bar{u}, \bar{v})$ 
	such that
	$\bar{v}=1-\bar{u}$. This entails that
	\begin{equation*}
	1-\bar{u} =\bar{v}=(v(0)/ u^{\rho}(0))  \bar{u}^{\rho},
	\end{equation*}
which is exactly equation~\eqref{1650}.	

\medskip

{\emph{(iii)}} We observe that the point~$(u_s^0, v_s^0)$ given in~\eqref{u0v0}
lies on the straight line~$u+v=1$, and therefore, thanks to~(i) here, it is
an equilibrium of the system~\eqref{model}, which corresponds
to a strictly negative eigenvalue~$-u_s^0-\rho v_s^0$ and a null one.

Hence, by the Center Manifold Theorem
(see e.g. Theorem~1 on page~16
of~\cite{MR635782}), the point~$(u_s^0, v_s^0)$ 
has a stable manifold, which has dimension~$1$ and
is tangent to the eigenvector of the linearized system associated to the
strictly negative eigenvalue~$-u_s^0-\rho v_s^0$.

Also, the monotonicity and the property of being a graph follow from the strict sign of~$\dot{u}$
and~$\dot{v}$. The smoothness of the graphs follows from the regularity of the center manifold.
The fact that~$\gamma_0(0)=0$ is a consequence of the monotonicity property of~${u}$ and~${v}$, which ensures that the $\omega-$limit exists, and the fact that this limit has to lie on the prime integral
in~\eqref{intprim}.
The fact that~$\gamma_0(u_{\mathcal{M}}^0)=v_{\mathcal{M}}^0$
follows from formula~\eqref{form} and the monotonicity property.
Formula~\eqref{def:gamma0}
follows from the fact that any trajectory has to lie on the prime integral
in~\eqref{intprim}.
\end{proof}

To state our next result concerning the dependence of~$\mathcal{E}$ defined in~\eqref{DEFE} on the parameter~$a$,
we give some notation.
We will make it explicit the dependence of the sets~$\mathcal{E}$
and~$\mathcal{B}$
on the parameter~$a$, by writing
explicitly~$\mathcal{E}(a)$ and~$\mathcal{B}(a)$, and we will call
\begin{equation*}
\mathcal{E}_0:=\underset{a'>0}{\bigcap} \, \underset{a'>a>0}{\bigcup} \mathcal{E}(a)
\end{equation*}
and
\begin{equation}\label{def:Einfty}
\mathcal{E}_{\infty}:=\underset{a'>0}{\bigcap} \, \underset{a>a'}{\bigcup} \mathcal{E}(a).
\end{equation}
In this setting, we have the following statements:

\begin{proposition}[Dependence of the dynamics on~$a$] \label{prop:bhva}
	\quad
	\begin{itemize}
		\item[(i)] We have that
		\begin{equation}\label{char:E0}
			\mathcal{G}\subseteq\mathcal{E}_0
			\subseteq\overline{\mathcal{G}},
		\end{equation}
		where
		\begin{eqnarray*}
			\mathcal{G}&:=& \big\{  (u,v)\in [0,1]\times [0,1]\;{\mbox{ s.t. }}\; 
			v < \gamma_{0}(u) \,\text{ if } \, u\in[0, u_{\mathcal{M}}^0]\\
			&&\qquad\qquad\qquad\qquad{\mbox{and }}		\;
			v \leq 1  \, \text{ if }\,  u\in(u_{\mathcal{M}}^0, 1]  \big\},
		\end{eqnarray*}
		and~$\gamma_0$ and~$u_{\mathcal{M}}^0$ are given in~\eqref{def:gamma0}.
		\item[(ii)] It holds that 
\begin{equation}\label{asdfgert019283}
\mathcal{S}_c\subseteq
\mathcal{E}_{\infty} \subseteq\overline{\mathcal{S}_c},
\end{equation}
where
		\begin{equation} \label{def:S_c}
		\mathcal{S}_c:=\left\{ (u,v)\in[0,1] \times [0,1]\;
		{\mbox{ s.t. }}\; v-\frac{u}{c}<0 \right\}.
		\end{equation}	
	\end{itemize}
\end{proposition}

We point out that the set~$\mathcal{E}_0$ in~\eqref{char:E0}
does not coincide with
the basin of attraction for the system~\eqref{model} when~$a=0$.
Indeed, as already mentioned, formula~\eqref{form}
in Proposition~\ref{prop:bhvaPRE}
says that for~$a=0$ every point in the interior
of~$[0,1]\times[0,1]$ tends to a coexistence equilibrium and thus
if~$v(0)\neq0$ then~$v(t)$ does not get extinct in finite time.

Also, as~$a\to+\infty$, we have that the set~$\mathcal{E}_{\infty}$
is determined by~$\mathcal{S}_c$, defined in~\eqref{def:S_c},
that depends only on the parameter~$c$.

\medskip

The statement in~(i) of Proposition~\ref{prop:bhva}
will be a direct consequence of the following result. 
Recalling the function~$\gamma$ introduced in
Propositions~\ref{lemma:M} and~\ref{M:p045}, 
we express here the dependence on the parameter
$a$ by writing~$\gamma_a$,~$u_a$,~$v_a$,
$u_s^a$,~$u_{\mathcal{M}}^a$.
We will also denote by~$\mathcal{M}^a$ the stable manifold
of the point~$(u_s, v_s)$ in~\eqref{usvs}, and by~$\mathcal{M}^0$
the stable manifold
of the point~$(u_s^0, v_s^0)$ in~\eqref{u0v0}.
The key lemma is the following:

\begin{lemma}\label{lemma:conv_gamma}
	For all~$u\in[0,1]$, we have that~$\gamma_a(u) \to \gamma_0(u)$
uniformly as~$a\to0$, where~$\gamma_0(u)$
	is the function defined in~\eqref{def:gamma0}.
\end{lemma}

\begin{proof}
Since we are dealing with the limit as~$a$ goes to zero,
throughout this proof we will always assume that
we are in the case~$ac<1$.

Also, we denote by~$\phi_p^{a}(t)$ the flow at time~$t$
of the point~$p\in[0,1]\times[0,1]$ associated with~\eqref{model},
and similarly by~$\phi_p^{(0)}(t)$ the flow at time~$t$
of the point~$p$ associated with~\eqref{model} when~$a=0$. 
With a slight abuse of notation,
we will also write $$\phi_p^{a}(t)=(u_a(t),v_a(t)) \qquad
{\mbox{ with }}\; p=(u_a(0),v_a(0)).$$

Let us start by proving that
\begin{equation}\label{340}
\mathcal{M}^a\cap\big([0,u_s^0]\times[0,v_s^0]\big)\to
\mathcal{M}^0\cap\big([0,u_s^0]\times[0,v_s^0]\big) \quad {\mbox{ as }}a\to0.
\end{equation}
For this, we claim that, for every~$\varepsilon>0$,
if
\begin{equation}\label{aqwzero}
(u_a(0))^2+(v_a(0))^2 \ge\frac{\varepsilon^2}{4}
\end{equation}
and
	\begin{equation}\label{zz}
	\big| (u_a(t),v_a(t) ) - (u_s^a, v_s^a)\big| > \frac{\varepsilon}{2},
	\end{equation}
then
	\begin{equation}  \label{z}
	|\dot{u}_a(t)|^2 +|\dot{v}_a(t)|^2 > \frac{\varepsilon^{4}}{C_0},
	\end{equation}
for some~$C_0>0$, depending only on~$\rho$ and~$c$.

Indeed, by recalling~(v) of
Theorem~\ref{thm:dyn} and~\eqref{zz}, we see that
the trajectory~$(u_a(t), v_a(t))$ belongs to the set $$[0, u_s^a] 
\times [0, v_s^a] \setminus B_{\frac{\varepsilon}{2}}(u_s^a, v_s^a).$$

Moreover, we claim that
\begin{equation}\label{123456poi}
1-ac-u_a(t)-v_a(t)\ge \frac{\varepsilon \sqrt{2}}{4},
\end{equation}
for any~$t>0$ such that~\eqref{zz} is satisfied.
To prove this, we recall that~$(u_s^a, v_s^a)$ lies on the straight
line~$\ell$ given by~$v=-u+1-ac$
when~$0<ac<1$ (see~\eqref{curve:u'}). 
Clearly, there is no point of 
the set $$[0, u_s^a] 
\times [0, v_s^a] \setminus B_{\frac{\varepsilon}{2}}(u_s^a, v_s^a)$$
lying on~$\ell$, and we notice that the points 
in the set $$[0, u_s^a] 
\times [0, v_s^a] \setminus B_{\frac{\varepsilon}{2}}(u_s^a, v_s^a)$$
with minimal distance from~$\ell$ are given by
$$p:=\left(u_s^a-\frac{\varepsilon}2,
v_s^a\right) \qquad {\mbox{ and }} \qquad q:=\left(u_s^a, v_s^a-\frac{\varepsilon}2\right).$$
Also, the distance of the point~$p$ from the straight line~$\ell$
is given by $$\frac{\varepsilon}2\cdot \tan\frac\pi4=
\frac{\varepsilon \sqrt{2}}{4}.$$
Thus, the distance between~$(u_a(t),v_a(t) )$ 
and the line~$\ell$ is greater than 
$$\frac{\varepsilon \sqrt{2}}{4},$$
and this gives~\eqref{123456poi}.

As a consequence of~\eqref{123456poi}, we obtain that
	\begin{equation}\label{pggdeyw087968754}\begin{split}
	(\dot{u}_a(t))^2 = \;&
\big(u_a(t)(1-ac-u_a(t)-v_a(t)) \big)^2 \\> \;&
(u_a(t))^2\left(\frac{\varepsilon \sqrt{2}}{4}\right)^2\end{split}
	\end{equation}
and that
\begin{equation}\begin{split}\label{pggdeyw087968754BIS}
(\dot{v}_a(t))^2\,=& 
\big(\rho v_a(t)(1-u_a(t)-v_a(t))-au_a(t) \big)^2 \\
\ge&
\left(\rho v_a(t)\left(ac+\frac{\varepsilon \sqrt{2}}{4}\right)-au_a(t) \right)^2.
\end{split}\end{equation}
Now, if~$u_a(t)\ge\rho cv_a(t)$, then from~\eqref{pggdeyw087968754}
and~\eqref{aqwzero}
we obtain that
\begin{eqnarray*}&&
(\dot{u}_a(t))^2+(\dot{v}_a(t))^2 \ge 
(\dot{u}_a(t))^2\\&&\qquad\qquad>
(u_a(t))^2\left(\frac{\varepsilon \sqrt{2}}{4}\right)^2\\&&\qquad\qquad
\ge \frac{(u_a(t))^2}2\left(\frac{\varepsilon \sqrt{2}}{4}\right)^2
+\frac{(\rho cv_a(t))^2}2\left(\frac{\varepsilon \sqrt{2}}{4}\right)^2
\\&&\qquad\qquad
\ge \min \{1,\rho^2c^2\} \frac{\varepsilon^2}{16}
\big((u_a(t))^2+(v_a(t))^2\big)\\
&&\qquad\qquad
\ge \min \{1,\rho^2c^2\} \frac{\varepsilon^2}{16}
\big((u_a(0))^2+(v_a(0))^2\big)\\
&&\qquad\qquad
\ge \min \{1,\rho^2c^2\} \frac{\varepsilon^4}{64},
\end{eqnarray*}
which proves~\eqref{z} in this case.

If instead~$u_a(t)<\rho cv_a(t)$, we use~\eqref{pggdeyw087968754BIS}
to see that
\begin{eqnarray*}&&
(\dot{u}_a(t))^2+(\dot{v}_a(t))^2 \ge 
(\dot{v}_a(t))^2\\&&\qquad\qquad\ge
\left(\rho v_a(t)\left(ac+\frac{\varepsilon \sqrt{2}}{4}\right)-au_a(t) \right)^2
\\&&\qquad\qquad =
\left(\frac{\varepsilon \sqrt{2}\rho v_a(t)}{4}+a\big(\rho cv_a(t)-u_a(t)\big) \right)^2\\&&\qquad\qquad
\ge
\left(\frac{\varepsilon \sqrt{2}\rho v_a(t)}{4}\right)^2\\&&\qquad\qquad
\ge
\frac12\left(\frac{\varepsilon \sqrt{2}\rho v_a(t)}{4}\right)^2
+\frac12\left(\frac{\varepsilon \sqrt{2} u_a(t)}{4c}\right)^2\\&&\qquad\qquad
\ge \min \left\{\rho^2,\frac1{c^2}\right\}\frac{\varepsilon^2}{16}
\big( (u_a(t))^2 +(v_a(t))^2\big)\\
&&\qquad\qquad
\ge \min \left\{\rho^2,\frac1{c^2}\right\}\frac{\varepsilon^2}{16}
\big( (u_a(0))^2 +(v_a(0))^2\big)\\
&&\qquad\qquad
\ge \min \left\{\rho^2,\frac1{c^2}\right\}\frac{\varepsilon^4}{64},
\end{eqnarray*}
which completes the proof of~\eqref{z}.

Now, for any~$\eta>0$, we define
\begin{eqnarray*}\mathcal{P}_\eta&:=&\Bigg\{(u,v)\in[0,1]\times[0,1]\;
{\mbox{ s.t. }}\\ &&\qquad v=\frac{v_s^0-\eta'}{(u_s^0+\eta')^\rho}u^\rho\;
{\mbox{ with }} |\eta'|\le\eta
\Bigg\}.\end{eqnarray*}
Notice that this is the union of graphs of functions close to the curve~$\gamma_0$.
Given~$\varepsilon>0$, we define
\begin{equation}\label{mettiin}
{\mbox{$\eta(\varepsilon)$ to be the smallest~$\eta$
for which~$\mathcal{P}_\eta \supset B_{\varepsilon}(u_s^0,v_s^0)$.}} \end{equation}
We remark that
\begin{equation}\label{ricorda}
\lim_{\varepsilon\to0}\eta(\varepsilon)=0.
\end{equation}
Also, given~$\delta>0$, we define a tubular
neighborhood~$\mathcal{U}_\delta$
of~$\mathcal{M}^0$ as
$$ \mathcal{U}_\delta :=\bigcup_{q\in\mathcal{M}^0}
B_{\delta}(q).$$
Furthermore, we define
\begin{equation}\label{lometto}
{\mbox{$\delta(\varepsilon)$ the smallest~$\delta$ such that
$\mathcal{U}_\delta\supset \mathcal{P}_{\eta(\varepsilon)}$.}}
\end{equation}
Recalling~\eqref{ricorda}, we have that
\begin{equation}\label{ricorda2}
\lim_{\varepsilon\to0}\delta(\varepsilon)=0.
\end{equation}

We remark that, as~$a\to0$, the point~$(u_s^a, v_s^a)$ in~\eqref{usvs},
which is a saddle point for the dynamics of~\eqref{model}
when~$ac<1$ (recall Theorem~\ref{thm:dyn}),
tends to the point~$(u_s^0,v_s^0)$ in~\eqref{u0v0}, that belongs
to the line~$v+u=1$, which is an equilibrium point for the dynamics of~\eqref{model}
when~$a=0$, according to Proposition~\ref{prop:bhvaPRE}.

As a consequence, for every~$\varepsilon>0$, there exists~$a_\varepsilon>0$
such that if~$a\in(0,a_\varepsilon)$,
\begin{equation}\label{qwertyuisdfghjxcvbn}
|(u_s^a,v_s^a)-(u_s^0,v_s^0)|\le\frac\varepsilon8.
\end{equation}
This gives that the intersection of~$\mathcal{M}^a$ with~$B_{\varepsilon/2}(u_s^0,v_s^0)$
is nonempty.

Furthermore, since~$\gamma_a(0)=0$, in light of Proposition~\ref{lemma:M},
we have that the intersection of~$\mathcal{M}^a$ with~$B_{\varepsilon/2}$
is nonempty. Hence, there exists~$p_{\varepsilon,a}\in \mathcal{M}^a\cap
\partial B_{\varepsilon/2}$.

We also notice that
\begin{equation}\label{swqdbvsdjvksdv097654}
\mathcal{M}^a=\phi_{p_{\varepsilon,a}}^{a}(\R).\end{equation}
In addition, 
\begin{equation}\label{a12ewgerheh}
 \phi_{p_{\varepsilon,a}}^{a}\big((-\infty,0]\big)\subset B_{\varepsilon/2}.
\end{equation}
Also, since the origin belongs to~$\mathcal{M}^0$, we have that~$
B_{\varepsilon/2}\subset \mathcal{U}_\varepsilon$. {F}rom
this and~\eqref{a12ewgerheh}, we deduce that
\begin{equation}\label{lsdgrdhtrjb yrweur748v6348900}
\phi_{p_{\varepsilon,a}}^{a}\big((-\infty,0]\big)
\subset \mathcal{U}_\varepsilon.\end{equation}

Now, we let~$C_0$ be as in~\eqref{z} and
we claim that there exists~$t_{\varepsilon,a}\in(0,3\sqrt{C_0}\varepsilon^{-2})$
such that
\begin{equation}\label{esisteuntempo}
\phi_{p_{\varepsilon,a}}^{a}(t_{\varepsilon,a})\in\partial B_{3\varepsilon/4}
(u_s^0,v_s^0).
\end{equation}
To check this, we argue by contradiction and we suppose that
$$ \phi_{p_{\varepsilon,a}}^{a}\big((0,3\sqrt{C_0}\varepsilon^{-2})\big)
\cap B_{3\varepsilon/4}(u_s^0,v_s^0)=\varnothing.$$
Then, for every~$t\in(0,3\sqrt{C_0}\varepsilon^{-2})$,
recalling also~\eqref{qwertyuisdfghjxcvbn},
\begin{eqnarray*} \big|\phi_{p_{\varepsilon,a}}^{a}(t)-(u_s^a,v_s^a)\big|&\ge&
 \big|\phi_{p_{\varepsilon,a}}^{a}(t)-(u_s^0,v_s^0)\big| -
\big|(u_s^a,v_s^a)-(u_s^0,v_s^0)\big|\\&\ge&
\frac{3\varepsilon}4-\frac\varepsilon8\\&>&\frac{\varepsilon}2,
\end{eqnarray*}
and consequently~\eqref{zz} is satisfied for every~$t\in(0,3\sqrt{C_0}
\varepsilon^{-2})$.

Moreover,
we observe that~$p_{\varepsilon,a}$ satisfies~\eqref{aqwzero},
and therefore, by~\eqref{z},
$$ | \dot{u}_a(t)|^2 +| \dot{v}_a(t)|^2
> \frac{\varepsilon^{4}}{C_0},$$
for all~$t\in(0,3\sqrt{C_0}\varepsilon^{-2})$,
where we used the notation$${\phi}_{p_{\varepsilon,a}}^{a}(t)=
(u_a(t),v_a(t)),$$being$$p_{\varepsilon,a}=(u_a(0),v_a(0)).$$
As a result,
$$ \big( \dot{u}_a(t)+  \dot{v}_a(t)\big)^2>\frac{\varepsilon^{4}}{C_0},$$
and thus
$$ \dot{u}_a(t)+  \dot{v}_a(t)>\frac{\varepsilon^{2}}{\sqrt{C_0}}.$$
This leads to
\begin{eqnarray*}&&
u_a\left(\frac{3\sqrt{C_0}}{\varepsilon^2}\right)
+v_a\left(\frac{3\sqrt{C_0}}{\varepsilon^2}\right)\\&=&u_a(0)+v_a(0)
+\int_0^{\frac{3\sqrt{C_0}}{\varepsilon^2}}\big( \dot{u}_a(t)+  \dot{v}_a(t)\big)\,dt
\\&\ge& u_a(0)+v_a(0)+
\int_0^{\frac{3\sqrt{C_0}}{\varepsilon^2}}\frac{\varepsilon^{2}}{\sqrt{C_0}}\,dt\\&
=&u_a(0)+v_a(0) +3\\&\ge& 3,
\end{eqnarray*}
which forces the trajectory to exit the region~$[0,1]\times[0,1]$.
This is against the assumption that~$p_{\varepsilon, a}\in\mathcal{M}^a$,
and therefore the proof of~\eqref{esisteuntempo} is complete.

In light of~\eqref{esisteuntempo}, we can set~$q_{\varepsilon,a}:=
\phi_{p_{\varepsilon,a}}^{a}(t_{\varepsilon,a})$,
and we deduce from~\eqref{mettiin} that~$q_{\varepsilon,a}\in\mathcal{P}_{\eta(\varepsilon)}$.
We also observe that the set~$\mathcal{P}_\eta$ is invariant for the
flow with~$a=0$, thanks to~\eqref{intprim678}. These observations
give that~$\phi_{q_{\varepsilon,a}}^{0}(t)\in\mathcal{P}_{\eta(\varepsilon)}$
for all~$t\in\R$. 

As a result, using~\eqref{lometto}, we conclude that
\begin{equation}\label{ASDFGHJtergyfhgj}
\phi_{q_{\varepsilon,a}}^{0}(t)\in\mathcal{U}_{\delta(\varepsilon)}\quad
{\mbox{ for all }} t\in\R.
\end{equation}
In addition, by the continuous dependence of the flow
on the parameter~$a$ in closed intervals of time (see e.g. Section~2.4
in~\cite{MR3791466},
or Theorem~2.4.2 in~\cite{MR3186036}),
$$ \big|\phi_{q_{\varepsilon,a}}^{0}(t)-\phi_{q_{\varepsilon,a}}^{a}(t)\big|
<\varepsilon,$$
for all~$t\in[-3\sqrt{C_0}\varepsilon^{-2},0]$, provided that~$a$
is sufficiently small, possibly in dependence of~$\varepsilon$.
This fact and~\eqref{ASDFGHJtergyfhgj} entail that
$$ \phi_{q_{\varepsilon,a}}^{a}(t)\in\mathcal{U}_{\delta(\varepsilon)+\varepsilon}\quad
{\mbox{ for all }} t\in[-3\sqrt{C_0}\varepsilon^{-2},0].
$$
In particular, for all~$t\in[0,t_{\varepsilon,a}]$,
\begin{equation}\label{andatosu}
\phi_{p_{\varepsilon,a}}^{a}(t)=\phi_{q_{\varepsilon,a}}^{a}(t-t_{\varepsilon,a})
\in\mathcal{U}_{\delta(\varepsilon)+\varepsilon}.\end{equation}

We now claim that
for all~$t\ge t_{\varepsilon,a}$,
\begin{equation}\label{qwertyuiop}
\phi_{p_{\varepsilon,a}}^{a}(t)\subset B_{\varepsilon}(u_s^a,v_s^a).
\end{equation}
Indeed, this is true when~$t=t_{\varepsilon,a}$ thanks to~\eqref{qwertyuisdfghjxcvbn}
and~\eqref{esisteuntempo}. 

Hence, since
the trajectory~$\phi_{p_{\varepsilon,a}}^{a}(t)$
is contained in the domain where~$\dot{u}\ge0$ and~$\dot{v}\ge0$,
thanks to~\eqref{aggiunto}, we deduce that~\eqref{qwertyuiop} holds true.

{F}rom~\eqref{qwertyuisdfghjxcvbn} and~\eqref{qwertyuiop}, we conclude that
$$ \phi_{p_{\varepsilon,a}}^{a}(t)\subset B_{2\varepsilon}(u_s^0,v_s^0),$$
for all~$t\ge t_{\varepsilon,a}$.

Using this,~\eqref{lsdgrdhtrjb yrweur748v6348900}
and~\eqref{andatosu}, we obtain that
$$ \phi_{p_{\varepsilon,a}}^{a}(\R)\subset\mathcal{U}_{\delta(\varepsilon)+
2\varepsilon}.$$
This and~\eqref{ricorda2} give that~\eqref{340} is satisfied, as desired.

One can also show that
\begin{equation}\label{340BIS}
\mathcal{M}^a\cap\big([u_s^0, u_{\mathcal{M}}^0]\times[v_s^0,v_{\mathcal{M}}^0]\big)\to
\mathcal{M}^0\cap\big([u_s^0, u_{\mathcal{M}}^0]\times[v_s^0,v_{\mathcal{M}}^0]\big)
\end{equation}
as~$a\to0$.
The proof of~\eqref{340BIS} is similar to that of~\eqref{340},
just replacing~$p_{\varepsilon,a}$ with~$(u_{\mathcal{M}}^a,v_{\mathcal{M}}^a)$ (in this case
the analysis near the origin is simply omitted since the trajectory
has only one limit point).

With~\eqref{340} and~\eqref{340BIS}
the proof of Lemma~\ref{lemma:conv_gamma} is thereby complete.
\end{proof}

Now we are ready to give the proof of Proposition~\ref{prop:bhva}:

\begin{proof}[Proof of Proposition~\ref{prop:bhva}]
{\emph{(i)}} We aim at proving that~$\mathcal{G}\subseteq\mathcal{E}_0
\subseteq\overline{\mathcal{G}}$.

For this, we observe that, by Lemma~\ref{lemma:conv_gamma},
$\gamma_a(u)$ converges to~$\gamma_{0}(u)$ pointwise as~$a\to0$.
In particular,~$u_{\mathcal{M}}^a\to u_{\mathcal{M}}^0$ as~$a\to0$.
 
Also, recalling~\eqref{def:gamma0},
we notice that if~$u_{\mathcal{M}}^0= u_s^0 / (v_s^0)^{\frac{1}{\rho}}<1$,
then~$\gamma_0(u_{\mathcal{M}}^0)= 1$,
otherwise if~$u_{\mathcal{M}}^0=1$
then~$\gamma_0(u_{\mathcal{M}}^0)<1$, being~$\gamma_0(u)$ strictly
monotone increasing.
	
Furthermore, thanks to Proposition~\ref{prop:char},
we know that
the set~$\mathcal{E}(a)$ is bounded from above by the graph of the function~$\gamma_a(u)$ for~$u\in [0, u_{\mathcal{M}}^a]$ and from the
straight line~$v=1$ for~$u\in(u_{\mathcal{M}}^a, 1]$ (that is non empty for~$u_{\mathcal{M}}^a<1$). 

Now we claim that, for all~$a'>0$,
	\begin{equation}\label{gocont123}
	\mathcal{G} \subseteq \underset{0<a<a'}{\bigcup} \mathcal{E}(a).
	\end{equation}
To show this, we take a point~$(u,v)\in\mathcal{G}$.
Hence, in light of the considerations
above, we have that~$(u,v)\in\mathcal{E}(a)$ for any~$a$ sufficiently small,
which proves~\eqref{gocont123}.
	
{F}rom~\eqref{gocont123}, we deduce that
	\begin{equation}\label{gocont1232233}
	\mathcal{G} \subseteq \underset{a'>0}{\bigcap} \, \underset{0<a<a'}{\bigcup} \mathcal{E}(a).
	\end{equation}
	Now we show that
		\begin{equation}\label{gocont12322}
  \underset{a'>0}{\bigcap} \,\underset{0<a<a'}{\bigcup} \mathcal{E}(a)\subseteq
	\overline{\mathcal{G}} .
	\end{equation}
	For this, we take 
	$$(\widehat{u},\widehat{v})\in  \underset{a'>0}{\bigcap} \, \underset{0<a<a'}{\bigcup} \mathcal{E}(a),$$
	then it must hold that for every~$a'>0$
	there exists~$a<a'$ such that~$(\widehat{u},\widehat{v})\in\mathcal{E}(a)$,
	namely~$\widehat v < \gamma_{a}(\widehat u)$ if~$\widehat u\in[0, u_{\mathcal{M}}^a]$ and~$\widehat v
	 \leq 1$ if~$\widehat u\in(u_{\mathcal{M}}^a, 1]$.  
	Thus, by the pointwise convergence,
	we have that~$  \widehat{v} \le\gamma_0(\widehat{u})~$ if~$\widehat u\in[0, u_{\mathcal{M}}^0]$ and~$\widehat v
	 \leq 1$ if~$\widehat u\in(u_{\mathcal{M}}^0, 1]$, which proves~\eqref{gocont12322}.

{F}rom~\eqref{gocont1232233} and~\eqref{gocont12322},
we conclude that
	\begin{equation*}
	\mathcal{G}\subseteq
	\underset{a'>0}{\bigcap} \, \underset{0<a<a'}{\bigcup} \mathcal{E}(a) =\mathcal{E}_0\subseteq \overline{\mathcal{G}} ,
	\end{equation*}
	as desired.
	\medskip	

{\emph{(ii)}} Since we deal with the limit case as~$a\to+\infty$, from now on
we suppose from now on that~$ac>1$.
We fix~$\varepsilon>0$ and we consider the set  
	\begin{equation*}
	\mathcal{S}_{\varepsilon^+} := 
\left\{  (u,v)\in [0,1]\times[0,1]\;{\mbox{ s.t. }}\;
v>u \left( \frac{1}{c}+\varepsilon \right)      \right\}.
	\end{equation*}
We claim that
\begin{equation}\label{prova1}
\mathcal{S}_{\varepsilon^+} \subseteq \mathcal{B}(a)
\end{equation}
for~$a$ big enough, possibly in dependence of~$\varepsilon$.
For this, 
we first analyze the component of the velocity in the inward normal directions
along the boundary of~$\mathcal{S}_{\varepsilon^+}$.
The only side on the interior of $[0,1]\times[0,1]$ is given by the straight line~$v-u(\varepsilon +1/c )=0$.
Ignoring the scaling constant $1/\sqrt{1+\varepsilon^2 +\frac{2\varepsilon}{c}+\frac{1}{c^2}}$, we compute
	\begin{align*}& (\dot{u},\dot{v})\cdot\left(-\left(\varepsilon+\frac1c\right),
	1\right)=
	\dot{v}- \dot{u}\left(\varepsilon+ \frac{1}{c} \right) \\
	&\quad = {\rho}v(1-u-v)-au - \left(\varepsilon+ \frac{1}{c} \right)u(1-u-v) + \left(\varepsilon+ \frac{1}{c} \right) acu \\
	&\quad=  \Bigg[ {\rho}v - \left(\varepsilon+ \frac{1}{c} \right) u     \Bigg] (1-u-v) +\varepsilon ac u	.	
	\end{align*} 	
Thus, by using that~$v-u(\varepsilon +1/c )=0$, we obtain that
	\begin{align*}
 (\dot{u},\dot{v})\cdot\left(-\left(\varepsilon+\frac1c\right),
	1\right) = u\left[a\varepsilon c + ({\rho}-1)(1-u-v)   \left( \varepsilon+ \frac{1}{c} \right)  \right]. 
	\end{align*}
	Notice that~$u\leq 1$ and~$|1-u-v|\leq 2$, and therefore
	$$ (\dot{u},\dot{v})\cdot\left(-\left(\varepsilon+\frac1c\right),
	1\right) \geq u\left[a\varepsilon c -2 ({\rho}+1) \left( \varepsilon+ \frac{1}{c} \right)  \right] .$$
Accordingly, the normal velocity is positive for~$a \geq {a}_1$, where
	\begin{equation*}
	{a}_1:= 2({\rho}+1) \left( \varepsilon+ \frac{1}{c}  \right)\frac{1}{\varepsilon c}.
	\end{equation*}
 Hence, by Lemma \ref{lemma:exit}, no trajectory can exit  $\mathcal{S}_{\varepsilon^+}$.
These considerations, together with the fact that there are no cycles
in~$[0,1]\times[0,1]$, that $\mathcal{S}_{\varepsilon^+}\cap \{v=0\}= \varnothing$, and the
Poincar\'e-Bendixson Theorem (see e.g.~\cite{TESCHL}),
give that the~$\omega$-limit set of any trajectory starting
in the interior of~$\mathcal{S}_{\varepsilon^+}$
can be either an equilibrium or a union of (finitely many)
equilibria and non-closed orbits connecting these equilibria.

We remark that 
\begin{equation}\label{asdfgzxcv098t76re}
{\mbox{the~$\omega$-limit set of any trajectory cannot
be the equilibrium~$(0,0)$.}}\end{equation}
Indeed, if the~$\omega$-limit of a trajectory
were~$(0,0)$, then this trajectory must lie on the stable manifold of~$(0,0)$,
and moreover it must be contained in~$\mathcal{S}_{\varepsilon^+}$,
since no trajectory can exit~$\mathcal{S}_{\varepsilon^+}$.
On the other hand, by Proposition~\ref{lemma:M},
we have that at~$u=0$ the stable manifold is tangent to the
line
$$v=\frac{a}{\rho-1+ac}u=\frac{1}{\frac{\rho-1}{a}+c}u.
$$
Now, if we take~$a$ sufficiently large, this line lies
below the line~$v=u(1/c+\varepsilon)$, thus providing a contradiction.
Hence, the proof of~\eqref{asdfgzxcv098t76re} is complete.

Accordingly, since~$(0,1)$ is a sink, the only possibility is that
the~$\omega$-limit set of any trajectory starting
in the interior of~$\mathcal{S}_{\varepsilon^+}$ is the equilibrium~$(0,1)$.
Namely, we have established~\eqref{prova1}.
 
As a consequence of~\eqref{prova1}, we deduce that for every~$\varepsilon>0$
there exists~$a_{\varepsilon}>0$ such that
\begin{equation}\label{qwt5uktkjer464586897}
\underset{a\ge a_\varepsilon}{\bigcup}  \mathcal{E}(a) \subseteq
\left\{  (u,v)\in [0,1]\times[0,1]\;{\mbox{ s.t. }}\;
v\le u \left( \frac{1}{c}+\varepsilon \right)      \right\}.\end{equation}
In addition,
	\begin{equation*}\begin{split}
&	\underset{\varepsilon >0 }{\bigcap} \left\{  (u,v)\in [0,1]\times[0,1]\;{\mbox{ s.t. }}\;
v\le u \left( \frac{1}{c}+\varepsilon \right)      \right\}\\&\qquad
=\left\{  (u,v)\in [0,1]\times[0,1]\;{\mbox{ s.t. }}\;
v\le \frac{u}{c} \right\}=\overline{\mathcal{S}_c}.\end{split}
	\end{equation*}
	This and~\eqref{qwt5uktkjer464586897} entail that 
	\begin{equation*}
	\underset{a'>0}{\bigcap} \, \underset{a>a'}{\bigcup}
 \mathcal{E}(a)\subseteq \overline{\mathcal{S}_c},
	\end{equation*}
which implies the second inclusion in~\eqref{asdfgert019283}.

Now, to show the first inclusion in~\eqref{asdfgert019283},
for every~$\varepsilon\in(0,1/c)$ we consider the set
	\begin{equation*}
	\mathcal{S}_{\varepsilon^-} := \left\{  
(u,v)\in [0,1]\times[0,1] \;{\mbox{ s.t. }}\; v<u \left( \frac{1}{c}-\varepsilon \right)      \right\}.
	\end{equation*}
We claim that, for all~$\varepsilon\in(0,1/c)$,
\begin{equation}\label{chefus}
\mathcal{S}_{\varepsilon^-} \subseteq \mathcal{E}_{\infty}.
\end{equation}
For this, we first show that if~$a$ is sufficiently large, possibly
in dependence of~$\varepsilon$,
\begin{equation}\label{forse33}
\mathcal{S}_{\varepsilon^-} \subseteq \mathcal{E}(a).
\end{equation}
Indeed, no trajectory can leave $\mathcal{S}_{\varepsilon^-}$. 
	In fact, on the side given by~$v-(-\varepsilon+1/c)u=0$, the component of the velocity
in the direction of the
outward normal vector is 
	\begin{eqnarray*}&&
	(\dot{u},\dot{v})\cdot\left(- \left( \frac{1}{c} -\varepsilon \right),1
\right)\\&=&
\dot{v} - \dot{u} \left( \frac{1}{c} -\varepsilon \right)\\
	&=& \rho v(1-u-v) -au - \left( \frac{1}{c} -\varepsilon \right)u(1-u-v)   + \left( \frac{1}{c} -\varepsilon \right)acu\\
	&
=&u\left[\left( \frac{1}{c} -\varepsilon \right)(\rho-1)(1-u-v)    - 
\varepsilon ac \right]\\
& \le& u\left[2\left( \frac{1}{c} -\varepsilon \right)(\rho+1)  - 
\varepsilon ac \right]
,
	\end{eqnarray*}
which is negative if~$a \geq {a}_2$, with
	\begin{equation*}
	{a}_2:= 2\left( \frac{1}{c} -\varepsilon \right)
 \left( \rho+1  \right) \frac{1}{\varepsilon c} .
	\end{equation*}
	Hence, if~$(u(0), v(0))\in \mathcal{S}_{\varepsilon^-}$, then
 either~$T_s(u(0), v(0)) <\infty$ or~$(u(t), v(t))\in \mathcal{S}_{\varepsilon^-}$
 for all~$t\geq 0$, where the notation in~\eqref{def:T_s} has been used.
We also notice that,
for~$a>1/c$, the points~$(0,1)$ and~$(0,0)$ are the only equilibria
of the system, and there are no cycles. 
	We have that~$(0,1) \notin \overline{\mathcal{S}_{\varepsilon^-}}$
and~$(0,0) \in \overline{\mathcal{S}_{\varepsilon^-}}$, thus if
\begin{equation}\label{dweioterygvhsdjk}
{\mbox{$(u(t), v(t))\in
 \mathcal{S}_{\varepsilon^-}$ for all~$t\geq 0$}}\end{equation} then
	\begin{equation}\label{tendere}
	(u(t), v(t)) \to (0,0).
	\end{equation}
On the other hand, by Proposition~\ref{lemma:M},
we have that at~$u=0$ the stable manifold is tangent to the
line
$$v=\frac{a}{\rho-1+ac}u=\frac{1}{\frac{\rho-1}{a}+c}u,
$$
and, if we take~$a$ large enough, this line lies
above the line~$v=u(1/c-\varepsilon)$. This says that, for sufficiently large~$t$,
the trajectory must lie outside~$ \mathcal{S}_{\varepsilon^-}$,
and this is in contradiction with~\eqref{dweioterygvhsdjk}.

As a result of these considerations, we conclude 
that if $$(u(0), v(0))\in \mathcal{S}_{\varepsilon^-}$$ then $$T_s(u(0), v(0)) <\infty ,$$
which implies~\eqref{forse33}.

As a consequence of~\eqref{forse33}, we obtain that for every~$\varepsilon\in(0,1/c)$
there exists~$a_\varepsilon>0$ such that
$$\mathcal{S}_{\varepsilon^-}\subseteq\underset{a\ge a_\varepsilon}{\bigcap}
 \mathcal{E}(a).$$
In particular for all~$\varepsilon\in(0,1/c)$ it holds that
	\begin{equation*}
	\mathcal{S}_{\varepsilon^-}\subseteq \underset{a'>0}{\bigcap} \,
 \underset{a> a'}{\bigcup} \mathcal{E}(a)=\mathcal{E}_\infty,
	\end{equation*}
which proves~\eqref{chefus}, as desired.

Then, the first inclusion in~\eqref{asdfgert019283} plainly follows
from~\eqref{chefus}.
\end{proof}

\chapter{Strategies of the first population}\label{STRATE}

\begin{center}
\begin{minipage}{25em}
\noindent{\bf Abstract of Chapter~\ref{STRATE}.}
{\sl Here we characterize the winning strategies for the aggressive populations, i.e.
the setting of the parameters which lead to the victory of the aggressive population
starting from a favorable initial condition.

The analysis is different for different parameter ranges. In particular, the case of equal fitness between the two populations boils down to constant strategies, but the general case
is not exhausted by them.

In any case, we prove that also in the general case the winning strategies can always be found among the ``bang-bang'' functions, i.e. piecewise constant functions with at most one jump.}
\end{minipage}\end{center}\bigskip\bigskip\bigskip\bigskip\bigskip\bigskip

The main theorems on the winning strategy have
been stated in Section~\ref{ss:strategy}.
In particular, Theorem~\ref{thm:Vbound} gives the characterization of the set~$\mathcal{V}_{\mathcal{A}}$
of points that have a winning strategy
in~\eqref{DEFNU},
and Theorem~\ref{thm:W} establishes the non equivalence of constant
and non-constant strategies when~$\rho\ne1$
(and their equivalence when~$\rho=1$).
Nonetheless, in Theorem~\ref{thm:H} we state that
Heaviside functions are enough to construct a
winning strategy for every point in~$\mathcal{V}_{\mathcal{A}}$.

In the following subsections we will give the proofs of these results.

\section{Winning non-constant strategies}\label{explowrewwt}

We want to put in light the construction of non-constant winning
strategies for the points for which constant strategies fail.

For this, we recall the notation introduced in~\eqref{u0v0},
\eqref{def:gamma0}
and~\eqref{ZETADEF}, and we have the following statement:

\begin{proposition}\label{prop:construction}
	Let~$M>1$. Then we have:
	\begin{itemize}
		\item[1.] For~$\rho<1$, let~$(u_0, v_0)$ be a point in the set
		\begin{equation}\label{PPDEFA}\begin{split}
			\mathcal{P}:=\;&\Bigg\{ (u, v)\in [u_s^0,1]\times[0,1]\;{\mbox{ s.t. }} \\&\qquad\qquad  \gamma_{0}(u) \leq v < \frac{u}{c} + \frac{1-\rho}{1+\rho c}    \Bigg\}.	\end{split}\end{equation}
		Then there exist~$a^*>M$,~$a_*<\frac{1}{M}$,
		and~$T\ge0$, depending on~$(u_0, v_0)$,~$c$, and~$\rho$, such that
		for the Heaviside strategy defined by
		\begin{equation}\label{NSJmldsf965to}
		a(t) = \left\{
		\begin{array}{lr}
		a^*,  & {\mbox{ if }} t<T, \\
		a_*,  &  {\mbox{ if }} t\geq T,
		\end{array}
		\right.
		\end{equation} we have $(u_0, v_0)\in  \mathcal{V}_{\mathcal{A}}$.
		\item[2.] For~$\rho>1$, let~$(u_0, v_0)$ be a point of the set
		\begin{equation}\label{DEFQ}
		\mathcal{Q}:=\left\{ (u, v)\in [u_{\infty},1]\times[0,1] \;{\mbox{ s.t. }}\;\frac{u}{c}  \leq v < \zeta(u)   \right\}.
		\end{equation}
		Then there exist~$a^*>M$,~$a_*<\frac{1}{M}$, and~$T\ge0$, depending on~$(u_0, v_0)$,~$c$, and~$\rho$, such that for the Heaviside strategy
		defined by
		\begin{equation*}
		a(t) = \left\{
		\begin{array}{lr}
		0,  &{\mbox{ if }} t<T, \\
		a^*,  &{\mbox{ if }} t\geq T,
		\end{array}
		\right.
		\end{equation*}
	we have $(u_0, v_0)\in  \mathcal{V}_{\mathcal{A}}$.
	\end{itemize}
\end{proposition}

\begin{proof} We start by proving the first claim in Proposition~\ref{prop:construction}.
To this aim, we take~$(\bar{u}, \bar{v})\in \mathcal{P}$, and we observe that
\begin{equation*}
		\bar{v}-\frac{\bar{u}}{c}< \frac{1-\rho}{1+\rho c} = v_s^0-\frac{u_s^0}{c}.
	\end{equation*}
Therefore, there exists~$\xi>0$ such that 
	\begin{equation*}
	\xi < \frac{v_s^0- \bar{v}-\frac{1}{c}(u_s^0-\bar{u})}{\bar{u}-u_s^0}.
	\end{equation*}
Hence, setting
	\begin{equation}\label{VUESSE0}
		v_S:=\left(\frac{1}{c} -\xi \right)(u_s^0-\bar{u}) + \bar{v},
	\end{equation}	
we see that
	\begin{equation}\label{VUESSE}v_S<v_s^0.\end{equation}	
	Now, we want to show that there exists~$a^*>0$ such that, for any~$a>a^*$ and~$u>u_s^0$, we have that
	\begin{equation}\label{1632}
		\frac{\dot{v}}{\dot{u}} > \frac{1}{c}- \xi.
	\end{equation}
	To prove this, we first notice that
	\begin{equation}\label{questab}
	{\mbox{if~$a>\displaystyle\frac2c$, then~$\dot{u}\le -u<0$ for all $u,v\in[0,1]\times[0,1]$.}}\end{equation}
	 Moreover, we set~$$a_1:=\frac{1+\rho c}{4c} , ~$$
	and we claim that, 
	\begin{equation}\label{questae}
	{\mbox{if~$a>a_1$ and~$u>u_s^0$, then~$\dot{v}<0$.}}\end{equation}
	Indeed, we recall that the function~$\sigma$
	defined in~\eqref{f:sigma} represents the points
	in~$[0,1]\times[0,1]$ where~$\dot v=0$
	and separates the points where~$ \dot v>0$, which lie on the left of
	the curve described by~$\sigma$, from the points where~$ \dot v<0$, which lie on the right of
	the curve described by~$\sigma$.	

Therefore, in order to show~\eqref{questae}, it is sufficient to prove that
the curve described by~$\sigma$
is contained in~$\{u\le u_s^0\}$ whenever~$a>a_1$. For this, one computes that, if~$u=\sigma(v)$
and~$a>a_1$, then
\begin{eqnarray*}&&
u-u_s^0=\sigma(v)-\frac{\rho c}{1+\rho c}\\&&\qquad=
1-\frac{\rho v^2+a}{\rho v+a}-\frac{\rho c}{1+\rho c}\\&&\qquad=
\frac{\rho v-\rho v^2}{\rho v+a}-\frac{\rho c}{1+\rho c}\\&&\qquad=\frac{\rho v(1-v)}{\rho v+a}-\frac{\rho c}{1+\rho c}
\\&&\qquad\le\frac{\rho }{4(\rho v+a)}-\frac{\rho c}{1+\rho c}\\&&\qquad\le
\frac{\rho }{4a}-\frac{\rho c}{1+\rho c}\\&&\qquad\le
\frac{\rho }{4a_1}-\frac{\rho c}{1+\rho c}\\&&\qquad\le0.
\end{eqnarray*}
This completes the proof of~\eqref{questae}.

Now we define
\begin{equation*}
		a_2:=\left( \rho+\frac{1}{c}+\xi  \right) \frac{2}{u_s^0 c \xi}.
	\end{equation*}
and we claim that
	\begin{equation}\label{questaX}
	{\mbox{if~$a>a_2$ and~$u>u_s^0$, then }}
		\dot{v} < \left( \frac{1}{c}- \xi \right) \dot{u}.
\end{equation}
Indeed, under the assumptions of~\eqref{questaX},
we deduce that
\begin{eqnarray*}&&
\dot{v} -\left( \frac{1}{c}- \xi \right) \dot{u}\\&&\qquad
=\rho v(1-u-v)-au-\left( \frac{1}{c}- \xi \right)\Big(
u(1-u-v)-acu
\Big)\\
&&\qquad=(1-u-v)\left(
\rho v-\left( \frac{1}{c}- \xi \right)u\right)-ac \xi u\\&&\qquad
\le 2\left(
\rho v+ \frac{u}{c}+ \xi u\right)-ac \xi u\\&&\qquad< 2\left(
\rho + \frac{1}{c}+ \xi \right)-a_2\,c \xi {u_s^0}\\&&\qquad=0,
\end{eqnarray*}
and this establishes the claim in~\eqref{questaX}.

Then, choosing
$$ a^*:=
\max\left\{\displaystyle\frac2c,a_1,a_2,M\right\},$$
we can exploit~\eqref{questab},~\eqref{questae} and~\eqref{questaX}
to deduce~\eqref{1632}, as desired.
 
Now we claim that, for any~$a>a^*$, there exists~$T\ge0$ 
such that the trajectory~$(u(t), v(t))$ starting from~$(\bar{u}, \bar{v})$ satisfies
\begin{equation}\label{SM -kg}
{\mbox{$u(T)=u_s^0$ and~$v(T)< v_S$.}}
\end{equation}
Indeed, we define~$T\ge0$ to be the first time for which~$u(T)=u_s^0$.
This is a fair definition, since~$u(0)=\bar{u}\ge u_s^0$
and~$\dot u$ is negative, and bounded away from zero
till~$u\ge u_s^0$, thanks to~\eqref{questab}.
Then, we see that
\begin{eqnarray*}
v(T)&=&\bar{v}+\int_0^T \dot v(t)\,dt\\&<&
\bar{v}+\int_0^T\left(
\frac{1}{c}- \xi\right)\,\dot u(t)\,dt\\&=&
\bar{v}+\left(
\frac{1}{c}- \xi\right)(u(T)-u(0))\\&=&
\bar{v}+\left(
\frac{1}{c}- \xi\right)(u_s^0-\bar u)\\&=&v_S,
\end{eqnarray*}
thanks to~\eqref{VUESSE0} and~\eqref{1632}, and this establishes~\eqref{SM -kg}.

Now we observe that
$$ v(T)<v_S<v_s^0=\gamma_0(u_s^0)=\gamma_0(u(T))$$
due to~\eqref{VUESSE}
and~\eqref{SM -kg}

As a result, recalling Lemma~\ref{lemma:conv_gamma}, we can choose~$a_*<1/M$
such that
$$ v(T)<\gamma_{a_*}(u(T)).$$
Accordingly, by Proposition~\ref{prop:char},
we obtain that~$(u(T),v(T))\in{\mathcal{E}}(a_*)$.
Hence, applying the strategy
in~\eqref{NSJmldsf965to},
we accomplish the desired result and complete the proof of the first claim
in Proposition~\ref{prop:construction}.\medskip

Now we focus on the proof of the second claim in Proposition~\ref{prop:construction}.
For this, let \begin{equation}\label{8ygdw}
(u_0,v_0)\in \mathcal{Q},\end{equation} and consider the trajectory~$(u_0(t),v_0(t))$
starting from~$(u_0,v_0)$ for the strategy~$a=0$. In light of formula~\eqref{form}
of Proposition~\ref{prop:bhvaPRE}, we have that
\begin{equation}\label{Ecijerrin}\begin{split}&
{\mbox{the trajectory~$(u_0(t),v_0(t))$
converges}}\\&{\mbox{to a point of the form~$(u_F, 1-u_F)$ as~$t\to+\infty$.}}\end{split}
\end{equation}

We define
\begin{equation}\label{1713}
	v_F:=1-u_F, \quad	v_{\infty}:=1-u_{\infty}=\frac{1}{c+1},
\end{equation} 
where the last equality can be checked starting from the value of~$u_{\infty}$ given in~\eqref{ZETADEF}.
Using the definition of~$\zeta$ in~\eqref{ZETADEF} 
we have
\begin{equation*}
	\zeta(u_{\infty})= \frac{1}{c(u_{\infty})^{\rho-1}} u_{\infty}^{\rho}=\frac{c}{c(c+1)}=v_{\infty}
\end{equation*}
and, recalling~\eqref{1713} and formula~\eqref{form} of Proposition~\ref{prop:bhvaPRE}, we get that 
the graph of~$\zeta$ is the union of trajectories for~$a=0$ that converges to~$(u_\infty,1-u_\infty)$ as~$t\to+\infty$.

Also, by~\eqref{8ygdw}, we have that~$v_0 < \zeta(u_0)$. Thus, since by Cauchy's uniqueness result for ODEs, two orbits never intersect, we have that
\begin{equation}\label{JS145DD-0}
{\mbox{the orbit~$(u_0(t),v_0(t))$ must lie below the graph of~$\zeta$.}}\end{equation}
Since both~$(u_F, v_F)$ and~$(u_{\infty}, v_{\infty})$ belong to the line given by~$v=1-u$, from~\eqref{JS145DD-0} we get that
\begin{equation}\label{2304}
	u_{\infty} < u_F
\end{equation}
and 
\begin{equation}\label{2305}
	v_{\infty} > v_F.
\end{equation}
Thanks to~\eqref{2304} and~\eqref{2305} and recalling the values of~$u_{\infty}$ from~\eqref{ZETADEF} and of~$v_{\infty}$ from~\eqref{1713}, we get that
\begin{equation}\label{1524}
v_F < v_{\infty} = \frac{u_{\infty}}{c} < \frac{u_F}{c}.
\end{equation}
As a consequence, since the inequality in~\eqref{1524} is strict,
there exists~$T'>0$
such that
\begin{equation}\label{2334}
	v_0(T') < \frac{u_0(T')}{c}.
\end{equation}
Moreover, since~$\dot{u}<0$ for~$v>1-u$ and~$a=0$, we get that 
$u_0(t)$ is decreasing in~$t$, and therefore~$u_F < u_0(T') < u_0$.

By the strict inequality in~\eqref{2334}, and 
claim~(ii) in Proposition~\ref{prop:bhva},
we have that~$(u_0(T'), v_0(T')) \in \mathcal{E}_{\infty}$, where~$\mathcal{E}_{\infty}$ is defined in~\eqref{def:Einfty}.
In particular, we have that $$(u_0(T'), v_0(T')) \in
\underset{a>a'}{\bigcup} \mathcal{E}(a),$$ for every~$a'>0$.

Consequently, there exists~$a^*>M$ such that
$$(u_0(T'),v_0(T'))\in\mathcal{E}(a^*).$$ Therefore, applying the strategy
		\begin{equation*}
		a(t) = \left\{
		\begin{array}{lr}
		0,  & t<T', \\
		a^*,  & t\geq T,
		\end{array}
		\right.
		\end{equation*}
	we reach the claim $(u_0,v_0)\in \mathcal{V}_{\mathcal{A}}$.
\end{proof}

\section{Winning strategies}

To avoid repeating passages in the proofs of Theorems \ref{thm:Vbound} and \ref{thm:W}, we first state and prove the following lemma:

\begin{lemma}\label{lemma:rho=1}
	If~$\rho=1$, then for all~$a>0$:
	\begin{enumerate}
		\item the curve
			\begin{equation*}
				v(\tau)=\frac{e^{c\tau} \left( c\frac{v(0)}{u(0)}-1 \right) +1}{c} u(\tau) \quad \text{for}  \ \tau>0,
			\end{equation*}
			is a parametrisation of the trajectory starting in $(u(0), v(0))$.
			
		\item  we have~$\mathcal{E}(a)=\mathcal{S}_c$, where~$\mathcal{S}_c$ is defined in~\eqref{def:S_c}.	
	\end{enumerate}

\end{lemma}

\begin{proof}
	Let~$(u(t),v(t))$ be a trajectory starting at a point
	in~$[0,1]\times[0,1]$.
	For any~$a>0$, we consider the function
	$$\mu(t):= \frac{\displaystyle
		v(t/a)}{\displaystyle u\left(t/a\right)}.$$
	Notice that
	\begin{equation}\label{-1-e}\begin{split}&
	(u(0),v(0))\in\mathcal{E}(a)
	{\mbox{ if and only if}}\\ &{\mbox{there exists~$T>0$ such that }}\mu(T)=0.\end{split}
	\end{equation}
	In addition, we observe that
	\[\begin{split}
	\dot{\mu}(t) \,&= 
	\frac{\displaystyle\dot v\left(t/a\right)u\left(t/a\right)-
		v\left(t/a\right)\dot u\left(t/a\right)}{\displaystyle au^2\left(t/a\right)}
	\\&=
	\frac{\displaystyle -u^2 \left(t/a\right)+c u\left(t/a\right)v\left(t/a\right)}{\displaystyle u^2\left(t/a\right)}\\&=c\mu(t) -1.
	\end{split}\]
	We deduce that
	\begin{equation}\label{18131801}
	\mu(t)=\frac{e^{ct} \left( c\mu(0)-1 \right) +1}{c}. 
	\end{equation}
	This proves the first part of the Lemma.
	
	{F}rom \eqref{18131801} and~\eqref{-1-e}, we deduce that
	\begin{equation*}
	(u(0),v(0))\in\mathcal{E}(a)
	{\mbox{ if and only if }}
	c\mu(0)-1<0.
	\end{equation*}
	This leads to
	\begin{equation*}
	(u(0),v(0))\in\mathcal{E}(a)
	{\mbox{ if and only if }}\,
	\frac{v(0)}{u(0)}<\frac1c,
	\end{equation*}
	which, recalling the definition of~$\mathcal{S}_c$
	in~\eqref{def:S_c}, ends the proof.
\end{proof}

Now we provide the proof of Theorem~\ref{thm:Vbound}, exploiting the result obtained
in Section~\ref{explowrewwt}.

\begin{proof}[Proof of Theorem~\ref{thm:Vbound}] \quad \\
{\em (i)} Let~$\rho=1$. 
We claim that
\begin{equation}\label{Thns932}
{\mbox{$\mathcal{V}_{\mathcal{A}}=\mathcal{S}_c$,}}\end{equation} where~$\mathcal{S}_c$ was
defined in~\eqref{def:S_c}
(incidentally,~$\mathcal{S}_c$
is precisely the right-hand-side of equation~\eqref{Vbound1}).

{F}rom Lemma \ref{lemma:rho=1} we have that for $\rho=1$ and $a>0$ it holds $\mathcal{S}_c=\mathcal{E}(a)\subseteq \mathcal{V}_{\mathcal{A}}$. Thus, to show \eqref{Thns932} we just need to check that
\begin{equation}\label{inturn}
\mathcal{V}_{\mathcal{A}} \subseteq \mathcal{S}_c,\end{equation}
which is equivalent to
\begin{equation}\label{dotdot}
\mathcal{S}_c^C \subseteq \mathcal{V}_{\mathcal{A}}^C,
\end{equation}
where the superscript~$C$ denotes the complement of the set in the topology of~$[0,1]\times[0,1]$.

In order to apply Lemma \eqref{lemma:exit}, we analyze the behavior of the trajectories at~$\partial \mathcal{S}_c^C$. 
It holds that~$$\partial \mathcal{S}_c^C \cap \big(
[0,1]\times[0,1]\big)= \left\{ (u,v)\in (0,1)\times(0,1) \;{\mbox{ s.t. }}\; v=\frac{u}{c} \right\}.$$
Then, let us analyse the curve $v=\frac{u}{c}$ for $u\in[0, u_M]$ where $u_M=\min \{ 1, c  \}$.
By Lemma \ref{lemma:rho=1}, for all $a> 0$ and by Proposition \ref{prop:bhvaPRE}  if $a=0$, the considered curve is a parametrisation of a trajectory.
Thus, for $\check u\in[0, u_M]$, for all points in the form $(\check u,\check v)=(\check u,\check u/c)$, it holds that the trajectory $(u(t), v(t))$ starting at $(\check u,\check v)$ at $t=0$ satisfies
\begin{equation}\label{1742}
	(\dot{u}(0),\dot{v}(0))\cdot \nu =0 \quad \text{for} \ \check u\in[0, u_M], \check v=\check u/c
\end{equation} 
where $\nu$ is the outward unit normal vector to $\partial{S}_c^C$ at $(\check u, \check v)$.

Hence, by choosing $K_1=[0,u_M]$, thanks to \eqref{1742} we can apply  Lemma \ref{lemma:exit} and conclude that
\begin{equation*}\label{1649}
{\mbox{no trajectory can exit~$\mathcal{S}_c^C$.}}
\end{equation*}

Next, we observe that
\begin{equation*}\label{1617}
\mathcal{S}_c^C \cap ((0,1]\times\{0\})=\varnothing,
\end{equation*}
so no trajectory starting in~$\mathcal{S}_c^C$ can reach the set~$(0,1]\times\{0\}$.

Therefore, $$\mathcal{S}_c^C \cap \mathcal{V}_{\mathcal{A}}= \varnothing$$ and this implies that~\eqref{dotdot} is true. 
As a result, the proof of~\eqref{inturn} is established and the proof is completed for $\rho=1$.

	\medskip
	
	{\em (ii)} Let~$\rho<1$. 
Let~$\mathcal{Y}$ be the set in the right-hand-side of~\eqref{bound:rho<1}, and
\begin{equation}\label{qwertyuiolkjhgf}
\mathcal{F}_0:=
\big\{  (u,v)\in [0,1]\times [0,1]\;{\mbox{ s.t. }}\; 
v < \gamma_{0}(u) \,\text{ if } \, u\in[0, 1]  \big\}.\end{equation}
Notice that 
\begin{equation}\label{8ujff994-p-1}
\mathcal{Y} = \mathcal{F}_0 \cup \mathcal{P},
\end{equation}
being~$\mathcal{P}$ the set defined in~\eqref{PPDEFA}.

Moreover,
\begin{equation}\label{8ujff994-p-2}
\mathcal{P}\subseteq \mathcal{V}_{\mathcal{A}},
\end{equation}
thanks to
Proposition~\ref{prop:construction}.

We also claim that
\begin{equation}	\label{8ujff994-p-3BIS}\mathcal{F}_0\subseteq \mathcal{V}_{\mathcal{K}},\end{equation}
where~$\mathcal{K}$ is the set of constant functions.
Indeed, if~$(u,v)\in\mathcal{F}_0$, we have that~$v < \gamma_{0}(u)$
and consequently~$v < \gamma_{a}(u)$, as long as~$a$ is small enough,
due to Lemma~\ref{lemma:conv_gamma}.

{F}rom this and Proposition~\ref{prop:char}, we deduce that~$(u,v)$
belongs to~${\mathcal{E}}(a)$, as long as~$a$ is small enough,
and this proves~\eqref{8ujff994-p-3BIS}.

{F}rom~\eqref{8ujff994-p-3BIS} and the fact that~$\mathcal{K}\subseteq\mathcal{A}$,
we obtain that
\begin{equation}	\label{8ujff994-p-3}\mathcal{F}_0\subseteq \mathcal{V}_{\mathcal{A}}.\end{equation}
Then, as a consequence of~\eqref{8ujff994-p-1},
\eqref{8ujff994-p-2} and~\eqref{8ujff994-p-3},
we get that~$\mathcal{Y}\subseteq \mathcal{V}_{\mathcal{A}}$.

Hence, we are left with proving that 
\begin{equation}\label{8iujdpp-1}
\mathcal{V}_{\mathcal{A}} \subseteq \mathcal{Y}.\end{equation}
For this, we show that
\begin{equation}\label{8iujdpp-2}\begin{split}&
{\mbox{no trajectory enters $\mathcal{Y}$ }}\\&{\mbox{through $\partial \mathcal{Y}\cap\big([0,1]\times[0,1]
		\big) $.}}\end{split}
\end{equation}
To prove this, in order to be able to apply Lemma \ref{lemma:entrance},
we calculate the outward normal derivative on the part of~$\partial \mathcal{Y}$ lying on the graph of~$v=\gamma_0(u)$ for $u\in[0, u_s]$, that is, up to a positive constant of normalization,
\begin{equation*}
\dot{v}-\frac{\rho v^0_s \, u^{{\rho} -1}\dot{u}}{(u^0_s)^{\rho} }={\rho} v(1-u-v)-a u -\frac{\rho v^0_s \, u^{{\rho} }(1-u-v-ac)}{ (u^0_s)^{\rho} }.
\end{equation*}
By substituting for~$v=\gamma_0(u)=\frac{v^0_s \, u^\rho}{(u_s^0)^{\rho}}$ we get 
\begin{eqnarray*}&&
	\dot{v}-\frac{\rho v^0_s \, u^{{\rho} -1}\dot{u}}{ (u^0_s)^{\rho-1} }\\&=&
	\frac{\rho v^0_s \, u^\rho}{ (u_s^0)^{\rho}}(1-u-v)-a u -\frac{\rho v^0_s \, u^{{\rho} }(1-u-v-ac)}{ (u^0_s)^{\rho} }\\
	&=& -a u +\frac{ ac\rho v_s^0 \, u^{{\rho} }}{ (u^0_s)^{\rho} }	\\&
	=&au^{\rho}  \left( - u^{1-{\rho} } + \frac{c\rho v_s^0}{(u^0_s)^{\rho} } \right)\\&=&au^{\rho}  \left( - u^{1-{\rho} } + \frac{1}{(u^0_s)^{\rho-1} } \right)
	.\end{eqnarray*}
As a result, since~${\rho} <1$, we have
\begin{equation}\label{ygfbv7r9yty4}
\dot{v}-\frac{ \rho v_s^0\, u^{{\rho} -1}\dot{u}}{(u^0_s)^{\rho} }> 0 \quad \text{for} \ u\in(0, u_s)
\end{equation}
and
\begin{equation}\label{ygfbv7r9yty42}
\dot{v}-\frac{ \rho v_s^0\, u^{{\rho} -1}\dot{u}}{(u^0_s)^{\rho} }= 0 \quad \text{for} \ u\in\{0, u_s\}.
\end{equation}
Notice that $\partial \mathcal{Y}$ coincides on the line~$v=\frac{u}{c} + \frac{1-{\rho} }{1+{\rho} c}$ for
$$ u\in \left[ u_s^0, \min\left\{1, \frac{\rho c(c+1)}{1+\rho c}  \right\}  \right]  $$
For the sake of simplicity, we suppose
that $$\frac{\rho c(c+1)}{1+\rho c}\ge1.$$
Then, on the part of~$\partial \mathcal{Y}$ contained on the line~$v=\frac{u}{c} + \frac{1-{\rho} }{1+{\rho} c}$, the outward normal derivative is, up to the constant $1/\sqrt{1+(1/c^2)}$,
\begin{equation}\label{7undws8uf8v}\begin{split}&
\dot{v}-\frac{\dot{u}}{c}\\ =\;& {\rho} v(1-u-v) -au -\frac{u(1-ac-u-v)}{c}\\=\;&\left({\rho} v-\frac{u}{c}\right)(1-u-v)\\=\;&
\left(
\frac{\rho u}{c} + \frac{\rho(1-{\rho}) }{1+{\rho} c}-\frac{u}{c}\right)\left(1-u-
\frac{u}{c}- \frac{1-{\rho} }{1+{\rho} c}\right)\\
=\;&
\left(
\frac{(\rho-1) u}{c} + \frac{\rho(1-{\rho}) }{1+{\rho} c}\right)\left(-
\frac{u(c+1)}{c}+ \frac{\rho(1+c) }{1+{\rho} c}\right)
.\end{split}
\end{equation}
We also observe that, when~$u > u_s^0=\frac{\rho c}{1+\rho c}$, the condition~$\rho<1$ gives that
\begin{eqnarray*}
	\frac{(\rho-1) u}{c} + \frac{\rho(1-{\rho}) }{1+{\rho} c} < 
	\frac{\rho (\rho-1)}{1+\rho c}+ \frac{\rho(1-{\rho}) }{1+{\rho} c}=0
\end{eqnarray*}
and
\begin{eqnarray*}-\frac{u(c+1)}{c}+ \frac{\rho(1+c) }{1+{\rho} c}<
	-\frac{\rho(c+1)}{1+\rho c}+ \frac{\rho(1+c) }{1+{\rho} c}=0.
\end{eqnarray*}
Therefore, when~$u > u_s^0$, we deduce from~\eqref{7undws8uf8v} that~$$
\dot{v}-\frac{\dot{u}}{c}> 0,$$
and 
$$\dot{v}-\frac{\dot{u}}{c} =0 \quad \text{for } \ u=u_s^0.$$
Combining this,~\eqref{ygfbv7r9yty4} and~\eqref{ygfbv7r9yty4}, we
can say that for
$K_1=\{0\}$, $K_2=\{u_s^0\}$, $I_1=(0, u_s^0)$, $I_2=(u_s^0, 1]$
we can apply  Lemma \ref{lemma:entrance}
 obtaining~\eqref{8iujdpp-2}, as desired.

Since for any value of~$a$, no trajectory starting in~$\big([0,1]\times[0,1]
\big)\setminus\mathcal{Y}$ can enter in~$\mathcal{Y}$ passing through $\partial \mathcal{Y}\cap ([0,1]\times[0,1])$, 
in particular no trajectory starting
in~$\big([0,1]\times[0,1]
\big)\setminus\mathcal{Y}$
can hit~$\{v=0\}$, which ends the proof of~\eqref{8iujdpp-1}.
\medskip

{\em (iii)} Let~$\rho>1$. For the sake of simplicity, we suppose
that$$\frac{c}{(c+1)^\rho}\ge1.$$
Let~$\mathcal{X}$ be the right-hand-side of~\eqref{bound:rho>1}.
We observe that
\begin{equation}\label{7hperpre923i5}
\mathcal{X}= \mathcal{S}_{c} \cup \mathcal{Q},
\end{equation}
where~$\mathcal{S}_{c}$ was defined in~\eqref{def:S_c} and~$\mathcal{Q}$ in~\eqref{DEFQ}.

Thanks to Proposition~\ref{prop:bhva},
one has that
$$\mathcal{S}_{c}\subseteq \underset{a>a'}{\bigcup} \mathcal{E}(a),$$ for every~$a'>0$, and therefore~$\mathcal{S}_{c}\subseteq\mathcal{V}_{\mathcal{A}}$.

Moreover, by the second claim in
Proposition~\ref{prop:construction}, one also has that~$\mathcal{Q}\subseteq \mathcal{V}_{\mathcal{A}}$. Hence, 
\begin{equation}\label{1923}
\mathcal{X}\subseteq \mathcal{V}_{\mathcal{A}}.
\end{equation}

Accordingly, to prove equality in~\eqref{1923} and thus complete
the proof of~\eqref{bound:rho>1}, we need to show that~
\begin{equation}\label{2026}
	\mathcal{V}_{\mathcal{A}} \subseteq \mathcal{X}.
\end{equation}
First, we prove that
\begin{equation}\label{1229}
(0,1]\times\{0\} \subseteq \mathcal{X}.
\end{equation}
Indeed, for~$u>0$ we have~$v=\frac{u}{c}>0$, therefore~$(u, 0)\in \mathcal{X}$ for~$u\in (0, u_{\infty}]$. Then,~$\zeta(u)$ is increasing in~$u$ since it is a positive power function, therefore~$v=\zeta(u)>0$ for~$u\in(u_{\infty}, 1]$, hence ~$(u, 0)\in \mathcal{X}$ for~$u\in ( u_{\infty}, 1]$. These observations prove~\eqref{1229}.

Now, to show that
\begin{equation}
    \mbox{no trajectory enters} \ \mathcal{X},
\end{equation}
we want to apply Lemma \ref{lemma:entrance} with $K_1=\{0\}$, $I_1=(0, u_{\infty})$, $K_2=\{u_{\infty}\}$, $I_2=\left(u_{\infty},1 \right]$.

First, we notice that $(0,0)\in \partial \mathcal{X}$ is an equilibrium, thus the component of the velocity are 0 in every direction. This proves \eqref{property2} for $K_1=\{0\}.$

We now prove that the component of the velocity field in the outward normal direction with respect to~$\mathcal{X}$  is positive on
\begin{equation*} 
\left\{ (u,v)\in(0,u_{\infty})\times(0,1) \ : \ v=\frac{u}{c} \right\}.
\end{equation*}
In fact, it is 
\begin{equation}\label{1853}\begin{split}
\dot{v}-\frac{1}{c}\dot{u}&= \rho v(1-u-v)-au -\frac{u}{c}(1-ac-u-v)\\&=\left(\rho v -\frac{u}{c}\right)(1-u-v).\end{split}
\end{equation}
The first term is positive because for~$\rho >1$ we have
\begin{equation*}
\rho v > v =\frac{u}{c}.
\end{equation*}
Moreover, for~$u< u_{\infty}$ we have that~$$1-u-v>1-u_{\infty}-\frac{u_{\infty}}{c}=0,$$
thanks to~\eqref{ZETADEF}.
Thus, the left hand side of~\eqref{1853} is positive for $u\in I_1=(0, u_{\infty})$.

On the part of~$\partial \mathcal{X}$ lying in the graph of~$v=\zeta(u)$, that is, for $u\in I_2$, the component of the velocity field in the outward normal direction is
given by
\begin{equation}\label{po091326uthgbvfjf}\begin{split}&
\dot{v}-\frac{\rho u^{{\rho} -1}\dot{u}}{\rho c(u_{\infty})^{\rho-1} }\\&\qquad ={\rho} v(1-u-v)-a u -\frac{{\rho} u^{{\rho} }}{\rho c(u_{\infty})^{\rho-1} }(1-u-v-ac).\end{split}
\end{equation}
Now we substitute for
$$v=\zeta(u)= \frac{u^{{\rho} }}{\rho c(u_{\infty})^{\rho-1} }$$ in~\eqref{po091326uthgbvfjf} and we get 
\begin{equation}\label{1850}
\dot{v}-\frac{u^{{\rho} -1}\dot{u}}{ c(u_{\infty})^{\rho-1} } =	au  \left( - 1 + \frac{u^{\rho-1}}{(u_{\infty})^{\rho-1} } \right)
\end{equation}
which leads to
\begin{equation*}
\dot{v}-\frac{\rho u^{{\rho} -1}\dot{u}}{\rho c(u_{\infty})^{\rho-1} }>0 \quad \text{if} \ u>u_{\infty},
\end{equation*}
as desired. 
This proves the hypothesis \eqref{property} for $u\in I_2$.

Moreover, again by the expression in \eqref{1853} and to~\eqref{ZETADEF}, for $u=u_{\infty}$, the scalar product of the direction of the trajectory with the normal vector to~$\left\{v=\frac{u}{c} \right\}$ is zero. 
Also, by \eqref{1850}, for $u=u_{\infty}$, the scalar product of the direction of the trajectory with the normal vector to~$\left\{v=\zeta (u) \right\}$ is zero. Thus, the hypothesis \eqref{property2} of Lemma \ref{lemma:entrance} is satisfied in $K_2=\{u_{\infty}\}$.

As a consequence of these considerations, by Lemma \ref{lemma:entrance} we find that no trajectory starting in~$\mathcal{X}^C$ can enter in~$\mathcal{X}$ and therefore hit~$\{v=0\}$, by~\eqref{1229}. 

Hence, we conclude that~\eqref{2026} holds true, which, together with~\eqref{1923}, establishes~\eqref{bound:rho>1}.
\end{proof}

\section{The role of the constant strategies}

In order to prove Theorem~\ref{thm:W},
we will establish a geometrical lemma in order to
understand the reciprocal position of the function~$\gamma$,
as given by Propositions~\ref{lemma:M} and~\ref{M:p045},
and the straight line where the saddle equilibria lie. To emphasize the dependence of~$\gamma$
on the parameter~$a$ we will often use the notation~$\gamma=\gamma_a$.
Moreover, we recall the notation of the saddle points~$(u_s,v_s)$
defined in~\eqref{usvs} and of the points~$(u_{\mathcal{M}},v_{\mathcal{M}})$ given by
Propositions~\ref{lemma:M} and~\ref{M:p045}, with the convention
that
\begin{equation}\label{usvs2}
{ \mbox{$(u_s,v_s)=(0,0)$ if~$ac\ge1$,}  }
\end{equation}
and we state the following
result:

\begin{lemma}\label{lemma:vett_tg}
If~$\rho<1$, then 
\begin{equation}\label{gamma1>r}
\frac{u}{{\rho}c} \leq \gamma_a(u) \quad \text{ for }  u\in[0, u_s]
\end{equation}
and
\begin{equation}\label{gamma<r}
\gamma_a(u) \leq \frac{u}{{\rho}c} \quad \text{ for }  u\in[u_s, u_{\mathcal{M}}].
\end{equation}
If instead~${\rho}>1$, then
\begin{equation}\label{gamma1<r}
\gamma_a(u) \leq \frac{u}{{\rho}c} \quad \text{ for }  u\in[0, u_s]
\end{equation}
and
\begin{equation}\label{gamma>r}
\frac{u}{{\rho}c} \leq \gamma_a(u)  \quad \text{ for }  u\in[u_s, u_{\mathcal{M}}].
\end{equation}
Moreover equality holds in~\eqref{gamma1>r} and~\eqref{gamma1<r} 
if and only if either~$u=u_s$ or~$u=0$. Also, strict inequality
holds in~\eqref{gamma<r}
and~\eqref{gamma>r} for~$u\in(u_s, u_{\mathcal{M}})$.
\end{lemma}

\begin{proof}
We focus here on the proof of~\eqref{gamma<r}, since
the other inequalities are proven in a similar way. Moreover,
we deal with the case~$ac<1$, being the case~$ac\ge1$ analogous with
obvious modifications.

We suppose by contradiction that~\eqref{gamma<r} does not hold true.
Namely, we assume that there exists~$\tilde{u}\in(u_s,u_{\mathcal{M}}]$ such that
$$  \gamma_a(\tilde u) > \frac{\tilde u}{{\rho}c}.$$
Since~$\gamma_a$ is continuous thanks to Propositions~\ref{lemma:M},
we have that
$$  \gamma_a( u) > \frac{ u}{{\rho}c} \quad \mbox{in a neighborhood of~$\tilde u$. }$$
Hence, we consider the largest open
interval~$(u_1,u_2)\subset(u_s,u_{\mathcal{M}}]$ containing~$\tilde u$ and such that
\begin{equation} \label{g>r}
\gamma_a(u) > \frac{u}{{\rho}c} \quad {\mbox{ for all }} u \in (u_1,u_2).
\end{equation}
Moreover, in light of~\eqref{usvs}, we see that
\begin{equation}\label{togdfgheter}
\gamma_a(u_s)=v_s= \frac{1-ac}{1+\rho c}=
\frac{u_s}{\rho c}.\end{equation}
Hence, by the continuity of~$\gamma_a$, we have that~$\gamma_a(u_1)=\frac{u_1}{\rho c}$
and 
\begin{equation}\label{doesnotcon}
{\mbox{either~$\gamma_a(u_2)=\displaystyle \frac{u_2}{\rho c}$ or~$u_2=u_{\mathcal{M}}$.}}\end{equation}

Now, we consider the set
\begin{equation*}
\mathcal{T}:= \left\{  (u,v)\in [u_1,u_2]\times[0,1] \;
{\mbox{ s.t. }}\; \frac{u}{{\rho}c} < v<  \gamma_a(u)      \right\},
\end{equation*}
that is non empty, thanks to~\eqref{g>r}. 
We claim that
\begin{equation}\label{lkjhgfds1234567}\begin{split}&
{\mbox{for all~$(u(0), v(0))\in \mathcal{T}$,}}\\&{\mbox{the~$\omega$-limit
		of its trajectory is~$(u_s,v_s)$.}}\end{split}\end{equation}
To prove this, 
we analyze the normal derivative on 
\begin{equation*}\begin{split}
&\partial \mathcal{T} =\mathcal{T}_1\cup\mathcal{T}_2\cup
\mathcal{T}_3,\\
{\mbox{where }}\quad &
\mathcal{T}_1:=\big\{ (u, \gamma_a(u)) \;{\mbox{ with }} 
u \in (u_1,u_2) \big\},\\
&\mathcal{T}_2:= \left\{  \left( u, \frac{u}{{\rho}c} \right) \;{\mbox{ with }}
u \in (u_1,u_2)  \right\}\\
{\mbox{and }}\quad &
\mathcal{T}_3:=\left\{ (u_2, v) \;{\mbox{ with }} 
v \in \left( \frac{u_2}{{\rho}c},\min\{\gamma_a(u_2),1\}\right) \right\}
,
\end{split}\end{equation*}
with the convention that~$\partial \mathcal{T}~$ does contain~$\mathcal{T}_3$
only if the second possibility in~\eqref{doesnotcon} occurs.

We notice that the set~$\mathcal{T}_1$ is an orbit for the system, and
thus the component of the velocity in the normal direction is null. 
On~$\mathcal{T}_2$, we have that the sign of
the component of the velocity in the inward normal direction is given by
\begin{equation}\label{der:r}
\begin{split}&
(\dot{u},\dot{v})\cdot\left(-\frac1{\rho c},1\right)\\
=\;&
\dot{v} - \frac{1}{{\rho}c} \dot{u} \\&= {\rho} v(1-u-v) 
-au - \frac{u}{{\rho}c}(1-u-v) + \frac{au}{\rho} \\
=\;& \frac{u}{c} \left( 1-u-\frac{u}{{\rho}c}   \right)
\left( 1 - \frac{1}{{\rho}}   \right) -au\left( 1-\frac{1}{\rho}  
\right)  \\
 =\;& \frac{u}{c} \left(1-\frac{1}{\rho} \right) \left(  
1-u-\frac{u}{{\rho}c}  -ac \right) .
\end{split}
\end{equation}
Notice that  for~$u \geq u_s$ we have that
\begin{equation}\label{acapo}
1-u-v -ac  \leq0,
\end{equation}
thus the sign of last term in~\eqref{der:r} depends only on
the quantity~$1-\frac{1}{\rho}$.
Consequently, since~${\rho}<1$ the sign of
the component of the velocity in the inward normal direction is positive. 

Furthermore, in the case in which the second possibility in~\eqref{doesnotcon} occurs,
we also check the sign of the component of the velocity in the inward normal direction
along~$\mathcal{T}_3$. In this case, if~$\gamma_a(u_2)<1$ then~$u_2=1$,
and therefore we find that
$$(\dot{u},\dot{v})\cdot\left(-1 ,0 \right)=-\dot{u}=-u(1-u-v)+acu=
v+ac,
$$
which is positive. If instead~$\gamma_a(u_2)=1$
$$(\dot{u},\dot{v})\cdot\left(-1 ,0 \right)=-\dot{u}=-u(1-u-v)+acu=
-u(1-ac-u-v)
,
$$
which is positive, thanks to~\eqref{acapo}.

We also point out that there are no cycle in~$\mathcal{T}$, since~$\dot{u}$
has a sign. These considerations and the
Poincar\'e-Bendixson Theorem (see e.g.~\cite{TESCHL})
give that the~$\omega$-limit set of~$(u(0),v(0))$
can be either an equilibrium or a union of (finitely many)
equilibria and non-closed orbits connecting these equilibria.
Since~$(0,0)$ and~$(0,1)$ do not belong to the closure of~$\mathcal{T}$,
in this case the only possibility is that the~$\omega$-limit is the equilibrium~$(u_s,v_s)$.
Consequently, we have that~$u_1=u_s$, and that~\eqref{lkjhgfds1234567}
is satisfied.

Accordingly, in light of~\eqref{lkjhgfds1234567}, we have that the
set~$\mathcal{T}$
is contained in the stable manifold of~$(u_s,v_s)$, which is in contradiction
with the definition of~$\mathcal{T}$.
Hence,~\eqref{gamma<r} is established, as desired.
\medskip

Now we show that strict inequality holds true in~\eqref{gamma<r}
if~$u\in(u_s,u_{\mathcal{M}})$.	
To this end, we suppose by contradiction that
there exists~$\bar{u}\in (u_s,u_{\mathcal{M}})$ such that
\begin{equation}\label{equality}
\gamma_a(\bar{u})=\frac{\bar{u}}{{\rho}c}.
\end{equation} 
Now, since~\eqref{gamma<r} holds true, we have that
the line~$v-\frac{u}{{\rho}c}=0$ is tangent to the curve~$
v=\gamma_a(u)$ at~$(\bar{u}, \gamma_a(\bar{u}))$,
and therefore at this point the components of the velocity
along the normal directions to the curve and to the line coincide. 
On the other hand,
the normal derivative at a point on
the line has a sign, as computed in~\eqref{der:r}, while the normal derivative
to~$v=\gamma_a(u)$ is~$0$ because the curve is an orbit. 

This, together with~\eqref{togdfgheter}, proves that equality
in~\eqref{gamma<r} holds true if~$u=u_s$, but strict inequality holds true
for all~$u\in(u_s,u_{\mathcal{M}})$,
and thus
the proof of Lemma~\ref{lemma:vett_tg} is complete.
\end{proof}

For each~$a>0$, we define~$(u_d^a, v_d^a)\in [0,1]\times[0,1]$ as the unique intersection of the graph of~$\gamma_a$ with the line~$\{v=1-u\}$, that is the solution of the system
\begin{equation}\label{ki87yh556g}
\left\{
\begin{array}{l}
v_d^a=\gamma_a(u_d^a),\\
v_d^a=1- u_d^a.
\end{array}
\right.
\end{equation}
We recall that the above intersection is
unique since the function~$\gamma_a$
is increasing. Also, by construction,
\begin{equation}\label{CALM}
u_d^a\le u_{\mathcal{M}}.
\end{equation}
Now, recalling~\eqref{usvs}  and making explicit the dependence on $a$ by writing $u_s^a$
(with the convention in~\eqref{usvs2}), we give the following result:

\begin{lemma}\label{lemma:ord}
We have that:
\begin{enumerate}
	\item For $\rho<1$, for all $a^*>0$ it holds that
	\begin{equation}\label{1304b}\begin{split}
	& \gamma_a(u) \leq \gamma_{a^*}(u) \\& \text{for all} \ a > a^* \ \text{and for all} \ u\in[u_s^{a^*}, u_d^{a^*}].\end{split}
	\end{equation}
	
	\item  For $\rho>1$, for all $a^*>0$ it holds that
	\begin{equation}\label{1819b}\begin{split}
&	\gamma_a(u) \leq \gamma_{a^*}(u) \\&\text{for all} \ a < a^* \ \text{and for all} \ u\in[u_s^{a^*}, u_d^{a^*}].\end{split}
	\end{equation}
\end{enumerate}
\end{lemma}

\begin{proof}
We claim that
\begin{equation}\label{1306}
u_s^{a^*} < u_d^{a^*}.
\end{equation}
Indeed, when~$a^* c\ge1$, we have that~$u_s^{a^*} =0< u_d^{a^*}$
and thus~\eqref{1306} holds true. If instead~$
a^* c<1$, by \eqref{usvs} and~\eqref{ki87yh556g} we have that
\begin{equation} \label{8uhj76tuyg6446r6f6}\gamma_{a^*}(u_s^{a^*})+u_s^{a^*}=1-a^* c< 1=
\gamma_{a^*}(u_d^{a^*})+u_d^{a^*}.\end{equation}
Also, since~$\gamma_{a^*}$ is increasing,
we have that the map~$r\mapsto \gamma_{a^*}(r)+r$
is strictly increasing. Consequently, we deduce from~\eqref{8uhj76tuyg6446r6f6} that~\eqref{1306}
holds true in this case as well.

Now we suppose that $\rho<1$ and we prove~\eqref{1304b}.
For this, we claim that, for every~$a^*>0$ and every~$a>a^*$,
\begin{equation}\label{xcvbn881300}\begin{split}&\gamma_{a}(
u_s^{a^*})\le\gamma_{a^*}( u_s^{a^*})\\&{\mbox{ with strict inequality when }}
a^*\in\left(0,\frac{1}{c}\right).\end{split}
\end{equation}
To check this, we distinguish two cases.
If $a^*\in\left(0,\frac{1}{c}\right)$, then for all $a>a^*$
\begin{equation}\label{1300}
u_s^a=\max \left\{ 0, \rho c \frac{1-ac}{1+ \rho c} \right\} <  \rho c \frac{1-a^*c}{1+ \rho c}  =u_s^{a^*}.
\end{equation}
By \eqref{1300} and formula \eqref{gamma<r} in Lemma \ref{lemma:vett_tg}, we have that
\begin{equation}\label{1647}
\gamma_a(u_s^{a^*}) < \frac{u_s^{a^*}}{\rho c} = \gamma_{a^*}(u_s^{a^*}) \quad \text{for all} \ a> a^*.
\end{equation}
If instead $a^*\geq \frac{1}{c}$, then $u_s^{a^*}=0$ and for all $a>a^*$ we have $u_s^a=0$. As a consequence, 
\begin{equation}\label{1647b}
\gamma_{{a^*}}(u_s^{a^*})=\gamma_{{a}}(u_s^{a^*}) \quad \text{for all} \ a> a^* .
\end{equation}
The claim in \eqref{xcvbn881300} thus follows from~\eqref{1647} and~\eqref{1647b}.

Furthermore, by Propositions \ref{lemma:M} and \ref{M:p045},
\begin{equation}\label{1641b}\begin{split}&
\gamma_a'(0)= \frac{a}{\rho+ac-1} < \frac{{a^*}}{\rho+{a^*}c-1}=\gamma_{a^*}'(0) \\& \text{for all} \ a> a^*\ge\frac1c.
\end{split}\end{equation}
Moreover, for all $a\ge a^*$ and $u>u_s^{a^*}$  it holds that, when~$v=\gamma_{a^*}(u)$,
\begin{equation}\label{1623b}\begin{split}
-\big(
acu-u(1-u-v)
\big)&=
u(1-u-\gamma_{a^*}(u)- ac)\\& < u(1-u_s^{a^*}-v_s^{a^*}-ac)\\&\le 0.\end{split}
\end{equation}
Now, we establish that
\begin{equation}\label{767675747372}\begin{split}&
u(\rho c  v-u)(1-u-v)(a-a^*) < 0 \\ &\qquad \text{ for all} \ a > a^*, \   u
\in(u_s^{a^*}, u_d^{a^*}), \  v=\gamma_{a^*}(u).\end{split}
\end{equation}
Indeed,
for the values of $a$, $u$ and $v$ as in \eqref{767675747372} we have
that~$v\le\gamma_{a^*}(u_d^{a^*})$ and hence
\begin{equation}\label{767675747372-2}
(1-u-v)> (1-u_d^{a^*}-\gamma_{a^*}(u_d^{a^*}))=0.
\end{equation}
Moreover, by formula \eqref{gamma<r} in Lemma \ref{lemma:vett_tg},
for $u\in(u_s^{a^*}, u_d^{a^*})$ and $v=\gamma_{a^*}(u)$ and  we have that $$\rho c v -u =
\rho c \gamma_{a^*}(u) -u<
0.$$
{F}rom this and~\eqref{767675747372-2}, we see that~\eqref{767675747372} plainly follows, as desired.

As a consequence of~\eqref{1623b} and~\eqref{767675747372}, one deduces that,
for all~$a > a^*$, $u\in(u_s^{a^*}, u_d^{a^*})$ and~$v=\gamma_{a^*}(u)$,
\begin{equation}\label{1335}\begin{split}
&
\frac{au- \rho v(1-u-v)}{acu-u(1-u-v)} - \frac{a^* u- \rho v(1-u-v)}{a^* cu-u(1-u-v)} \\=\,&
\frac{(a-a^*)c \rho uv(1-u-v)-(a-a^*) u^2(1-u-v)}{\big(a cu-u(1-u-v)\big)\big(a^* cu-u(1-u-v)\big)}
\\=\,&
\frac{(a-a^*)(1-u-v)u( c \rho v- u)}{\big(a cu-u(1-u-v)\big)\big(a^* cu-u(1-u-v)\big)}
\\ \le\,&0.
\end{split}
\end{equation}
Now, we define
\begin{equation}\label{3456784jncdkc6knsbd vc83456789} {\mathcal{Z}}(u):=\gamma_a(u)-\gamma_{a^*}(u)\end{equation}
and we claim that
\begin{equation}\label{4jncdkc6knsbd vc8}
{\mbox{if~$\check u\in(u_s^{a^*}, u_d^{a^*})$ is such that~${\mathcal{Z}}(\check u)=0$,
		then~${\mathcal{Z}}'(\check u)<0$.}}
\end{equation}
Indeed,
since~$\gamma_a$ is a trajectory for~\eqref{model},
if~$(u_a(t),v_a(t))$ is a solution of~\eqref{model}, we have that~$
v_a(t)=\gamma_a(u_a(t))$, whence
\begin{equation}\label{989u:SMNDnb csn44}
\begin{split}&\rho v_a(t)(1-u_a(t)-v_a(t)) -au_a(t)\\&\qquad=
\dot v_a(t)\\&\qquad=\gamma_a'(u_a(t))\,\dot u_a(t)\\&\qquad=
\gamma_a'(u_a(t))\big( u_a(t)(1-u_a(t)-v_a(t)) - acu_a(t)\big)
.\end{split}\end{equation}
Then, we let~$\check v:=\gamma_a(\check u)$
and we notice that~$\check v$ coincides also with~$\gamma_{a^*}(\check u)$.
Hence, we
take trajectories of the system with parameter~$a$
and~$a^*$ starting at~$(\check u,\check v)$,
and by~\eqref{1335} we obtain that
\begin{eqnarray*}0
	&>&\frac{a\check u- \rho v(1-\check u-\check v)}{ac\check u-\check u(1-\check u-\check v)}-
	\frac{a^*\check u- \rho v(1-\check u-\check v)}{a^*c\check u-\check u(1-\check u-\check v)}\\&=&
	\frac{au_a(0)- \rho v(1-u_a(0)-v_a(0))}{acu_a(0)-u(1-u_a(0)-v_a(0))}\\&&\qquad\qquad-
	\frac{a^*u_{a^*}(0)- \rho v(1-u_{a^*}(0)-v_a(0))}{a^*cu_{a^*}(0)-u(1
		-u_{a^*}(0)-v_{a^*}(0))}\\&
	=&\gamma'_a(u_a(0))-\gamma'_{a^*}(u_{a^*}(0))\\&=&
	\gamma'_a(\check u)-\gamma'_{a^*}(\check u),
\end{eqnarray*}
which establishes~\eqref{4jncdkc6knsbd vc8}.

Now we claim that
\begin{equation}\label{xx124ff469}\begin{split}&
{\mbox{there exists~$\underline{u}\in[u_s^{a^*}, u_d^{a^*}]$
		such that~${\mathcal{Z}}(\underline{u})<0$}}\\&{\mbox{and~${\mathcal{Z}}(u)\le0$ for every~$u\in[u_s^{a^*},\underline{u}]$.}}
\end{split}\end{equation}
Indeed, if~$a^*\in\left(0,\frac{1}{c}\right)$,
we deduce from~\eqref{xcvbn881300}
that~${\mathcal{Z}}( u_s^{a^*})<0$
and therefore~\eqref{xx124ff469} holds true with$$\underline{u}:=
u_s^{a^*}.$$ If instead~$a^*\ge\frac{1}{c}$, 
we have that~$u_s^{a}=u_s^{a^*}=0$
and we deduce from~\eqref{xcvbn881300}
and~\eqref{1641b} that~${\mathcal{Z}}(u_s^{a^*})=0$
and~${\mathcal{Z}}'(u_s^{a^*})<0$, from which~\eqref{xx124ff469}
follows by choosing$$\underline{u}:=u_s^{a^*}+\epsilon$$
with~$\epsilon>0$ sufficiently small.

Now we claim that
\begin{equation}\label{TBP-SP-EL-34}
{\mathcal{Z}}(u)\le0\qquad{\mbox{for every }}u\in[u_s^{a^*}, u_d^{a^*}].
\end{equation}
To prove this, 
in light of~\eqref{xx124ff469}, it suffices to check that~${\mathcal{Z}}(u)\le0$
for every~$u\in(\underline{u}, u_d^{a^*}]$.
Suppose not. Then there exists~$u^\sharp\in(\underline{u}, u_d^{a^*}]$
such that~${\mathcal{Z}}(u)<0$ for all~$[\underline{u},u^\sharp)$
and~${\mathcal{Z}}(u^\sharp)=0$. This gives that~$
{\mathcal{Z}}'(u^\sharp)\ge0$.
But this inequality is in contradiction with~\eqref{4jncdkc6knsbd vc8}
and therefore the proof of~\eqref{TBP-SP-EL-34}
is complete.

The desired claim in~\eqref{1304b} follows easily
from~\eqref{TBP-SP-EL-34}, hence we focus now
on the proof of~\eqref{1819b}.
\medskip 

To this end, we take $\rho>1$ and we 
claim that, for every~$a^*>0$ and every~$a\in(0,a^*)$,
\begin{equation}\label{xcvbn881300-ALT56}\begin{split}&\gamma_{a}(
u_s^{a^*})\le\gamma_{a^*}( u_s^{a^*})\\&{\mbox{ with strict inequality when }}
a^*\in\left(0,\frac{1}{c}\right).\end{split}
\end{equation}
To prove this, we first notice that,
if $a<a^*<\frac{1}{c}$, then
\begin{equation*}
u_s^{a^*}=\rho c \frac{1-a^*c}{1+\rho c} < \rho c \frac{1-ac}{1+\rho c} = u_s^a.
\end{equation*}
Hence by \eqref{gamma1<r} in
Lemma \ref{lemma:vett_tg} we have
\begin{equation*}
\gamma_a(u_s^{a^*}) < \frac{u_s^{a^*}}{\rho c} = \gamma_{a^*}(u_s^{a^*}) \quad \text{for} \ a<a^*<\frac{1}{c}, 
\end{equation*}
and this establishes~\eqref{xcvbn881300-ALT56}
when~$a^*\in\left(0,\frac{1}{c}\right)$.
Thus, we now focus on the case~$a^*\geq \frac{1}{c}$.
In this situation, we have that~$u_s^{a^*}=0$
and accordingly~$\gamma_a(u_s^{a^*})=
\gamma_a(0)=\gamma_{a^*}(0) =
\gamma_{a^*}(u_s^{a^*})$, that completes the proof
of~\eqref{xcvbn881300-ALT56}.

In addition, by Propositions \ref{lemma:M} and \ref{M:p045} we have that
\begin{equation}\label{1821}\begin{split}&
\gamma_a'(0)= \frac{a}{\rho-1+ac} \leq \frac{{a^*}}{\rho-1+{a^*}c}=\gamma_{a^*}'(0) \\& \text{for} \ a\in\left[\frac{1}{c}, {a^*}\right].\end{split}
\end{equation}
Moreover, for $u>u_s^a$, if~$v=\gamma_a(u)$ we have that~$v>\gamma_a(u_s^a)=v_s^a$, thanks to the monotonicity of~$\gamma_a$,
and, as a result,
\begin{equation}\label{1804}
u(1-u-v-ac)<u(1-u_s^a-v_s^a-ac)=0.
\end{equation}

Now we claim that, for all~$ a< {a^*}$,
$u\in(u_s^{a^*}, u_d^{{a^*}})$ and~$  v=\gamma_{{a^*}}(u)$,
we have
\begin{equation}\label{473-bniu-1}
u(1-u-v)({a^*}-a)(u-\rho c v)<0
.\end{equation}
Indeed, 
by the monotonicity of~$\gamma_{{a^*}}$,
in this situation we have that~$v\le\gamma_{{a^*}}(u^{a^*}_d)$,
and therefore,
by \eqref{ki87yh556g}, \begin{equation}\label{473-bniu-2}
1-u-v >1-u_d^{a^*}-\gamma_{{a^*}}(u_d^{a^*})=1-u_d^{a^*}-1+u_d^{a^*}=0. \end{equation}
Moreover, by~\eqref{gamma>r}
in Lemma \eqref{lemma:vett_tg},
we have that~$ \gamma_{a^*}(u) > \frac{u}{\rho c}$, and
hence $u-\rho c v> 0$. Combining this inequality with~\eqref{473-bniu-2},
we obtain~\eqref{473-bniu-1}, as desired.

Now, by~\eqref{1804},
for all~$ a < {a^*}$,
$u\in(u_s^{a}, u_d^{{a^*}})$
and~$v=\gamma_{{a^*}}(u)$,
$$ 
0<
-u(1-u-v-ac)=acu-u(1-u-v) <
{a^*} cu-u(1-u-v)$$
and then, by~\eqref{473-bniu-1},
\begin{equation}\label{1800}
\begin{split}&
\frac{au- \rho v(1-u-v)}{acu-u(1-u-v)} - \frac{{a^*} u- \rho v(1-u-v)}{{a^*} cu-u(1-u-v)} \\=\,&
\frac{u(1-u-v)({a^*}-a)(u-\rho c v)}{\big(acu-u(1-u-v)\big)
	\big({a^*} cu-u(1-u-v)\big)}
\\ <\,&0.
\end{split}
\end{equation}
Now we recall the definition of~${\mathcal{Z}}$
in~\eqref{3456784jncdkc6knsbd vc83456789}
and we claim that
\begin{equation}\label{567890-4jncdkc6knsbd vc8}
{\mbox{if~$\check u\in(u_s^{a^*}, u_d^{a^*})$ is such that~${\mathcal{Z}}(\check u)=0$,
		then~${\mathcal{Z}}'(\check u)<0$.}}
\end{equation}
To prove this, we let~$\check v:=\gamma_a(\check u)$,
we notice that~$\check v=\gamma_{a^*}(\check u)$, we
recall~\eqref{989u:SMNDnb csn44}
and apply it to a trajectory starting at~$(\check u,\check v)$,
thus finding that
\begin{eqnarray*}
	&&\rho \check v(1-\check u-v_a(t)) -a\check u=
	\gamma_a'(\check u)\big( \check u(1-\check u-\check v) - ac\check u\big)
	.\end{eqnarray*}
This and~\eqref{1800} yield that
\begin{eqnarray*}
	0&>&\frac{au- \rho v(1-u-v)}{acu-u(1-u-v)} - \frac{{a^*} u- \rho v(1-u-v)}{{a^*} cu-u(1-u-v)}\\& =&\gamma_a'(\check u)-\gamma_{a^*}'(\check u)\\&=&{\mathcal{Z}}'(\check u),
\end{eqnarray*}
which proves the desired claim in~\eqref{567890-4jncdkc6knsbd vc8}.

We now point out that
\begin{equation}\label{6879977xx124ff469}\begin{split}&
{\mbox{there exists~$\underline{u}\in[u_s^{a^*}, u_d^{a^*}]$
		such that~${\mathcal{Z}}(\underline{u})<0$}}\\&{\mbox{and~${\mathcal{Z}}(u)\le0$ for every~$u\in[u_s^{a^*},\underline{u}]$.}}
\end{split}\end{equation}
Indeed, if~$a^*\in\left(0,\frac{1}{c}\right)$,
this claim follows directly from~\eqref{xcvbn881300}
by choosing$$\underline{u}:=
u_s^{a^*},$$ while if~$a^*\ge\frac{1}{c}$,
the claim follows from~\eqref{xcvbn881300}
and~\eqref{4jncdkc6knsbd vc8}
by choosing $$\underline{u}:=u_s^{a^*}+\epsilon$$
with~$\epsilon>0$ sufficiently small.

Now we claim that
\begin{equation}\label{jjjjdnfnfTBP-SP-EL-34}
{\mathcal{Z}}(u)\le0\qquad{\mbox{for every }}u\in[u_s^{a^*}, u_d^{a^*}].
\end{equation}
Indeed, by~\eqref{6879977xx124ff469},
we know that the claim is true for all~$u\in[u_s^{a^*},\underline{u}]$.
Then, the claim for~$u\in(\underline{u}, u_d^{a^*}]$
can be proved by contradiction,
supposing that there exists~$u^\sharp\in(\underline{u}, u_d^{a^*}]$
such that~${\mathcal{Z}}(u)<0$ for all~$[\underline{u},u^\sharp)$
and~${\mathcal{Z}}(u^\sharp)=0$. This gives that~$
{\mathcal{Z}}'(u^\sharp)\ge0$, which is in contradiction with~\eqref{4jncdkc6knsbd vc8}.

Having completed the proof of~\eqref{jjjjdnfnfTBP-SP-EL-34},
one can use it to obtain the desired claim in~\eqref{1819b}.
\end{proof}

Now we perform the proof 
of Theorem~\ref{thm:W}, analyzing separately the cases~$\rho=1$,~$\rho<1$
and~$\rho>1$.

\begin{proof}[Proof of Theorem~\ref{thm:W}, case $\rho=1$]
We notice that
\begin{equation}\label{1959}
\mathcal{V_{\mathcal{K}}}\subseteq 
\mathcal{V_{\mathcal{A}}},
\end{equation}
since~$\mathcal{K}\subset \mathcal{A}$.

Also, from Theorem~\ref{thm:Vbound}, part (i), we get that~$\mathcal{V}_{\mathcal{A}}=\mathcal{S}_c$, where~$\mathcal{S}_c$ was defined in~\eqref{def:S_c}. 
On the other hand, by Lemma~\ref{lemma:rho=1}, we know that for~$\rho=1$ and for all~$a>0$ we have~${\mathcal{E}(a)}=\mathcal{S}_c$.
But since every constant~$a$ belongs to the set~$\mathcal{K}$, we have~$\mathcal{E}(a)\subseteq \mathcal{V}_{\mathcal{K}}$.
This shows that~$\mathcal{V}_{\mathcal{A}} = {\mathcal{E}(a)}\subseteq \mathcal{V}_{\mathcal{K}}$, and together with~\eqref{1959} concludes the proof.
\end{proof}

\begin{proof}[Proof of Theorem~\ref{thm:W}, case~$\rho<1$]
We notice that
\begin{equation}\label{00--fg5996}
\mathcal{V_{\mathcal{K}}}\subseteq 
\mathcal{V_{\mathcal{A}}},\end{equation} since~$\mathcal{K}\subset \mathcal{A}$.
To prove that the inclusion is strict, we aim to find a point~$
(\bar{u}, \bar{v})\in \mathcal{V}_{\mathcal{A}} \setminus  \mathcal{V}_{\mathcal{K}}$.
Namely, we have to prove that there exists~$
(\bar{u}, \bar{v})\in \mathcal{V}_{\mathcal{A}}$ such that,
for all constant strategies~$a>0$, we have that~$(\bar{u}, \bar{v})\notin \mathcal{E}(a)$,
that is, by the characterization in Proposition~\ref{prop:char},
it must hold true that~$\bar{v} \geq \gamma_a(\bar{u})$ and $\bar{u}\leq u_{\mathcal{M}}^a$.

To do this, we define
\begin{equation}\label{def:f}
f(u):= \frac{u}{c} +  \frac{1-\rho}{1+\rho c}\quad {\mbox{ and }}\quad
m:= \min \left\{\frac{\rho c (c+1)}{1+\rho c}, 1 \right\}.
\end{equation}
By inspection, one can see that~$(u, f(u))\in[0,1]\times[0,1]$ if and only
if~$u\in [0,m]$.
We point out that, by~(ii) of Theorem~\ref{thm:Vbound},
for~$\rho <1$ and~$u\in [u_s^0, m]$,
a point~$({u}, {v})$ belongs to~$\mathcal{V}_{\mathcal{A}}$
if and only if~${v} < f({u})$. Here~$u_s^0$ is defined in~\eqref{u0v0}. We underline that the interval~$[u_s^0, m]$ is non empty since
\begin{equation}\label{2101}
u_s^0=\frac{\rho c}{1+\rho c}<\min \left\{\frac{\rho c (c+1)}{1+\rho c}, 1 \right\}= m.
\end{equation}	
Now we point out that
\begin{equation}\label{1633}
m \leq u_{\mathcal{M}}^a .
\end{equation}
Indeed, by \eqref{def:f} we already know that $m\leq 1$, thus if $u_{\mathcal{M}}^a=1$ the inequality in \eqref{1633} is true. On the other hand, when $u_{\mathcal{M}}^a<1$ we have that $(u_{\mathcal{M}}^a,1)\times(0,1)\subseteq{\mathcal{E}}(a)$.
This and~\eqref{00--fg5996} give that~$
(u_{\mathcal{M}}^a,1)\times(0,1)\subseteq\mathcal{V_{\mathcal{K}}}\subseteq 
\mathcal{V_{\mathcal{A}}}$.

Hence, in view of~\eqref{bound:rho<1}, we deduce that $$\frac{\rho c(c+1)}{1+\rho c}\le u_{\mathcal{M}}^a.$$ In particular, we find that~$m\le u_{\mathcal{M}}^a$,
and therefore~\eqref{1633} is true also in this case.

With this notation, we claim the existence of a
value~$\bar{v}\in(0,1]$
such that for all~$a>0$ we have 
$\gamma_a(m)\leq  \bar{v} < f(m)$.
That is, we prove now that  there
exists~$\theta>0$ such that
\begin{equation}\label{0000}
\gamma_a(m)+ \theta < f(m) \quad {\mbox{ for all }} a>0.
\end{equation}
The strategy is to study two cases separately, namely we prove~\eqref{0000}
for sufficiently small values of~$a$ and then for the other 
values of~$a$.

To prove~\eqref{0000} for small values of~$a$, we start by
looking at
the limit function~$\gamma_0$ defined in~\eqref{def:gamma0}.
One observes that
\begin{equation}\label{wqffwoe3u8ry4}
\gamma_0(u_s^0) = v_s^0= \frac{1}{1+\rho c} = \frac{\rho c}{c(1+\rho c)}+\frac{1-\rho}{1+\rho c}=
f(u_s^0).\end{equation}
Moreover, for all~$u\in(u_s^0, m]$,
we have that 
$$\gamma_0'(u) =\frac{v_s^0}{(u_s^0)^\rho} \,\rho u^{\rho-1}<
\frac{v_s^0}{(u_s^0)^\rho} \,\rho (u_s^0)^{\rho-1}=\frac{\rho v_s^0}{u_s^0}=
\frac{1}{c}= f'(u).$$
Hence, using the fundamental theorem of calculus on the continuous functions $\gamma_{0}(u)$ and $f(u)$, we get 
\begin{equation*}\begin{split}
\gamma_{0}(m)&= \gamma_{0}(u_s^0) +\int_{u_s^0}^{m} \gamma_{0}'(u)\,du \\&<  f(u_s^0) +\int_{u_s^0}^{m} f'(u)\,du \\&= f(m).\end{split}
\end{equation*}
Then, the quantity  $$\theta_1:= \frac{f(m)-\gamma_0(m)}{4}$$
is positive and
we have
\begin{equation}\label{1608}
\gamma_0(m)+ 2\theta_1 < f(m).
\end{equation}

Now, by the uniform convergence of~$\gamma_a$ to~$\gamma_0$
given by Lemma~\ref{lemma:conv_gamma},
we know that there exists~$\varepsilon\in \left(0,\frac1c\right)$ such that, if~$a\in(0,\varepsilon]$, 
\begin{equation}\label{KJ444S}\underset{u\in [u_s^0,m]}{\sup } |\gamma_a(u)-\gamma_0(u)| < {\theta_1}.\end{equation}
By this and~\eqref{1608}, we obtain that
\begin{equation}\label{0000BIS}
\gamma_a(m) + {\theta_1} < f(m) \quad {\mbox{ for all }} a\in(0,\varepsilon] .
\end{equation}
We remark that formula~\eqref{0000BIS} will give the desired claim
in~\eqref{0000} for conveniently small values of~$a$.

We are now left with considering the case~$a> \varepsilon$.
To this end, recalling~\eqref{usvs}, \eqref{ki87yh556g}, by the first statement in
Lemma~\ref{lemma:ord}, used here with~$a^*:=\varepsilon$,
we get 
\begin{equation}\label{1304}
\gamma_a(u) \leq \gamma_{\varepsilon}(u) \quad \text{for all} \ a > \varepsilon \ \text{and for all} \ u\in[u_s^{\varepsilon}, u_d^{\varepsilon}].
\end{equation}
Now we observe that
\begin{equation}\label{8j8j8fb8i903-1}
u^a_d\ge u_s^\varepsilon.
\end{equation}
Indeed, suppose not, namely
\begin{equation}\label{8j8j8fb8i903-2}
u^a_d< u_s^\varepsilon.
\end{equation}
Then, by the monotonicity of~$\gamma_a$, we have that~$\gamma_a(u^a_d)\le\gamma_a( u_s^\varepsilon)$.
This and~\eqref{1304} yield that~$\gamma_a(u^a_d)\le\gamma_\varepsilon( u_s^\varepsilon)$.

Hence,
the monotonicity of~$\gamma_\varepsilon$ gives that~$
\gamma_a(u^a_d)\le\gamma_\varepsilon( u_d^\varepsilon)$.
This and~\eqref{ki87yh556g} lead to~$1-u^a_d\le1-u_d^\varepsilon$,
that is~$u_d^\varepsilon\le u^a_d$. {F}rom this inequality,
using again~\eqref{8j8j8fb8i903-2}, we deduce that~$u_d^\varepsilon<
u_s^\varepsilon$. This is in contradiction with~\eqref{1306}
and thus the proof of~\eqref{8j8j8fb8i903-1}
is complete.

We also notice that
\begin{equation}\label{8j8j8fb8i903-11}
u^a_d\ge u_d^\varepsilon.
\end{equation}
Indeed, suppose not, say
\begin{equation}\label{8j8j8fb8i903-12}
u^a_d< u_d^\varepsilon.\end{equation}
Then, by~\eqref{8j8j8fb8i903-1}, we have that~$u^a_d\in[u_s^\varepsilon, u_d^\varepsilon]$ and therefore we can apply~\eqref{1304}
to say that~$
\gamma_a(u^a_d) \leq \gamma_{\varepsilon}(u^a_d)$.
Also, by the monotonicity of~$\gamma_{\varepsilon}$,
we have that~$\gamma_{\varepsilon}(u^a_d)\le \gamma_{\varepsilon}(u^\varepsilon_d)$.

With these items of information and~\eqref{ki87yh556g}, we find that
$$ 1-u^a_d=\gamma_a(u^a_d) \leq
\gamma_{\varepsilon}(u^\varepsilon_d)=1-u^\varepsilon_d,$$
and accordingly~$u^a_d\ge u^\varepsilon_d$.
This is in contradiction with~\eqref{8j8j8fb8i903-12}
and establishes~\eqref{8j8j8fb8i903-11}.

Moreover, by~\eqref{usvs} and~\eqref{u0v0},
we know that~$u_s^0>u_s^{a^*}$, for every~$a^*>0$.
Therefore, setting
$$\tilde u_d^{a^*}:=\min
\{u_d^{a^*},u_s^0\},$$ we have that~$\tilde u_d^{a^*}\in
[u_s^{a^*},u_d^{a^*}]$. Thus, we are in the position of
using the first statement 
in Lemma~\ref{lemma:ord} with~$a:=\varepsilon$
and deduce that
\begin{equation}\label{90i3883jj889203}
\gamma_\varepsilon (\tilde u_d^{a^*})\le\gamma_{a^*}(
\tilde u_d^{a^*})\qquad{\mbox{for all}}\quad a^*<\varepsilon.
\end{equation}
We also remark that
\begin{equation} \label{79ihkf843767676}u_d^{a^*}\to u_s^0\qquad{\mbox{as }}\;a^*\to0.\end{equation} Indeed, 
up to a subsequence we can assume that~$u_d^{a^*}\to \tilde u$ as~$a^*\to0$, for some~$\tilde u\in[0,1]$. Also, by~\eqref{ki87yh556g},
$$ \gamma_{a^*} (u_d^{a^*})=1-u_d^{a^*},$$
and then the uniform convergence of~$ \gamma_{a^*}$
in Lemma~\ref{lemma:conv_gamma} yields that
$$ \gamma_{0} (\tilde u)=1-\tilde u.$$
This and~\eqref{ki87yh556g} lead to~$\tilde u=u_d^0$.
Since
\begin{equation}\label{SQund0-dis0}
u_d^0=u_s^0\end{equation} in virtue of~\eqref{u0v0},
we thus conclude that~$\tilde u=u_s^0$
and the proof of~\eqref{79ihkf843767676}
is thereby complete.

As a consequence of~\eqref{79ihkf843767676},
we have that~$\tilde u_d^{a^*}\to
u_s^0$ as~$a^*\to0$. Hence,
using again the uniform convergence of~$ \gamma_{a^*}$
in Lemma~\ref{lemma:conv_gamma},
we obtain that~$\gamma_{a^*}(
\tilde u_d^{a^*})\to\gamma_{0}(u^0_s)$.

{F}rom this and~\eqref{90i3883jj889203}, we conclude that
\begin{equation}\label{SQund0-dis1}
\gamma_\varepsilon (u^0_s)\le\gamma_{0}(
u^0_s).\end{equation}

Now we claim that
\begin{equation}\label{9i9i9i78i9u-3934}
u_d^{\varepsilon} > u_s^0 .\end{equation}
Indeed, suppose, by contradiction,
that
\begin{equation}\label{9i9i9i78i9u-3934-0}u_d^{\varepsilon} \le u_s^0.\end{equation} Then, the monotonicity
of~$\gamma_{\varepsilon} $, together with~\eqref{SQund0-dis0}
and~\eqref{SQund0-dis1}, gives that
$$ 1-u_d^{\varepsilon} =
\gamma_{\varepsilon} (u_d^{\varepsilon}) \le
\gamma_{\varepsilon} ( u_s^0)=1-u_s^0.$$
{F}rom this and~\eqref{9i9i9i78i9u-3934-0} we deduce that~$u_d^{\varepsilon}=u_s^0$. In particular, we
have that~$u^0_s\in(u_s^\varepsilon,u^\varepsilon_{\mathcal{M}})$.
Accordingly, by~\eqref{gamma<r},
$$ 1- u^0_s=
1-u_d^{\varepsilon}=
\gamma_{\varepsilon}(u_d^{\varepsilon})=
\gamma_{\varepsilon}(u^0_s)< \frac{u^0_s}{{\rho}c} 
.$$
As a consequence,
$$ u^0_s>\frac{\rho c}{1+\rho c},$$
and this is in contradiction with~\eqref{u0v0}.
The proof of~\eqref{9i9i9i78i9u-3934} is thereby complete.

As a byproduct of~\eqref{SQund0-dis0}
and~\eqref{9i9i9i78i9u-3934}, we have that
\begin{equation}\label{89ujfvdjjgjh599fghjkl6}\begin{split}&
v^{\varepsilon}_d=
\gamma_{\varepsilon}(u^{\varepsilon}_d)=
1-u_d^{\varepsilon} <1- u_s^0=1- u_d^0\\&\qquad\qquad=\gamma_0(u^0_d)
=\gamma_0(u^0_s)=v^0_s.\end{split}\end{equation}
Similarly, by means of~\eqref{8j8j8fb8i903-11},
\begin{equation}\label{Qcbvolr9fjevcanf9d4}
v^a_d=
\gamma_a( u^a_d)= 1-u^a_d\le1- u_d^\varepsilon=
\gamma_\varepsilon(u_d^\varepsilon)=v_d^\varepsilon.
\end{equation}
In light of~\eqref{8j8j8fb8i903-11}, \eqref{9i9i9i78i9u-3934},
\eqref{89ujfvdjjgjh599fghjkl6} and~\eqref{Qcbvolr9fjevcanf9d4},
we can write that
\begin{equation}\label{1731}
1>u_d^a \geq u_d^{\varepsilon} > u_s^0 >0 \quad \text{and} \quad 1> v_s^0 > v_d^{\varepsilon} \geq v_d^a >0.
\end{equation}

\begin{figure} 
	\begin{subfigure}{1\textwidth}
		\centering
		\includegraphics[scale=1]{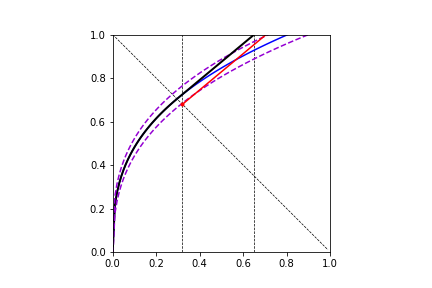}
	\end{subfigure}\\
	\begin{subfigure}{1\textwidth}
		\centering
		\includegraphics[scale=1]{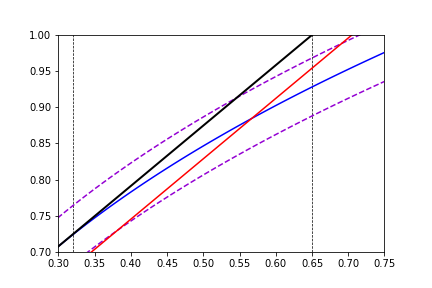}
	\end{subfigure}
	\caption{\em The figures illustrate the functions involved in the proof of Theorem \ref{thm:W} for the case $\rho < 1$.
		The two vertical lines correspond to the values $u_d^{\varepsilon}$ and $m$. The thick black line represents the boundary of $\mathcal{V}_{\mathcal{A}}$; the blue line is the graph of $\gamma_0(u)$; the dark violet lines delimit the area where $\gamma_{a}(u)$ for $a\leq\varepsilon$ might be; the red line is the upper limit of $\gamma_a(u)$ for $a>\varepsilon$. The image was realized using a simulation in Python for the values $\rho=0.35$ and $c=1.2$. }
	\label{fig:thm141}
\end{figure}

Now,
to complete the proof of~\eqref{0000}
when~$a>\varepsilon$,
we consider two cases depending on the order of $m$ and $u_d^{\varepsilon}$. If $u_d^{\varepsilon}\geq m$, by \eqref{1731}  we have that~$m<1$ and $f(m)=1$. Then, 
\begin{equation}\label{1802}
\gamma_a(m) \leq \gamma_a(u_d^{\varepsilon}) \le
\gamma_{\varepsilon}(u_d^{\varepsilon})= v_d^{\varepsilon} < 1 = f(m),
\end{equation}
thanks to the monotonicity of $\gamma_a$, \eqref{1304} and \eqref{1731}.
We define $$\theta_2:=\frac{1-v_d^{\varepsilon
}}{2},$$ which is positive  thanks to \eqref{1731}. {F}rom \eqref{1802}, we get that
\begin{equation}\label{1815}
\gamma_{a}(m) +\theta_2 \leq v_d^{\varepsilon} +\theta_2 <1= f(m).
\end{equation}
This formula proves the claim  in \eqref{0000} for $a>\varepsilon$ and $u_d^{\varepsilon}\geq m$.

If instead~$u_d^{\varepsilon}< m$, then we proceed as follows.
By \eqref{1731} we have
\begin{equation}\label{1722}
\gamma_a(u_d^{a}) =v_d^{a} \leq v_d^{\varepsilon} <  v_s^0 =  f(u_s^0).
\end{equation}
Now we set $$\theta_3:=\frac{f(u_d^{\varepsilon})-f(u_s^0)}{2}.$$
Using the definition of $f$ in \eqref{def:f}, we see that $$	\theta_3
= \frac{u_d^{\varepsilon}-u_s^0}{2c} ,$$ 
and accordingly~$\theta_3$ is positive, due to \eqref{1731}. 

{F}rom \eqref{1722} we have
\begin{equation}\label{1740}
\gamma_a(u_d^{a}) + \theta_3 < f(u_s^0) +\theta_3 < f(u_d^{\varepsilon}).
\end{equation}

Now we show that, on any trajectory~$(u(t),v(t))$ lying
on the graph of~$\gamma_{a}$, it holds that
\begin{equation} \label{1640}
\dot{v}(t) > \frac{\dot{u}(t)}{c} \quad \text{provided that} \ u(t)\in( u_d^a,u^a_{\mathcal{M}}) .
\end{equation}
To prove this, we first observe that  $u(t)>u_d^{a}> u_s^{a}$,
thanks to~\eqref{1306}.
Hence, we can exploit formula \eqref{gamma<r} of Lemma~\ref{lemma:vett_tg} and get that 
\begin{equation}\label{8yh78749in2fd9}
\gamma_a(u(t)) -
\frac{u(t)}{\rho c}<0.\end{equation}
Also, by the monotonicity of~$\gamma_a$
and~\eqref{ki87yh556g},
$$\gamma_a(u(t))\ge \gamma_a(u_d^a) = 1-u_d^a > 1-u(t).$$
{F}rom this and \eqref{8yh78749in2fd9} it follows that
\begin{equation*}
\left(\dot{v}(t) - \frac{\dot{u}(t)}{c} \right)= \rho \left(\gamma_a(u(t)) -
\frac{u(t)}{\rho c} \right) (1-u(t)-\gamma_a(u(t))) > 0 
\end{equation*}
provided that~$ u(t)
\in( u_d^a,u^a_{\mathcal{M}}) $, and this proves~\eqref{1640}.

In addition, for such a trajectory~$(u(t),v(t))$ we have that
\begin{equation*}\begin{split}
\dot{u}(t)&=u(t)\,(1-u(t)-\gamma_a(u(t))- ac) \\&
< u(t)\,(1-u(t)-\gamma_a(u_d^a))\\&=u(t)\,(1-u(t)-1+u_d^a)\\&<0,\end{split}
\end{equation*}
provided that~$ u(t)
\in( u_d^a,u^a_{\mathcal{M}}) $.

{F}rom this and \eqref{1640}, we get
\begin{equation*}
\gamma_a'(u(t))= \frac{\dot{v}(t)}{\dot{u}(t)} < \frac{1}{c} = f'(u(t))
,\end{equation*}
provided that~$ u(t)
\in( u_d^a,u^a_{\mathcal{M}}) $.

Consequently, taking as initial datum of the trajectory
an arbitrary point~$(u,\gamma_a(u))$ with~$u\in
( u_d^a,u^a_{\mathcal{M}}) $, we can write that, for all~$u\in( u_d^a,u^a_{\mathcal{M}})$,
\begin{equation*}
\gamma_a'(u)< f'(u).
\end{equation*}
As a result, integrating and using~\eqref{1304}, 
for all~$u\in( u_d^a,u^a_{\mathcal{M}})$, we have
\begin{equation*}\begin{split}
\gamma_a(u)&=
\gamma_{a}(u_d^a)+ \int_{u_d^a}^{u}\gamma_a'(u)\,du\\&
< \gamma_{a}(u_d^a)+ \int_{u_d^a}^{u}f'(u)\,du\\&=\gamma_{a}(u_d^a)+f(u)-f(u_d^a)
.\end{split}
\end{equation*}
Then, making use \eqref{1740}, for~$u\in( u_d^a,u^a_{\mathcal{M}})$,
\begin{equation}\label{8781jh98172omOS} 	\begin{split}\gamma_a(u) + \theta_3 &<  \gamma_{a}(u_d^a)+f(u)-f(u_d^a)  + \theta_3\\&\le
f(u)-f(u_d^a)+f(u_d^{\varepsilon}).\end{split}
\end{equation}
Also, recalling~\eqref{1731}
and the monotonicity of~$f$, we see that~$f(u_d^{\varepsilon})\le
f(u_d^{a})$. Combining this and~\eqref{8781jh98172omOS},
we deduce that
\begin{equation}\label{8781jh98172omOS-0987654-PRE}
\gamma_a(u) + \theta_3 <f(u)\qquad{\mbox{for all }}u\in( u_d^a,u^a_{\mathcal{M}})
.\end{equation}
We also observe that if~$u\in (u_d^\varepsilon, u_d^a]$,
then the monotonicity of~$\gamma_a$ yields that~$\gamma_a(u)\le
\gamma_a(u_d^a)$. It follows from this and~\eqref{1740}
that~$\gamma_a(u)+\theta_3 < f(u_d^{\varepsilon})$.
This and the monotonicity of~$f$ give that
$$ \gamma_a(u)+\theta_3 < f(u)
\qquad{\mbox{for all }}u\in(u_d^\varepsilon, u_d^a]
.$$
Comparing this with~\eqref{8781jh98172omOS-0987654-PRE},
we obtain
\begin{equation*}
\gamma_a(u) + \theta_3 <f(u)\qquad{\mbox{for all }}u\in( u_d^\varepsilon,u^a_{\mathcal{M}})
\end{equation*}
and
therefore
\begin{equation}\label{8781jh98172omOS-0987654}
\gamma_a(u) + \theta_3 \le f(u)\qquad{\mbox{for all }}u\in[u_d^\varepsilon,u^a_{\mathcal{M}}]
.\end{equation}

Now, in view of~\eqref{1633}, we have that~$m\in
[u_d^\varepsilon,u^a_{\mathcal{M}}]$.
Consequently, we can utilize~\eqref{8781jh98172omOS-0987654}
with~$u:=m$ and find that
\begin{equation}\label{a}
\gamma_a(m) + \theta_3 \le f(m)
\end{equation}
which gives \eqref{0000} in the case $a>\varepsilon$ and $u_d^{\varepsilon} \leq m$
(say, in this case with~$\theta\le\theta_3/2$).

That is, by~\eqref{0000BIS},~\eqref{1815}  and~\eqref{a}
we obtain that~\eqref{0000} holds true
for
\begin{equation*}
\theta :=\frac12\, \min \left\{ \theta_1, \ \theta_2 , \ \theta_3   \right\}.
\end{equation*}
If we choose $\bar{v}:= f(m)-\frac{\theta}{2}$ we have that
\begin{equation}\label{1643}
0 < \gamma_{a}(m) \leq \bar{v} < f(m) \leq 1.
\end{equation}
This completes the proof of Theorem~\ref{thm:W} when~$\rho<1$, in light of 
the characterizations of $\mathcal{E}(a)$ and $\mathcal{V}_{\mathcal{A}}$ from Proposition \ref{prop:char} and Theorem \ref{thm:Vbound}, respectively.
\end{proof}

Now we focus on the case~$\rho>1$. 

\begin{proof}[Proof of Theorem~\ref{thm:W}, case~$\rho>1$]
As before, the inclusion~$\mathcal{V_{\mathcal{K}}}\subseteq \mathcal{V_{\mathcal{A}}}$ is trivial since~$\mathcal{K}\subset \mathcal{A}$.
To prove that it is strict, we aim to find a point~$(\bar{u}, \bar{v})\in \mathcal{V}_{\mathcal{A}}$ such that~$(\bar{u}, \bar{v})\notin \mathcal{V}_{\mathcal{K}}$.
Thus, we have to prove that there exists~$
(\bar{u}, \bar{v})\in \mathcal{V}_{\mathcal{A}}$ such that,
for all constant strategies~$a>0$, we have that~$(\bar{u}, \bar{v})\notin \mathcal{E}(a)$.

To this end, using the characterizations given in Proposition~\ref{prop:char} and Theorem~\ref{thm:Vbound}, we claim that
\begin{equation}\begin{split}\label{1219}
&{\mbox{there exists a point~$(\bar{u}, \bar{v})\in[0,1]\times[0,1]$ satisfying}}\\ &{\mbox{$u_{\infty}\leq\bar{u}\leq u_{\mathcal{M}}^a$ and~$\gamma_a(\bar{u}) \leq \bar{v} < \zeta (\bar{u})$ for all~$a>0$.}}   
\end{split}\end{equation}
For this, we let
\begin{equation*}
m:=  \min\left\{1, \frac{c}{(c+1)^{\frac{\rho-1}\rho}} \right\} .
\end{equation*}
By \eqref{ZETADEF} one sees that
\begin{equation}\label{1657}
u_{\infty}<m.
\end{equation}
In addition, we point out that
\begin{equation}\label{1633-0987654gfhyf}
m \leq u_{\mathcal{M}}^a .
\end{equation}
Indeed, since $m\leq 1$, if $u_{\mathcal{M}}^a=1$ the desired inequality is obvious. If instead $u_{\mathcal{M}}^a<1$ we have that $(u_{\mathcal{M}}^a,1)\times(0,1)\subseteq{\mathcal{E}}(a)\subseteq\mathcal{V_{\mathcal{K}}}\subseteq 
\mathcal{V_{\mathcal{A}}}$.

Hence, by~\eqref{bound:rho>1},
it follows that $$\frac{c}{(c+1)^{\frac{\rho-1}\rho}}\le u_{\mathcal{M}}^a,$$
which leads to~\eqref{1633-0987654gfhyf}, as desired.

Now we claim that there exists~$\theta >0$ such that
\begin{equation}\label{0017}
\gamma_a(m) + \theta < \zeta(m) \quad {\mbox{for all }}\; a>0.
\end{equation}

We first show some preliminary facts for~$\gamma_a(u)$.
For all~$a>0$, we have that~$\mathcal{E}(a)\subseteq \mathcal{V}_{\mathcal{A}}$. Owing
to the characterization of~$\mathcal{E}(a)$ from Proposition~\ref{prop:char} and of~$\mathcal{V}_{\mathcal{A}}$ from Theorem~\ref{thm:Vbound} (which can be used here, thanks
to~\eqref{1657} and~\eqref{1633-0987654gfhyf}), we get that
\begin{equation}\label{1826}
\gamma_a(u)\le \frac{u}{c} \quad \text{for all } \ u\in(0, u_{\infty}]\ \text{ and }\ a>0.
\end{equation}
This is true in particular for~$u=u_{\infty}$.

We choose 
\begin{equation}\label{pthm3:def:delta}
\delta \in\left(0, \frac{\rho-1}{c}\right)\quad{\mbox{ and }}\quad M:= \max\left\{  \frac{1}{c}, \frac{\rho + \frac{1}{c}+\delta}{\delta c u_{\infty}}  \right\} ,
\end{equation}
and we prove~\eqref{0017} by treating separately the cases~$a>M$ and~$a\in(0, M]$.

We first consider the case~$a>M$. 
We let~$(u(t),v(t))$ be a trajectory for \eqref{model} lying on~$\gamma_a$ and we show that
\begin{equation}\label{1721}\begin{split}&
\dot{v}(t)- \left( \frac{1}{c}+\delta  \right)\dot{u} (t)> 0 \\& \text{provided that } \ u(t)>u_{\infty}\
{\mbox{ and }}\ a>M.\end{split}
\end{equation}
To check this, we observe that
\begin{align*}&
\dot{v}(t)- \left( \frac{1}{c}+\delta  \right)\dot{u} (t)\\&\qquad=  \left[ \rho \gamma_{a}(u(t))
-\left( \frac{1}{c}+\delta  \right) u(t)\right](1-u(t)-\gamma_{a}(u(t)))+ \delta acu(t) \\ 
& \qquad\geq  - \left\vert \rho + \frac{1}{c} + \delta \right\vert
+ \delta a c u_{\infty} \\&\qquad>0,	
\end{align*}
where the last inequality is true thanks to the hypothesis~$a>M$ and the definition of~$M$
in~\eqref{pthm3:def:delta}. This proves~\eqref{1721}.

Moreover, for~$a>M\geq \frac{1}{c}$ we have $\dot{u}<0$.
{F}rom this, \eqref{1721} and the invariance of~$\gamma_a$ for the flow, we get
\begin{equation}\label{1905}
\gamma_a'(u(t))=\frac{\dot v(t)}{\dot u(t)}< 
\frac{1}{c}+ \delta ,
\end{equation}	
provided that $u(t)>u_{\infty}$ and $a>M$.

For this reason and~\eqref{1826}, we get
\begin{equation}\label{pthm3:1643}\begin{split}
\gamma_a(u(t)) &= \gamma_a(u_{\infty})  + \int_{u_{\infty}}^{u(t)} \gamma_a'(\tau)\,d\tau\\&\le \frac{u_{\infty}}{c} + \left( \frac{1}{c}+\delta \right)(u(t)-u_{\infty}
),\end{split}
\end{equation}
provided that $u(t)>u_{\infty}$ and $a>M$.

Furthermore, thanks to the choice of~$\delta$ in~\eqref{pthm3:def:delta}, we have
\begin{equation*}
\zeta'(u)= \frac{\rho u^{\rho-1}}{c u_{\infty}^{\rho-1}}>\frac{\rho}{c} > \frac{1}{c}+\delta \quad \text{for all } \ u>u_{\infty}.
\end{equation*} 
Since also~$\zeta(u_{\infty})=\frac{u_{\infty}}{c}$, by \eqref{pthm3:1643} we deduce that
\begin{equation}\label{1621-PRE}\begin{split}
\gamma_a(u(t))&\leq    \frac{u_{\infty}}{c} + \left( \frac{1}{c}+\delta \right)(u(t)-u_{\infty})\\& < \zeta(u_{\infty}) + \int_{u_{\infty}}^{u(t)} \zeta'(\tau)\, d\tau\\& = \zeta(u(t)),\end{split}
\end{equation}
provided that $u(t)>u_{\infty}$ and $a>M$.

In particular, given any~$u>u_\infty$, we can take a trajectory starting at~$(u,\gamma_a(u))$
and deduce from~\eqref{1621-PRE} that
\begin{equation*}\begin{split}
\gamma_a(u)&\leq    \frac{u_{\infty}}{c} + \left( \frac{1}{c}+\delta \right)(u-u_{\infty})\\& < \zeta(u_{\infty}) + \int_{u_{\infty}}^{u} \zeta'(\tau)\, d\tau \\&= \zeta(u),\end{split}
\end{equation*}
whenever $a>M$. We stress that, in light of~\eqref{1657}, we can take~$u:=m$
in the above chain of inequalities, concluding that
\begin{equation*}
\gamma_a(m)\le	  \frac{u_{\infty}}{c} + \left( \frac{1}{c}+\delta \right)(m-u_{\infty}) < \zeta(m)
.\end{equation*}
We rewrite this in the form
\begin{equation}\label{1621}
\gamma_a(m)\le	\left( \frac{1}{c}+\delta  \right)m-{\delta}u_{\infty}< \zeta(m)
.\end{equation}
We define
\begin{equation}\label{pthm3:def:theta}
\theta_1:=\frac{1}{2}\left[ \zeta(m)-\left( \frac{1}{c}+\delta  \right)m +{\delta}u_{\infty} \right],
\end{equation}
that is positive thanks to the last inequality in~\eqref{1621}. Then by the first inequality in~\eqref{1621} we have
\begin{equation*}\begin{split}
\gamma_a(m)+\theta_1 &
\le \left( \frac{1}{c}+\delta  \right)m-{\delta}u_{\infty}+\theta_1\\&=
\frac12\left[ \left( \frac{1}{c}+\delta  \right)m-{{\delta}u_{\infty}}\right]
+\frac{ \zeta(m)}2.	\end{split}\end{equation*}
Hence, using again the last inequality in~\eqref{1621}, we obtain that
\begin{equation}\label{2001}
\gamma_a(m)+\theta_1<\zeta(m),\end{equation}
which gives the claim in \eqref{0017} for the case~$a>M$.

Now we treat the case~$a\in(0, M]$.
We claim that
\begin{equation}\label{1204}
u_d^M > u_{\infty}.
\end{equation}
Here, we are using the notation~$u_d^M$ to denote the point~$u_d^a$ when~$a:=M$. To prove~\eqref{1204} we argue as follows.
Since $M\geq \frac{1}{c}$, by Propositions \ref{lemma:M} and~\ref{M:p045} we have
\begin{equation}\label{1719}
\gamma_M'(0) = \frac{M}{\rho-1+Mc} < \frac{1}{c}.
\end{equation}
Moreover, since the graph of $\gamma_M(u)$ is a parametrization of a trajectory for \eqref{model} with $a=M$, we have that $ \dot{v}(t)= \gamma_M'(u(t)) \dot{u}(t)$.
Hence, at all points $(\bar{u}, \bar{v})$ with~$\bar{u}\in(0, u_{\infty})$ and $\bar{v}=\gamma_M(\bar{u})$ we have
\begin{equation}\label{1734}
\gamma_M ' (\bar{u}) = \frac{M \bar{u} - \rho \bar{v} (1-\bar{u}-\bar{v}) }{Mc \bar{u} - \bar{u}(1-\bar{u}-\bar{v})}.
\end{equation}
We stress that the denominator in the right hand side of~\eqref{1734}
is strictly positive, since~$M\geq \frac{1}{c}$ and~$\bar u>0$.

In addition, we have that  
\begin{equation}\label{1757}
\begin{split}
\frac{1}{c} - \frac{M \bar{u} - \rho \bar{v} (1-\bar{u}-\bar{v}) }{Mc \bar{u} - \bar{u}(1-\bar{u}-\bar{v})} = \frac{(\rho c \bar{v}-\bar{u})(1-\bar{u}-\bar{v}) }{Mc^2 \bar{u} - c\bar{u}(1-\bar{u}-\bar{v})}.
\end{split}
\end{equation}
Also,
$$
u_s^M=0<\bar{u}<u_\infty<m\le u^M_{\mathcal{M}},
$$
thanks to~\eqref{1657} and~\eqref{1633-0987654gfhyf}.
Hence, we can exploit formula~\eqref{gamma>r} in Lemma \ref{lemma:vett_tg}
with the strict inequality, thus obtaining that
\begin{equation}\label{8ujINtdensnumeok3965}
\rho c \bar{v}-\bar{u}=\rho c\gamma_M(\bar{u})-\bar{u} >0.\end{equation}
Moreover, by \eqref{1826},
$$ 1-\bar{u}-\bar{v} = 
1-\bar{u}-\gamma_M(\bar{u})\ge
1-\bar{u}-\frac{\bar{u}}{c} > 1-u_{\infty}- \frac{u_{\infty}}{c}=0. $$
Therefore, using the latter estimate and~\eqref{8ujINtdensnumeok3965}
into~\eqref{1757}, we get that
\begin{equation*}
\begin{split}
\frac{1}{c} - \frac{M \bar{u} - \rho \bar{v} (1-\bar{u}-\bar{v}) }{Mc \bar{u} - \bar{u}(1-\bar{u}-\bar{v})} > 0.
\end{split}
\end{equation*}
{F}rom this and \eqref{1734}, we have that
\begin{equation*}
\gamma_M'(u) < \frac{1}{c} \quad \text{for all} \ u\in(0,u_{\infty}).
\end{equation*}
This, together with \eqref{1719} and the fact that $\gamma_M(0)=0$, gives
\begin{equation*}
\gamma_M(u)
=\gamma_M(u)-\gamma_M(0)=\int_0^u \gamma_M'(\tau)\,d\tau
< \frac{u}{c} 
\end{equation*}
for all~$ u\in(0,u_{\infty}]$. 
This inequality yields that
\begin{equation}\label{1810}
\gamma_M(u_{\infty}) <  \frac{u_{\infty}}{c}= 1-u_{\infty}.
\end{equation}
Now, to complete the proof of~\eqref{1204} we argue by contradiction
and suppose that the claim in~\eqref{1204} is false, hence
\begin{equation}\label{1814}
u_d^M \leq u_{\infty}.
\end{equation}
Thus, by \eqref{1810}, the monotonicity of $\gamma_M(u)$ and the definition of $u_d^M$ given in \eqref{ki87yh556g}, we get
\begin{equation*}
1-u_d^M  = \gamma_M(u_d^M) \le \gamma_M(u_{\infty}) < 1-u_{\infty} 
\end{equation*}
which is in contraddiction with \eqref{1814}. Hence, \eqref{1204} holds true, as desired.

Also, by the second statement in Lemma \ref{lemma:ord}, used here with~$a^*:=M$,
\begin{equation}\label{1819}
\gamma_a(u) \leq \gamma_{M}(u) \quad \text{for all } \ u\in[0, u_d^M].
\end{equation}
We claim that
\begin{equation}\label{181967890p67890-4567890456789}
u_d^M\le u^a_d.
\end{equation}
Indeed, suppose, by contradiction, that
\begin{equation}\label{181967890p67890-4567890456789PRE}
u_d^M>u^a_d.
\end{equation}
Then, by the monotonicity of~$\gamma_a$ and~\eqref{1819}, used here with~$u:=u^M_d$,
we find that
$$ 1-u^a_d=\gamma_a(u^a_d)\le\gamma_a(u^M_d) \leq \gamma_{M}(u^M_d)=1-u^M_d.$$
This entails that~$u^a_d\ge u^M_d$, which is in contradiction with~\eqref{181967890p67890-4567890456789PRE},
and thus establishes~\eqref{181967890p67890-4567890456789}.

We note in addition that
\begin{equation}\label{09876543988878-00-181967890p67890-4567890456789PRE}
v_d^M =\gamma_M(u_d^M)=1-u_d^M
<1-u_{\infty},
\end{equation}
thanks to the definition of $(u_d^M, v_d^M)$ and~\eqref{1204}.

Similarly, by~\eqref{181967890p67890-4567890456789},
\begin{equation}\label{09876543988878-00-181967890p67890-4567890456789PRE098-08}
v_d^a =\gamma_a(u_d^a)=1-u_d^a\le
1-u_d^M=\gamma_M(u_d^M)
= v_d^M.
\end{equation}
Collecting the pieces of information in~\eqref{1204}, \eqref{181967890p67890-4567890456789},
\eqref{09876543988878-00-181967890p67890-4567890456789PRE}
and~\eqref{09876543988878-00-181967890p67890-4567890456789PRE098-08},
we thereby conclude that, for all~$a\in(0,M]$,
\begin{equation}\label{1834}\begin{split}&
0<	u_{\infty} < u_d^M \leq u_d^a<1 \\ \text{and} \qquad &0< v_d^a \leq v_d^M < 1-u_{\infty} =:v_\infty<1.\end{split}
\end{equation}

Now we consider two cases depending on the order of $m$ and $u_d^M$. If $u_d^M\geq m$,  by \eqref{1834} we have $m<1$ and $\zeta(m)=1$. Accordingly,
for $a\in(0,M]$, by \eqref{1834} and \eqref{1819} we have 
\begin{equation*}
\gamma_a(m) \le \gamma_{a}(u_d^M) \leq  \gamma_{M}(u_d^M)=v_d^M < 1=\zeta (m).
\end{equation*}
Hence, we can define $$\theta_2:=\frac{1-v_d^M}{2},$$
and observe that~$\theta_2$ is positive by \eqref{1834}, thus obtaining that
\begin{equation}\label{1916}
\gamma_{a}(m) + \theta_2 < \zeta(m).
\end{equation}
This is the desired claim in \eqref{0017} for $a\in(0,M]$ and $u^*\geq m$.

\begin{figure} 
	\centering
	\includegraphics[scale=1]{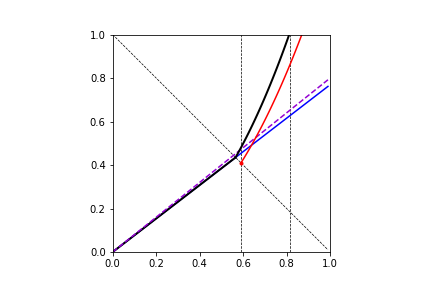}
	\caption{\em The figure illustrates the functions involved in the proof of Theorem \ref{thm:W} for the case $\rho > 1$.
		The two vertical lines correspond to the values $u_d^{M}$ and $m$. The thick black line represents the boundary of $\mathcal{V}_{\mathcal{A}}$; the blue line is the graph of the line $v=\frac{u}{c}$; the dark violet line is the upper bound for $\gamma_{a}(u)$ for $a>M$; the red line is  $\phi(u)$. The image was realized using a simulation in Python for the values $\rho=2.3$ and $c=1.3$. }
	\label{fig:thm142}
\end{figure}

If instead $u_d^M<m$, we consider the function
\begin{equation*}
\phi(u) := v_d^M\,\left(\frac{u}{u_d^M}\right)^{\rho} , \quad{\mbox{ for }} u\in[u_d^M, m]
\end{equation*}and we claim that 
\begin{equation}\label{1730}
\gamma_a(u)\leq\phi(u) \quad \text{for all } \ a\in(0,M]\ {\mbox{ and }} \ u\in[u_d^M,m].
\end{equation}
To prove this, we recall \eqref{1834} and the fact that~$\gamma_a$ is an increasing function
to see that
\begin{equation}\label{1946}\gamma_a(u_d^M)\le 
\gamma_a(u_d^a) =v_d^a \leq v_d^M = \phi(u_d^M) .
\end{equation} 

Now we remark that
$$ \gamma_M(u_d^M)+u_d^M=1>1-Mc=\gamma_M(u_s^M)+u_s^M,$$
and therefore~$u_d^M>u_s^M$. Notice also that~$u_d^M<m\le 
u^M_{\mathcal{M}}$, thanks to~\eqref{1633-0987654gfhyf}.
As a result, we find that~$\rho c \gamma_M(u_d^M) > u_d^M$ by inequality~\eqref{gamma>r}
in Lemma \ref{lemma:vett_tg}. Therefore, if~$u\ge u_d^M$ and~$v=\phi(u)$, then
\begin{equation*}\begin{split}
au \left(   1-\rho c \frac{v_d^M}{(u_d^M)^{\rho}} u^{\rho -1}   \right)&
= au \left(   1- \frac{\rho c\gamma_M(u_d^M) }{(u_d^M)^{\rho}} u^{\rho -1}   \right)\\&<
au \left(   1- \left(\frac{u}{u_d^M} \right)^{\rho -1}   \right)\\&
\leq 0\\&=\rho \left( v- \frac{v_d^M}{(u_d^M)^{\rho}} u^{\rho}  \right) (1-u-v).\end{split}
\end{equation*}
Using this and	\eqref{1804}, we deduce that, if~$a\in[0,M]$, $u\in[u_d^M, m]$ and~$v=\phi(u)$,
\begin{equation}\label{1858}\begin{split}&
\frac{au-\rho v (1-u-v)}{acu - u(1-u-v)}-\frac{v_d^M}{(u_d^M)^{\rho}} \rho u^{\rho-1}
\\=\;&
\frac{au-\rho v (1-u-v)
	-\big(acu - u(1-u-v) \big)\,\frac{v_d^M}{(u_d^M)^{\rho}} \rho u^{\rho-1}
}{acu - u(1-u-v)}\\=\;&
\frac{au\left(1-\rho c\frac{v_d^M}{(u_d^M)^{\rho}} u^{\rho-1}\right)-\rho (1-u-v)\left(v-
	\frac{v_d^M}{(u_d^M)^{\rho}}  u^{\rho}\right)
}{acu - u(1-u-v)}\\ <\;&0.
\end{split}
\end{equation}
Now we take~$ a\in(0,M]$, $u\in[u_d^M, m]$ and
suppose that~$v=\phi(u)=\gamma_a(u)$,
we consider an orbit~$(u(t),v(t))$ lying on~$\gamma_a$ with~$(u(0),v(0))=(u,v)$,
and we notice that, by~\eqref{1804} and~\eqref{1858},
\begin{equation}\label{1951}\begin{split}
\gamma_a'(u)&=	\gamma_a'(u(0))\\&=\frac{\dot v(0)}{\dot u(0)}\\&
=	\frac{au(0)-\rho v(0)\, (1-u(0)-v(0))}{acu (0)- u(0)(1-u(0)-v(0))}\\&
=	\frac{au-\rho v\, (1-u-v)}{acu - u(1-u-v)}\\&
<
\frac{v_d^M}{(u_d^M)^{\rho}} \rho u^{\rho-1}\\&
= \phi'(u).
\end{split}\end{equation}

To complete the proof of~\eqref{1730},
we define
$$ {\mathcal{H}}(u):=\gamma_a(u)-\phi(u)$$
and we claim that for every~$a\in(0,M]$ there exists~$\underline{u}\in[u_d^M, m]$
such that
\begin{equation}\label{98989898kjkjkjkdfbv}
{\mbox{${\mathcal{H}}(\underline u)<0$ and~${\mathcal{H}}(u)\le0$ for every~$u\in[u_d^M,\underline u]$.}}
\end{equation}
Indeed, by~\eqref{1946}, we know that~${\mathcal{H}}(u_d^M)\le0$.
Thus, if~${\mathcal{H}}(u_d^M)<0$ then we can choose~$
\underline u:=u_d^M$ and obtain~\eqref{98989898kjkjkjkdfbv}.
If instead~${\mathcal{H}}(u_d^M)=0$, we have that~$
\gamma_a(u_d^M)=\phi(u_d^M)$ and thus we can
exploit~\eqref{1951} and find that~${\mathcal{H}}'(u_d^M)<0$, from which
we obtain~\eqref{98989898kjkjkjkdfbv}.

Now we claim that, for every~$ a\in(0,M]$ and $u\in[u_d^M, m]$,
\begin{equation}\label{KL:0ksf3566}
{\mathcal{H}}(u)\le0.
\end{equation}
For this, given~$ a\in(0,M]$, we define
$$ 
{\mathcal{L}}:=\{ u_*\in [u_d^M, m]{\mbox{ s.t. }}{\mathcal{H}}(u)\le0 {\mbox{ for every }}u\in[u_d^M,u_*]\}$$and $$
\overline u:=\sup {\mathcal{L}}.$$
We remark that~$\underline u\in{\mathcal{L}}$, thanks to~\eqref{98989898kjkjkjkdfbv}
and therefore~$\overline u$ is well defined.
We have that
\begin{equation}\label{12jnSikjm239gfvhb37}
\overline u=m,
\end{equation}
otherwise we would have that~${\mathcal{H}}(\overline u)=0$
and thus~${\mathcal{H}}'(\overline u)<0$, thanks to~\eqref{1951},
which would contradict the maximality of~$\overline u$.
Now, the claim in~\eqref{KL:0ksf3566} plainly follows from~\eqref{12jnSikjm239gfvhb37}.

We notice that by the inequalities in~\eqref{1834} we have
\begin{equation}\label{2007}
\zeta(u)= \frac{v_{\infty}}{(u_{\infty})^{\rho}} u^{\rho}>
\frac{v_d^M}{(u_d^M)^{\rho}} u^{\rho}
= \phi(u).
\end{equation}
Then, we define
\begin{equation}\label{1958}
\theta_3:= \frac{\zeta(m)-\phi(m)}{2},
\end{equation}	
that is positive thanks to \eqref{2007}.
We get that
\begin{equation}\label{1912}
\phi(m)+\theta_3 < \zeta(m). 
\end{equation}
{F}rom this and~\eqref{1730}, we conclude that  
\begin{equation}\label{0115}
\gamma_a(m) + \theta_3 \leq \phi(m) + \theta_3 < \zeta(m) \quad \ \text{for} \ a\in(0,M].
\end{equation}

By~\eqref{2001},~\eqref{1916} and~\eqref{0115} we have that~\eqref{0017} is true for~$\theta = \min \{\theta_1, \ \theta_2, \ \theta_3  \}$. 

This also establishes the claim in \eqref{1219}, and the proof is completed.  
\end{proof}

\section{The role of Heaviside functions}

Now, we can complete the proof of Theorem~\ref{thm:H} by building on the previous work.

\begin{proof}[Proof of Theorem~\ref{thm:H}]
Since the class of Heaviside functions~$\mathcal{H}$ is contained in the class of piecewise continuous functions~$\mathcal{A}$, we have that
\begin{equation}
\mathcal{V}_{\mathcal{H}}\subseteq \mathcal{V}_{\mathcal{A}},
\end{equation}
hence we are left with proving the converse inclusion. We treat separately the cases $\rho=1$, $\rho<1$ and $\rho>0$.

If~$\rho=1$, the desired
claim follows from Theorem \ref{thm:W}, part (i).

If~$\rho<1$, we deduce from~\eqref{bound:rho<1} and~\eqref{8ujff994-p-1} that
\begin{equation}\label{iwfewuguew387627}
\mathcal{V}_{\mathcal{A}}= \mathcal{F}_0 \cup \mathcal{P},
\end{equation}
where~$\mathcal{P}$ has been defined 
in~\eqref{PPDEFA}
and~$\mathcal{F}_0$ in~\eqref{qwertyuiolkjhgf}.

Moreover, by~\eqref{8ujff994-p-3BIS}, we have that
\begin{equation}\label{iwfewuguew38762722} \mathcal{F}_0\subseteq \mathcal{V}_{\mathcal{K}}\subseteq \mathcal{V}_{\mathcal{H}}.	\end{equation}
Also, in Proposition~\ref{prop:construction} we construct a Heaviside winning strategy for every point in~$ \mathcal{P}$. Accordingly, it follows that~$ \mathcal{P} \subseteq \mathcal{V}_{\mathcal{H}}$. 
This,~\eqref{iwfewuguew387627} and~\eqref{iwfewuguew38762722} entail that~$ \mathcal{V}_{\mathcal{A}} \subseteq \mathcal{V}_{\mathcal{H}}$, which completes the proof of
Theorem~\ref{thm:H} when~$\rho<1$.

Hence, we now focus on the case~$\rho>1$.  By~\eqref{bound:rho>1}
and~\eqref{7hperpre923i5},
\begin{equation}\label{8877SA}
\mathcal{V}_{\mathcal{A}}= \mathcal{S}_{c} \cup \mathcal{Q},
\end{equation}
where~$\mathcal{S}_{c}$ was defined in~\eqref{def:S_c} and~$\mathcal{Q}$ in~\eqref{DEFQ}.

For every point~$(u_0, v_0)\in\mathcal{S}_{c}$ there exists~$\bar{a}$ that is a constant winning strategy for~$(u_0, v_0)$, thanks to Proposition~\ref{prop:bhva},
therefore~$\mathcal{S}_{c}\subseteq
\mathcal{V}_{\mathcal{H}}$.
Moreover, in Proposition~\ref{prop:construction} for every point~$(u_0, v_0)\in \mathcal{Q}$ we constructed a Heaviside winning strategy, whence~$ \mathcal{Q} \subseteq \mathcal{V}_{\mathcal{H}}$. In light of these observations
and~\eqref{8877SA}, we see that also in this case~$ \mathcal{V}_{\mathcal{A}} \subseteq \mathcal{V}_{\mathcal{H}}$ and the proof is complete.
\end{proof}

\section{Pointwise
constraints}

This subsection is dedicated to the analysis of~$\mathcal{V}_{\mathcal{A}}$ when we put some constraints on~$a(t)$. In particular, we consider $M\geq m \geq 0$ with $M>0$ and the set $\mathcal{A}_{m,M}$ of the functions $a(t)\in\mathcal{A}$ with $m\leq a(t)\leq M$ for all $t>0$.  We will prove Theorem \ref{thm:limit} via
a technical proposition giving informative bounds on $\mathcal{V}_{{m,M}}$.

For this, we denote by~$(u_s^m,v_s^m)$ the point~$(u_s,v_s)$ introduced in~\eqref{usvs}
when~$a(t)=m$ for all~$t>0$ (this when~$mc<1$, and we use the convention that~$(u_s^m,v_s^m)=(0,0)$
when~$mc\ge1$).
In this setting, we have the following result obtaining explicit
bounds on the favorable set~$\mathcal{V}_{{m,M}}$:

\begin{proposition} \label{prop:limit}
Let~$M\geq m\geq 0$ with $M>0$
and
\begin{equation}\label{RANGEEP}
\varepsilon\in\left(0,\,\min\left\{\frac{M(c+1)}{M+1},1\right\}\right).\end{equation}
Then
\begin{itemize}
	\item[(i)] If~$\rho<1$, we have
	\begin{equation}\label{8uj6tg574tygh}
	\begin{split}
	\mathcal{V}_{{m,M}} \subseteq \Big\{ (u,v)\in[0,1] \times [0,1] \;{\mbox{ s.t. }}\; v< f_{\varepsilon}(u)\Big\}
	\end{split}
	\end{equation}
	where~$f_{\varepsilon} : [0, u_{\mathcal{M}}]\to [0,1]$ is the continuous function given by
	\begin{equation*}
	f_{\varepsilon}(u)=\left\{ 
	\begin{array}{ll}
	\displaystyle		\frac{(u_s^m)^{1-\rho}u^{\rho}}{\rho c } & \text{if} \ u\in [0, u_s^m), \\		
	\displaystyle\frac{u}{\rho c}  & \text{if} \ u\in [u_s^m, u_s^0), \\
	\displaystyle\dfrac{u}{c}+\frac{1-\rho}{1+\rho c}  & \text{if} \ u\in [u_s^0, u_1), \\
	hu +p & \text{if} \ u\in [u_1, 1],
	\end{array}	\right.
	\end{equation*}
	with the convention that the first interval is empty if $m\geq \frac{1}{c}$, the second interval is empty if $m=0$,
	and $h$, $u_1$ and $p$ take the following values:
$$ h := \frac{1}{c}\left(1-\dfrac{\varepsilon^2(1-\rho)}{M (1+\rho c)(c+1-\varepsilon)^2 + \varepsilon (\rho c +\rho + \varepsilon-\varepsilon \rho)}\right),$$
$$u_1:=\frac{c(\rho c+\rho+\varepsilon-\varepsilon \rho)}{(1+\rho c)(c+1-\varepsilon)} $$
and $$p :=\frac{c+1-hc(\rho c+\rho+\varepsilon-\varepsilon \rho)}{(1+\rho c)(c+1-\varepsilon)}.
	$$
	
	\item[(ii)]  If~$\rho>1$,  we have
	\begin{equation*}
	\begin{split}
	\mathcal{V}_{{m,M}} \subseteq \Big\{ (u,v)\in[0,1] \times [0,1] \;{\mbox{ s.t. }}\; v< g_{\varepsilon}(u)\Big\},
	\end{split}
	\end{equation*}
	where~$g_{\varepsilon} : [0, u_{\mathcal{M}}] \to [0,1]$ is the continuous function given by
	\begin{equation*}
	g_{\varepsilon}(u)=
	\begin{dcases}
	k\,u & \text{if} \ u\in [0, u_2), \\		
	\displaystyle\dfrac{u}{c} + q  & \text{if} \ u\in [u_2, u_3), \\
	\displaystyle\dfrac{(1-u_3)u^{\rho}}{(u_3)^{\rho}}  & \text{if} \ u\in [u_3, 1]
	\end{dcases}	
	\end{equation*}
	for the following values:
	\begin{align*}&
	k:= \frac{(c+1-\varepsilon)M}{(\rho -1)\varepsilon c + (c+1-\varepsilon) Mc}, \\& q:=
	\frac{(kc-1)(1-\varepsilon)}{c(k-k\varepsilon+1)}, \\& u_2:=\frac{1-\varepsilon}{k-k\varepsilon+1}\\ {\mbox{and}}\qquad&
	u_3:=\frac{c+1-\varepsilon}{(c+1)(k-k\varepsilon +1)}.
	\end{align*}
\end{itemize}
\end{proposition} 

We observe that it might be that for some $u\in[0,1]$ we have $f_{\varepsilon}(u)>1$ or $g_{\varepsilon}(u)>1$. In this case, Proposition~\ref{prop:limit} would produce the trivial result that $\mathcal{V}_{{m,M}} \cap (\{u\}\times[0,1]) \subseteq \{  u\}\times [0,1]$.

On the other hand, a suitable choice of~$\varepsilon$ would lead to nontrivial
consequences entailing, in particular, the proof of
Theorem \ref{thm:limit}.

\begin{proof}[Proof of Proposition~\ref{prop:limit}]

We start by proving the claim in~(i). For this, we will show that
\begin{equation}\label{1642}
\mathcal{V}_{m, M} \subset \mathcal{D}:=\Big\{ (u,v)\in[0,1] \times [0,1] \;{\mbox{ s.t. }}\; v<  f_{\varepsilon}(u)\Big\}.
\end{equation}

We remark that once~\eqref{1642} is established, then the desired claim in~\eqref{8uj6tg574tygh}
plainly follows by taking the complement sets.

To prove~\eqref{1642} we first show that
\begin{equation}\label{rPjhnfvvcc}
0 \leq u_s^m < u_s^0 < u_1 < 1.\end{equation}
Notice, as a byproduct, that the above inequalities also give that~$f_{\varepsilon}$ is well defined.
To prove~\eqref{rPjhnfvvcc} we notice that, by \eqref{usvs}, \eqref{u0v0} and \eqref{usvs2},
\begin{equation*}
0\leq 	u_s^m=\max \left\{ 0, \frac{1-mc}{1+\rho c}\,\rho c\right\} < \frac{\rho c}{1+\rho c} =u_s^0.
\end{equation*}
Actually the first inequality is strict if $m<\frac{1}{c}$.

Next, one can check that, since~$\varepsilon>0$,
$$  u_s^0-u_1=\frac{\rho c}{1+\rho c}
-\frac{c(\rho c+\rho+\varepsilon-\varepsilon \rho)}{(1+\rho c)(c+1-\varepsilon)}=-\frac{c\varepsilon}{(1+\rho c)(c+1-\varepsilon)}
<0.$$
Furthermore, since~$\varepsilon<1$,
$$u_1-1=
\frac{c(\rho c+\rho+\varepsilon-\varepsilon \rho)}{(1+\rho c)(c+1-\varepsilon)}-1=\frac{(\varepsilon-1)(c+1)}{(1+\rho c)(c+1-\varepsilon)}<0.
$$
These observations prove~\eqref{rPjhnfvvcc}, as desired.

Now we point out that
\begin{equation}\label{97896705689045-0}
{\mbox{$f_{\varepsilon}$ is a continuous function. }}\end{equation}
Indeed,
\begin{equation}\label{97896705689045-1}
\frac{(u_s^m)^{1-\rho}}{\rho c} (u_s^m)^\rho = \frac{u_s^m}{\rho c}\qquad{\mbox{ and }}\qquad
\frac{u_s^0}{\rho c} = \frac{u_s^0}{ c}+ \frac{1-\rho}{1+\rho c}.\end{equation}
Furthermore, by the definitions of~$p$ and~$u_1$ we see that
\begin{equation}\label{767thisbc0-i6yjh00}
\begin{split} p\,&=
\frac{c+1}{(1+\rho c)(c+1-\varepsilon)}
-
\frac{hc(\rho c+\rho+\varepsilon-\varepsilon \rho)}{(1+\rho c)(c+1-\varepsilon)}\\&
=\frac{c+1}{(1+\rho c)(c+1-\varepsilon)}-hu_1.\end{split}\end{equation}
Moreover, from the definition of~$u_1$,
$$ \frac{u_1}{c}+\frac{1-\rho}{1+\rho c}  = \frac{c+1}{(1+\rho c)(c+1-\varepsilon)}.$$
Combining this and~\eqref{767thisbc0-i6yjh00}, we deduce that
\begin{equation}\label{indeh8idenf4596}
\frac{u_1}{c}+\frac{1-\rho}{1+\rho c}  = h u_1+p.
\end{equation}
This observation and~\eqref{97896705689045-1}
entail the desired claim in~\eqref{97896705689045-0}.

Next, we show that
\begin{equation}\label{1601}
f_{\varepsilon}(u)>0 \quad \text{for} \ u>0.
\end{equation}
To prove this, we note that for $u\in(0,u_s^m)$ the function~$f_\varepsilon(u)$ is an exponential times the positive constant $$\frac{(u_s^m)^{1-\rho}}{\rho c},$$ hence it is positive.

If $u\in[u_s^m, u_s^0)$ then $f_{\varepsilon}(u)$ is a linear function and it is positive since~$\rho c >0$. 

On $[u_s^0, u_1)$, $f_{\varepsilon}(u)$ coincide with a linear function with positive angular coefficient, hence we have
$$ f_{\varepsilon}(u) \geq \underset{u\in[u_s^0, u_1)}{\min} f_{\varepsilon}(u)= f_{\varepsilon}(u_s^0)= \frac{u_s^0}{\rho c} >0.  $$
By inspection one can check that $h>0$. 

Therefore, 
in the interval $[u_1,1]$ we have
$$ f_{\varepsilon}(u) \geq \underset{u\in[u_1, 1]}{\min} f_{\varepsilon}(u)= f_{\varepsilon}(u_1)\geq \frac{u_s^0}{\rho c} >0.  $$
This completes the proof of \eqref{1601}.

Let us notice that, as a consequence of~\eqref{1601}, 
\begin{equation}\label{2344}
\big( (0,1]\times \{0\} \big) \subset \mathcal{D}.
\end{equation}

Now we show that 
\begin{equation}\label{2352}
{ \mbox{for any strategy~$a\in\mathcal{A}_{m, M}$, no trajectory enters~$\mathcal{D}$.}  }
\end{equation}

To apply Lemma \eqref{lemma:entrance}, we compute the velocity of a trajectory in the outward normal direction at~$\partial \mathcal{D}\cap \{v=f_{\varepsilon}(u)\}$.

For every~$u\in[0, u_s^m)$ we have that this normal velocity is
\begin{equation}\begin{split}\label{1614}&
\dot{v}-	\frac{(u_s^m)^{1-\rho} \rho (u)^{\rho-1} \dot{u}}{\rho c }  \\
 =\;&\rho \left( v-  \frac{(u_s^m)^{1-\rho} \, u^{\rho} }{\rho c }  \right)  (1-u-v) -au\left(1- \frac{(u_s^m)^{1-\rho} }{u^{1-\rho}}  \right).
\end{split}\end{equation}
Notice that the term $$ 
v-  \frac{(u_s^m)^{1-\rho} \, u^{\rho} }{\rho c } $$
vanishes on~$\partial \mathcal{D}\cap \{v=f_{\varepsilon}(u)\}$
when~$u\in[0, u_s^m]$.

Also, for all~$u\in[0, u_s^m]$ we have 
\begin{equation*}
1- \frac{(u_s^m)^{1-\rho} }{u^{1-\rho}}\leq0,
\end{equation*} 
thus the left hand side in~\eqref{1614} is nonnegative with equality only in $u_s^m$. 
So \eqref{property} is satisfied for $u\in(0, u_s^m)$, and \eqref{property2} is satisfied for $u=0$ and $u=u_s^m$.

It remains to verify the hypothesis at points of~$\partial \mathcal{D}\cap\{v=f_{\varepsilon}(u)\}$
with~$u\in[ u_s^m,1]$. We first consider this type of points when~$[u_s^m,u_s^0]$.
At these points, we have that the
velocity in the
outward normal direction on~$\{ v=\frac{u}{\rho c} \}$ is
\begin{equation*}
\dot{v}- \frac{\dot{u}}{\rho c}= \left( \rho v - \frac{u}{\rho c} \right)(1-u-v) + au\left( \frac{1}{\rho} -1\right).
\end{equation*}  
Expressing~$u$ with respect to~$v$ on~$\partial \mathcal{D}\cap \{v=f_{\varepsilon}(u)\}$
with~$u\in[ u_s^m,u_s^0]$, we have
\begin{equation}\begin{split}\label{Moiuyted645JN}
\dot{v}- \frac{\dot{u}}{\rho c}&=v \left( \rho-1  \right)(1-\rho c v-v) + a \rho cv\,\frac{1-\rho}{\rho} \\
&= v(1-\rho)(\rho c v + v-1 +ac).\end{split}
\end{equation}  
We also remark that, for these points,
\begin{equation*}
v\geq v_s^m= \frac{1-mc}{1+\rho c}\ge\frac{1-ac}{1+\rho c}	,\end{equation*} 
thanks to~\eqref{usvs}. This gives that the quantity in~\eqref{Moiuyted645JN}
is strictly positive for $u\in(u_s^m,u_s^0)$  and null for $u=u_s^m$ and $u=u_s^0$, as a consequence,
we have proved \eqref{property} for $u\in(u_s^m,u_s^0)$  and \eqref{property2} for $u=u_s^m$ and $u=u_s^0$.

It remains to consider the case~$u\in[u_s^0,1]$; we now consider the interval~$u\in[u_s^0,u_1]$.
In this interval, the boundary $\partial \mathcal{D}\cap\{v=f_{\varepsilon}(u)\}$
lies on the line~$v=\frac{u}{c}+ \frac{1-\rho}{1+\rho c}$.
The velocity of a trajectory starting at a point~$(u,v)\in\partial \mathcal{D}\cap\{v=f_{\varepsilon}(u)\}$
in the outward normal direction with respect to~$\partial \mathcal{D}$ is given by
\begin{equation}\label{tqwfe3857uvcjycer4cubt}
\dot{v}- \frac{1}{c}\dot{u}= \left(\rho v - \frac{u}{c} \right)(1-u-v).
\end{equation}
We also observe that, in light of~\eqref{u0v0},
$$ u\geq u_s^0=\frac{\rho c}{1+\rho c}, $$
and therefore, for any~$u\in[ u_s^0,u_1]$ lying on the above line,
\begin{equation*}
1-u-v= 1-u-\frac{u}{c} - \frac{1-\rho}{1+\rho c} =(c+1)\left(\frac{\rho}{1+\rho c} -\frac{u}{c} \right) \leq 0
\end{equation*}
and
\begin{equation*}
\rho v - \frac{u}{c} = \frac{\rho u}{c}+ \frac{\rho (1-\rho)}{1+\rho c}- \frac{u}{c}  =(1-\rho)\left(  \frac{\rho}{1+\rho c} - \frac{u}{c} \right) \leq 0,
\end{equation*}
with equality in $u_s^0$.
Using these pieces of information in~\eqref{tqwfe3857uvcjycer4cubt}, we conclude that
\eqref{property}
for a point~$(u,v)\in\partial \mathcal{D}\cap \{v=f_{\varepsilon}(u) \}$ is satisfied for~$u\in (u_s^0,u_1]$ and \eqref{property2} is satisfied for $u=u_s^0$. 
We need to verify the case~$u\in [u_1,1]$.

We consider now the interval~$[u_1,1]$. 
In this interval, the component of the velocity of a trajectory at a point on the straight
line given by~$v=hu+p$ in the orthogonal outward pointing direction is 
\begin{equation}\label{8gqwfOJHNsmeoout43906}\begin{split}&
(\dot{u}, \dot{v}) \cdot \frac{(-h, 1)}{\sqrt{1+h^2}} = \frac{   (\rho v -h u)(1-u-v) -au(1-hc) }{\sqrt{1+h^2}}\\
&\qquad =
\frac{((1-\rho)hu-\rho p )(u+v-1) -au(1-hc)}{\sqrt{1+h^2}}
\end{split}\end{equation}
We observe that, if~$u\in[u_1,1]$,
\begin{equation}\label{jd723u9007432yhgvythgkliew}\begin{split}&(1-\rho)hu-\rho p \ge
(1-\rho)hu_1-\rho p\\&\qquad=hu_1-\rho (hu_1+p) \\ 
&\qquad	= h u_1-\rho \left(\frac{u_1}{c}+\frac{1-\rho}{1+\rho c} \right) \\&\qquad=
h u_1-\rho \left(
\frac{\rho c+\rho+\varepsilon-\varepsilon \rho}{(1+\rho c)(c+1-\varepsilon)}
+\frac{1-\rho}{1+\rho c} \right) \\&\qquad=
h u_1-\frac{\rho (c+1)}{(1+\rho c)(c+1-\varepsilon)},
\end{split}\end{equation}
thanks to~\eqref{indeh8idenf4596}.

We also remark that
\begin{equation*}\begin{split}
& hu_1\\ =\,&
\left(1-\dfrac{\varepsilon^2(1-\rho)}{M (1+\rho c)(c+1-\varepsilon)^2 + \varepsilon (\rho c +\rho + \varepsilon-\varepsilon \rho)}\right)\,\frac{ \rho c+\rho+\varepsilon-\varepsilon \rho}{(1+\rho c)(c+1-\varepsilon)} \\
=\,&
\frac{ \rho c+\rho+\varepsilon-\varepsilon \rho}{(1+\rho c)(c+1-\varepsilon)}\\&\qquad
-
\dfrac{\varepsilon^2(1-\rho)\big(\rho c+\rho+\varepsilon-\varepsilon \rho\big)}{\big(
	M (1+\rho c)(c+1-\varepsilon)^2 + \varepsilon (\rho c +\rho + \varepsilon-\varepsilon \rho)\big)(1+\rho c)(c+1-\varepsilon)}
.
\end{split}
\end{equation*}	
Accordingly,
\begin{eqnarray*}&&
	h u_1-\frac{\rho (c+1)}{(1+\rho c)(c+1-\varepsilon)}\\&=&
	\frac{ \varepsilon(1- \rho)}{(1+\rho c)(c+1-\varepsilon)}\\&&\qquad
	-
	\dfrac{\varepsilon^2(1-\rho)\big(\rho c+\rho+\varepsilon-\varepsilon \rho\big)}{\big(
		M (1+\rho c)(c+1-\varepsilon)^2 + \varepsilon (\rho c +\rho + \varepsilon-\varepsilon \rho)\big)(1+\rho c)(c+1-\varepsilon)}\\
	&=&
	\frac{ \varepsilon(1- \rho)}{(1+\rho c)(c+1-\varepsilon)}\Bigg(1
	-
	\dfrac{\varepsilon \big(\rho c+\rho+\varepsilon-\varepsilon \rho\big)}{
		M (1+\rho c)(c+1-\varepsilon)^2 + \varepsilon (\rho c +\rho + \varepsilon-\varepsilon \rho)}\Bigg)\\&=&
	\frac{ \varepsilon(1- \rho)}{(1+\rho c)(c+1-\varepsilon)}\cdot
	\dfrac{M (1+\rho c)(c+1-\varepsilon)^2}{
		M (1+\rho c)(c+1-\varepsilon)^2 + \varepsilon (\rho c +\rho + \varepsilon-\varepsilon \rho)}\\&=&
	\dfrac{\varepsilon M(1- \rho)(c+1-\varepsilon)}{
		M (1+\rho c)(c+1-\varepsilon)^2 + \varepsilon (\rho c +\rho + \varepsilon-\varepsilon \rho)}
	.\end{eqnarray*}
{F}rom this and~\eqref{jd723u9007432yhgvythgkliew}, we gather that
\begin{equation}\label{ILpredmnow55}\begin{split}&
(1-\rho)hu-\rho p\\ &\qquad\ge
\dfrac{\varepsilon M(1- \rho)(c+1-\varepsilon)}{
	M (1+\rho c)(c+1-\varepsilon)^2 + \varepsilon (\rho c +\rho + \varepsilon-\varepsilon \rho)}
.\end{split}\end{equation}

Furthermore, we point out that, when~$[u_1, 1)$ and~$v=hu+p$,
\begin{equation*}\begin{split}
u+v-1&\ge u_1+hu_1+p-1 \\&= 
u_1+\frac{u_1}c+\frac{1-\rho}{1+\rho c}-1\\&
=\frac{(c+1)(\rho c+\rho+\varepsilon-\varepsilon \rho)}{(1+\rho c)(c+1-\varepsilon)}
-\frac{\rho(c+1)}{1+\rho c}
\\&=\frac{\varepsilon(c+1)}{(1+\rho c)(c+1-\varepsilon)}\\& >\frac{\varepsilon}{c+1-\varepsilon},
\end{split}\end{equation*}
thanks to~\eqref{indeh8idenf4596}.

Combining this inequality and~\eqref{ILpredmnow55}, we deduce that
\begin{equation*}\begin{split}&
((1-\rho)hu-\rho p )(u+v-1)\\&\qquad >
\dfrac{\varepsilon^2 M(1- \rho)}{
	M (1+\rho c)(c+1-\varepsilon)^2 + \varepsilon (\rho c +\rho + \varepsilon-\varepsilon \rho)}.
\end{split}\end{equation*}	
Therefore, noticing that~$h<\frac{1}{c}$,
\begin{eqnarray*}&&
	((1-\rho)hu-\rho p )(u+v-1) -au(1-hc)\\&>&
	\dfrac{\varepsilon^2 M(1- \rho)}{
		M (1+\rho c)(c+1-\varepsilon)^2 + \varepsilon (\rho c +\rho + \varepsilon-\varepsilon \rho)}-Mu(1-hc)\\
	&=&
	\dfrac{\varepsilon^2 M(1- \rho)(1-u)}{
		M (1+\rho c)(c+1-\varepsilon)^2 + \varepsilon (\rho c +\rho + \varepsilon-\varepsilon \rho)},
\end{eqnarray*}
which is strictly positive.

Using this information in~\eqref{8gqwfOJHNsmeoout43906},
we can thereby conclude that \eqref{property} is verified for~$(u, f_{\varepsilon}(u))\in
\partial \mathcal{D}$
with~$u\in [u_1,1)$.

In this way, we have shown that either \eqref{property} or \eqref{property2} holds for $(u,f_{\varepsilon}(u))\in \partial \mathcal{D}$ for $u\in[0,1]$ and \eqref{property2} holds in a finite number of points, so we can apply Lemma \ref{lemma:entrance}. Hence, no trajectory
 can enter~${\mathcal{D}}$ and the proof
of~\eqref{2352}
is complete.

By~\eqref{2344} and~\eqref{2352}, no trajectory starting outside~$\mathcal{D}$ can arrive in~$(0,1]\times[0,1]$ when the bound~$m\leq a(t)\leq M$ holds, hence \eqref{1642} is true. 
Therefore the statement~(i) in Proposition \ref{prop:limit} is true.

\medskip

Now we establish the claim in~(ii). To this end,
we point out that claim (ii) is equivalent to
\begin{equation}\label{2359}
\mathcal{V}_{m, M} \subseteq \mathcal{G}:=\Big\{ (u,v)\in[0,1] \times [0,1] \;{\mbox{ s.t. }}\; v <  g_{\varepsilon}(u)\Big\}.
\end{equation}

First, we point out that
\begin{equation}\label{Ov54io0v9ik4gfvh}
{\mbox{$g_{\varepsilon}$ is a well defined continuous function. }}\end{equation}
Indeed, one can easily check for $\varepsilon\in(0,1)$ that 
\begin{equation}\label{1409}\begin{split}
0 &< u_2\\&
=\frac{1-\varepsilon}{k-k\varepsilon+1}-\frac{c+1-\varepsilon}{(c+1)(k-k\varepsilon +1)}+u_3\\&
=-\frac{c\varepsilon }{(c+1)(k-k\varepsilon +1)}+u_3\\&
<u_3\\&
<\frac{c+1}{(c+1)(k-k\varepsilon +1)}\\&<1.\end{split}
\end{equation}
Then, one checks that 
\begin{align*}
ku_2=\frac{u_2}{c}+q,
\end{align*}
hence $g_{\varepsilon}$ is continuous at the point $u_2$. In addition, one can check that $g_{\varepsilon}$ is continuous
at the point~$u_3$ by observing that
\begin{equation}\label{isceocessvcpoo}\begin{split}&
\frac{u_3}c+q-(1-u_3)=\frac{(c+1)u_3}{c}+q-1\\&\qquad=
\frac{c+1-\varepsilon}{c(k-k\varepsilon +1)}+\frac{(kc-1)(1-\varepsilon)}{c(k-k\varepsilon+1)}-1\\&\qquad=
\frac{c+1-\varepsilon+(kc-1)(1-\varepsilon)-c(k-k\varepsilon+1)}{c(k-k\varepsilon+1)}=0.
\end{split}\end{equation}
This completes the proof of \eqref{Ov54io0v9ik4gfvh}.

Now we show that
\begin{equation}\label{1411}
g_{\varepsilon}(u)>0 \quad \text{for every} \ u\in(0,1].
\end{equation}
We have that~$k>0$ for every~$\varepsilon<1$, and therefore~$g_{\varepsilon}(u)>0$ for all~$u\in(0, u_2)$.
Also, since $g_{\varepsilon}(u_2)=ku_2>0$ and $g_{\varepsilon}$ is linear in $(u_2, u_3)$, we have that $g_{\varepsilon}(u)>0$ for all~$u\in(u_2, u_3)$.

Moreover, in the interval~$\in[u_3,1]$ we have that~$g_{\varepsilon}$
is an exponential function multiplied by a  positive constant, thanks to~\eqref{1409}, hence it is positive. These considerations prove~\eqref{1411}.

As a consequence of~\eqref{1411}, we have that 
\begin{equation}\label{2360}
\big((0,1]\times \{0\}\big) \subset \mathcal{G}.
\end{equation}
Now we claim that 
\begin{equation}\label{0002}
{ \mbox{for any strategy~$a\in\mathcal{A}_{m,M}$, no trajectory enters~$\mathcal{G}$.}  }
\end{equation}

To prove~\eqref{0002}, we want to apply Lemma \ref{lemma:entrance}.
We do this
by showing that the outward pointing derivative of the trajectory is positive up to a finite number of points, where it is zero, according to the computation below.

At a point on the line~$v=k u$, the velocity of a trajectory in the direction that is orthogonal to~$\partial \mathcal{G}$ for~$u\in(0,u_2]$ and pointing outward is:
\begin{equation}\label{1741}
(\dot{u}, \dot{v})\cdot \frac{(-k, 1)}{\sqrt{1+k^2}} =\frac{(\rho v- ku)(1-u-v)-au(1-kc) }{\sqrt{1+k^2}}    .   
\end{equation}
We also note that
\begin{equation}\label{2018}
kc
= \frac{(c+1-\varepsilon)M}{(\rho -1)\varepsilon + (c+1-\varepsilon)M }
<1,\end{equation} and therefore, at a point on~$v=k u$ with $u\in(0, u_2]$,
\begin{equation*}\begin{split}
1-u-v &\;\geq  1-u_2-k u_2 \\&\;= 
1-	\frac{(1+k)(1-\varepsilon)}{k-k\varepsilon+1}
\\&\;=	\frac{\varepsilon}{k(1-\varepsilon)+1}\\&\;=\frac{\varepsilon c}{kc(1-\varepsilon)+c}\\&\;
> \frac{\varepsilon c}{1+c-\varepsilon}.\end{split}
\end{equation*}
This inequality entails that
\begin{equation*}\begin{split}
k&\;= \frac{(1+c-\varepsilon)M}{(\rho-1)\varepsilon c+(1+c-\varepsilon)Mc } \\&\;
=\frac{M}{\frac{(\rho-1)\varepsilon c}{1+c-\varepsilon}+Mc } 
\\&\;>  \frac{M}{(\rho-1)(1-u-v)+Mc}.\end{split}
\end{equation*}
Consequently,
\begin{equation*}
(\rho-1)(1-u-v)k > M (1-kc).
\end{equation*} 
{F}rom this and~\eqref{1741}, one deduces that, for all~$u\in(0, u_2]$, $a\leq M$, and $v=k u$,
\begin{equation*}\begin{split}
(\dot{u}, \dot{v})\cdot \frac{(-k, 1)}{\sqrt{1+k^2}} &\;=\frac{ku(\rho - 1)(1-u-v)-au(1-kc) }{\sqrt{1+k^2}} \\&\;  >
\frac{Mu (1-kc)-au(1-kc) }{\sqrt{1+k^2}}\\&\;>0,\end{split}
\end{equation*}
thus satisfying \eqref{property}. 
Moreover, since $(0,0)$ is an equilibrium, \eqref{property2} holds for $u=0$.

It remains to consider the portions of~$\partial\mathcal{G}\cap((0,1)\times(0,1))$ given by
\begin{equation}\label{9u:9idkj:0oekdjfjfj81763yhrf}
\left\{ u\in[ u_2,u_3){\mbox{ and }} v=\frac{ u}c+q\right\}\end{equation}
and by
\begin{equation}\label{9u:9idkj:0oekdjfjfj81763yhrf2}\left\{ u\in[u_3,1]{\mbox{ and }} v=\frac{(1-u_3)u^\rho}{(u_3)^\rho}\right\}.\end{equation}

Let us deal with the case in~\eqref{9u:9idkj:0oekdjfjfj81763yhrf}.
In this case,
the velocity of a trajectory in the direction orthogonal to~$\partial \mathcal{G}$ for~$u\in[u_2,u_3)$ and pointing outward is
\begin{equation}\label{2027}
(\dot{u}, \dot{v})\cdot \frac{(-1, c)}{\sqrt{1+c^2}}=\frac{(\rho c v -u)(1-u-v)}{\sqrt{1+c^2}}.
\end{equation}
Recalling~\eqref{RANGEEP}, we also observe that
\begin{equation}\label{1536}\begin{split}k-
\frac{1}{\rho c}
&\;=\frac1c\left(
\frac{(c+1-\varepsilon)M}{(\rho -1)\varepsilon + (c+1-\varepsilon) M}-\frac1\rho\right)\\&\;=
\frac{(\rho-1)\big((c+1-\varepsilon)M
	-\varepsilon\big)}{\rho c\big( (\rho -1)\varepsilon + (c+1-\varepsilon) M\big)}
\\&\;>0.\end{split}
\end{equation}
Thus, on the line given by~$v=\frac{u}{c}+q$ we have that
\begin{equation}\label{Cnodizeos80p4}\begin{split}
\rho c v -u&\;= (\rho-1)u +\rho c q
\\&\;\ge (\rho-1)u_2 +\rho c q\\&\;=
\frac{(\rho-1)(1-\varepsilon)}{k-k\varepsilon+1}
+\frac{\rho(kc-1)(1-\varepsilon)}{k-k\varepsilon+1}\\&\;
= (1-\varepsilon)\frac{(\rho-1)+\rho(kc-1)}{k-k\varepsilon+1}\\&\;=\frac{(1-\varepsilon)(\rho k c -1)}{k-k\varepsilon+1}\\&\;>0,\end{split}
\end{equation}
where~\eqref{1536} has been used in the latter inequality.

In addition, recalling~\eqref{isceocessvcpoo},
\begin{equation*}
1-u-v > 1-u_3 - \frac{u_3}{c} -q = 1-u_3-1+u_3=0.
\end{equation*}
{F}rom this and~\eqref{Cnodizeos80p4}, we gather that the velocity calculated in~\eqref{2027} is positive in~$[u_2, u_3)$ (satisfying \eqref{property}) and null in $u_3$ (satisfying \eqref{property2}). 

Next, we focus
on the portion of the boundary described in~\eqref{9u:9idkj:0oekdjfjfj81763yhrf2}
by considering~$u\in[u_3, 1]$.
That is, we now compute the component of the velocity at a point on~$\partial \mathcal{G}$ for ~$u\in[u_3, 1]$ in the direction that is orthogonal to~$\partial \mathcal{G}$ and pointing outward, that is
\begin{equation}\label{1803}
\begin{split}&
(\dot{u}, \dot{v})\cdot \frac{(-\rho \frac{1-u_3}{(u_3)^{\rho}}u^{\rho-1}, 1)}{\sqrt{1+\rho^2\frac{(1-u_3)^2}{(u_3)^{2\rho}}u^{2\rho-2}}} \\=\,& \frac{\rho(1-u-v)\left(v-   \frac{1-u_3}{(u_3)^{\rho}} u^{\rho} \right) - au\left( 1-\rho c \frac{1-u_3}{(u_3)^{\rho}} u^{\rho-1}   \right) }{\sqrt{1+\rho^2\frac{(1-u_3)^2}{(u_3)^{2\rho}}u^{2\rho-2}}}
\\=\,& \frac{ au\left( \rho c \frac{1-u_3}{(u_3)^{\rho}} u^{\rho-1} -1   \right) }{\sqrt{1+\rho^2\frac{(1-u_3)^2}{(u_3)^{2\rho}}u^{2\rho-2}}}\\ \ge\,&
\frac{ au\left( \rho c \frac{1-u_3}{u_3} -1   \right) }{\sqrt{1+\rho^2\frac{(1-u_3)^2}{(u_3)^{2\rho}}u^{2\rho-2}}}.
\end{split}
\end{equation}
Now we notice that
\begin{eqnarray*}
	\rho c (1-u_3)&=& \rho c \left(\frac{u_3}{c}+q \right) \\&=&
	\rho u_3+ \rho c q\\&=&\rho u_3+\frac{\rho (kc-1)(1-\varepsilon)(c+1) u_3}{c+1-\varepsilon},
\end{eqnarray*} 
thanks to~\eqref{isceocessvcpoo}.

As a result, using~\eqref{1536}, 
\begin{eqnarray*}
		\rho c (1-u_3) &>&\rho u_3+\frac{ (1-\rho)(1-\varepsilon)(c+1) u_3}{c+1-\varepsilon}
	\\&
	= & \frac{ u_3}{c+1-\varepsilon}
	\Big(\rho (c+1-\varepsilon)+(1-\rho)(1-\varepsilon)(c+1) \Big)\\& =&
	\frac{ u_3\big( (1-\varepsilon)(c+1)+\varepsilon \rho c\big)}{c+1-\varepsilon}	\\&
=& u_3+
	\frac{ \varepsilon c u_3( \rho-1)}{c+1-\varepsilon}\\&>&u_3.
\end{eqnarray*}
This gives that the quantity in  \eqref{1803} is positive, proving \eqref{property} for $u\in[u_3,1]$.

Hence, since \eqref{property2} is verified at $u=0, u_3$ and \eqref{property} is verified for $u\in [0,1]\setminus \{0, u_3 \}$, we can apply Lemma \ref{lemma:entrance} and get \eqref{0002}.

Since no trajectory can enter~$\mathcal{G}$ for any~$a$ with~$m\leq a \leq M$, we get that no point~$(u,v)\in \mathcal{G}^c$ is mapped into~$(0,1]\times\{0\}$ because of~\eqref{2360}, thus~\eqref{2359} is true and the proof is complete.
\end{proof}

We end this chapter with the proof of Theorem \ref{thm:limit}.

\begin{proof}[Proof of Theorem \ref{thm:limit}]
Since by definition $\mathcal{A}_{m,M}\subseteq \mathcal{A}$, we have that~$\mathcal{V}_{{m,M}}\subseteq \mathcal{V}_{\mathcal{A}}$. Hence, we are left with proving that the latter inclusion is strict.

We start with the case $\rho<1$. We choose
\begin{equation}\label{1934567890-dfghjk4567890-fd11}
\varepsilon\in\left(0,\,\min\left\{ \frac{ \rho c(c+1)}{1+\rho c}, \frac{M(c+1)}{M+1},1  \right\} \right). \end{equation}
We observe that this choice is compatible with
the assumption on~$\varepsilon$ in~\eqref{RANGEEP}. We note that
\begin{equation}\label{1911}
u_1 < \min\left\{ \frac{ \rho c(c+1)}{1+\rho c}, 1  \right\},
\end{equation}
thanks to~\eqref{1934567890-dfghjk4567890-fd11}.
Moreover, by \eqref{indeh8idenf4596} and the fact that~$h<\frac1c$, it holds that
\begin{equation}\label{1933}\begin{split}
h u + p 
&\;=h (u-u_1)+hu_1 + p\\
&\;=h (u-u_1)+\frac{u_1}{c}+\frac{1-\rho}{1+\rho c}\\
&\;<
\frac{u}{c}+ \frac{1-\rho }{1+\rho c} \end{split}
\end{equation}
for all~$u>u_1$.

Now we choose $$\bar{u}\in \left( u_1,  \min\left\{ \frac{ \rho c(c+1)}{1+\rho c}, 1  \right\} \right),$$ which is possible thanks to \eqref{1911}, and 
\begin{equation}\label{1925}
\bar{v}: = \frac{1}{2}\left(  h \bar{u} + p   \right) + \frac{1}{2}\left( \frac{\bar{u} }{c}+ \frac{1-\rho }{1+\rho c}   \right).
\end{equation}
By \eqref{1933} we get that
\begin{equation}\label{1937}
h \bar{u} + p <\frac12\left(h \bar{u} + p\right)
+ \frac{1}{2}\left( \frac{\bar{u} }{c}+ \frac{1-\rho }{1+\rho c}   \right)=
\bar{v} < \frac{\bar{u}}{c}+ \frac{1-\rho }{1+\rho c}.
\end{equation}
Using Proposition \ref{prop:limit} and \eqref{1937}, we deduce that  $(\bar{u}, \bar{v})\not\in \mathcal{V}_{{m,M}}$. By Theorem \ref{thm:Vbound} and \eqref{1937} we obtain instead that $(\bar{u}, \bar{v})\in \mathcal{V}_{\mathcal{A}}$. Hence, the set  
$\mathcal{V}_{{m,M}}$ is strictly included in~$\mathcal{V}_{\mathcal{A}}$
when~$\rho<1$.

Now we consider the case~$\rho>1$, using again the notation of Proposition \ref{prop:limit}. We recall that~$u_2>0$
and~$ u_{\infty}>0$, due to~\eqref{ZETADEF} and~\eqref{1409}, hence we can choose
$$ \bar{u} \in \left( 0, \min\{u_2, u_{\infty}\}   \right) .$$
We also define 
\begin{equation*}
\bar{v} := \frac12\left( \frac{1}{c} +k   \right) \bar{u}. 
\end{equation*}
By \eqref{2018}, we get that 
\begin{equation}\label{2031}
k \bar{u} < \frac{k\bar{u}}2+\frac{\bar{u}}{2c}
=	\bar{v} < \frac{ \bar{u}}{c}.
\end{equation}
Exploiting this and the characterization in Proposition \ref{prop:limit}, it holds that  $(\bar{u}, \bar{v})\not\in \mathcal{V}_{{m,M}}$.

On the other hand, by Theorem \eqref{thm:Vbound} and \eqref{2031} we have instead that $(\bar{u}, \bar{v})\in \mathcal{V}_{\mathcal{A}}$. As a consequence, the set~$\mathcal{V}_{{m,M}}$
is strictly contained in~$ \mathcal{V}_{\mathcal{A}}$ for $\rho>1$.
This concludes the proof of Theorem~\ref{thm:limit}.
\end{proof}

\section{Minimization of the war duration}\label{s:dimthmmin}

We now deal with the strategies leading to the quickest possible victory of the first population.

\begin{proof}[Proof of Theorem~\ref{thm:min}]
Our aim is to establish the existence 
of the strategy leading to the quickest possible victory
and to determine its range.
For this, we consider
the following minimization problem under constraints for~$x(t):=(u(t), v(t))$: 
\begin{equation}\label{sys:min}
\left\{ 
\begin{array}{ll}
\dot{x}(t)=f(x(t), a(t) ), \\ x(0)=(u_0, v_0), \\ x(T_s)\in (0,1]\times\{0\}, \\
\displaystyle\min_{a(t)\in [m, M]} \displaystyle\int_{0}^{T_s} 1 \,dt, 
\end{array}
\right.
\end{equation}
where 
\begin{equation*}
f(x, a) := \Big( u(1-u-v-ac), \ \rho v(1-u-v) -au    \Big).
\end{equation*}
Here~$T_s$ corresponds to the exit time introduced in~\eqref{def:T_s}, in dependence of the strategy~$a(\cdot)$.

Theorem 6.15 in~\cite{trelat2005controle} assures the existence of a minimizing solution~$(\tilde{a}, \tilde{x})$ with~$\tilde{a}(t)\in[m, M]$ for all~$t\in[0,T]$, and~$\tilde{x}(t)\in[0,1]\times[0,1]$ absolutely continuous, such that~$\tilde{x}(T)=(\tilde u(T), 0)$ with~$\tilde u(T)\in [0,1]$,
where~$T$ is the exit time for~$\tilde{a}$.

We now prove that
\begin{equation}\label{90o-045}
\tilde{u}(T)>0.\end{equation}
Indeed, if this were false, then~$(\tilde{u}(T), \tilde{v}(T))=(0,0)$. Let us call~$d(t):
= \tilde{u}^2(t)+ \tilde{v}^2(t)$. Then, we observe that
the function~$d(t)$ satisfies the following differential inequality:
\begin{equation}\label{1955}
- \dot{d}(t) \le C d , \qquad \text{for}  \quad C:=4+4\rho+2Mc+M.
\end{equation}
To check this, we compute that
\begin{align*}
- \dot{d}  &= 2\left(  -\tilde{u}^2(1-\tilde{u}-\tilde{v}-\tilde ac) - \tilde{v}^2 \rho(1-\tilde{u}-\tilde{v}) + \tilde{u}\tilde{v}\tilde a     \right) \\
& \le2\tilde{u}^2(2+Mc) +  4\rho\tilde{v}^2 + (\tilde{u}^2+\tilde{v}^2)M \\
& \le C (\tilde{u}^2+\tilde{v}^2)\\&= C d,
\end{align*}
which proves~\eqref{1955}.

{F}rom~\eqref{1955}, one has that
\begin{equation*}
0<(u_0^2+v_0^2 ) e^{-CT} \le d(T)=\tilde{u}^2(T)+ \tilde{v}^2(T)=\tilde{u}^2(T),
\end{equation*}
and this leads to~\eqref{90o-045}, as desired.	
We remark that, in this way, we have found a trajectory~$\tilde{a}$ which
leads to the victory of the first population in the shortest possible time.

Theorem 6.15 in~\cite{trelat2005controle} assures that~$\tilde{a}(t)\in L^{1}[0,T]$, so~$\tilde{a}(t)$ is measurable.
We have that the two vectorial functions~$F$ and~$G$, defined by
\begin{equation*}
F(u,v):= \left( 
\begin{array}{c}
u(1-u-v)\\
\rho v (1-u-v)
\end{array}
\right)\qquad{\mbox{and}}\qquad G(u,v):= \left( 
\begin{array}{c}
-cu\\
-u
\end{array}
\right),
\end{equation*}
and satisfying~$f(x(t), a(t))= F(x(t))+a(t)G(x(t))$, 
are analytic.

Moreover the set $\overline{\mathcal{V}}_{\mathcal{A}_{m,M}}$ is a subset of $\R^2$, therefore it can be seen as an analytic manifold with border which is also a compact set.

For all $x_0\in{\mathcal{V}}_{\mathcal{A}_{m,M}}$ and $t>0$ we have that the trajectory starting from $x_0$ satisfies $x(\tau)\in\overline{\mathcal{V}}_{\mathcal{A}_{m,M}}$ for all $\tau\in[0,t]$.

Then, by Theorem~3.1 in~\cite{sussmann1987C}, there exists a couple $(\tilde{a}, \tilde{x})$ analytic a part from a finite number of points, such that $(\tilde{a}, \tilde{x})$ solves \eqref{sys:min}.

\medskip

Now, to study the range of~$\tilde{a}$, we apply the Pontryagin Maximum Principle (see for example~\cite{trelat2005controle} or the original book \cite{pontryagin2018mathematical}). The Hamiltonian associated with system~\eqref{sys:min} is
\begin{equation*}
H(x,p, p_0, a ): =  p\cdot f(x,a)  + p_0
\end{equation*}
where~$p=(p_u, p_v)$ is the adjoint to~$x=(u,v)$ and~$p_0$ is the adjoint to the cost function identically
equal to~$1$.

The Pontryagin Maximum Principle tells us that, since~$\tilde{a}(t)$ and~$\tilde{x}(t)=(\tilde{u}(t), \tilde{v}(t))$ give the optimal solution, there exist a vectorial function~$\tilde p : [0, T] \to \R^2$ and a scalar~$\tilde p_0\in(-\infty, 0]$ such that 
\begin{equation}\label{HJA}
\left\{
\begin{array}{ll}
\dfrac{d\tilde{x}}{dt} (t)= \dfrac{\partial H}{\partial p} (\tilde{x}(t), \tilde p(t), \tilde p_0, \tilde{a}(t) ), & \text{for a.a.} \ t\in[0, T], \\
\\
\dfrac{d	\tilde{p}}{dt} (t)=- \dfrac{\partial H}{\partial x} (\tilde{x}(t), \tilde p(t), \tilde p_0, \tilde{a}(t) ), & \text{for a.a.} \ t\in[0, T],
\end{array}
\right.
\end{equation} 
and 
\begin{equation}\label{2349}\begin{split}
& H(\tilde{x}(t), \tilde p(t), \tilde p_0, \tilde{a}(t) ) = \underset{a(\cdot)\in[m,M]}{\max} H(\tilde{x}(t), \tilde p(t), \tilde p_0, a ) \\& \text{for a.a.} \ t\in[0, T].\end{split}
\end{equation}
Moreover, since the final time is free, we have
\begin{equation}\label{1244}
H(\tilde{x}(T), \tilde p(T),\tilde p_0, \tilde{a}(T) ) =0. 
\end{equation}
Also, since~$H(x,p,p_0,a)$ does not depend on~$t$, we get
\begin{equation}\label{2343}
H(\tilde{x}(t), \tilde p(t), \tilde p_0, \tilde{a}(t) ) ={\mbox{constant}}=0, \quad \text{for a.a.} \ t\in[0, T], 
\end{equation}
where the value of the constant is given by~\eqref{1244}.
By substituting the values of~$f(x,a)$ in~$H(x,p,p_0,a)$ and using~\eqref{2343}, we get, for a.a.~$
t\in[0, T]$,
\begin{equation*}
\tilde p_u \tilde{u}(1-\tilde{u}- \tilde{v}-\tilde{a}c)+ \tilde p_v\rho  \tilde{v}(1-\tilde{u}- \tilde{v}) -\tilde p_v \tilde{a} \tilde{u} + \tilde p_0 =0,
\end{equation*}
where~$\tilde p=(\tilde p_u,\tilde p_v)$.

Also, by~\eqref{2349} we get that
\begin{equation}\label{0oskdfee}\begin{split}&
\underset{a\in[m,M]}{\max} H(\tilde{x}(t), \tilde p(t), \tilde p_0, a )\\=& \underset{a\in[m,M]}{\max} \Big[-a\tilde{u}(c\tilde p_u + \tilde p_v ) + \tilde p_u \tilde{u}(1-\tilde{u}- \tilde{v})+ \tilde p_v\rho  \tilde{v}(1-\tilde{u}- \tilde{v})+\tilde p_0\Big].\end{split}
\end{equation}Thus, to maximize the term in the square brackets we must choose appropriately the value of~$\tilde{a}$ depending on the sign of~$\varphi(t):=c\tilde p_u(t)+\tilde p_v(t)$, that is
we choose
\begin{equation}\label{1631}
\tilde{a}(t):=
\left\{
\begin{array}{ll}
m &{\mbox{ if }} \varphi(t)>0, \\
M &{\mbox{ if }} \varphi(t)<0.
\end{array}
\right.
\end{equation}
When~$\varphi(t)=0$, we are for the moment free
to choose~$	\tilde{a}(t):=a_s(t)$ for every~$a_s(\cdot)$
with range in~$[m,M]$,
without affecting the maximization problem in~\eqref{0oskdfee}.

Our next goal is to determine that~$a_s(t)$ has the expression stated in~\eqref{KSM94rt3rjjjdfe} for
a.a.~$t\in[0,T]\cap \{\varphi=0\}$.

To this end, we claim that
\begin{equation}\label{9id0-3rgjj}
{\mbox{$\dot\varphi(t)=0$ a.e.~$t\in[0,T]\cap \{\varphi=0\}$.}}
\end{equation}
Indeed, by~\eqref{HJA}, we know that~$\tilde p$ is Lipschitz continuous in~$[0,T]$,
hence almost everywhere differentiable, and thus the same holds for~$\varphi$.

Therefore, up to a set of null measure, given~$t\in[0,T]\cap \{\varphi=0\}$,
we can suppose that~$t$ is not an isolated point in such a set,
and that~$\varphi$ is differentiable at~$t$.

That is, there exists an infinitesimal sequence~$h_j$
for which~$\varphi(t+h_j)=0$ and
$$ \dot\varphi(t)=\lim_{j\to+\infty}\frac{\varphi(t+h_j)-\varphi(t)}{h_j}
=\lim_{j\to+\infty}\frac{0-0}{h_j}=0,$$
and this establishes~\eqref{9id0-3rgjj}.

Consequently, in light of~\eqref{9id0-3rgjj}, a.a.~$t\in[0,T]\cap \{\varphi=0\}$
satisfies	\begin{equation*}\begin{split}
	0& \;=\dot\varphi(t)\\&\;= c\frac{d\tilde p_u}{dt}(t)+ \frac{d\tilde p_v}{dt}(t) \\&\; = c\big[ -\tilde p_u(t)(1-2\tilde{u}(t)-\tilde{v}(t)-ca_s(t))+\tilde p_v(t) (\rho \tilde{v}(t)+a_s(t))  \big]\\&\qquad\qquad
+ \tilde p_u(t)\tilde u(t)-\tilde p_v(t) \rho(1-\tilde{u}(t)-2\tilde{v}(t)).\end{split}
\end{equation*}
Now, since~$\varphi(t)=0$, we have that~$ \tilde p_v(t)=- c\tilde p_u(t)$; inserting this information in the last equation, we get
\begin{equation}\label{0004}\begin{split}&
0= -\tilde p_u c (1-2\tilde u-\tilde v-a_s c) -\tilde p_u \rho c^2 \tilde v \\&\qquad- \tilde p_u a_s c^2 + \tilde p_u \tilde u+ \tilde p_u \rho c (1-\tilde u-2\tilde v).\end{split}
\end{equation}
Notice that if~$\tilde p_u=0$, then~$\tilde p_v=-c \tilde p_u=0$; moreover, by~\eqref{2343}, one gets~$\tilde p_0=0$. But by the Pontryagin Maximum Principle one cannot have~$(\tilde p_u, \tilde p_v, \tilde p_0)=(0,0,0)$, therefore one obtains~$\tilde p_u\neq 0$ in~$\{ \varphi=0 \}$.

Hence, dividing~\eqref{0004} by~$\tilde p_u$ and rearranging the terms, one gets
\begin{equation}\label{0007}
\tilde{u}(2c+1-\rho c) + c\tilde{v}(1-\rho c-2\rho)+c(\rho-1)=0.
\end{equation}
Differentiating the expression in~\eqref{0007} with respect to time, we get
\begin{equation*}
\tilde{u} (2c+1-\rho c) (1-\tilde{u}-\tilde{v}-ac) + c(1-\rho c-2\rho) [  \rho \tilde{v} (1-\tilde{u}-\tilde{v}) -a\tilde{u} ]=0,
\end{equation*}
that yields
\begin{equation*}
a_s = \frac{(1-\tilde{u}-\tilde{v}) (	\tilde{u} (2c+1-\rho c)+\rho c) }{2c\tilde{u}(c+1)},
\end{equation*}
which is the desired expression. 
By a slight abuse of notation, we define the function~$a_s(t)= a_s(\tilde{u}(t), \tilde{v}(t))$ for~$t\in[0,T]$. 
Notice that since~$\tilde{u}(t)>0$ for~$t\in[0,T]$,~$a_s(t)$ is continuous for~$t\in[0,T]$.
\end{proof}


\chapter*{Bibliographical notes}

\begin{center}
\begin{minipage}{25em}
\noindent
{\sl We give here some reference to the classical books, where the reader can find more detailed explanations of the theories we exploit, and some guidance in the literature of Lotka-Volterra systems.}
\end{minipage}\end{center}\bigskip\bigskip\bigskip\bigskip\bigskip\bigskip

We start with some classical references to dynamical systems. In this book, we often refer to the book of Perko \cite{dynsyst}, which covers those topics necessary for a clear understanding of the qualitative theory of ordinary differential equations and the concept of a dynamical system. It is written for advanced undergraduates and for beginning graduate students, and it is useful as a primer for young people interested in research in dynamical systems. In particular, it focuses on describing the qualitative behavior of the solution set of a given system of differential equations, including the invariant sets and limiting behavior of the dynamical system. This includes the Stable Manifold Theorem and the Pointcar\'e-Bendixon Theorem, which we often exploit.

Another excellent reference is the book \cite{smale1974differential} by Hirsh and Smale. 
The advantage of this book is the vast collection of examples, including the predatory-prey and the SIR  model, which are presented to visualize the theorems and their hypothesis.
This text is more accessible to non-mathematicians and undergraduate students, being at the same time very complete and satisfactory as a background for our text.

The book \cite{MR1056699} by Wiggings takes a step forward. In the first chapters, it presents the geometrical analysis of dynamical systems already covered in \cite{smale1974differential}, but in a more general and abstract form. Then, it continues with advanced material, working with Hamiltonians and their relations with the other tools, and exposing questions that are relevant in research.
The text is intended for an audience with ``mathematical maturity'' and interested in research on dynamical systems. The author says that the material is abundant and ambitious even for a three terms program.

Pontryaigin's Maximum Principle was first presented in the book \cite{pontryagin2018mathematical} by Pontryagin, Boltayanskii, Gamkrelidze, and Mishchenko, in the form of multiple theorems based on the same principle. This text is however hard to read, and nowadays more accessible texts are available.

For control theory and optimization, the textbook \cite{trelat2005controle} by Tr\'elat is a very good starting point. Targeted at undergraduate students, it provides a clear and concise introduction to problems and tools in control theory, including controllability and Pontyagin Maximum Principle, and it is completed by examples and numerical methods. The proposed material is enough backgorund for the comprehension of our book.
We can also recommend \cite{macki2012introduction} by Macki and Strauss as an advanced undergraduate text. Other reliable options are the lecture notes \cite{fernandez2003control} by F\'ernandez Cara and Zuazua or \cite{micu2004introduction} by Micu and Zuazua.

More advanced books in control theory focus on geometric control, a branch that has been developed since the 60s using the tools of differential geometry and Lie theory to prescribe optimal strategies, called optimal feedback, for a dynamic optimization problem.
To the scope of our book, a great reference is the book \cite{boscain2003optimal} by Boscain and Piccoli. In fact, their work focuses on time minimization problems in 2-D systems, for which the authors provide a complete theory, based on the research papers of the early 2000s. 

For a comprehensive treatment of geometric control applied to both linear and nonlinear systems, including stabilization, consider the references \cite{coron2007control} by Coron and \cite{agrachev2013control} Agrachev and Sachkov. 

Control theory also originated the branch of game theory, that analyses systems in which one or more parties take strategic decisions in order to achieve an optimal situation for their interests. Started with a finite number of players taking decisions from a discrete set of choices, this fascinating field has evolved in various way. When the situation evolves according to a dynamical system, where some of the parameters represents the strategic decisions of the players, we talk about differential games. 
An overview of the connection between control theory and differential games can be found for instance in the books \cite{BCD,dockner2000differential}, 
see also \cite{evans1984differential} where the subject is approached within the theory of viscosity solutions.
We also refer to the book \cite{isaacs1999differential}, which focuses in particular on some differential games modeling war conflicts.

The classic Lotka-Volterra equations for modeling predator-prey systems were first introduced independently in \cite{1eddd3e8-a442-3aa5-91bf-ea3ab4ee18ac} and \cite{volterra1926variatzioni} in the 20s. The work of Volterra \cite{volterra1926variatzioni} focused on three behaviors - the predation, the competition for the same resources, and the mutualism.
In the same decades, the emergence of numerous models inspired by ecology gave birth to the branch of mathematics that we call mathematical biology. In the 80s, Murray wrote the milestone books \cite{murray1, murray2} to bring together the fundamentals of this new subject. 
We also cite the text \cite{cantrell2004spatial} by Cantarell and Cosner, which collects the rigorous mathematics formalizing many of the insights that have been had throughout the twentieth century in mathematical biology. However, this book treats mainly single-species equations.
In fact, ecological modelisation has grown so much that it is impossible to give a complete picture.

A series of six papers by Hirsh and Morris answered many open questions on the dynamics of competitive and collaborative systems. The ones regarding two-dimensional models are \cite{hirsch1982systems, hirsch1985systems, hirsch1988systems}.

On the other hand, the Lotka-Volterra model was extended in different ways. A very successful additive feature is the spatial diffusion of the individuals of the two species. As examples of studies on the subject, a long list of works on traveling waves for Lotka-Volterra competiton systems are available, see \cite{gardner1982existence} by Gardner, \cite{tang1980propagating} by Tang and Fife, \cite{kan1997fisher} by Kan-On, \cite{guo2013sign} by Guo and Lin just to cite a few.
Spatial segregation caused by competition was studied by Tavares, Terracini, Verzini and their collaborators in \cite{conti2005asymptotic, tavares2012regularity} and the following works.

To conclude, from the viewpoint of our book, another relevant use of Lotka-Volterra competitive systems consists of applications to diffusion of new technologies that substitute old ones. The Bass model introduced in \cite{bass1969new} became one of the most influential in Management Science of the last century, see also  \cite{krishnan2000impact} and the subsequent research papers.

	
\bibliography{biblio}

\newpage$\,$\vfill

{\em Elisa Affili}, Laboratoire de Math\'{e}matiques Rapha\"{e}l Salem
UFR des Sciences et Techniques
Avenue de l'Universit\'{e}, BP.12
76801 Saint-\'{E}tienne-du-Rouvray, France \\
{\tt elisa.affili@univ-rouen.fr}\medskip\medskip

{\em Serena Dipierro}, Department of Mathematics and Statistics,
University of Western Australia,
35 Stirling Highway,
Crawley WA 6009, Australia. \\ {\tt serena.dipierro@uwa.edu.au}\medskip\medskip

{\em Luca Rossi}, Dipartimento di Matematica, Sapienza Universit\`a di Roma, Piazzale Aldo Moro 5,
00185 Roma, Italy. \\ {\tt l.rossi@uniroma1.it}\medskip\medskip

{\em Enrico Valdinoci}, Department of Mathematics and Statistics,
University of Western Australia,
35 Stirling Highway,
Crawley WA 6009, Australia. \\ {\tt enrico.valdinoci@uwa.edu.au}

\end{document}